\documentclass[11pt]{amsart}
\usepackage{amsmath}
\usepackage{amssymb}
\usepackage{amsthm}
\usepackage{amscd}
\usepackage{amsxtra}
\usepackage{verbatim}
\usepackage{xcolor}
\usepackage{color}
\usepackage{enumerate}
\usepackage{mathrsfs}
\usepackage{stmaryrd}

\usepackage{url,hyperref}
\allowdisplaybreaks

\numberwithin{equation}{section}

\definecolor{dblue}{rgb}{0,0,0.45}
\definecolor{red}{rgb}{0.7,0,0}

 \RequirePackage{geometry}
\newtheorem{theorem}{Theorem}[section]
\newtheorem{lemma}[theorem]{Lemma}
\newtheorem*{lemma*}{Lemma}
\newtheorem{corollary}[theorem]{Corollary}
\newtheorem{proposition}[theorem]{Proposition}

\theoremstyle{definition}

\newtheorem{remark}[theorem]{Remark}
\newtheorem{definition}[theorem]{Definition}
\newtheorem{example}[theorem]{Example}

\theoremstyle{remark}

\newcommand{\cD}{{\mathcal D}}

\newcommand{\cI}{{\mathcal I}}
\newcommand{\cL}{{\mathcal L}}

\newcommand{\ve}{\varepsilon}

\newcommand{\cE}{\mathcal{E}}
\newcommand{\ImUnit}{\mathsf{i}}

\newcommand{\SmoothFunc}[3]{C^{#1}({#2},{#3})}
\newcommand{\HolBesSp}[1]{\mathcal{C}^{#1}}

\newcommand{\drivers}[2]{{\mathcal{X}_{#1}^{#2}}}
\newcommand{\sols}[3]{{\mathcal{D}_{#1}^{#2,#3}}}

\newcommand{\FourierTrans}{\mathcal{F}}

\newcommand{\FourierCoeff}[1]{\hat{#1}}
\newcommand{\FourierBase}[1][]{\mathbf{e}_{#1}}
\newcommand{\DyaPartOfUnit}[1][]{\rho_{#1}}
\newcommand{\LPBlock}[1][]{\triangle_{#1}}

\newcommand{\LaplaceOp}{\Delta}

\DeclareMathOperator{\supp}{supp}

\newcommand{\const}[1][]{K_{#1}} % renormalization constant
\newcommand{\reso}{\varodot}
\newcommand{\rpara}{\varolessthan}
\newcommand{\lpara}{\varogreaterthan}
\DeclareMathOperator{\com}{com}
\newcommand{\comC}{C}

\newcommand{\RealNum}{\mathbf{R}}
\newcommand{\Integers}{\mathbf{Z}}

\newcommand{\Torus}{\mathbf{T}}

\begin{document}

\title{
Paracontrolled quasi-geostrophic equation with space-time white noise
%\footnote{Preliminary version: 24 May 2017.}
}
\author{ Yuzuru Inahama and Yoshihiro Sawano
%\footnote{ }
}
\maketitle

\begin{abstract}
We study the stochastic dissipative quasi-geostrophic equation 
with space-time white noise on the two-dimensional torus.
This equation is highly singular and 
basically ill-posed in its original form. 
The main objective of the present paper is to formulate and
solve this equation locally in time
in the framework of paracontrolled calculus
when the differential order of the main term,
the fractional Laplacian, is larger than $7/4$.
No renormalization has to be done for this model.
Quasi-geostrophic equation: white noise:
\end{abstract}

\section{Introduction}

\subsection{Setting and main result}
The dissipative 2D quasi-geostrophic equation (QGE for short) is a partial differential equation
(PDE) which describes geophysical fluid dynamics on a two-dimensional space
(see \cite{pe} for example).
It has the fractional Laplacian $(-\LaplaceOp)^{\theta /2}$, 
$0< \theta \le 2$,
as its main term. 
%Among others 
%he dissipative 2D quasi-geostrophic equation
%with $\theta=1$
%The dissipative 2D QGE 
%\eqref{eq:180520} below 
%(without the noise term)
It is used as a model in oceanography and meteorology to
describe large-scale, 
stratified oceanic and atmospheric circulation exhibiting geostrophic balance.
The solution $u$ and its ``rotated Riesz transform"
$R^{\perp} u$ represent the potential temparature 
and the velocity field, respectively.
This type of QGE has been intensively studied (see \cite{CV, CMT, CW, CC, ju, KNV, mm} among many others)
and so has its stochastic counterpart (see \cite{HGH, lrz, rss, ZZ2} among others).

In this paper we study the following
stochastic dissipative QGE with {\it additive} space-time
white noise $\xi$ on the two-dimensional torus $\Torus^2={\bf R}^2/ {\bf Z}^2$
with a given initial condition:
\begin{equation}\label{eq:180520}
\partial_t u 
= - (-\LaplaceOp)^{\theta /2} u + R^{\perp} u \cdot \nabla u +\xi,
\qquad
t >0, \,\, x \in \Torus^2.
\end{equation}
Here, {\rm (i)} $\nabla =(\partial_1, \partial_2)$ is the usual gradient on $\Torus^2$,
{\rm (ii)} $R^{\perp} = (R_2, -R_1)$ with $R_j$ being the $j$th Riesz transform 
on $\Torus^2$ ($j=1,2$),
{\rm (iii)} the dot stands for the standard inner product on ${\bf R}^2$,
{\rm (iv)} $\xi =\xi (t,x)$ is a (generalized) 
centered Gaussian random field
 associated with 
$L^2 ({\bf R} \times \Torus^2)$ whose covariance is heuristically given by 
${\mathbb E} [\xi (t,x) \xi (s,y)]= \delta (t-s) \delta (x-y)$, where $\delta$ 
stands for Dirac's delta function at $0$.

Although stochastic QGE \eqref{eq:180520} may look natural and innocent at first glance, 
it does not even make sense in the classical sense 
since the regularity of $\xi$ is quite bad.
Let us take a quick look at this ill-definedness in the best case ($\theta =2$).
If the nonlinear term $R^{\perp} u \cdot \nabla u$ were absent, 
the solution of \eqref{eq:180520} would be an Ornstein--Uhlenbeck process.
It is a well-known Gaussian process and its space regularity at a fixed time is 
 $-\kappa$ for every small $\kappa >0$.
Therefore, it is natural to guess that the
regularity of $u (t, \cdot)$ is $-\kappa$ at best.
Then, the regularities of 
$R^{\perp} u (t, \cdot)$ and $\nabla u (t, \cdot)$ are 
$-\kappa$ and $-1-\kappa$ at best, respectively,
and their (inner) product therefore cannot be defined.

Recently, three theories were invented (see \cite{GubinelliImkellerPerkowski2015, Hai, Kup})
and it became possible to study very singular stochastic PDEs of this kind.
Now, studying singular stochastic PDEs is certainly one of the
cutting edges in both probability theory and PDE theory. 
In the present paper we will use
the so-called Gubinelli--Imkeller--Perkowski's paracontrolled calculus
\cite{GubinelliImkellerPerkowski2015} to solve \eqref{eq:180520}
locally in time when $7/4 <\theta \le 2$.

Before stating our main result (Theorem \ref{thm_2017013154042} below), 
let us fix some notation.
We denote by $\HolBesSp{\alpha} = B^{\alpha}_{\infty \infty}$
the Besov-H\"older space on $\Torus^2$ of regularity $\alpha \in \RealNum$.
We set $q := 2 -\frac{5}{2\theta}$ for $11/6 < \theta \le 2$ 
and $q := 5 -\frac{8}{\theta}- \frac{2(4 \theta -7)}{10^{100}\theta}$ 
for $7/4 < \theta \le 11/6$.
The value of $q$ is a little bit messy 
but is not very important. 
Let 
 $S=\frac{3}{2}\theta -2 - \theta q+ (\theta -1)\kappa'$
be
the regularity of the initial condition
$u(0,\cdot)$.
One should note that
$S$
is smaller than that of the stationary Ornstein--Uhlenbeck process $X$
(i.e., the stationary solution of $\partial_t X= -(-\Delta)^{\theta /2} X -X + \xi$).

For example, if $\theta =2$ then 
the Besov space in Theorem \ref{thm_2017013154042} is $\HolBesSp{-\frac12 +\kappa'}$.
Hence, the initial condition can be a distribution of regularity approximately $-1/2$.
(The smaller $\theta$ is, the less regular everything becomes.)

Let $\chi \colon {\mathbf R}^2 \to {\mathbf R}$
be a smooth and even
function with compact support such that $\chi (0)=1$
and write $\chi^{\ve} (x)= \chi (\ve x)$ for $0<\ve <1$.
Mollify $\xi$ using $\chi^{\ve}$ and denote it by $\xi^{\ve}$. 
See \eqref{eq:180522-1} (and Remark \ref{re:180524}) for the precise 
definition of $\xi^{\ve}$.
Then, $\xi^{\ve}$ is a smooth noise and converges to $\xi$ as $\ve \searrow 0$
in an appropriate sense.
Since $\xi^{\ve}$ is smooth, the following QGE driven by $\xi^{\ve}$ 
is well-posed (up to the explosion time) in the classical mild sense:
\begin{equation}\label{eq:rcgl}
\partial_t u^{\ve}
= - (-\LaplaceOp)^{\theta /2} u^{\ve} + R^{\perp} u^{\ve}\cdot \nabla u^{\ve} +\xi^{\ve},
\qquad
t >0, x \in \Torus^2.
\end{equation}
Due to the structure of QGE, 
no renormalization is needed for \eqref{eq:rcgl}.
This phenomenon 
is rare among singular stochastic PDEs, but has already happened
in \cite{ZZ}
for the three-dimensional Navier-Stokes equation with space-time white noise
and also in \cite{gp} for the stochastic Burgers equation.

Now we are in the position to state our main result. 
The proof of the following theorem is immediate from Propositions \ref{pr:180215-1},
\ref{pr.20171116} and Theorem \ref{pr:180220}:
\begin{theorem}\label{thm_2017013154042}
Let $7/4 < \theta \le 2$ and then take sufficiently small $\kappa' >0$ for this $\theta$.
Assume that $u_0\in\HolBesSp{\frac{3}{2}\theta -2 - \theta q+ (\theta -1)\kappa'}$.
 Then, for every $0<\epsilon<1$,
 there exist a random time $T_\ast^\epsilon\in(0,1]$
 and a unique process $u^\epsilon$ defined up to time $T_\ast^\epsilon$ such that
	\begin{itemize}
		\item	$u^\epsilon$ solves \eqref{eq:rcgl} on $[0,T_\ast^\epsilon]$
		 with the initial condition $u_0$,
		\item	$T_\ast^\epsilon$ converges to some a.s.~positive random time $T_\ast$ in probability
as $\ve \downarrow 0$,
		\item	$u^\epsilon$ converges to some 
		$\HolBesSp{\frac{3}{2}\theta -2 - \theta q+ (\theta -1)\kappa'}$-valued		process $u$ defined on $[0,T_\ast)$ in the following sense:
				\begin{align*}
					\lim_{\varepsilon\searrow 0}
\left(
					\sup\limits_{0\leq s\leq T_\ast/2}
						\|u^\epsilon_s-u_s\|_{\HolBesSp{\frac{3}{2}\theta -2 - \theta q+ (\theta -1)\kappa' }}
\right)
					=
						0
				\end{align*}
				in probability
as $\ve \downarrow 0$,.
  Here, we understand
  $
   \sup\limits_{0\leq s\leq T_\ast/2}
   \|u^\epsilon_s-u_s\|_{\HolBesSp{\frac{3}{2}\theta -2 - \theta q+ (\theta -1)\kappa' }}
   =
   \infty
  $
  on the event $\{T_\ast^\epsilon<T_\ast/2\}$.
				Furthermore, $u$ is independent of the choice of $\chi$.
	\end{itemize}
\end{theorem}

Before we go further,
we make clarifying remarks.

\begin{remark}
The restriction $\theta >7/4$ does not seem essential.
A quick and heuristic computation shows that QGE with space-time 
white noise is subcritical in the sense of \cite{Hai} if and only if $\theta >4/3$.
So, while the case $\theta \le 4/3$ looks hopeless,
there is a possibility that the case $4/3< \theta \le 7/4$ can be solved.
However, as soon as $\theta$ gets smaller than $7/4$, 
computations get much more complicated.
For example, more ``symbols" are needed to define drivers.
Moreover, a higher order version of paracontrolled calculus 
developed in \cite{bb} is probably needed, too.
For these reasons we restrict ourselves to the case $7/4 <\theta \le 2$,
although the case of smaller $\theta$ still looks quite interesting. 
\end{remark}

\begin{remark}
Nevertheless,
the postulate $\theta>7/4$ is necessary in the proof
of Theorem \ref{thm_2017013154042}.
In fact,
$\theta>7/4$ is used to prove
(\ref{eq:20200113-1});
see Lemma \ref{lem_20171219}.
Lemma \ref{lem_20171219} is in turn used
for the proof of Proposition \ref{prop_20161003143917},
which is one of the components of Proposition \ref{pr:180213-1}.
Finally 
the proof of Proposition \ref{pr:180215-1} uses Proposition \ref{pr:180213-1}.
\end{remark}

\begin{remark}
Although the time $T_*$ is random,
the solution is unique since we are using the fixed point theorem.
\end{remark}

As far as the authors know, 
the only paper which deals with a singular stochastic PDE 
with the fractional Laplacian $(-\Delta)^{\theta/2}$ is 
Gubinelli--Imkeller--Perkowski \cite{GubinelliImkellerPerkowski2015},
in which paracontrolled calculus was invented.
They studied a Burgers-like equation on $\Torus$ with $\theta > 5/3$.
(It is almost certain that
two other theories for singular stochastic PDE,
Hairer's regularity structure theory \cite{Hai} and 
Kupiainen's renormalization group approach \cite{Kup}, 
also work very well for singular stochastic PDEs with the fractional Laplacian. 
At the moment, however, there seems to be no literature which 
actually elaborates to solve QGE with white noise or 
other concrete examples of singular stochastic PDEs with fractional Laplacian 
with these theories.)

In this paper we slightly refine the argument in 
\cite{GubinelliImkellerPerkowski2015} in the following way:
First, we will use Mourrat--Weber's method \cite{mw} 
to formulate our paracontrolled QGE.
In their work, a solution is defined to be a fixed point 
of a certain integration map
in a Besov space-valued function space consisting of 
functions in time which may explode near $t=0$.
A merit of doing so is that we become able to treat a relatively bad initial condition.
In turn, it allows us to use
a stationary Ornstein--Uhlenbeck process, 
even though
the Ornstein--Uhlenbeck process in the original problem (in the mild form)
starts at $0$.
Thanks to the stationarity, enhancing white noise becomes less cumbersome
and can be done with a fractional generalization of 
Gubinelli--Perkowski's enhancing method developed in \cite{gp}.

Before we go further,
we explain how to use the notation $\lesssim$ and $\gtrsim$.
Let $A,B \ge 0$.
Then $A \lesssim B$ and $B \gtrsim A$ mean
that there exists a constant $C>0$
such that $A \le C B$,
where $C$
does not depend on the parameters
or the functions
of importance.
The symbol $A \sim B$ means
that $A \lesssim B$ and $B \lesssim A$
happen simultaneously,
while $A \simeq B$ means that there exists
a constant $C>0$ such that $A=C B$.
When we need to emphasize or
keep in mind that
the constant $C$
depends on the parameters
$\alpha,\beta,\gamma$ etc:
Instead of $A\lesssim B$,
we write
$A\lesssim_{\alpha,\beta,\gamma,\ldots}B$.
Instead of
$A \sim B$,
we write
$A \sim_{\alpha,\beta,\gamma,\ldots}B$.
We also write $A \vee B=\max\{A,B\}$ and $A \wedge B=\min\{A,B\}$.

The organization of this paper is as follows:
Section 2 is a preliminary section. 
We recall basic results on Besov spaces and paradifferential calculus on the torus.
The main aim is to study the heat semigroup associated to the 
fractional Laplacian and give a Schauder-type estimate for this semigroup.
In Section 3 we give a bird's eye view 
of our proof of the main theorem in a heuristic way.
In Section 4 
we formulate paracontrolled QGE for $\theta \in (7/4, 2]$, following the method in \cite{mw}.
In the first half we define a driver of this equation.
In the latter half we introduce a Besov space-valued function space in time
in which solutions live.
As usual for the mild formulation, 
a solution of our paracontrolled QGE is defined to be a fixed point
of a suitable integration map.
In Section 5 we estimate this integration map.
This is the most important part in solving the equation.
Using these estimates, we prove the local well-posedness of 
our paracontrolled QGE for $\theta \in (7/4, 2]$ in Section 6. 
In Section 7, we prove that for a nice driver a solution in the paracontrolled sense 
coincides with a solution in the classical mild sense.
So far, everything is deterministic.
Section 8 is the probabilistic part and devoted to 
enhancing the space-time white noise $\xi$ to a stationary
random driver of paracontrolled QGE
(Theorem \ref{pr:180220}). 
In Subsections 8.1 
we generalize the enhancing method in \cite{gp}
to the fractional case in a rather general setting.
In Subsections 8.2--8.7 we then prove Theorem \ref{pr:180220} with this method.

%%%%%%%%%%%%%%%%%
%%%%%%%%%%%%%%%%%%%%%%%%%%%%%%%%%%%%\newpage
%%%%%%%%%%%%%%%%%
%%%%%%%%%%%%%%%%%%%%%%%%%%%%%%%%

\section{Preliminaries 
}

In this section
%
%\footnote{This section is essentially a copy-and-paste from \cite{hin}.}
%
 we introduce the Besov-H\"older space 
$\HolBesSp{\alpha} = B^{\alpha}_{\infty \infty}$, 
$\alpha \in \RealNum$,
and paradifferential calculus on the $d$-dimensional torus $\Torus^d$.
Unlike the counterpart in most of the 
preceding works on paracontrolled calculus,
we study the heat semigroup generated by the fractional Laplacian.
In particular, a fractional version of the Schauder estimate is provided.

\subsection{Besov spaces and paraproduct}

We introduce Besov spaces over $\Torus^d={\mathbf R}^d/{\mathbf Z}^d$
and recall their basic properties.
Let $\cD =\cD(\Torus^d,\RealNum)$ be
the space of all smooth $\RealNum$-valued functions on $\Torus^d$.
We denote by $\cD'$ the dual of $\cD$, 
that is, the space of Schwartz distributions in $\Torus^d$.

We set $\FourierBase[k](x)=e^{2\pi\ImUnit k\cdot x}$ for every $k\in\Integers^d$ and $x\in\Torus^d$.
The Fourier transform $\FourierTrans{f}$ for $f\in\cD$ is defined by
$
	\FourierTrans f(k)
	=
		\int_{\Torus^d}
			\FourierBase[-k](x)
			f(x)\,
			dx
$
and its inverse $\FourierTrans^{-1}g$ for a rapidly decreasing sequences $\{g(k)\}_{k\in\Integers^d}$ is defined by
$
	\FourierTrans^{-1} g
	=
		\sum\limits_{k\in\Integers^d}
			g(k)
			\FourierBase[k]
$.
For a rapidly decreasing smooth function $\phi:{\mathbf R}^d \to {\mathbf C}$, we set
$
	\phi(D)f
	=
		\FourierTrans^{-1}\phi\FourierTrans f
	=
		\sum\limits_{k\in\Integers^d}
			\phi(k)
			\FourierCoeff{f}(k)
			\FourierBase[k]
$.

We denote by $\{\DyaPartOfUnit[m]\}_{m=-1}^\infty$ a dyadic partition of unity, that is,
it satisfies the following:
(1) each $\DyaPartOfUnit[m]\colon\RealNum^d \to[0,1]$ is radial and smooth,
(2) $
		\supp(\DyaPartOfUnit[-1])
		\subset
			B(0,\frac{4}{3})
	$,
	$
		\supp(\DyaPartOfUnit[0])
		\subset
			B(0,\frac{8}{3})\setminus B(0,\frac{3}{4})
	$,
(3) $\DyaPartOfUnit[m](\cdot)=\DyaPartOfUnit[0](2^{-m}\cdot)$ for $m\geq 0$,
(4) $
		\sum\limits_{m=-1}^\infty
			\DyaPartOfUnit[m](\cdot)
		=
			1
	$.
Here $B(0,r)=\{x\in\RealNum^d ;|x|<r\}$.
Such a function $\rho$
can be constructed with ease.
In fact, choose $\rho_{-1} \in C^\infty$ 
so that $\chi_{B(0,\frac{1}{2})} \le \rho_{-1} \le \chi_{B(0,1)}$.
If we let
$\rho_m(x)=\rho_{-1}(2^{-m-1}x)-\rho_{-1}(2^{-m}x)$,
then $\rho_0$ satisfies (2).
Finally if we let 
$\rho_m(x)=\rho_0(2^{-m}x)$ for any $m \in {\mathbb N}$
and $x \in {\mathbb R}^d$,
then we obtain (3) and (4).

The Littlewood-Paley blocks 
(or the Littlewood-Paley operator)
$\{\LPBlock[m]\}_{m=-1}^\infty$ is defined by $\LPBlock[m]=\DyaPartOfUnit[m](D)$.
As usual we set $S_j f = \sum\limits_{m \le j-1} \LPBlock[m] f$.
We note that
\begin{equation}\label{eq:180329-1}
\sharp\{k \in {\mathbf Z}^d\,:\,\rho_j(k)>0\}={\rm O}(2^{j d})
\end{equation}
for each $j \in {\mathbf N}_0=\{0,1,\ldots\}$.

We are ready to define the Besov space 
$B^{\alpha}_{pq} 
= B^{\alpha}_{pq} (\Torus^d)$
for $\alpha\in\RealNum$, $1 \le p, q \le \infty$.
For $f \in \cD'$,
the $B^{\alpha}_{pq}$-norm $\|f\|_{B^{\alpha}_{pq}}$ is defined by
\begin{eqnarray*}
	\|f\|_{B^{\alpha}_{pq} }
	=
\begin{cases}
			\left(
\sum\limits_{m=-1}^\infty(2^{m \alpha}\|\LPBlock[m] f\|_{L^p(\Torus^d)})^q
\right)^{\frac1q}			
			 & \mbox{(if $1 \le q <\infty$)},\\
		\sup\limits_{m\geq -1}
			2^{m\alpha}
			\|\LPBlock[m] f\|_{L^p(\Torus^d)} & \mbox{(if $q =\infty$)}.
\end{cases}
			\end{eqnarray*}
We set $B^{\alpha}_{pq} = \{ f \in \cD^{\prime} ; \|f\|_{B^{\alpha}_{pq} }<\infty \}$.	
See \cite{Sawano18,Triebel-text-83}
for more about this function space.
For simplicity, we will write
$\HolBesSp{\alpha}
=\HolBesSp{\alpha}(\Torus^d) = B^{\alpha}_{\infty \infty}(\Torus^d)$
and sometimes call it the H\"{o}lder-Zygmund class of order $\alpha$.
We note that
$\displaystyle
{\mathcal D}'=\bigcup_{\alpha \in {\mathbf R}}\HolBesSp{\alpha}.
$

A couple of helpful remarks may be in order.
\begin{remark}\label{rem:171027-1}
\
\begin{enumerate}
\item
Let $f \in \cD'$.
By using the Young inequality
one can verify that the following norm 
is equivalent to the original one:
\[
\|f\|_{B^{\alpha}_{p q}*}
=
\left(
\sum\limits_{m=-1}^\infty(2^{m \alpha}\| \LPBlock[m]{}^2 f\|_{L^p(\Torus^d)})^q
\right)^{\frac1q}
\]
for $\alpha \in \RealNum$ and $1 \le p,q \le \infty$. 
See \cite[p. 286]{Triebel-text-83}.
\item
Our space $B^{\alpha}_{p q}$ may be different from the popular space 
${\mathcal B}^{\alpha}_{p q}$
which is defined to be the closure
of ${\mathcal D}$ in $B^{\alpha}_{p q}$.
However it causes no troubles, anyway.
Some prefer to use the notation
$b^\alpha_{p q}$
to denote the Besov space
${\mathcal B}^\alpha_{p q}$.
\end{enumerate}
\end{remark}

In this paper we will mainly use 
 $\HolBesSp{\alpha} = B^{\alpha}_{\infty \infty}$.
The following inequalities are frequently used results on these Banach spaces:
\begin{proposition}
%[{\cite[Theorem~2.80]{BahouriCheminDanchin2011}}]
\label{prop_20160928054620}
	We have the following:
	\begin{enumerate}
		\item[$(1)$]
Let $\alpha \le \beta$,
and let $f \in \HolBesSp{\beta}$.
Then $f \in \HolBesSp{\alpha}$
and $
					\|f\|_{\HolBesSp{\alpha}}
					\lesssim
						\|f\|_{\HolBesSp{\beta}}$.
		\item[$(2)$]	
Let $\alpha_1,\alpha_2\in\RealNum$
and $0\le \nu \le 1$. Then
				$
					\|f\|_{\HolBesSp{(1-\nu)\alpha_1+\nu\alpha_2}}
					\leq
						\|f\|_{\HolBesSp{\alpha_1}}^{1-\nu}
						\|f\|_{\HolBesSp{\alpha_2}}^\nu
				$
for all $f \in \HolBesSp{\alpha_1} \cap \HolBesSp{\alpha_2}$.
	\end{enumerate}
\end{proposition}

\begin{proof}
\
\begin{enumerate}
\item
This is trivial from the definition of 
two norms
$\|f\|_{\HolBesSp{\alpha}}$
and
$\|f\|_{\HolBesSp{\beta}}$;
simply compare $2^{m\alpha}$ and $2^{m\beta}$
in their definitions.
See also \cite[p. 47]{Triebel-text-83}
or \cite[Proposition 2.3]{Sawano18}.
\item
This is also trivial since
$
\sup\limits_{m \ge -1}2^{(1-\nu)\alpha_1+\nu \alpha_2}|a_m|
\le
\left(\sup\limits_{m \ge -1}2^{\alpha_1}|a_m|\right)^{1-\nu}
\left(\sup\limits_{m \ge -1}2^{\alpha_2}|a_m|\right)^{\nu}.
$
\end{enumerate}
\end{proof}

Proposition \ref{prop_20160928054620}(1) 
and
(2)
are referred to 
as the embedding inequality
and 
the interpolation inequality,
respectively.
We use
the fact that $2^{j\alpha} \lesssim 2^{j\beta}$ for $j \ge -1$
to prove the embedding inequality,
while we use
$2^{m \alpha}\|\LPBlock[m] {}^2 f\|_{L^\infty(\Torus^d)}
=
(2^{m \alpha_1}\| \LPBlock[m]{}^2 f\|_{L^\infty(\Torus^d)})^{1-\nu}
\cdot
(2^{m \alpha_2}\|\LPBlock[m]{}^2 f\|_{L^\infty(\Torus^d)})^{\nu}
$
to prove the interpolation inequality.

We recall that the derivative stands for the smoothness order
of functions.

\begin{lemma}\label{lem:180122-1}
Let $u \in \HolBesSp{\alpha+1}$ with $\alpha \in {\mathbf R}$.
Then $\|\partial_j u\|_{\HolBesSp{\alpha}} \lesssim \|u\|_{\HolBesSp{\alpha+1}}$
for all $j=1,2,\ldots,d$.
\end{lemma}

\begin{proof}
The proof being simple, we recall it.
We write the left-hand side out in full:
\begin{align*}
	\|\partial_j u\|_{\HolBesSp{\alpha}}
	=
		\sup\limits_{m\geq -1}
			2^{m\alpha}
			\|\partial_j \LPBlock[m]u\|_{L^\infty(\Torus^d)}
	=
		\sup\limits_{m\geq -1}
			2^{m\alpha+m}
			\|2^{-m}\partial_j (\LPBlock[m-1]+\LPBlock[m]+\LPBlock[m+1])\LPBlock[m]u\|_{L^\infty(\Torus^d)}.
\end{align*}
We observe that $2^{-m}\partial_j(\LPBlock[m-1]+\LPBlock[m]+\LPBlock[m+1])=\tau(2^{-m}D)$
for some $\tau \in C^\infty_{\rm c}$.
As a result, we obtain
\[
	\|\partial_j u\|_{\HolBesSp{\alpha}}
	\lesssim
		\sup\limits_{m\geq -1}
			2^{m\alpha+m}
			\| \LPBlock[m]u\|_{L^\infty(\Torus^d)}
			=
	\|u\|_{\HolBesSp{\alpha+1}},
\]
as required.
\end{proof}

\begin{proposition}\label{prop:180209-1}
Let $\alpha>0$.
Then $\HolBesSp{\alpha} \subset L^\infty$.
\end{proposition}

\begin{proof}
Let $f \in\HolBesSp{\alpha}$.
Simply observe
$
f=\sum\limits_{j=0}^\infty \LPBlock[j]f
$
in $L^\infty$.
\end{proof}
%%%%%%%%%%%%%%%%%%%%%%%%%%%%%%%%%
%\newpage
%%%%%%%%%%%%%%%%%%%%%
%\subsection{Paraproducts and Commutator estimates}

Now we introduce the paraproduct 
and the resonant product.
For every $f\in\HolBesSp{\alpha}$, $g\in\HolBesSp{\beta}$, 
we define the paraproduct by
\[
		f\lpara g
		=
			\sum\limits_{m_1\geq m_2+2}
				\LPBlock[m_1]f
				\LPBlock[m_2]g, \quad
		f\rpara g
		=
			\sum\limits_{m_1+2\leq m_2}
				\LPBlock[m_1]f
				\LPBlock[m_2]g,
	\]
and the resonant product by
\[
	f\reso g
	=
		\sum\limits_{|m_1-m_2|\leq 1}
			\LPBlock[m_1]f
			\LPBlock[m_2]g.
\]
Observe that
$fg =f\lpara g+ f\reso g+ f\rpara g$
at least formally.
We have to establish that these definitions make sense.
Here in Proposition \ref{prop_20160926062106}
we collect the cases where these definitions are justified.

The following are basic properties of the paraproduct and the resonant.
\begin{proposition}[Paraproduct and resonant estimate]
\label{prop_20160926062106}
Let $\alpha,\beta \in \RealNum$.
	\begin{enumerate}
		\item[$(1)$]
For all $f \in L^\infty(\Torus^d)$ and $g \in \HolBesSp{\beta}$
$
					\|f\rpara g\|_{\HolBesSp{\beta}}
					\lesssim
						\|f\|_{L^\infty(\Torus^d)}
						\|g\|_{\HolBesSp{\beta}}.
$
		\item[$(2)$]
If $\alpha <0$,
then
$\|f\rpara g\|_{\HolBesSp{\alpha+\beta}}
					\lesssim
						\|f\|_{\HolBesSp{\alpha}}
						\|g\|_{\HolBesSp{\beta}}$
for $f \in \HolBesSp{\alpha}$ and $g \in \HolBesSp{\beta}$.
		\item[$(3)$]	
Assume $\alpha+\beta>0$.
Then
$					\|f\reso g\|_{\HolBesSp{\alpha+\beta}}
					\lesssim
						\|f\|_{\HolBesSp{\alpha}}
						\|g\|_{\HolBesSp{\beta}}
$
for $f \in \HolBesSp{\alpha}$ and $g \in \HolBesSp{\beta}$.
	\end{enumerate}
\end{proposition}

\begin{proof}
The corresponding assertions to the ones over ${\mathbf R}^d$ are
well known.
%{\cite[Theorem~2.82 and 2.85]{BahouriCheminDanchin2011}}
Since
$\HolBesSp{\alpha}(\Torus^d)
=\HolBesSp{\alpha}({\mathbf R}^d) \cap \cD'(\Torus^d)$,
we can readily transplant these results over ${\mathbf R}^d$
into the ones over $\Torus^d$.
\end{proof}

A couple of remarks may be in order.
\begin{remark}\label{rem:180114-1}
Let $\alpha,\beta\in\RealNum$,
and let
$f \in \HolBesSp{\alpha}$ and $g \in \HolBesSp{\beta}$.
\begin{enumerate}
\item
We need $\alpha+\beta>0$
for
Proposition \ref{prop_20160926062106}(3).
This assumption allows
us not to take into account
the \lq \lq so called"
moment condition
as in
\cite{Triebel-text-97}.
\item
Let $\varepsilon>0$.
In view of Proposition \ref{prop_20160928054620}(1) and Proposition \ref{prop_20160926062106}(2)
the definition
$f\rpara g$
makes sense as an element in
$\HolBesSp{\min(\alpha,-\varepsilon)+\beta}$.
As a result,
the element
$f\rpara g$ and hence
$f\lpara g$ make
sense
for any $f,g \in {\mathcal D}'$.
\end{enumerate}
\end{remark}

We convert our observation to the form we use below.
\begin{corollary}\label{cor:180131-1}
Let 
$\alpha,\beta \in {\mathbf R}$, and let
$f \in \HolBesSp{\alpha}$, 
$g \in \HolBesSp{\beta}$
and
$h \in L^\infty(\Torus^d)$.
\begin{enumerate}
\item[$(1)$]
$
					\|f\lpara h\|_{\HolBesSp{\alpha}}
					\lesssim
						\|f\|_{\HolBesSp{\alpha}}
						\|h\|_{L^\infty(\Torus^d)},
$
\item[$(2)$]
Assume $\beta>0$.
Then
$
					\|f\lpara g\|_{\HolBesSp{\alpha}}
					\lesssim
						\|f\|_{\HolBesSp{\alpha}}
						\|g\|_{\HolBesSp{\beta}}.
$
\item[$(3)$]
Assume $\alpha+\beta>0$ and $\alpha\beta\ne 0$.
Then
$\|f g\|_{\HolBesSp{\min(\alpha,\beta)}} \lesssim
						\|f\|_{\HolBesSp{\alpha}}
						\|g\|_{\HolBesSp{\beta}}$.
\end{enumerate}
\end{corollary}

\begin{proof}\
\begin{enumerate}
\item
Simply swap the role of $f$ and $g$ 
in
Proposition \ref{prop_20160926062106}(1).
\item
Combine
${\mathcal C}^\beta \subset L^\infty$
and
Proposition \ref{prop_20160926062106}(1).
\item
If $\min(\alpha,\beta)<0$,
then use
Proposition \ref{prop_20160926062106}(2) and (3)
as well as
Proposition \ref{prop:180209-1} and
Corollary \ref{cor:180131-1}(2).
If $\min(\alpha,\beta)>0$,
then use
Proposition \ref{prop_20160926062106}(1) and (3)
as well as
Proposition \ref{prop:180209-1} and
Corollary \ref{cor:180131-1}(1).
\end{enumerate}
\end{proof}
%%%%%%% partial derivative %%%%%%%

Now we record an identity on the partial derivative $\partial_j$ on the torus.
\begin{lemma}\label{lem:161117-106}
Whenever
$f,g \in {\mathcal D}'$
and
$l=1,2,\ldots,d$,
$
\partial_l (f \rpara g)
=
(\partial_l f)\rpara g
+
f\rpara(\partial_l g)
$.
\end{lemma}

%%%%%% Commutator estimates %%%%%

Concerning the paraproduct and the resonant product,
the following trilinear form is significant:
	Define the map $\comC$ by
	\begin{align}\label{eq:180106-4}
		\comC (f,g,h)
		=
			(f\rpara g)\reso h-f(g\reso h)
	\end{align}
	for $f,g,h\in\cD$.
We prove the following estimate:
\begin{proposition}[Commutator estimates]
\label{prop_comm_20160919055939}
	Let 
$\alpha, \beta, \gamma\in\RealNum$ satisfy 
$0<\alpha<1$, 
$\beta+\gamma<0$ 
and $\alpha+\beta+\gamma>0$.
	Then $\comC$ uniquely extends to a continuous trilinear map from
	$
		\HolBesSp{\alpha}
		\times\HolBesSp{\beta}
		\times\HolBesSp{\gamma}
	$
	to
	$\HolBesSp{\alpha+\beta+\gamma}$
and the map satisfies
$
\|\comC (f,g,h)\|_{\HolBesSp{\alpha+\beta+\gamma}}
\lesssim
\|f\|_{\HolBesSp{\alpha}}
\|g\|_{\HolBesSp{\beta}}
\|h\|_{\HolBesSp{\gamma}}
$
for all
$f\in \HolBesSp{\alpha}$,
$g\in \HolBesSp{\beta}$
and
$h\in \HolBesSp{\gamma}$.
\end{proposition}

\begin{proof}
We invoke
\cite[Lemma 2.4]{GubinelliImkellerPerkowski2015},
where a counterpart to ${\mathbf R}^d$ is shown.
Since
$B^s_{p q}(\Torus^d)=\cD' \cap B^s_{p q}({\mathbf R}^d)$
with the identical norms.
Thus
the desired estimate follows from
\cite[Lemma 2.4]{GubinelliImkellerPerkowski2015}.
\end{proof}

%%%%%%%%%%%% Riesz %%%%%%%%%%%%

Next we provide two lemmas on the Riesz transform $R_l$ on the torus.
Denote by $R_l$ the $l$th Riesz transform
for $l=1,2,\ldots,d$, which is defined by
\begin{equation}\label{eq:Riesz}
R_l h
=
\sqrt{-1}
\sum\limits_{k \in {\mathbf Z}^d \setminus \{0\}}
\frac{k_l}{|k|}\hat{h}(k) {\bf e}_k
\quad (h \in {\mathcal D}').
\end{equation}

\begin{lemma}\label{lm.20171227}
Let $u \in \HolBesSp{\alpha}$ with $\alpha \in {\mathbf R}$.
Then $\|R_l u\|_{\HolBesSp{\alpha}} \lesssim \|u\|_{\HolBesSp{\alpha}}$
for all $l=1,2,\ldots,d$.
\end{lemma}

\begin{proof}
Since $1 \in \ker(R_l)$, we may handle
$(1-\rho_{-1}(D))u$ instead of $u$ itself.
The proof being simple once again, we recall it.
We write the left-hand side out in full:
\begin{align*}
	\|R_l (1-\rho_{-1}(D)) u\|_{\HolBesSp{\alpha}}
	&=
		\sup\limits_{m\geq -1}
			2^{m\alpha}
			\|R_l (1-\rho_{-1}(D))\LPBlock[m]u\|_{L^\infty(\Torus^d)}\\
	&=
		\sup\limits_{m\geq -1}
			2^{m\alpha}
			\|R_l (1-\rho_{-1}(D))(\LPBlock[m-1]+\LPBlock[m]+\LPBlock[m+1])\LPBlock[m]u\|_{L^\infty(\Torus^d)}\\
	&\le
		\sup\limits_{m\geq -1}
			2^{m\alpha}
			\|R_l (\LPBlock[m-1]+\LPBlock[m]+\LPBlock[m+1])\LPBlock[m]u\|_{L^\infty(\Torus^d)}.
\end{align*}
We observe that $R_l(\LPBlock[m-1]+\LPBlock[m]+\LPBlock[m+1])=\tau_l(2^{-m}D)$
for some $\tau_l \in C^\infty_{\rm c}$ if $m \gg 1$.
As a result, we obtain
\[
	\|R_l(1-\rho_{-1}(D)) u\|_{\HolBesSp{\alpha}}
	\lesssim
		\sup\limits_{m\geq -1}
			2^{m\alpha}
			\| \LPBlock[m]u\|_{L^\infty(\Torus^d)}
			=
	\|u\|_{\HolBesSp{\alpha}},
\]
as required.
\end{proof}

If we consider the commutator,
the postulate on $\alpha$ is loosened.
In fact,
we needed $\alpha<0$
in Proposition \ref{prop_20160926062106}(2),
while
Lemma \ref{lem:161117-107}
requires $\alpha<1$.

\begin{lemma}\label{lem:161117-107}
Let $\alpha<1$ and $\beta \in {\mathbf R}$,
and let
$f \in {\mathcal C}^\alpha$ 
and $g \in {\mathcal C}^\beta$.
Then
$
R_l(f \rpara g)-f \rpara R_l g
\in {\mathcal C}^{\alpha+\beta},
$
and it satisfies
$\|R_l(f \rpara g)-f \rpara R_l g\|_{{\mathcal C}^{\alpha+\beta}}
\lesssim
\|f\|_{{\mathcal C}^\alpha}
\|g\|_{{\mathcal C}^\beta}.
$
\end{lemma}

\begin{proof}
Let $j \ge 5$ and $l=1,2\ldots,n$. We disregard the lower frequency terms
because we can incorporate them later easily.
We observe that 
$R_l \LPBlock[j]$
has a kernel $K_j =K_{j,l}$, that is,
\begin{align*}
\displaystyle 
 \LPBlock[j] (R_l(f \rpara g)-f \rpara R_l g)
&=
\int_{{\mathbf R}^d}K_{j}(x-y)(S_jf(y)-S_jf(x)) \LPBlock[j] g(y)\,dy
\end{align*}
and
${\mathcal F}K_j \in C^\infty$ satisfies the scaling relation:
$
{\mathcal F}K_5(2^{-j+5}\cdot)={\mathcal F}K_{j}.
$
%Here and below we suppress the subscript $l$
%keeping in mind that $l=1,2\ldots,n$.
By the mean value theorem, we have
\[
|S_jf(y)-S_jf(x)|
\le 
\|\nabla S_j f\|_{L^\infty(\Torus^d)}|x-y|
=
2^{-j}\|\nabla S_j f\|_{L^\infty(\Torus^d)}
\cdot
2^j|x-y|.
\]
Since $\alpha<1$, we have
\[
\|\nabla S_j f\|_{L^\infty(\Torus^d)}
\le
\sum\limits_{m=-1}^{j+2}
\|\nabla S_j \LPBlock[m] f\|_{L^\infty(\Torus^d)}
\lesssim
\sum\limits_{m=-1}^{j+2}
2^{m}\| \LPBlock[m] f\|_{L^\infty(\Torus^d)}
\lesssim
2^{j(1-\alpha)}
\|f\|_{{\mathcal C}^{\alpha}}.
\]
Since
\[
\int_{{\mathbf R}^d}2^j|x-y|\cdot|K_j(x-y)|\,dy=
\int_{{\mathbf R}^d}2^5|y|\cdot|K_5(y)|\,dy<\infty
\quad (x \in {\mathbf R}^d),
\]
we have
\begin{align*}
\| \LPBlock[j](R_l(f \rpara g)-f \rpara R_l g)\|_{L^\infty(\Torus^d)}
\lesssim
2^{-j(\alpha+\beta)}\|f\|_{{\mathcal C}^\alpha}\|g\|_{{\mathcal C}^\beta}.
\end{align*}
\end{proof}

%%%%%%%%%%%%%%%%%%%%%%%%%%%%
%\newpage
%%%%%%%%%%%%%%%%%%%%%%%%%%%%%
\subsection{Fractional heat semigroup} 

In this subsection we study effects of the heat semigroup
generated by the fractional Laplacian
$(- \LaplaceOp)^{\theta /2}$ for $0 <\theta \le 2$.

For later considerations,
we need the following simple fact deduced from the Minkowski inequality:
\begin{proposition}\label{prop:170822-1}
Let $1 \le p \le \infty$,
and let $f \in L^1({\mathbf R}^d)$ and $g \in L^p(\Torus^d)$.
Then
\[
\|f*g\|_{L^p(\Torus^d)} \le \|f\|_{L^1({\mathbf R}^d)}\|g\|_{L^p(\Torus^d)}.
\]
\end{proposition}

Although the multiplier of $(-\Delta)^{\theta/2}$ has singularity at $\xi=0$,
as the following lemma shows,
this singularity does not affect
the integrability of the Fourier transform of $\exp(-|2\pi\cdot|^\theta)$
 so much.
\begin{proposition}\label{prop:170822-2}
Let $\theta>0$.
The function
\begin{equation}\label{eq:171026-1}
x \in {\mathbf R}^d
\mapsto
K^\dagger(x)=\int_{{\mathbf R}^d}\exp(-|2\pi\xi|^\theta)\exp(2\pi\sqrt{-1}x \cdot \xi)\,d\xi
\in {\mathbf C}
\end{equation}
is an integrable function.
\end{proposition}
This fact is somewhat well known.
However,
since we will need to use the proof
of Proposition \ref{prop:170822-2}
in the proof of Proposition \ref{prop_20160927051159},
we supply the proof.
See \cite{Press} as well as \cite{BBKRSV}
for example.
%We can find the case of $d=1$ in \cite{Feller}.
\begin{proof}
Let $u>0$.
We set
\[
K^\dagger_u(x)=\int_{{\mathbf R}^d}(\rho_{-1}(u\xi)-\rho_{-1}(2u\xi))
\exp(-|2\pi\xi|^\theta)\exp(2\pi\sqrt{-1}x \cdot \xi)\,d\xi
\quad (x \in {\mathbf R}^d).
\]
We decompose
\[
K^\dagger=\sum\limits_{j=-\infty}^\infty K^\dagger_{2^{-j}}.
\]
We claim that we have
an important bound for $K^\dagger_u$:
\[
|K^\dagger_u(x)| \lesssim u^{-d}
\quad (x \in {\mathbf R}^d).
\]
In fact,
\[
|K^\dagger_u(x)|
\le\int_{{\mathbf R}^d}
|\rho_{-1}(u\xi)-\rho_{-1}(2u\xi)|\,d\xi
={\rm O}(u^{-d}).
\]
Let $u<1$.
Then we have 
\[
|\partial^\alpha[(\rho_{-1}(u\xi)-\rho_{-1}(2u\xi))\exp(-|2\pi\xi|^\theta)]|
\lesssim u^{L}\chi_{\xi\sim u^{-1}}(\xi)
\]
for all $\alpha \in {\mathbf N}_0{}^d$ and $L \in {\mathbf N}$.
As a consequence
\[
(|x|^{d-1}+|x|^{d+1})|K^\dagger_u(x)| \lesssim u^{L-d}
\]
for all $u \in (0,1]$ and $L \in {\mathbf N}$.
Consequently,
\[
\|K^\dagger_u\|_{L^1({\mathbf R}^d)}
\lesssim u^{L-d}
\]
for any $L \in {\mathbb N}$,
where the implicit constant depends on $L$.
As a consequence,
\[
\left\|\sum\limits_{j=0}^{\infty} K^\dagger_{2^{-j}}\right\|_{L^1({\mathbf R}^d)}
\lesssim
\sum\limits_{j=0}^\infty 2^{-j(L-d)}\sim 1.
\]
Let $u>1$.
Fix $x \in {\mathbf R}^d$.
We write
\begin{align*}
L^\dagger_u(x)
&=\int_{{\mathbf R}^d}(\rho_{-1}(u\xi)-\rho_{-1}(2u\xi))
[\exp(-|2\pi\xi|^\theta)-1]\exp(2\pi\sqrt{-1}x \cdot \xi)\,d\xi\\
M^\dagger_u(x)
&=\int_{{\mathbf R}^d}(\rho_{-1}(u\xi)-\rho_{-1}(2u\xi))
\exp(2\pi\sqrt{-1}x \cdot \xi)\,d\xi.
\end{align*}
Then we decompose
$
K^\dagger_u(x)
=
L^\dagger_u(x)+M^\dagger_u(x).
$
We remark that
\begin{equation}\label{eq:180120-1}
\left\|\sum\limits_{j=-\infty}^0 M^\dagger_{2^{-j}}\right\|_{L^1({\mathbf R}^d)}
\lesssim 1
\end{equation}
by the use of the Fourier transform,
which maps ${\mathcal S}({\mathbf R}^d)$ to itself continuously.
Next,
keeping in mind that $\theta>0$ and that
\[
\exp(-|2\pi\xi|^\theta)-1=\sum\limits_{k=1}^\infty \frac{(-1)^k}{k!}|2\pi \xi|^{k\theta},
\]
we observe
\[
|\partial^\alpha[
(\rho_{-1}(u\xi)-\rho_{-1}(u\xi))
(\exp(-|2\pi\xi|^\theta)-1)]|
\lesssim_\alpha
\chi_{\xi \sim u^{-1}}(\xi)u^{-\theta+|\alpha|}.
\]
Thus,
\begin{equation}\label{eq:171026-5}
|x|^{N}|L^\dagger_u(x)|
\lesssim u^{-\theta+N}
\int_{\xi \sim u^{-1}}\,d\xi
\sim
u^{-d-\theta+N}
\end{equation}
for any integer $N$.
As a result,
\begin{equation}\label{eq:180117-1}
\left|\sum\limits_{j=-\infty}^0 L^\dagger_{2^{-j}}(x)\right|
\lesssim 1.
\end{equation}
Meanwhile,
(\ref{eq:171026-5}) can be interpolated and we have
\[
|x|^{d+\frac{\theta}{2}}|L^\dagger_u(x)|
\lesssim
u^{-\frac{\theta}{2}}.
\]
Recall that $\theta>0$.
We consider the case where $u=2^{-j}$.
If we add this estimate over $j$,
then we have
\[
|x|^{d+\frac{\theta}{2}}
\left|\sum\limits_{j=-\infty}^0 L^\dagger_{2^{-j}}(x)\right|
\lesssim 1.
\]
If we combine (\ref{eq:180117-1}) with this estimate,
then we obtain
\[
(1+|x|^{d+\frac{\theta}{2}})
\left|\sum\limits_{j=-\infty}^0 L^\dagger_{2^{-j}}(x)\right|
\lesssim 1,
\]
together with (\ref{eq:180120-1})
proving the integrability of $K^\dagger$.
\end{proof}

We also need the following analogy to Proposition
\ref{prop:170822-1}:
\begin{proposition}\label{prop:170822-3}
Let $\theta>0$.
Then the function
\[
x \in {\mathbf R}^d \mapsto
\int_{{\mathbf R}^d}|\xi|^{-\theta}(\exp(-|2\pi\xi|^{\theta})-1)\exp(-2\pi\sqrt{-1} x \cdot \xi)\,d\xi
\in {\mathbf R}
\]
is integrable.
\end{proposition}
Although we have 
\[
\exp(-2\pi\sqrt{-1} x \cdot \xi)
\]
in the integrand,
the right-hand side is the Fourier transform
of an even function.\\
So, the function is real valued.
\begin{proof}
Since
\begin{align*}
\lefteqn{
\int_{{\mathbf R}^d}|\xi|^{-\theta}(\exp(-|2\pi\xi|^{\theta})-1)\exp(-2\pi\sqrt{-1} x \cdot \xi)\,d\xi
}\\
&\simeq_d
\int_0^1
\left(
\int_{{\mathbf R}^d} \exp(-2t|\pi\xi|^{\theta}-2\pi\sqrt{-1} x \cdot \xi)\,d\xi
\right)dt\\
&\simeq_d
\int_0^1
t^{-\frac{d}{\theta}}
\left(
\int_{{\mathbf R}^d} \exp(-2|\pi\xi|^{\theta}-2\pi\sqrt{-1} t^{-\frac{1}{\theta}}x \cdot \xi)\,d\xi
\right)dt\\
&=
\int_0^1
t^{-\frac{d}{\theta}}K^\dagger(t^{-\frac{1}{\theta}}x)dt,
\end{align*}
the function is integrable thanks to 
Proposition \ref{prop:170822-2}.
\end{proof}

%%%%%%%%%%%%%%%%%

\begin{lemma}\label{lm:171102-1}
Let $T>0$.
Assume 
$1 \le p,q \le \infty$,
$s \in {\mathbf R}$, 
$0<\theta \le 2$ and $\delta \ge 0$.
Then 
for all $f \in B^s_{p q}$ and $0<t \le T$,
$
\|e^{-t(-\Delta)^{\theta/2}}f\|_{B^{s+\theta\delta}_{p q}}
\lesssim_{p,q,s,T} t^{-\delta}
\|f\|_{B^s_{p q}}
$.
\end{lemma}
\begin{proof}
We assume $q<\infty$;
otherwise we can readily modify the proof below.
The case where $\delta=0$ is clear
in view of Proposition \ref{prop:170822-2} since the function
$e^{-t(-\Delta)^{\theta/2}}f$ can be written as
$e^{-t(-\Delta)^{\theta/2}}f=K_t*f$,
where $\|K_t\|_{L^1({\mathbf R}^d)} \lesssim 1$.
Let us suppose $\delta>0$.
Based on Remark \ref{rem:171027-1},
it suffices to show
\[
\|e^{-t(-\Delta)^{\theta/2}}f\|_{B^{s+\theta\delta}_{p q}}
\sim
\left(
\sum\limits_{j=-1}^\infty(2^{j(s+\delta\theta)}\|e^{-t(-\Delta)^{\theta/2}} \LPBlock[j]{}^2 f\|_{L^p(\Torus^d)})^q
\right)^{\frac1q}
\lesssim\,t^{-\delta}
\|f\|_{B^s_{p q} }.
\]
We notice
\[
e^{-t(-\Delta)^{\theta/2}} \LPBlock[j] {}^2 f
=c_d
{\mathcal F}^{-1}
\left(
e^{-t|2\pi\cdot|^{\theta}}\rho_0(2^{-j}\cdot)
\right)* \LPBlock[j] f
\]
and hence thanks to Proposition \ref{prop:170822-1}
\[
\|e^{-t(-\Delta)^{\theta/2}} \LPBlock[j] {}^2 f\|_{L^p(\Torus^d)}
\le
c_d
\left\|{\mathcal F}^{-1}
\left(
e^{-t|2\pi\cdot|^{\theta}}\rho_0(2^{-j}\cdot)
\right)\right\|_{L^1({\mathbf R}^d)}
\|\LPBlock[j] f\|_{L^p(\Torus^d)}.
\]
Here we denoted by $c_d$ some unimportant positive constants,
whose precise value is irrelevant in the proof.
Since
\[
{\mathcal F}^{-1}
\left(
e^{-t|2\pi\cdot|^{\theta}}\rho_0(2^{-j}\cdot)
\right)(x)
=
t^{-d/\theta}
{\mathcal F}^{-1}
\left(
e^{-|2\pi\cdot|^{\theta}}\rho_0(2^{-j}t^{-1/\theta}\cdot)
\right)\left(\frac{x}{t^{1/\theta}}\right)
\quad (x\in {\mathbf R}^d),
\]
we have
\[
\left\|{\mathcal F}^{-1}
\left(
e^{-t|2\pi\cdot|^{\theta}}\rho_0(2^{-j}\cdot)
\right)\right\|_{L^1({\mathbf R}^d)}
=
\left\|
{\mathcal F}^{-1}
\left(
e^{-|2\pi\cdot|^{\theta}}\rho_0(2^{-j}t^{-1/\theta}\cdot)
\right)
\right\|_{L^1({\mathbf R}^d)}.
\]
Note that
$
\|{\mathcal F}^{-1}F\|_{L^1({\mathbf R}^d)}
\lesssim_d
\|\Delta^{d}F\|_{L^1({\mathbf R}^d)}
+
\|F\|_{L^1({\mathbf R}^d)}
$
for any $F \in C^\infty_{\rm c}({\mathbb R}^d)$.
Hence
\[
\left\|{\mathcal F}^{-1}
\left(
e^{-t|2\pi\cdot|^{\theta}}\rho_0(2^{-j}\cdot)
\right)\right\|_{L^1({\mathbf R}^d)}
\lesssim_d
\|\Delta^{d}e^{-|2\pi\cdot|^{\theta}}\rho_0(2^{-j}t^{-1/\theta}\cdot)\|_{L^1({\mathbf R}^d)}
+
\|e^{-|2\pi\cdot|^{\theta}}\rho_0(2^{-j}t^{-1/\theta}\cdot)\|_{L^1({\mathbf R}^d)}.
\]
Let $x \in {\mathbb R}^d$.
Since
$
\rho_0(x) \ne 0
$
implies
$\frac34 \le |x| \le \frac83$,
we have
\[
e^{-|2\pi x|^{\theta}}\rho_0(2^{-j}t^{-1/\theta}x)
\le
e^{-\left|\frac34 \cdot 2^{j+1}t^{1/\theta}x\right|^\theta}
\rho_0(2^{-j}t^{-1/\theta}x).
\]
Taking the $L^1$-norm,
we have
\[
\|e^{-|2\pi\cdot|^{\theta}}\rho_0(2^{-j}t^{-1/\theta}\cdot)\|_{L^1({\mathbf R}^d)}
\le
e^{-\left|\frac34 \cdot 2^{j+1}t^{1/\theta}x\right|^\theta}
\|\rho_0(2^{-j}t^{-1/\theta}\cdot)\|_{L^1({\mathbf R}^d)}
\lesssim(2^{j}t^{1/\theta})^{-N}
\]
for any $N \in {\mathbf N}$.
Meanwhile,
by the Leibniz rule,
\[
\|\Delta^{d}e^{-|2\pi\cdot|^{\theta}}\rho_0(2^{-j}t^{-1/\theta}\cdot)\|_{L^1({\mathbf R}^d)}
\lesssim_d
\sum_{|\alpha|+|\beta|=d}
\|\partial^\alpha[e^{-\frac12|2\pi\cdot|^{\theta}}]\partial^\beta[\rho_0(2^{-j}t^{-1/\theta}\cdot)]\|_{L^1({\mathbf R}^d)}.
\]
Since
\[
|\partial^\alpha[e^{-|2\pi x|^{\theta}}]|
\lesssim_{\alpha,\theta}
e^{-\frac12|2\pi x|^{\theta}},
\]
we have
\[
|\partial^\alpha[e^{-|2\pi x|^{\theta}}]\partial^\beta[\rho_0(2^{-j}t^{-1/\theta}x)|
\lesssim
e^{-\left|\frac34 \cdot 2^{j+1}t^{1/\theta}x\right|^\theta}
\rho_0(2^{-j}t^{-1/\theta}x).
\]
Taking the $L^1$-norm,
we have
\[
\|\Delta^{d}e^{-|2\pi\cdot|^{\theta}}\rho_0(2^{-j}t^{-1/\theta}\cdot)\|_{L^1({\mathbf R}^d)}
\lesssim(2^{j}t^{1/\theta})^{-N}
\]
as before.

Thus, if $2^{j}t^{1/\theta} \ge 1$,
then we obtain
\[
\left\|{\mathcal F}^{-1}
\left(
e^{-t|2\pi\cdot|^{\theta}}\rho_0(2^{-j}\cdot)
\right)\right\|_{L^1({\mathbf R}^d)}
\lesssim(2^{j}t^{1/\theta})^{-N}
\]
for any $N \in {\mathbf N}$.

Likewise, if we start from
\[
{\mathcal F}^{-1}
\left(
e^{-t|2\pi\cdot|^{\theta}}\rho_0(2^{-j}\cdot)
\right)(x)
=
2^{jn}
{\mathcal F}^{-1}
\left(
e^{-|2^{j+1}\pi t^{1/\theta}\cdot|^{\theta}}\rho_0
\right)(2^j x),
\]
then we obtain
\[
\left\|{\mathcal F}^{-1}
\left(
e^{-t|2\pi\cdot|^{\theta}}\rho_0(2^{-j}\cdot)
\right)\right\|_{L^1({\mathbf R}^d)}
\lesssim
1,
\]
whenever $2^{j}t^{1/\theta} \le 1$. 
Thus, since $\delta>0$,
it follows that
\begin{align*}
\lefteqn{
\left(
\sum\limits_{j=-1}^\infty(2^{j(s+\delta\theta)}\|e^{-t(-\Delta)^{\theta/2}}\LPBlock[j]{}^2 f\|_{L^p(\Torus^d)})^q
\right)^{\frac1q}
}\\
&\lesssim
\|f\|_{B^s_{p q}}
\left(
\sum\limits_{j=-1}^\infty
2^{j\delta\theta q}
\min(1,(2^{j}t^{1/\theta})^{-N})^q
\right)^{\frac1q}\\
&= t^{-\delta}
\|f\|_{B^s_{p q}}
\left(
\sum\limits_{j=-1}^\infty
(2^{j}t^{1/\theta})^{\theta \delta q}
\min(1,(2^{j}t^{1/\theta})^{-N})^q
\right)^{\frac1q}\\
&\le t^{-\delta}
\|f\|_{B^s_{p q}}
\left(
\sum\limits_{j=-\infty}^\infty
(2^{j}t^{1/\theta})^{\theta \delta q}
\min(1,(2^{j}t^{1/\theta})^{-N})^q
\right)^{\frac1q}\\
&\sim t^{-\delta}
\|f\|_{B^s_{p q}},
\end{align*}
as was to be shown.
\end{proof}

\begin{lemma}
\label{lm:171102-2}
Let $T>0$.
Asume $1 \le p,q \le \infty$, $s \in {\mathbf R}$, $0<\theta \le 2$ and $0 \le \delta \le 1$.
Then
for all $f \in B^s_{p q}$ and $0<t \le T$,
$
\|(e^{-t(-\Delta)^{\theta/2}}-1)f\|_{B^{s-\theta\delta}_{p q}}
\lesssim
t^\delta\|f\|_{B^{s}_{p q}}.
$
\end{lemma}

\begin{proof}
If $\delta=0$,
then simply use the fact that the integral kernel
of $e^{-t(-\Delta)^{\theta/2}}$ is integrable;
see Proposition \ref{prop:170822-2}.
So, we need to consider the opposite endpoint case;
suppose $\delta=1$.
The matter is reduced to investigating
\begin{align}\label{eq:161117-1}
\left\|
{\mathcal F}^{-1}[(e^{-t|2\pi\cdot|^\theta}-1)\varphi(2^{-j}\cdot)]
\right\|_{L^1({\mathbf R}^d)}
\end{align}
for $0<t \le T$ and $j \in {\mathbf N}$.
If $2^j t^{1/\theta} \ge 1$,
then we go through the same argument as above to have
\[
\left\|
{\mathcal F}^{-1}[(e^{-t|2\pi\cdot|^\theta}-1)\varphi(2^{-j}\cdot)]
\right\|_{L^1({\mathbf R}^d)}
=
\left\|
{\mathcal F}^{-1}[(e^{-2^{j\theta+1}t|\pi\cdot|^\theta}-1)\varphi]
\right\|_{L^1({\mathbf R}^d)}
=
{\rm O}((2^j t^{1/\theta})^{-N})
\]
for all $N \in {\mathbf N}$.

If $t^{1/\theta} 2^j \le 1$, we need to handle 
(\ref{eq:161117-1}) more carefully.
As before, we use
\[
\left\|
{\mathcal F}^{-1}[(e^{-t|2\pi\cdot|^\theta}-1)\varphi(2^{-j}\cdot)]
\right\|_{L^1({\mathbf R}^d)}
=
\left\|
{\mathcal F}^{-1}[(e^{-|2\pi\cdot|^\theta}-1)\varphi(2^{-j}t^{-1/\theta}\cdot)]
\right\|_{L^1({\mathbf R}^d)}.
\]
Note that
\[
|\nabla^l[e^{-|2\pi\xi|^\theta}-1]|
=
{\rm O}(|\xi|^{\theta-l})
\quad (|\xi| \to 0)
\]
for all $l \in {\mathbf N} \cup \{0\}$.
Thus,
\[
\|\nabla^l[e^{-|2\pi\xi|^\theta}-1]\|_{L^1(B(8r) \setminus B(r))}
={\rm O}(r^{\theta+d-l}).
\]
Hence
\begin{equation}\label{eq:161117-2}
|x|^{d+1}|{\mathcal F}^{-1}[(e^{-|2\pi\cdot|^\theta}-1)\varphi(2^{-j}t^{-1/\theta}\cdot)](x)|
\lesssim
(2^{j}t^{1/\theta})^{\theta-1}
\end{equation}
and
\begin{equation}\label{eq:161117-3}
|x|^{d}|{\mathcal F}^{-1}[(e^{-|2\pi\cdot|^\theta}-1)\varphi(2^{-j}t^{-1/\theta}\cdot)](x)|
\lesssim
(2^{j}t^{1/\theta})^{\theta}.
\end{equation}
Interpolating between
(\ref{eq:161117-2}) and (\ref{eq:161117-3}),
we obtain
\begin{equation}\label{eq:161117-4}
|x|^{d+\theta/2}|{\mathcal F}^{-1}[(e^{-|2\pi\cdot|^\theta}-1)\varphi(2^{-j}t^{-1/\theta}\cdot)](x)|
\lesssim
(2^{j}t^{1/\theta})^{\theta/2}.
\end{equation}
Meanwhile,
\[
|{\mathcal F}^{-1}[(e^{-|2\pi\cdot|^\theta}-1)\varphi(2^{-j}t^{-1/\theta}\cdot)](x)|
\lesssim
\|(e^{-|2\pi\cdot|^\theta}-1)\varphi(2^{-j}t^{-1/\theta}\cdot)\|_{L^1({\mathbf R}^d)}
\lesssim
(2^{j}t^{1/\theta})^{\theta+d}.
\]
Thus,
\[
|{\mathcal F}^{-1}[(e^{-|2\pi\cdot|^\theta}-1)\varphi(2^{-j}t^{-1/\theta}\cdot)](x)|
\lesssim
\min(
(2^{j}t^{1/\theta})^{\theta+d},(2^{j}t^{1/\theta})^{\theta/2}|x|^{-d-\theta/2}).
\]
As a result,
\[
\left\|
{\mathcal F}^{-1}[(e^{-t|2\pi\cdot|^\theta}-1)\varphi(2^{-j}\cdot)]
\right\|_{L^1({\mathbf R}^d)}
\lesssim
(2^{j}t^{1/\theta})^{\theta}.
\]
Thus, if we mimic the proof of Lemma \ref{lm:171102-1},
we obtain the desired result.
\end{proof}

\begin{lemma}\label{lm:171102-3}
Let $T>0$.
Assume 
$1 \le p,q \le \infty$, 
$s \in {\mathbf R}$, 
$0 \le \theta \le 2$,
$0 \le \delta \le 1$
and
$\eta \in [0,\infty)$.
Then
$\|e^{-t_2(-\Delta)^{\theta/2}}f-e^{-t_1(-\Delta)^{\theta/2}}f\|_{B^{s}_{p q}}
\lesssim|t_2-t_1|^{\delta}t_1{}^{-\eta}
\|f\|_{B^{s+\theta\delta-\theta\eta}_{p q}}$
for all $f \in B^{s+\theta\delta-\theta\eta}_{p q}$
and
$0 \le t_1 \le t_2 \le T$.
\end{lemma}

\begin{proof}
Combining Lemmas \ref{lm:171102-1} and \ref{lm:171102-2},
we obtain
\begin{align*}
\|e^{-t_2(-\Delta)^{\theta/2}}f-e^{-t_1(-\Delta)^{\theta/2}}f\|_{B^{s}_{p q}}
&=
\|e^{-(t_2-t_1)(-\Delta)^{\theta/2}}e^{-t_1(-\Delta)^{\theta/2}}f
-e^{-t_1(-\Delta)^{\theta/2}}f\|_{B^{s}_{p q}}\\
&\lesssim|t_2-t_1|^{\delta}
\|e^{-t_1(-\Delta)^{\theta/2}}f\|_{B^{s+\theta\delta}_{p q}}\\
&\lesssim|t_2-t_1|^{\delta}t_1{}^{-\eta}
\|f\|_{B^{s+\theta\delta-\theta\eta}_{p q}}.
\end{align*}
\end{proof}

In Lemmas \ref{lm:171102-1} and \ref{lm:171102-2}
we stated our results in terms of the Besov space
$B^s_{p q}$ with $1 \le p,q \le \infty$ and $s \in {\bf R}$.
Here we record the results which we actually use. 
We present results on smoothing effects of semigroup 
$\{e^{-t(-\Delta)^{\frac{\theta}{2}}}\}_{t\geq 0}$
on H\"{o}lder-Besov spaces.
\begin{proposition}[Effects of heat semigroup]\label{prop_20160927051055}
	Let $\alpha\in\RealNum$ and $0<\theta \le 2$.
	\begin{enumerate}
		\item[$(1)$]	Let $\delta \ge 0$. 
Then for all $f \in \HolBesSp{\alpha}$
				$
					\|e^{-t (-\LaplaceOp)^{\theta/2}} f\|_{\HolBesSp{\alpha+\theta\delta}}
					\lesssim
						t^{-\delta}
						\|f\|_{\HolBesSp{\alpha}}
				$
				uniformly in $t>0$.
		\item[$(2)$]
		Let $\delta\in[0,1]$. Then 
for all $f \in \HolBesSp{\alpha}$
				$
\|(e^{-t (-\LaplaceOp)^{\theta/2}}-1)f\|_{\HolBesSp{\alpha-\delta\theta}}
					\lesssim
						t^{\delta}
						\|f\|_{\HolBesSp{\alpha}}
				$
				uniformly in $t>0$.
				 	\end{enumerate}
\end{proposition}

%%%%%%%%%%%%%%%%%%%%%%%%%%%%%%%%%

%%%%%%%%%%%%%%%%%%%%%%%%%%%%%%%%%

%

Let ${\mathcal A}$ be an operator acting on ${\mathcal D}'$
and let $f \in {\mathcal D}'$.
		Define the commutator $[{\mathcal A}, f\rpara]$
generated by paraproduct and ${\mathcal A}$ by
		\begin{align*}
			[{\mathcal A},f\rpara]g
			=
				{\mathcal A}(f\rpara g)
				-
				f\rpara {\mathcal A}g
		\end{align*}
for $g \in {\mathcal D}'$ whenever this is well defined.
	\begin{proposition}[Commutators generated by paraproduct and heat semigroup]
		\label{prop_20160927051159}
		We let $\alpha<1$, $\beta\in\RealNum$, 
and let
		$0<\theta \le 2$ and $a<1$.
		Then
$$\|
				[e^{-t (-\LaplaceOp)^{\theta/2}},f\rpara]g
			\|_{\HolBesSp{\alpha+\beta-a\theta}}
			\lesssim
				t^a
				\|f\|_{\HolBesSp{\alpha}}
				\|g\|_{\HolBesSp{\beta}}
		$$
for $f \in \HolBesSp{\alpha}$ and $g \in \HolBesSp{\beta}$
		uniformly over $t>0$.
	\end{proposition}

\begin{proof}
Let $j \ge 5$. We disregard the lower frequency terms
because we can incorporate them later easily.
Let 
\[
K_{j,t}^\dagger(x)
=\int_{{\mathbf R}^d}\exp(-t|2\pi\xi|^\theta)(\rho_{j-1}(\xi)+\rho_{j}(\xi)+\rho_{j+1}(\xi))\exp(2\pi\sqrt{-1}x \cdot \xi)\,d\xi
\]
for $j \in {\mathbf N}$ and $x \in {\mathbf R}^d$.
We observe that 
$e^{-t (-\LaplaceOp)^{\theta/2}}\LPBlock[j-1]
+ e^{-t (-\LaplaceOp)^{\theta/2}} \LPBlock[j]
+e^{-t (-\LaplaceOp)^{\theta/2}} \LPBlock[j+1]$
has a kernel $K_{j,t}^\dagger$, that is,
for some suitable constant $c_d$,
\begin{align*}
\displaystyle \sum\limits_{k=j-1}^{j+1}
\LPBlock[k]
(e^{-t (-\LaplaceOp)^{\theta/2}}(f \rpara g)-f \rpara e^{-t (-\LaplaceOp)^{\theta/2}} g)
&=c_d
\int_{{\mathbf R}^d}K_{j,t}^\dagger(x-y)(S_jf(y)-S_jf(x)) \LPBlock[j] g(y)\,dy.
\end{align*}
Let $x,y \in {\mathbf R}^d$.
By the mean value theorem, we have
\[
|S_jf(y)-S_jf(x)|
\le 
\|\nabla S_j f\|_{L^\infty(\Torus^d)}|x-y|
=
2^{-j}\|\nabla S_j f\|_{L^\infty(\Torus^d)}
\cdot
2^j|x-y|.
\]
We notice
\[
K_{j,t}^\dagger(x)
=
\frac{1}{t^{d/\theta}}\int_{{\mathbf R}^d}\exp(-|2\pi\xi|^\theta)
(\rho_{j+1}(t^{-1/\theta}\xi)+\rho_{j}(t^{-1/\theta}\xi)+\rho_{j-1}(t^{-1/\theta}\xi))
\exp(2\pi\sqrt{-1} t^{-1/\theta}x \cdot \xi)\,d\xi.
\]
We set
\begin{align*}
K_{j,t,-}^\dagger(x)&:=
K_{j,t}^\dagger(x)
-
\frac{1}{t^{d/\theta}}
\sum\limits_{l=-1}^1\int_{{\mathbf R}^d}\rho_{-j+l}(t^{-1/\theta}\xi)
\exp(2\pi\sqrt{-1}t^{-1/\theta}x \cdot \xi)\,d\xi.
\end{align*}
If $2^{-j}t^{-1/\theta} \ge 1$,
then for all $N=0,1,\ldots$
\begin{align*}
|K_{j,t,-}^\dagger(x)|
&\lesssim_N
|t^{-1/\theta}x|^{-N}
(2^{-j}t^{-1/\theta})^{-d-\theta+N}
\end{align*}
from (\ref{eq:171026-5}).
In particular,
by considering the case of $N(N-d-1)=0$,
\begin{align*}
|K_{j,t,-}^\dagger(x)|
&\lesssim
t^{-d/\theta}
(2^{-j}t^{-1/\theta})^{-d-\theta}
\min(1,|t^{-1/\theta}x|^{-d-1}
(2^{-j}t^{-1/\theta})^{d+1}).
\end{align*}
Consequently,
\begin{align*}
\|K_{j,t,-}^\dagger\|_{L^1({\mathbf R}^d)}
&\lesssim
t^{-d/\theta}
	\int_{{\mathbf R}^d}
(2^{-j}t^{-1/\theta})^{-d-\theta}
\min(1,|t^{-1/\theta}x|^{-d-1}
(2^{-j}t^{-1/\theta})^{d+1})\,dx\\
&=
t^{-d/\theta}
	\int_{{\mathbf R}^d}
(2^{-j}t^{-1/\theta})^{-d-\theta}
\min(1,|2^{j}x|^{-d-1})\,dx\\
&=2^{j\theta}
t
	\int_{{\mathbf R}^d}
\min(1,|x|^{-d-1})\,dx\\
&\simeq 2^{j\theta}t.
\end{align*}
Thus we have
\begin{align*}
\lefteqn{
\|( \LPBlock[j -1] + \LPBlock[j] + \LPBlock[j+1])
(e^{-t (-\LaplaceOp)^{\theta/2}}(f \rpara g)-f \rpara e^{-t (-\LaplaceOp)^{\theta/2}}g)\|_{L^\infty(\Torus^d)}
}\\
&=
\|
( \LPBlock[j -1] + \LPBlock[j] + \LPBlock[j+1])
((e^{-t (-\LaplaceOp)^{\theta/2}}-1)(f \rpara g)
 -f \rpara (e^{-t (-\LaplaceOp)^{\theta/2}}-1) g)\|_{L^\infty(\Torus^d)}\\
&\lesssim
2^{-j(\alpha+\beta-\theta)}t\|2^{j\alpha}\nabla S_j f\|_{L^\infty(\Torus^d)}\|g\|_{{\mathcal C}^\beta}\\
&\lesssim
2^{-j(\alpha+\beta-\theta)}t\|f\|_{{\mathcal C}^\alpha}\|g\|_{{\mathcal C}^\beta}\\
&\le
2^{-j(\alpha+\beta-a\theta)}t^a\|f\|_{{\mathcal C}^\alpha}\|g\|_{{\mathcal C}^\beta}.
\end{align*}
Here for the last inequality, we used $a<1$.
If $2^{-j}t^{-1/\theta} \le 1$,
then the situation is simpler;
we simply use
\[
|K^\dagger_{j,t}(x)| \le 2^{j d}\exp(-2^{j\theta}t)(1+|2^j x|)^{-d-1}
\quad (x \in {\mathbf R}^d).
\]
\end{proof}

Let $0<\theta \le 2$.
We set 
$P^{\theta/2}_t=e^{-t ((-\LaplaceOp)^{\theta/2} +1)}= e^{-t} e^{-t (-\LaplaceOp)^{\theta/2}}$.
If $f$ is given by the Fourier series
\[
f=\sum\limits_{k \in {\mathbf Z}^d} a_k {\bf e}_k,
\]
then
\[
P^{\theta/2}_t f
=\sum\limits_{k \in {\mathbf Z}^d}e^{-t \{(2\pi |k|)^{\theta}+1\}} a_k {\bf e}_k.
\]

\begin{remark}
Since $e^{-t} \le 1$,
the results on the semigroup in this subsection
(Lemmas \ref{lm:171102-1}, \ref{lm:171102-2}, \ref{lm:171102-3},
Propositions \ref{prop_20160927051055}, \ref{prop_20160927051159}) 
still hold with possibly different positive constants 
even if $e^{-t (-\LaplaceOp)^{\theta/2}}$ is replaced by $P^{\theta/2}_t$.
\end{remark}

%%%%%%%%%%%
%%%%%%%%%%%%%%%%%%%%%% 
%\newpage
\subsection{Besov space-valued function spaces and
the fractional version of the Schauder estimate}
%%%%%%%%%%%%%%%%%%%%%%%%%%
In this subsection we consider $\HolBesSp{\alpha}$-valued functions in time
and prove the fractional version of the Schauder estimate,
which will play a key role in solving our QGE equation.

We introduce several function spaces as follows:
\begin{definition}\label{def_20171101}
Let $T>0$, $0 < \theta \le 2$, $\alpha \in \RealNum$, $\eta \ge 0$ and $\delta \in (0,1]$.
\begin{itemize}
\item $\SmoothFunc{}{(0,T]}{\HolBesSp{\alpha}}$
is defined to the space of all continuous functions
defined on $(0,T]$
which assume its value in $\HolBesSp{\alpha}$.
	\item	$C_T\HolBesSp{\alpha}$ is the space of all continuous functions 
$u \in \SmoothFunc{}{(0,T]}{\HolBesSp{\alpha}}$
for which the norm
			\begin{align*}
				\|u\|_{C_T\HolBesSp{\alpha}}
				=
					\sup\limits_{0< t\leq T}
						\|u_t\|_{\HolBesSp{\alpha}}
			\end{align*}
is finite.
	\item	$C_T^\delta\HolBesSp{\alpha}$ is the subspace of all $\delta$-H\"older continuous functions
$u \in \SmoothFunc{}{(0,T]}{\HolBesSp{\alpha}}$ for which the seminorm 
			\begin{align*}
				\|u\|_{C_T^\delta\HolBesSp{\alpha}}
				=
					\sup\limits_{0< s<t\leq T}
						\frac{\|u_t-u_s\|_{\HolBesSp{\alpha}}}{(t-s)^\delta}
			\end{align*}
is finite.
Given $u \in C_T^\delta\HolBesSp{\alpha}$,
define $u_0 \in \HolBesSp{\alpha}$
uniquely by continuity.
	\item	$
				\cE_T^\eta\HolBesSp{\alpha}
				=
					\{
						u\in \SmoothFunc{}{(0,T]}{\HolBesSp{\alpha}};
						\|u\|_{\cE_T^\eta\HolBesSp{\alpha}}<\infty
					\}
			$,
			where
			\begin{align*}
				\|u\|_{\cE_T^\eta\HolBesSp{\alpha}}
				=
					\sup\limits_{0<t\leq T}
						t^\eta\|u_t\|_{\HolBesSp{\alpha}}.
			\end{align*}
	\item	$
				\cE_T^{\eta,\delta}\HolBesSp{\alpha}
				=
					\{
						u\in \SmoothFunc{}{(0,T]}{\HolBesSp{\alpha}};
						\|u\|_{\cE_T^{\eta,\delta}\HolBesSp{\alpha}}<\infty
					\}
			$,
			where
			\begin{align*}
				\|u\|_{\cE_T^{\eta,\delta}\HolBesSp{\alpha}}
				=
					\sup\limits_{0<s<t\leq T}
						s^\eta
						\frac{\|u_t-u_s\|_{\HolBesSp{\alpha}}}{(t-s)^\delta}.
			\end{align*}
	\item	
We define
$
				\cL_T^{\alpha,\delta}
				=
					C_T\HolBesSp{\alpha}
					\cap
					C_T^\delta\HolBesSp{\alpha- \theta\delta}
			$,
where the norm of the left-hand side is given
by the intersection norm.
	\item	
We define
$
				\cL_T^{\eta,\alpha,\delta}
				=
					\cE_T^\eta\HolBesSp{\alpha}
					\cap
					\cE_T^{\eta,\delta}\HolBesSp{\alpha- \theta \delta}
					\cap
					C_T\HolBesSp{\alpha- \theta \eta}
			$,
where the norm of the left-hand side is given
by the intersection norm.
\end{itemize}
\end{definition}

Lemmas \ref{lm:171102-1}
and
\ref{lm:171102-3}
correspond to
$
				\cE_T^\eta\HolBesSp{\alpha}$
and
$
				\cE_T^{\eta,\delta}\HolBesSp{\alpha}$,
respectively.

\begin{remark}
	We introduced the norms on the spaces $\cE_T^\eta\HolBesSp{\alpha}$ and $\cE_T^{\eta,\delta}\HolBesSp{\alpha}$
	in order to control explosion at $t=0$.
	The definitions of $\cL_T^{\alpha,\delta}$ and $\cL_T^{\eta,\alpha,\delta}$
	 look natural from the viewpoint of the time-space scaling of 
	the fractional heat operator $\partial_t +( - \LaplaceOp)^{\theta/2}$.
	 Note that $\cL_T^{\alpha,\delta}$ and $\cL_T^{\eta,\alpha,\delta}$
	 actually 
	 depend on $\theta$, while the other spaces in Definition \ref{def_20171101} do not.
	 \end{remark}

To understand the structure of the above function spaces,
we prove the following embedding properties:
\begin{proposition}\label{rem_20171114-2}
Let $T>0$.
Let 
$\alpha \in {\mathbf R}$,
$0<\theta \le 2$,
$\eta>0$
and
$0<\delta'\le \delta<\infty$.
	\begin{enumerate}
		\item[$(1)$]
We have
$					\cL_T^{\alpha,\delta}
=
{\mathcal C}_T\HolBesSp{\alpha}
\cap
{\mathcal C}_T^\delta\HolBesSp{\alpha- \theta\delta}
					\subset
					\cL_T^{\alpha,\delta'}
=
{\mathcal C}_T\HolBesSp{\alpha}
\cap
{\mathcal C}_T^{\delta'}\HolBesSp{\alpha- \theta\delta'}.
$
		\item[$(2)$]
We have
$\cL_T^{\eta,\alpha,\delta}
=
{\mathcal E}_T^\eta {\mathcal C}^{\alpha}
\cap 
{\mathcal E}_T^{\eta,\delta}{\mathcal C}^{\alpha-\theta\delta}
\cap
C_T{\mathcal C}^{\alpha-\theta\eta}
					\subset
					\cL_T^{\eta,\alpha,\delta'}
=
{\mathcal E}_T^\eta {\mathcal C}^{\alpha}
\cap 
{\mathcal E}_T^{\eta,\delta'}{\mathcal C}^{\alpha-\theta\delta'}
\cap
C_T{\mathcal C}^{\alpha-\theta\eta}.
$
		\item[$(3)$]
Let $\gamma \in [\alpha-\theta\eta,\alpha]$.
Then $v_t\in\HolBesSp{\gamma}$ and
$
					\|v_t\|_{\HolBesSp{\gamma}}
					\lesssim
						t^{-\frac{\gamma-\alpha+\theta\eta}{\theta}}
						\|v\|_{\cL_T^{\eta,\alpha,\delta}}
$
for every $v\in\cL_T^{\eta,\alpha,\delta}$
and every $0<t\leq T$.
	\end{enumerate}
\end{proposition}

\begin{proof}
In principle,
the proof hinges on the interpolation;
see Proposition \ref{prop_20160928054620}(2).
The proof of each statement in
Proposition \ref{rem_20171114-2}
is made up of two steps:
the first step is to deduce endpoint inequalities.
We consider the case $\lq \lq \delta'=0"$ and $\delta'=\delta$
in
(1) and (2).
In (3)
we consider the cases
$\gamma=\alpha-\theta\delta$ and $\gamma=\alpha$.
The
second step is to interpolate between them.

Based on the general outline,
we provide a self-contained proof.
\begin{enumerate}
\item
It suffices to show
$
{\mathcal C}_T\HolBesSp{\alpha}
\cap
{\mathcal E}_T^\delta\HolBesSp{\alpha- \theta\delta}
					\subset
{\mathcal E}_T^{\delta'}\HolBesSp{\alpha- \theta\delta'}.
$
				Set $\nu=\delta'/\delta \in (0,1)$, 
so that
$\alpha-\theta\delta'=(\alpha- \theta\delta)\nu+\alpha(1-\nu)$.
Let $0<s<t \le T$.
				For every $v\in\cL_T^{\alpha,\delta}$, we have
				\begin{align*}
					\|v_t-v_s\|_{\HolBesSp{\alpha- \theta\delta'}}
					&\le
						\|v_t-v_s\|_{\HolBesSp{\alpha- \theta\delta}}^\nu
						\|v_t-v_s\|_{\HolBesSp{\alpha}}^{1-\nu}
\end{align*}
thanks to Proposition \ref{prop_20160928054620}(2).
By the definition of the norm
of $\HolBesSp{\alpha- \theta\delta'}$,
we have
\begin{align*}
					\|v_t-v_s\|_{\HolBesSp{\alpha- \theta\delta'}}
					&\le
						\{(t-s)^\delta
						\|v\|_{C_T^\delta\HolBesSp{\alpha- \theta\delta}}\}^\nu
						\|v\|_{C_T\HolBesSp{\alpha}}^{1-\nu}\\
					&\le
						(t-s)^{\delta'}
						\|v\|_{\cL_T^{\alpha,\delta}}.
				\end{align*}
\item
It suffices to show
$
\cL_T^{\eta,\alpha,\delta}
=
{\mathcal E}_T^\eta {\mathcal C}^{\alpha}
\cap 
{\mathcal E}_T^{\eta,\delta}{\mathcal C}^{\alpha-\theta\delta}
\cap
C_T{\mathcal C}^{\alpha-\theta\eta}
					\subset
{\mathcal E}_T^{\eta,\delta'}{\mathcal C}^{\alpha-\theta\delta'}.
$
Fix $v\in\cL_T^{\eta,\alpha,\delta}$ 
				 and $0<s<t\leq T$.
We list two inequalities which we interpolate between.
				\begin{gather*}
					\|v_t-v_s\|_{\HolBesSp{\alpha- \theta\delta}}
					\leq
						s^{-\eta}(t-s)^\delta
						\|v\|_{\cE_T^{\eta,\delta}\HolBesSp{\alpha -\theta\delta}},\\
					\|v_t-v_s\|_{\HolBesSp{\alpha}}
					\leq
						\|v_t\|_{\HolBesSp{\alpha}}
						+\|v_s\|_{\HolBesSp{\alpha}}
					\leq
						t^{-\eta}\|v\|_{\cE_T^\eta\HolBesSp{\alpha}}
						+s^{-\eta}\|v\|_{\cE_T^\eta\HolBesSp{\alpha}}
					\leq
						2s^{-\eta}\|v\|_{\cE_T^\eta\HolBesSp{\alpha}}.
				\end{gather*}
				Hence, for $0<\nu=\delta'/\delta<1$,
thanks to Proposition \ref{prop_20160928054620}(2) we have
				\begin{align*}
					\|v_t-v_s\|_{\HolBesSp{\alpha- \theta\delta'}}
					&\leq
						\|v_t-v_s\|_{\HolBesSp{\alpha- \theta\delta}}^\nu
						\|v_t-v_s\|_{\HolBesSp{\alpha}}^{1-\nu}\\
					&\lesssim
						\{
							s^{-\eta}
							(t-s)^\delta
							\|v\|_{\cE_T^{\eta,\delta}\HolBesSp{\alpha -\theta\delta}}
						\}^\nu
						\{
							s^{-\eta}
							\|v\|_{\cE_T^\eta\HolBesSp{\alpha}}
						\}^{1-\nu}\\
					&\lesssim
						s^{-\eta}
						(t-s)^{\delta'}
						\|v\|_{\cL_T^{\eta,\alpha,\delta}},
				\end{align*}
as required.
\item
Let
$v \in \cL_T^{\eta,\alpha,\delta}$.
Since
$v \in {\mathcal E}_T^\delta\HolBesSp{\alpha-\theta\delta}$,
we have
$\|v_t\|_{\HolBesSp{\alpha-\theta\delta}}
\le t^{-\delta}\|v\|_{C_T{\mathcal C}^{\alpha-\theta\eta}}$.
Since
$v \in {\mathcal E}^\eta_T\HolBesSp{\alpha}$,
we have
$\|v_t\|_{\HolBesSp{\alpha}} \le t^{-\eta}\|v\|_{ {\mathcal E}^\eta_T\HolBesSp{\alpha}}$.
				Take $\nu \in [0,1]$ such that
				$
					\gamma
					=
						(\alpha-\theta\eta)(1-\nu)
						+\alpha\nu
				$
				and use Proposition \ref{prop_20160928054620}(2)
once again to obtain
				\begin{align*}
					\|v_t\|_{\HolBesSp{\gamma}}
					\leq
						\|v_t\|_{\HolBesSp{\alpha- \theta\eta}}^{1-\nu}
						\|v_t\|_{\HolBesSp{\alpha}}^\nu
					\leq
						\|v\|_{C_T\HolBesSp{\alpha- \theta\eta}}^{1-\nu}
						\{t^{-\eta}\|v\|_{\cE_T^\eta\HolBesSp{\alpha}}\}^\nu
					\leq
						t^{-\eta\nu}
						\|v\|_{\cL_T^{\eta,\alpha,\delta}}.
				\end{align*}
				Noting that
				$
					\eta\nu
					=
						\frac{\gamma-(\alpha-\theta\eta)}{\theta}
				$,
				we obtain the desired result.
\end{enumerate}
\end{proof}

In the course of the proof of (2) we proved:
\begin{corollary}\label{cor:180131-2}
Let $T>0$.
Let 
$\alpha \in {\mathbf R}$,
$0<\theta \le 2$,
$\eta>0$
and
$0<\delta'\le \delta<\infty$.
Then
for every $v\in\cL_T^{\eta,\alpha,\delta}$,
and $0<s \le t \le T$,
$
					\|v_t-v_s\|_{\HolBesSp{\alpha- \theta\delta'}}
					\lesssim
						s^{-\eta}
						(t-s)^{\delta'}
						\|v\|_{\cL_T^{\eta,\alpha,\delta}}
$.
\end{corollary}

%\vspace{10mm}

Now we give a fractional version of the Schauder estimate.
We define 
\begin{equation}\label{eq_20171101}
I[u]_t
						=
							\int_0^t
								P^{\theta /2}_{t-s}
								u_s\,
								ds
\end{equation}
whenever the right-hand side makes sense
(dependency on $\theta$ is suppressed).
We denote by ${\mathcal B}(x,y)$, $x,y>0$ the beta function.
We will repeatedly use the following fact:
\begin{proposition}\label{prop:180131-1}
Let $a_1, a_2 \in [0,1)$ and $t>0$.
Then
$$\displaystyle
 \int_0^t s^{- a_1} 
 (t-s)^{- a_2 } ds=
{\mathcal B}(1-a_1,1-a_2)t^{1 - a_1- a_2}
\simeq
t^{1 - a_1- a_2}.
$$ 
\end{proposition}

\begin{proposition}
[Fractional Schauder estimates]\label{prop_20170804}
		Let $\alpha,\beta,\gamma,\delta\in {\mathbf R}$, $\theta \in (0,2]$,
$\eta\in[0,1)$ and $T \in (0,1]$.
		\begin{enumerate}
			\item[$(1)$]	Let $f \in\HolBesSp{\alpha}$.
Then 
$
						\|P^{\theta/2}_\cdot f
						\|_{\cL_T^{\frac{\beta-\alpha}{\theta},\beta,\delta}}
						\lesssim
							\|f\|_{\HolBesSp{\alpha}}
$
					for every $\alpha<\beta$ and $\delta\in[0,1]$.
			\item[$(2)$]
Assume
\[
\alpha <\beta,\quad
\alpha \le \gamma<\alpha-\theta\eta+\theta, \quad
\gamma \le \beta <\alpha+\theta, \quad
0<\delta\le\frac{\beta-\alpha}{\theta}.
\]
					Then
$
						\|I [u ]\|_{\cL_T^{\frac{\beta-\gamma}{\theta},\beta,\delta}}
						\lesssim
							T^{\frac{\alpha- \theta\eta+ \theta-\gamma}{\theta}}
							\|u\|_{\cE_T^\eta\HolBesSp{\alpha}}
					$
for all					 $u\in\cE_T^\eta\HolBesSp{\alpha}$.
		\end{enumerate}
	\end{proposition}

\bigskip
\begin{proof}
Recall that
$
\cL_T^{\frac{\beta-\alpha}{\theta},\beta,\delta}=
					\cE_T^{ \frac{\beta-\alpha}{\theta}} \HolBesSp{\beta}
					\cap
					C_T\HolBesSp{\alpha}
					\cap
					\cE_T^{\frac{\beta-\alpha}{\theta},\delta}
					 \HolBesSp{\beta-\theta\delta}
					$
					by definition.
\begin{enumerate}
\item				%
				Since $\beta = \alpha + \theta \cdot \frac{\beta-\alpha}{\theta}>\alpha$ and $0<T \le 1$,	
				we have
\begin{align}\label{eq:180115-1}
\| P^{\theta /2}_{\cdot} f\|_{ \cE_T^{ \frac{\beta-\alpha}{\theta}} \HolBesSp{\beta}} 
=
\sup\limits_{0< t \le T} t^{ \frac{\beta-\alpha}{\theta}} \| P^{\theta /2}_{t} f \|_{\HolBesSp{\beta} }
\lesssim
\|f\|_{\HolBesSp{\alpha} }
 \end{align}
by Lemma \ref{lm:171102-1}.
The estimate
$
\|P^{\theta /2}_{\cdot} f\|_{C_T\HolBesSp{\alpha}}
\lesssim
\|f\|_{\HolBesSp{\alpha}}
$
is also a consequence of Lemma \ref{lm:171102-1}.
Finally, when $T \ge t>s>0$,
we deduce from Lemma \ref{lm:171102-3}
that
\begin{align*}
\| e^{-t (- \LaplaceOp)^{\theta /2} } f 
 - e^{-s (- \LaplaceOp)^{\theta /2} } f \|_{ \HolBesSp{\beta -\theta \delta}} 
&\lesssim
(t-s)^{\delta} s^{- \frac{\beta-\alpha}{\theta}} 
 \| f \|_{\HolBesSp{\alpha} }.
 \end{align*}
 From this, we can easily see that
 $\| P^{\theta /2}_t f -P^{\theta /2}_s f \|_{ \HolBesSp{\beta -\theta \delta}} $ 
 satisfies the same estimate with a different constant.
Therefore, $P^{\theta /2}_{\cdot} f \in \cE_T^{\frac{\beta-\alpha}{\theta},\delta}
					 \HolBesSp{\beta-\theta\delta}$.
					 Thus, we have shown (1).
\item
We check that $I [u] \in C_T\HolBesSp{\gamma}$. Let $0 \le t \le T$.
By Lemma \ref{lm:171102-1}, 
\begin{align*}
 \| I [u ]_t \|_{\HolBesSp{\gamma} }
\le
\int_0^t \| e^{-(t-s) (- \LaplaceOp)^{\theta /2} } u_s \|_{\HolBesSp{\gamma} } ds
\lesssim
\int_0^t (t-s)^{\frac{\alpha - \gamma}{\theta}} \| u_s \|_{\HolBesSp{\alpha} } ds.
\end{align*}
Using Proposition \ref{prop:180131-1},
$0 \le t \le T$
and
$\gamma\in[\alpha,\alpha-\theta\eta+\theta)$,
we obtain
\begin{align*}
 \| I [u ]_t \|_{\HolBesSp{\gamma} }
\lesssim
\|u\|_{\cE_T^\eta\HolBesSp{\alpha}}
 \int_0^t (t-s)^{\frac{\alpha - \gamma}{\theta}} s^{-\eta} ds
\sim
 t^{\frac{\alpha- \theta\eta+ \theta-\gamma}{\theta}} 
 \|u\|_{\cE_T^\eta\HolBesSp{\alpha}}
\le
 T^{\frac{\alpha- \theta\eta+ \theta-\gamma}{\theta}} 
 \|u\|_{\cE_T^\eta\HolBesSp{\alpha}}.
 \end{align*}
In a similar way, 
using Lemma \ref{lm:171102-1} and Proposition \ref{prop:180131-1}
as well as the assumptions $0 \le t \le T$,
$\beta<\alpha+\theta$ and $\eta<1$
we obtain
\begin{align*}
t^{\frac{\beta - \gamma}{\theta}} \| I[u]_t \|_{\HolBesSp{\beta} }
&\le
t^{\frac{\beta - \gamma}{\theta}}
\int_0^t \| e^{-(t-s) (- \LaplaceOp)^{\theta /2} } u_s \|_{\HolBesSp{\beta} } ds\\
&\le
t^{\frac{\beta - \gamma}{\theta}}
\int_0^t (t-s)^{ -\frac{\beta -\alpha }{\theta}}\| u_s \|_{\HolBesSp{\alpha} } ds\\
&\lesssim
t^{\frac{\beta - \gamma}{\theta}}
\|u\|_{\cE_T^\eta\HolBesSp{\alpha}}
 \int_0^t (t-s)^{ -\frac{\beta -\alpha }{\theta}} s^{-\eta} ds
 \\
 &\simeq
 t^{\frac{\alpha- \theta\eta+ \theta-\gamma}{\theta}} 
 \|u\|_{\cE_T^\eta\HolBesSp{\alpha}}\\
 &\le
 T^{\frac{\alpha- \theta\eta+ \theta-\gamma}{\theta}} 
 \|u\|_{\cE_T^\eta\HolBesSp{\alpha}}.
 \end{align*}
 Hence, $\cE_T^{ \frac{\beta-\gamma}{\theta}} \HolBesSp{\beta}$-norm of 
 $I [u]$ is dominated by a constant multiple of 
 $T^{\frac{\alpha- \theta\eta+ \theta-\gamma}{\theta}} 
 \|u\|_{\cE_T^\eta\HolBesSp{\alpha}}$.

Finally, we estimate $\cE_T^{\frac{\beta-\gamma}{\theta},\delta}
					 \HolBesSp{\beta-\theta\delta}$-norm of $I [u]$.
Again by Lemma \ref{lm:171102-1},
\begin{align*}
\| I[u]_t - I[u]_s \|_{ \HolBesSp{\beta - \theta \delta} }
&=
\Bigl\| 
\int_s^t P^{\theta /2}_{t -r} u_r dr+
\int_0^s (P_{t-s}^{\theta/2}-I)P_{s-r}^{\theta/2}u_r\,dr
\Bigr\|_{ \HolBesSp{\beta - \theta \delta}}\\
&\lesssim
 \|u\|_{\cE_T^\eta\HolBesSp{\alpha}} 
 \int_s^t (t-r)^{\frac{\alpha-\beta+\theta \delta }{\theta}} r^{-\eta} dr+\Bigl\| 
\int_0^s (P_{t-s}^{\theta/2}-I)P_{s-r}^{\theta/2}u_r\,dr
\Bigr\|_{ \HolBesSp{\beta - \theta \delta}}.
\end{align*}
We estimate the second term of the most right-hand side.
By Proposition \ref{prop_20160927051055}(2),
\begin{align*}
\Bigl\| (P_{t-s}^{\theta/2}-I)P_{s-r}^{\theta/2}u_r
\Bigr\|_{ \HolBesSp{\beta - \theta \delta}}
&
\lesssim
(t-s)^\delta\|P_{s-r}^{\theta/2}u_r\|_{{\mathcal C}^\beta}.
\end{align*}
By Lemma \ref{lm:171102-1}
and the definition of the norm
${\mathcal E}^\eta_T{\mathcal C}^\alpha$,
\begin{align*}
\Bigl\| 
\int_0^s (P_{t-s}^{\theta/2}-I)P_{s-r}^{\theta/2}u_r\,dr
\Bigr\|_{ \HolBesSp{\beta - \theta \delta}}
&
\lesssim
(t-s)^\delta
\int_0^s (s-r)^{\frac{\alpha-\beta}{\theta}}\|u_r\|_{{\mathcal C}^\alpha}\,dr\\
&
\le
(t-s)^\delta
\int_0^s (s-r)^{\frac{\alpha-\beta}{\theta}}r^{-\eta}\|u\|_{{\mathcal E}^\eta_T{\mathcal C}^\alpha}\,dr.
\end{align*}
Recall that we are assuming
$0 \le \eta<1$
and
$\gamma <\alpha-\theta\eta+\theta$.
Let $0<s<t \le T$.
If we integrate this estimate against $r$
and use Proposition \ref{prop:180131-1} once again,
we obtain
\begin{align*}
(t-s)^{-\delta}
\Bigl\| 
\int_0^s (P_{t-s}^{\theta/2}-I)P_{s-r}^{\theta/2}u_r\,dr
\Bigr\|_{ \HolBesSp{\beta - \theta \delta}}
&\lesssim
\int_0^s(s-r)^{\frac{\alpha-\beta}{\theta}}r^{-\eta}\|u\|_{{\mathcal E}^\eta_T{\mathcal C}^\beta}\,dr\\
&\sim
s^{1-\eta-\frac{\beta-\alpha}{\theta}}\|u\|_{{\mathcal E}^\eta_T{\mathcal C}^\beta}\\
&\le
T^{\frac{\alpha+\theta-\eta\theta-\gamma}{\theta}}s^{-\frac{\beta-\gamma}{\theta}}\|u\|_{{\mathcal E}^\eta_T{\mathcal C}^\beta}.
\end{align*}
So, we need to estimate the first term of the most right-hand side.
We calculate
\begin{align*}
\lefteqn{
(t-s)^{-\delta} s^{ \frac{\beta-\gamma}{\theta} }
\| I[u]_t - I[u]_s \|_{ \HolBesSp{\beta - \theta \delta} }
}
\\
&\lesssim
(t-s)^{-\delta}
 \int_s^t (t-r)^{ -\frac{\beta -\theta\delta -\alpha }{\theta}} r^{-\eta+\frac{\beta-\gamma}{\theta}} dr
\cdot
 \|u\|_{\cE_T^\eta\HolBesSp{\alpha}}
 \\
&\lesssim
(t-s)^{-\delta}
 \int_s^t (t-r)^{ -\frac{\beta -\theta\delta -\alpha }{\theta}} 
((r-s)^{-\eta+\frac{\beta-\gamma}{\theta}}+t^{-\eta+\frac{\beta-\gamma}{\theta}}) dr
\cdot
 \|u\|_{\cE_T^\eta\HolBesSp{\alpha}}.
\end{align*}
Since
\[
\frac{\beta -\theta\delta -\alpha }{\theta}<1, \quad
\eta<\frac{\alpha-\gamma}{\theta}+1<\frac{\beta-\gamma}{\theta}+1,
\]
we are in the position of using
Proposition \ref{prop:180131-1} to have
\begin{align*}
\lefteqn{
(t-s)^{-\delta} s^{ \frac{\beta-\gamma}{\theta} }
\| I[u]_t - I[u]_s \|_{ \HolBesSp{\beta - \theta \delta} }
}
\\
&\lesssim
(t-s)^{-\delta}
 \int_s^t (t-r)^{ -\frac{\beta -\theta\delta -\alpha }{\theta}} r^{-\eta+\frac{\beta-\gamma}{\theta}} dr
\cdot
 \|u\|_{\cE_T^\eta\HolBesSp{\alpha}}
 \\
&\lesssim
\left(
(t-s)^{-\delta-\frac{\beta -\theta\delta -\alpha }{\theta}+1-\eta + \frac{\beta-\gamma}{\theta}}
+
(t-s)^{-\delta-\frac{\beta -\theta\delta -\alpha }{\theta}+1}t^{-\eta+\frac{\beta-\gamma}{\theta}}
\right)
 \|u\|_{\cE_T^\eta\HolBesSp{\alpha}}\\
&=
\left(
(t-s)^{-\frac{\gamma-\alpha }{\theta}+1- \eta }
+
(t-s)^{-\frac{\beta -\alpha }{\theta}+1}t^{-\eta+\frac{\beta-\gamma}{\theta}}
\right)
 \|u\|_{\cE_T^\eta\HolBesSp{\alpha}}.
\end{align*}
Since
\[
\gamma<\alpha-\theta\eta+\theta, \quad
\beta<\alpha+\theta,
\]
we obtain
\[
(t-s)^{-\frac{\gamma-\alpha }{\theta}+1- \eta }, \quad
(t-s)^{-\frac{\beta -\alpha }{\theta}+1}t^{-\eta+\frac{\beta-\gamma}{\theta}}\lesssim
t^{-\frac{\gamma-\alpha }{\theta}+1- \eta }
\quad (0<s<t<\infty).
\]
Thus
\begin{align*}
(t-s)^{-\delta} s^{ \frac{\beta-\gamma}{\theta} }
\| I[u]_t - I[u]_s \|_{ \HolBesSp{\beta - \theta \delta} }
 \lesssim
 t^{\frac{\alpha- \theta\eta+ \theta-\gamma}{\theta}} 
 \|u\|_{\cE_T^\eta\HolBesSp{\alpha}}.
 \end{align*}
This completes the proof.
\end{enumerate}
 \end{proof}

%%%%%%%%%%%%%%%%%
%%%%%%%%%%%%%%%%%%%%%%%%%%%%%%%%%%%%\newpage
%%%%%%%%%%%%%%%%%
%%%%%%%%%%%%%%%%%%%%%%%%%%%%%%%%
\section{A bird's eye view of our proof}
In this section, for the reader's convenience, 
we give a heuristic explanation 
of how our somewhat long proof is structured. 
Our aim is to take a bird's eye view of the proof of our main theorem. 
Therefore, one should keep in mind that
 arguments in this section are not intended to be rigorous.
The setting, symbols and definitions used in this section 
may not carry over to the core part (Sections 
\ref{s4}--\ref{s8}) of this paper.

To simplify our problem a little bit, 
we assume that $\theta =2$ and $u_0 =0$ in \eqref{eq:180520}.
We only consider the $C_T\HolBesSp{\alpha}$-norm
for continuous paths in a Besov-type function space.
(These simplifications do not destroy the essence of our problem.)
For every $\alpha \in {\bf R}$ and 
$z \in C_T \HolBesSp{\alpha}$, it holds that
$R^{\perp}z \in C_T\HolBesSp{\alpha}$, 
$\nabla z \in C_T\HolBesSp{\alpha -1}$ and 
 $I[z] \in C_T\HolBesSp{\alpha+2^-} 
 := \cap_{\kappa >0} C_T\HolBesSp{\alpha+2 -\kappa} $.
See (\ref{eq_20171101}) for the definition of $I[z]$.
(For simplicity, we will say that 
$z$ and $I[z]$ is of regularity $\alpha$ and $\alpha +2^-$,
respectively, etc.)
Hence,
it turns out that these three operations do not cause a problem.

It is well
known that $X := I[\xi]$ is a Gaussian process 
called 
an
Ornstein-Uhlenbeck process and 
is of regularity $0^-$ a.s., that is, $X \in C_T\HolBesSp{0^-}$ a.s.
(Since we assume $u_0 =0$, we also assume $X_0 \equiv 0$.
In the core part of this paper, the stationary 
Ornstein-Uhlenbeck process is used instead.)
We basically view $X$ instead of $\xi$ 
as a given input datum when we solve \eqref{eq:180520}. 
In the mild form, \eqref{eq:180520} 
with $\theta =2$ and $u_0 =0$ reads:
\begin{equation}\nonumber
u = I [R^{\perp} u \cdot \nabla u] + X.
 \end{equation}
Since $X$ is known, we may change the unknown 
to $\Psi:= u -X$.
Then, the above equation is equivalent to: 
\begin{align}
\Psi 
&= \underbrace{I [R^{\perp} X \cdot \nabla X]}_{=:Y} + 
\underbrace{I [R^{\perp} \Psi \rpara \nabla X] }_{=:v}
\nonumber
\\
&\qquad + \underbrace{
I [R^{\perp} \Psi (\reso+\lpara ) \nabla X] 
+
I [R^{\perp} X \cdot \nabla \Psi] +
I [R^{\perp} \Psi \cdot \nabla \Psi] }_{=:w}.
\label{eq:200413-1}
 \end{align}
Observe that the only possibly ill-defined operations in 
\eqref{eq:200413-1} 
are the usual product and the resonant product.
We will justify the definition of \eqref{eq:200413-1}.

Roughly speaking, our proof is oriented to doing so
and
consists of three steps:
\begin{enumerate}
 \item[(A)]
 Heuristic calculations to find out which symbols 
 made of $X$ alone are necessary and sufficient to 
 formulate \eqref{eq:200413-1} in a deterministic way. 
  This includes deforming \eqref{eq:200413-1} into 
  a more suitable 
  form by using the theory of paraproduct. 
   (One may get a rough sketch of this deforming 
    by reading Section 7 in the opposite direction, i.e.
    ``from right to left".)
     \item[(B)] 
      Proving the local well-posedness of (the new form of)
      the QGE equation for every given deterministic set 
      of such symbols. 
       This step is not so 
       different from the counterpart for usual 
       nonlinear (stochastic) PDEs in the mild form.
        (See Sections 4--6.)
           \item[(C)] 
    Proving that when $X$ is
an Ornstein-Uhlenbeck process, 
     the set of symbols can be constructed in a probabilistic way.
     In this step the theory of Wiener chaos plays a fundamental role.
     (See Section 8.)
     \end{enumerate}

Though cumbersome, Steps (B) and (C) are straightforward
and the reader will not be surprised or puzzled.
On the other hand, many non-experts may wonder
how to carry out Step (A).   
Below, we will loosely explain this point.

First,
take a look at the first term $Y:=I [R^{\perp} X \cdot \nabla X]$ 
on the right hand side of \eqref{eq:200413-1}.
In this case, $R^{\perp} X \cdot \nabla X$ is ill-defined.
However, it is made of $X$ alone. 
In other words, no information of the solution is involved. 
Therefore, we simply add a deterministic datum
which is denoted by $I [R^{\perp} X \cdot \nabla X]$ 
again to the list of input data we have 
at hand when we solve the equation
(so that this term obviously becomes well-defined).
Then, we will later construct this term in the probabilistic sense
(assuming that $X$ is an Ornstein-Uhlenbeck process)
in Step (C).
As for the regularity of the
{\it fictitious} product and the {\it fictitious} resonant product
of this kind,
an ``empirical rule" suggests that it is natural to guess that 
$f g \in \HolBesSp{\alpha \wedge \beta \wedge (\alpha +\beta)}$
and 
 $f\reso g \in \HolBesSp{\alpha +\beta}$
 even when $\alpha +\beta \le 0$ (i.e. these are ill-defined).
Hence, we suppose that the deterministic ``symbol"
$Y=I [R^{\perp} X \cdot \nabla X]$ is of regularity $1^-$. 
(Since we do not know a solution a priori,
we cannot use any other other terms
  in \eqref{eq:200413-1} as our input data.
To make sense of these terms, more ``symbols made of $X$
alone" are needed.)

Next, change the unknown again by 
$\hat\Psi := \Psi - Y$,
and plug it into \eqref{eq:200413-1}.
Then, we have a new form of equation for $\hat\Psi$,
in which the ill-defined terms 
$R^{\perp} Y \cdot \nabla X$ and $R^{\perp} X \cdot \nabla Y$
appear.
So, we must also assume these symbols 
of regularity $0^-$ a priori.
We will continue this kind of heuristic argument until 
 we find a well-defined formulation of our QGE.
 (This is non-trivial, but we do not give details here.)
The list of seven symbols necessary and sufficient 
are as follows
(see Definition \ref{def.200415} and Remark \ref{rem.200413}):
\[
{\bf X} :=
(X, V:= I [\nabla X], \,\,  Y, 
\,\,
R^\perp Y \reso \nabla X,
\,\,
R^\perp V \reso \nabla X,
\,\,
\nabla Y \cdot R^\perp X, 
 \,\,
\nabla V \reso R^\perp X) 
\]
Such a vector of symbols is called a driver.
According to the above rule for fictitious (resonant) products,
the regularity of a driver should be 
$(0^-, 1^-, 1^-, 0^-, 0^-, 0^-, 0^-)$.
A driver is defined to be 
 a (deterministic) element  
in the product of seven Banach spaces 
$C_T \HolBesSp{0^-} \times \cdots \times C_T\HolBesSp{0^-}$ with a certain reasonable algebraic constraint 
(see Definition \ref{def.200415}).

Now we turn to (the regularity of) a solution.
We would like make sense
of \eqref{eq:200413-1} for a given ${\bf X}$ as a fixed point 
problem in an appropriate Banach space.
After some heuristic calculations, 
we find that 
the term with the worst regularity 
 on the right hand side of \eqref{eq:200413-1} 
  (or in its deformation) is $Y$. 
Hence, it is natural to guess that the regularity of $\Psi$
is also $1^-$. 
The second worst term is $v:= I [R^{\perp} \Psi \rpara \nabla X]$,
which is equal to $R^{\perp} \Psi \rpara V + \mbox{(a nice term)}$
by Proposition \ref{prop_20160927051159}
and hence has regularity $1^-$.
Let us define $w:= \Psi - Y -v$. 
Then, it turns out that
$w$ is actually well defined and has regularity $2^-$.
We do not give a full proof of this fact here since it is too long.
To illustrate this, however, 
let us observe that $I [R^{\perp} \Psi \reso \nabla X]$ is a
well-defined
term with regularity $2^-$ if we suppose 
that the regularity of $w$ is $2^-$.
Since $\Psi = Y+v+w$, we have
\begin{align}
I [R^{\perp} \Psi \reso \nabla X]
&=
I [R^{\perp} Y \reso \nabla X]
+
I [R^{\perp} (R^{\perp} \Psi \rpara V ) \reso \nabla X]
\nonumber\\
&\qquad 
+ 
I [R^{\perp} ( \mbox{``nice"}) \reso \nabla X]
+I [R^{\perp} w \reso \nabla X].
\nonumber
\end{align}
The first term is clearly well-defined with regularity $2^-$
since the symbol inside $I$
is included in the definition of ${\bf X}$
to handle.
So are the third and the fourth terms, because 
$w$ and $\mbox{``nice"}$ are regular enough.
The second term is not so easy.
By Lemma \ref{lem:161117-107}
and Proposition \ref{prop_comm_20160919055939},
we see that
\begin{align*}
R^{\perp} (R^{\perp} \Psi \rpara V ) \reso \nabla X
&=
\{ (R^{\perp} \Psi \rpara R^{\perp} V ) 
 + \mbox{``nice"} \} \reso \nabla X
 \\
 &=
 R^{\perp} \Psi \cdot (R^{\perp} V \reso \nabla X)
 + C (R^{\perp} \Psi, R^{\perp} V, \nabla X )
 + \mbox{``nice"}\reso \nabla X
 \end{align*}
The second term above is called the commutator 
and has regularity $1^-$.
$R^{\perp} V \reso \nabla X$ is one of the components 
of ${\bf X}$ with regularity $0^-$, 
while the regularity of $R^{\perp} \Psi$ is $1^-$.
Hence, the first term is well-defined with regularity $0^-$.
From this we see that
$I [R^{\perp} \Psi \reso \nabla X]$
has a well-defined expression with regularity $2^-$
in terms of ${\bf X}$, $v$ and $w$.
So do the other terms in the definition of $w$
(though we do not illustrate them here).
Therefore, $\Psi -Y-v$ in \eqref{eq:200413-1}
can also be deformed to a well-defined form 
with regularity $2^-$ in terms of ${\bf X}$, $v$ and $w$,
which will be denoted by $G(v,w)$ in Subsection 4.3.

Combining the above considerations and educated guesses,
we arrive at the following systems of mild PDEs for $(v,w)$
for every ${\bf X}$:
\[
\begin{cases}
				\partial_t v
				=R^{\perp} (Y+v+w) \rpara \nabla X ,\\
				\partial_t w
				=
				 G(v,w).
			\end{cases}
			\]
We will adopt this kind of system 
as our definition of paracontrolled QGE 
in Definition \ref{eq_20160921004656}.
Observe that QGE is now decomposed into 
a system of two equations.
The first equation is linear, but its regularity is bad.
The second equation has complicated 
nonlinear terms, but its regularity is nice.

Before closing 
this section, we make two comments 
on the above heuristic ``algebraic" deformation of QGE.
First, if $X$ is regular enough, the above deformation
is of course rigorous. A prominent example of 
such $X$ is a sample path of 
a mollified Ornstein-Uhlenbeck process.

The second one is as follows.
There are two types of ill-definedness in \eqref{eq:200413-1}
(below, $\star$ is either the usual product or $\reso$):
\begin{enumerate}
 \item[(a):]~
 $\mbox{(a quantity that depends on a solution) $\star$ (a symbol made of $X$)}$,
    \item[(b):]~
   $\mbox{(a symbol made of $X$) $\star$ (a symbol made of $X$)}$.
        \end{enumerate} 
        For example, $Y$ is of Type (b), 
         while $I [R^{\perp} \Psi \reso \nabla X]$ and 
$I [R^{\perp} X \cdot \nabla \Psi] $ are of Type (a).
The main aim of the deformation is to convert the
ill-definedness of Type (a) to that of Type (b).
(For example, in our quick heuristic calculation above, 
we have $I [R^{\perp} \Psi \reso \nabla X]$ at first,
but we have instead $R^{\perp} Y \reso \nabla X$ 
and $R^{\perp} V \reso \nabla X$ at the end.) 
The reason is as follows.
Since we do not know a solution a priori, 
we have to use general theorems on the (resonant) product 
of two distributions in order to make sense of the 
(resonant) product in Type (a).
However, these general theorems are useless in our situation.
(That is why we have a difficulty in the first place.)
On the other hand, there is a possibility that we can
construct ill-defined terms of Type (b) 
by a probabilistic calculations concerning Ornstein-Uhlenbeck 
process. 
(In our case, we can actually do it in Section 8.)

Thus, we have heuristically explained 
where the strange-looking 
definition of our paracontrolled QGE comes from.
From the next section, we will rigorously discuss
it again.

%%%%%%%%%%%%%%%

%%%%%%%%%%%%%%%%%%%%%%
%% Section \newpage
%%%%%%%%%%%%%%%%%%%%%%%%%%%%%%%%
\section{Definition of drivers and solutions}
\label{s4}

From now on, we work on ${\bf R}^2$.
Let $T>0$ be arbitrary 
and let $7/4 <\theta \le 2$.
Once chosen, $\theta$ will be fixed throughout.
(Hence, the dependence on $\theta$ will often be implicit.)
For this $\theta$, we take $0 < \kappa<\kappa'$ sufficiently small.
More precisely, we basically assume 
\begin{equation}\label{eq:180213-3}
0 < \kappa<\kappa' \ll 1,
\qquad 
\frac13 < \frac{\kappa}{\kappa'} <\frac23.
\end{equation}
That is,
there exists $\kappa_0 > 0$ 
such that 
the relevant statements hold,
whenever
\[
0 < \kappa<\kappa' \le \kappa_0,
\qquad 
\frac13 < \frac{\kappa}{\kappa'} <\frac23.
\]

The exact values of 
$\kappa$ and $\kappa'$ are not important at all.

\subsection{Definition of a driver of paracontrolled QGE}
In this subsection we define a driver of 
the paracontrolled QGE.

\begin{definition}\label{def.200415}
An element 
\[
{\mathbf X}=(X, V, Y, Z, W, \hat{Z}, \hat{W})
\]
of the product Banach space 
\begin{eqnarray*}
\lefteqn{
C_T\HolBesSp{\frac{\theta}{2}-1-\kappa} \times
(C_T\HolBesSp{ \frac32 \theta -2 -\kappa})^2 \times
\cL_T^{2\theta -3-\kappa,\frac{2\theta -3-\kappa}{\theta}} 
}
\\
&\times&
C_T\HolBesSp{\frac52\theta -5-\kappa} \times
(C_T\HolBesSp{2\theta -4-\kappa})^2 \times
C_T\HolBesSp{\frac52\theta -5 -\kappa} \times
(C_T\HolBesSp{2\theta -4-\kappa})^2
\end{eqnarray*}
is said to be a driver if the following relation holds:
\[
V_t = P_t^{\theta /2} V_0 + I [\nabla X]_t,
\qquad
t \in [0,T].
\]
The set of all drivers is denoted by $\drivers{T}{\kappa}$.
The norm $\|{\mathbf X} \|_{\drivers{T}{\kappa}}$ is defined to be 
 the sum of the norms of all components.
\end{definition}
Assuming 
$7/4 <\theta \le 2$,
we can suppose that $\kappa,\kappa'>0$ are chosen
so that
$2\theta-3-\kappa',2\theta-3-\kappa>0$.
This means that the condition on $Y$ makes sense.

As usual, we wrote $\nabla = (\partial_1, \partial_2)$.
We suppresse
the
dependency on $\theta$ for notational simplicity.
Note that the third component $Y$ is assumed to have H\"older continuity in time.
Let $t \ge 0$ and $\kappa''>0$. 
Since
$u \in C_T{\mathcal C}^{\frac12\theta-1-\kappa} 
\mapsto I [\nabla u]_t
\in C_T{\mathcal C}^{\frac32\theta -2 -\kappa-\kappa''}$
and
$V_0 \in C_T{\mathcal C}^{\frac32\theta-2-\kappa}
\mapsto
P_t^{\theta/2}V_0 \in C_T{\mathcal C}^{\frac32\theta-2-\kappa}
\subset
C_T{\mathcal C}^{\frac32\theta-2-\kappa-\kappa''}$ 
are both continuous
thanks to Proposition \ref{prop_20170804},
we learn that $\drivers{T}{\kappa}$ is a closed subset 
of the product Banach space above.

\begin{lemma}\label{lem:180208-7}
Let 
${\mathbf X}=(X, V, Y, Z, W, \hat{Z}, \hat{W})$
be a driver.
Then $ \| \nabla X_t\|_{ \HolBesSp{\frac12 \theta -2 - \kappa}} 
\lesssim \| X_t\|_{ \HolBesSp{\frac12 \theta -1- \kappa}} \le \|{\mathbf X}\|_{\drivers{T}{\kappa}}$.
\end{lemma}

\begin{proof}
Simply resort to Lemma \ref{lem:180122-1}.
\end{proof}

Here is a simple remark on the $Y$-component of a driver.
\begin{lemma}\label{rem_20171114-1}
Let
$
Y \in \cL_T^{2\theta -3-\kappa,\frac{2\theta -3-\kappa}{\theta}} 
 = C_T\HolBesSp{2\theta -3-\kappa}
			\cap
			C_T^{ \frac{2\theta -3-\kappa}{\theta} } \HolBesSp{0}
$
with $7/4<\theta \le 2$ and $0<\kappa<\kappa' \ll 1$.
Then
$
					\|Y_t-Y_s\|_{\HolBesSp{\kappa'-\kappa}}
						\lesssim
						(t-s)^{ \frac{2\theta -3-\kappa'}{\theta} }
						\|Y\|_{{\mathcal L}^{2\theta -3-\kappa,\frac{2\theta -3-\kappa}{\theta}}}.
$
\end{lemma}

\begin{proof}
Since $0 < \kappa <\kappa'$, we see from Proposition \ref{rem_20171114-2}(1) that
				$
					\cL_T^{ 2\theta -3-\kappa,\frac{2\theta -3-\kappa}{\theta} }
					\subset
					\cL_T^{2\theta -3-\kappa,\frac{2\theta -3-\kappa'}{\theta}}
					\subset
						C_T^{ \frac{2\theta -3-\kappa'}{\theta}}
						\HolBesSp{\kappa'-\kappa}.
				$
This inclusion implies
				\begin{align*}
					\|Y_t-Y_s\|_{\HolBesSp{\kappa'-\kappa}}
					\lesssim
						(t-s)^{ \frac{2\theta -3-\kappa'}{\theta} }
						\|Y\|_{C_T^{\frac{2\theta -3-\kappa'}{\theta}} 
						 \HolBesSp{\kappa'-\kappa}}
						\lesssim
						(t-s)^{ \frac{2\theta -3-\kappa'}{\theta} }
						\|Y\|_{{\mathcal L}^{2\theta -3-\kappa,\frac{2\theta -3-\kappa}{\theta}}}.
				\end{align*}
				 \end{proof}

If the regularity of $X$ is nice enough, 
we can enhance $X$ to a driver ${\mathbf X}$ in a very natural way.
This driver is called the natural enhancement of $X$.

\begin{example}\label{ex.20171112}
Let
$\alpha >2$
and
$\beta >1$.
Let $X \in C_T\HolBesSp{\alpha}$,
$Y_0 \in \HolBesSp{\beta}$, 
and
$V_0 \in (\HolBesSp{\beta})^2$
be given.
Then, 
keeping in mind that
\[
X \in C_T\HolBesSp{\alpha} \subset C_T\HolBesSp{\frac12\theta-2-\kappa}
\]
we can define a driver 
${\mathbf X}= (X, V, Y, Z, W, \hat{Z}, \hat{W})$ by setting 
\begin{align*}
V &= P_{\cdot}^{\theta /2} V_0 + I[\nabla X]
\in C_T\HolBesSp{\beta} \cap C_T\HolBesSp{\alpha} \subset C_T\HolBesSp{\frac12\theta-1-\kappa},
\\
Y &=P_{\cdot}^{\theta /2} Y_0 + I[R^\perp X \cdot \nabla X]
= P_{\cdot}^{\theta /2} Y_0 + 
I [R_2 X \cdot \partial_1 X - R_1 X \cdot \partial_2 X]
\in C_T\HolBesSp{\min(\alpha-1,\beta)},
\\
 Z &=R^\perp Y \reso \nabla X = R_2 Y \reso \partial_1 X - R_1 Y \reso \partial_2 X
\in C_T\HolBesSp{2\alpha-2} \subset C_T\HolBesSp{\frac52\theta-5-\kappa},
 \\
 W &= R^\perp V \reso \nabla X 
 = \{ R_2 V_i \reso \partial_1 X - R_1 V_i \reso \partial_2 X \}_{i=1,2}
\in C_T\HolBesSp{2\theta-3-2\kappa}\subset C_T\HolBesSp{2\theta-4-\kappa},
 \\
 \hat{Z} &= \nabla Y \cdot R^\perp X = \partial_1 Y \cdot R_2 X
 - \partial_2 Y \cdot R_1 X
\in C_T\HolBesSp{2\alpha-2} \subset C_T\HolBesSp{\frac52\theta-5-\kappa},
 \\
 \hat{W} &=R^\perp X \reso \nabla V 
 = \{ R_2 X \reso \partial_1 V_i - R_1 X\reso \partial_2 V_i \}_{i=1,2}
\in C_T\HolBesSp{2\theta-3-2\kappa}\subset C_T\HolBesSp{2\theta-4-\kappa}.
 \end{align*}
Here, we write $R^{\perp} =( R_2, -R_1)$ and $V =(V_1, V_2)$.
Moreover,
since
\[
\sup\limits_{0<t \le T}
\|R^\perp X_t \cdot \nabla X_t\|_{{\mathcal C}^0}<\infty,
\]
we have
\[
\|I[R^\perp X \cdot \nabla X]_t-I[R^\perp X \cdot \nabla X]_s\|_{{\mathcal C}^0}
=
{\rm O}(|s-t|) \quad (0<s \le t \le T).
\]
Thus,
\[
I[R^\perp X \cdot \nabla X]
\in C^1_T{\mathcal C}^0
\subset C^{\frac{2\theta-3-\kappa}{\theta}}_T{\mathcal C}^0.
\]
As a result
$Y \in 
\cL_T^{2\theta -3-\kappa,\frac{2\theta -3-\kappa}{\theta}} $.
So,
it follows 
that
${\mathbf X}= (X, V, Y, Z, W, \hat{Z}, \hat{W})$
is a driver.

The most important example of such $X$ we have in mind 
is the sample path of the approximation of 
the stationary OU-like process 
\begin{equation}\label{eq.20171113_1}
X^{\ve}_t = \int_{-\infty}^t P_{t-s}^{\theta /2} \xi^{\ve}_s ds
=: \cI [\xi^{\ve}]_t, 
\qquad 
t \in \RealNum,
\end{equation}
where
$\xi^{\ve}$
stands for 
the smooth approximation
of the space-time white noise $\xi$
on $\RealNum \times \Torus^2$, respectively.
Moreover, if we choose
\[
V_0 = \int_{-\infty}^0 P_{-s}^{\theta /2} \nabla X^{\ve}_s ds,
\qquad
Y_0 = \int_{-\infty}^0 P_{-s}^{\theta /2} 
[R^\perp X^{\ve}_s \cdot \nabla X^{\ve}_s ]ds,
\]
Then, in this case we have $V^{\ve} = \cI [\nabla X^{\ve} ]$ 
and $Y^{\ve} = \cI [ R^\perp X^{\ve} \cdot \nabla X^{\ve}]$.
Using these, 
we define
the other four symbols as follows:
\begin{align*}
Z^{\ve} &=R^\perp Y^{\ve} \reso \nabla X^{\ve}
 = R^\perp \{ \cI [ R^\perp X^{\ve} \cdot \nabla X^{\ve}]\} \reso \nabla X^{\ve},
 \\
 W^{\ve} &= R^\perp V^{\ve} \reso \nabla X^{\ve}
 =R^\perp \{ \cI [\nabla X^{\ve} ] \} \reso \nabla X^{\ve},
 \\
 \hat{Z}^{\ve} &= \nabla Y^{\ve} \cdot R^\perp X^{\ve} 
 = \nabla \{ \cI [ R^\perp X^{\ve} \cdot \nabla X^{\ve}]\} \cdot R^\perp X^{\ve},
 \\
 \hat{W}^{\ve} &=R^\perp X^{\ve} \reso \nabla V^{\ve}
 = R^\perp X^{\ve} \reso \nabla \cI [\nabla X^{\ve} ].
 \end{align*}
 The random driver 
 ${\mathbf X}^{\ve} 
 = (X^{\ve}, V^{\ve}, Y^{\ve}, Z^{\ve}, W^{\ve}, \hat{Z}^{\ve}, \hat{W}^{\ve})$ 
obtained in this way is called the stationary natural enhancement
of the smooth approximation $X^{\ve}$ given in \eqref{eq.20171113_1}.
 \end{example}
 
%\newpage
 
\begin{remark}\label{rem.200413}
Now, we summarize some data on the driver
for the reader's convenience.

\begin{table}[htb]
 \begin{tabular}{|c|c|c |c|c|} \hline
 Symbol & Regularity & Space & Component of ${\mathbf X}^{\ve}$ & Order \\ \hline \hline
 $X$ & $\frac12 (\theta -2)$ & $C_T\HolBesSp{\frac12 (\theta -2)-\kappa}$&
 $X^{\ve}$ & 1 \\ \hline
 $V$ & $\frac32 \theta -2$ & $(C_T\HolBesSp{ \frac32 \theta -2 -\kappa})^2 $ &
 $V^{\ve} :=\cI [\nabla X^{\ve}]$ & 1 \\ \hline 
 $Y$ & $2\theta -3$ & 
 $\cL_T^{2\theta -3-\kappa,\frac{2\theta -3-\kappa}{\theta}} $ & 
 $Y^{\ve} :=\cI [R^\perp X^{\ve} \cdot \nabla X^{\ve}]$ & 2 \\ \hline 
 $Z$ & $\frac52 (\theta -2)$ & $C_T\HolBesSp{\frac52 (\theta -2)-\kappa}$ & 
 $R^\perp Y^{\ve} \reso \nabla X^{\ve}$ & 3 \\ \hline
$W$ & $2 (\theta -2)$ & $(C_T\HolBesSp{2 (\theta -2)-\kappa})^2$& 
 $R^\perp V^{\ve} \reso \nabla X^{\ve}$ & 2 \\ \hline
 $\hat{Z}$ & $\frac52 (\theta -2)$ & $C_T\HolBesSp{\frac52 (\theta -2)-\kappa}$&
 $\nabla Y^{\ve} \cdot R^\perp X^{\ve}$ &3 \\ \hline
 $\hat{W}$ & $2 (\theta -2)$ & $(C_T\HolBesSp{2 (\theta -2)-\kappa})^2$ & 
 $R^\perp X^{\ve} \reso \nabla V^{\ve}$ & 2 \\ \hline 
 \end{tabular}
\end{table} 
\end{remark}

In the table, ``Regularity $\alpha$" means that 
 the space regularity of the symbol is $\alpha -\kappa$
 for sufficiently small $\kappa >0$.

For example,
the regularity of the symbol $X$ is
$\frac{1}{2}(\theta-2)-\kappa$
for some $\kappa>0$.

 ``Component of ${\mathbf X}^{\ve}$" indicates how
the stationary natural enhancement 
of the smooth approximation of $X^{\ve}$ is defined.
``Order" shows the order of (inhomogeneous) Wiener chaos the corresponding 
component of ${\mathbf X}^{\ve}$ belongs to.

For instance, the ``$Y$-component" of a driver has the space regularity 
$2\theta -3 -\kappa$.
Precisely, it belongs to the Banach space
$\cL_T^{2\theta -3-\kappa,\frac{2\theta -3-\kappa}{\theta}} $.
The ``$Y$-component" of ${\mathbf X}^{\ve}$, 
the stationary natural enhancement 
of the smooth approximation of the stationary OU-like process,
is defined to be $\cI [R^\perp X^{\ve} \cdot \nabla X^{\ve}]$.
As a functional of white noise $\xi$, 
$\cI [R^\perp X^{\ve} \cdot \nabla X^{\ve}]_{t,x}$ belongs to 
the second order (inhomogeneous) Wiener chaos for each $(t,x)$.

%%%%%%%%%%%%%%%%%%%%%%%%%%%%%%%%%%%
%\newpage
%%%%%%%%%%%%%%%%%%%%%%%%%%%%%%

\subsection{Banach space for solutions of paracontrolled QGE}

In this subsection
we introduce the Banach space on which the fixed point problem 
for our paracontrolled QGE is formulated,
so that solutions of QGE will belong to this Banach space.

	Let $7/4 < \theta \le 2$. 
Write
\begin{equation}\label{eq:180114-2}
\rho :=\frac{4 \theta -7}{10^{100}\theta}.
\end{equation}
\begin{definition}
We set
	\begin{align*}
		\sols{T}{\kappa}{\kappa'}
		:=
			\cL_T^{ q-\kappa', 
			 \frac32\theta -2 -\kappa', 1-\frac{\kappa'}{\theta}}
			\times
			\cL_T^{q' -\kappa'+\kappa, 
			 \frac72 \theta - 5 - \theta\kappa',1-\kappa'},
	\end{align*}
where 
\begin{eqnarray}\label{eq:180114-1}
(q, q') 
:=
\begin{cases}
 ( 2 - \frac{5}{2\theta}, 1) &\frac{11}{6} < \theta \le 2, \\
 (5 - \frac{8}{\theta}- 2\rho, 7 - \frac{11}{\theta} - 3\rho)
 &\frac74 < \theta \le \frac{11}{6}.
\end{cases}
\end{eqnarray}
Its norm is denoted by $\|(v,w)\|_{\sols{T}{\kappa}{\kappa'}}$
for $(v,w) \in 
		\sols{T}{\kappa}{\kappa'}$.
\end{definition}
An element of this Banach space is usually denoted by $(v,w)$.

The parameter $\rho$ is defined to be sufficiently small
so as to justify computation below.
Note that
\begin{equation}\label{eq:180208-11}
0<q<1,~ 0<q' \le 1.
\end{equation}

\begin{remark}\label{rem_20171114-3}
For the readers' convenience, 
we compare the space regularities of the components 
$(X,V,Y,Z,W,\hat{Z},\hat{W})$
of a driver 
and those of $v$ and $w$. 
If $3/2 < \theta <2$, 
%\footnote{The condition on $\theta$ has to be checked later.}
it holds that
\begin{eqnarray*}
\frac52 (\theta-2)< 2 (\theta -2) < \frac12 (\theta -2) < 0 < 2\theta -3 
< \frac32 \theta -2 < \frac72 \theta -5 < 3\theta -4.
\end{eqnarray*}
\end{remark}

%\vspace{5mm}

By (\ref{eq:180114-1}), one can easily see that, for every $\theta \in (7/4, 2]$,
\begin{align}\label{eq_20171124_1}
\frac32 \theta -2 -\theta q < \min \Bigl\{ \frac72 \theta -5 -\theta q', \,
\frac12 (\theta -2)\Bigr\},
\quad
 q' \ge \frac{5}{\theta} -3 + 2q.
\end{align}

%\vspace{5mm}
 
Here we give a couple of remarks on these function spaces.
	We now make a comment on
	$
		\sols{T}{\kappa}{\kappa'}
	$.
\begin{remark}\label{rem_20160921064624}
Let 
$7/4<\theta \le 2$.
	Recall that
	\begin{align*}
		\cL_T^{
		q-\kappa', 
			 \frac32\theta -2 -\kappa', 1-\frac{\kappa'}{\theta}
			 }
		&=
			\mathcal{E}_T^{ q-\kappa' }
			 \HolBesSp{\frac32\theta -2-\kappa'}
			\cap
			C_T\HolBesSp{\frac32\theta -2 - \theta q+ (\theta -1)\kappa' }
			\cap
			\mathcal{E}_T^{q -\kappa',1-\frac{\kappa'}{\theta}}
			 \HolBesSp{ \frac12\theta -2 },\\
		\cL_T^{ q' -\kappa'+\kappa, 
			 \frac72 \theta - 5 - \theta\kappa',1-\kappa' }
		&=
			\mathcal{E}_T^{q' -\kappa'+\kappa}
			 \HolBesSp{\frac72 \theta - 5 - \theta\kappa' }
			\cap
			C_T\HolBesSp{ \frac72 \theta -5 -\theta q' -\theta\kappa}
			\cap
			\mathcal{E}_T^{q' -\kappa'+\kappa,1-\kappa'}
	 \HolBesSp{\frac52 (\theta -2)}.
	\end{align*}
Thanks to Proposition \ref{rem_20171114-2}(1), we have 
		\begin{align*}
					\sols{T}{\kappa}{\kappa'}
					&=
						\cL_T^{q-\kappa', 
			 \frac32\theta -2 -\kappa', 1-\frac{\kappa'}{\theta}}
						\times
						\cL_T^{q' -\kappa'+\kappa, 
			 \frac72 \theta - 5 - \theta\kappa',1-\kappa' 
			 }\\
					&\subset
						\cL_T^{ q-\kappa', 
			 \frac32\theta -2 -\kappa', \frac32- \frac{2}{\theta} - \kappa'
			 }
						\times
						\cL_T^{ q'-\kappa'+\kappa, 
			 \frac72 \theta - 5 - \theta\kappa', \frac{1}{\theta}-\kappa' 
			 }.
				\end{align*}
\end{remark}

\begin{proposition}\label{prop:180131-11}
Let
$(u,v) \in \sols{T}{\kappa}{\kappa'}$
with $0<\kappa<\kappa' \ll 1$,
and let $q$ be as above.
\begin{enumerate}
\item[$(1)$]
$\displaystyle
\|R^\perp (v_t- v_s)\|_{\HolBesSp{(\theta -1)\kappa'}}
+
					\|v_t-v_s\|_{\HolBesSp{(\theta -1)\kappa'}}
\lesssim_{\theta,\kappa'}
						s^{-q+\kappa'}
						(t-s)^{ \frac32 -\frac{2}{\theta} -\kappa'}
						\|v\|_{
						\cL_T^{q-\kappa', 
			 \frac32\theta -2 -\kappa', 1-\frac{\kappa'}{\theta} }
						 }.
$
In particular,
$$\displaystyle
\|R^\perp (v_t- v_s)\|_{L^\infty(\Torus^2)}
+
					\|v_t-v_s\|_{L^\infty(\Torus^2)}
\lesssim_{\theta,\kappa'}
						s^{-q+\kappa'}
						(t-s)^{ \frac32 -\frac{2}{\theta} -\kappa'}
						\|v\|_{
						\cL_T^{q-\kappa', 
			 \frac32\theta -2 -\kappa', 1-\frac{\kappa'}{\theta} }
						 }.
$$
\item[$(2)$]
$\displaystyle
\|R^\perp (w_t- w_s)\|_{\HolBesSp{\frac72\theta-6}}
					+
					\|w_t-w_s\|_{\HolBesSp{\frac72\theta-6}}
\lesssim_{\theta,\kappa,\kappa'}
						s^{-q+\kappa'-\kappa}
						(t-s)^{ \frac{1}{\theta} -\kappa'}
						\|v\|_{
						\cL_T^{q-\kappa'+\kappa, 
			 \frac72\theta -5 -\theta\kappa', 1-\frac{\kappa'}{\theta} }
						 }.
$
In particular,
$$\displaystyle
\|R^\perp (w_t-w_s)\|_{L^\infty(\Torus^2)}
					+
					\|w_t-w_s\|_{L^\infty(\Torus^2)}\lesssim_{\theta,\kappa,\kappa'}
						s^{-q+\kappa'-\kappa}
						(t-s)^{\frac{1}{\theta} -\kappa'}
						\|v\|_{
						\cL_T^{q-\kappa'+\kappa, 
			 \frac72\theta -5 -\theta\kappa', 1-\frac{\kappa'}{\theta} }
						 }.
$$
\end{enumerate}
\end{proposition}

\begin{proof}
We concentrate on (2) since (2)
is somewhat more delicate than (1).
Since
$\frac72 \theta -6 >\frac{1}{8}$,
we have only to show the estimate on H\"{o}lder-Zygmund spaces.
By Lemma \ref{lm.20171227}
we have only to handle
$
					\|w_t-w_s\|_{\HolBesSp{\frac72\theta-6}}$.
It remains to use Corollary \ref{cor:180131-2}.
\end{proof}

\begin{proposition}\label{prop:180131-12}
Under the above assumption on $\kappa$ and $\kappa'$,
for every $\gamma_1, \gamma_2$ satisfying
				\[
			 \frac32 \theta -2 -\theta q
								 +(\theta -1)\kappa'					
					\leq
						\gamma_1
					\leq
						\frac32 \theta -2-\kappa', \quad
					 \frac72 \theta -5 -\theta q' - \theta\kappa 
					\leq
						\gamma_2
					\leq
					 \frac72 \theta - 5 - \theta\kappa'.
				\]
and for every $(v,w)\in\sols{T}{\kappa}{\kappa'}$,
we have
\begin{enumerate}
\item[$(1)$]
$
\|v_t\|_{\HolBesSp{\gamma_1}}
					\lesssim_{\gamma_1,\theta,\kappa,\kappa'}
						t^{
							-\frac{1}{\theta}
							\left\{ 
								\gamma_1- ( \frac32 \theta -2 -\theta q
								 +(\theta -1)\kappa')
							\right\}
					}
\|v\|_{\cL_T^{q-\kappa',\frac32\theta -2 -\kappa',1-\frac{\kappa'}{\theta}}}
$
\item[$(2)$]
$
					\|w_t\|_{\HolBesSp{\gamma_2}}
					\lesssim_{\gamma_2,\theta,\kappa,\kappa'}
						t^{
							-\frac{1}{\theta}
							\left\{
								\gamma_2- ( \frac72 \theta -5 -\theta q'
 - \theta\kappa)
							\right\}
						}
						\|w\|_{\cL_T^{
						q' -\kappa'+\kappa, 
			 \frac72 \theta - 5 - \theta\kappa',1-\kappa' 
	 						}}.
$
\end{enumerate}
\end{proposition}

\begin{proof}
Use Proposition \ref{rem_20171114-2}(3).
\end{proof}

We can assume $\kappa'$ is small enough to have
\begin{equation}\label{eq:180206-1}
			 \frac32 \theta -2 -\theta q
								 +(\theta -1)\kappa'					
					<0
					<
						\frac32 \theta -2-\kappa', \quad
					 \frac72 \theta -5 -\theta q' - \theta\kappa 
					<0
<					 \frac72 \theta - 5 - \theta\kappa',
\end{equation}
since $0 < q < 1$.

\begin{corollary}\label{cor:180201-21}
Assume $(\ref{eq:180206-1})$.
\begin{enumerate}
\item[$(1)$]
$\displaystyle
					\|v_t\|_{{\mathcal C}^{(\theta-1)(\kappa'-\kappa)}}
\lesssim_{\theta,\kappa,\kappa'}
						t^{ 
						\frac{1}{\theta}
							\left( 
								 \frac32 \theta -2 -\theta q
								 +(\theta -1)\kappa 
							\right)
									}
						\|v\|_{\cL_T^{
						q-\kappa', 
			 \frac32\theta -2 -\kappa', 1-\frac{\kappa'}{\theta}
			 						}}
$
for each
$(u,v) \in \sols{T}{\kappa}{\kappa'}$.
In particular,
$$
\|v_t\|_{L^\infty(\Torus^2)}
+\|R^\perp v_t\|_{L^\infty(\Torus^2)}
\lesssim_{\theta,\kappa,\kappa'}
						t^{ 
						\frac{1}{\theta}
							\left( 
								 \frac32 \theta -2 -\theta q
								 +(\theta -1)\kappa 
							\right)
									}
						\|v\|_{\cL_T^{
						q-\kappa', 
			 \frac32\theta -2 -\kappa', 1-\frac{\kappa'}{\theta}
			 						}}.
$$
\item[$(2)$]
$\displaystyle
					\|w_t\|_{{\mathcal C}^{(\theta-1)(\kappa'-\kappa)}}
\lesssim_{\theta,\kappa,\kappa'}
						t^{
							\frac{1}{\theta}
							\left(
							\frac72 \theta -5 -\theta q'
 + (\theta -1)\kappa + \kappa'
							\right)
						}
						\|w\|_{\cL_T^{ 
						q'-\kappa'+\kappa, 
			 \frac72 \theta - 5 - \theta\kappa',1-\kappa'}}
$
for each
$(u,v) \in \sols{T}{\kappa}{\kappa'}$.
In particular,
$$
					\|w_t\|_{L^\infty(\Torus^2)}
+\|R^\perp w_t\|_{L^\infty(\Torus^2)}
\lesssim_{\theta,\kappa,\kappa'}
						t^{
							\frac{1}{\theta}
							\left(
							\frac72 \theta -5 -\theta q'
 + (\theta -1)\kappa + \kappa'
							\right)
						}
						\|w\|_{\cL_T^{ 
						q'-\kappa'+\kappa, 
			 \frac72 \theta - 5 - \theta\kappa',1-\kappa'}}.
$$
\end{enumerate}
\end{corollary}

\begin{proof}
We concentrate on (1), the proof of (2) being similar.
From
Proposition \ref{prop:180131-12}(1)
with $\gamma_1=(\theta -1)(\kappa'-\kappa) >0$
we deduce
the desired estimate.
It remains to 
use Lemma \ref{lm.20171227} 
to have
\[
					\|v_t\|_{L^\infty(\Torus^2)}
+\|R^\perp v_t\|_{L^\infty(\Torus^2)}
\lesssim
						\|v_t\|_{\HolBesSp{(\theta -1)(\kappa'-\kappa) }}.
\]
\end{proof}

\begin{corollary}\label{cor:180201-2}
Let
$(v,w) \in \sols{T}{\kappa}{\kappa'}$.
Assume
$(\ref{eq:180206-1})$.
If
\begin{equation}\label{eq:180206-200}		 
\frac72 \theta -5 -\theta q'
								 - \theta \kappa					
					\leq
						\gamma
					\leq
					 \frac32 \theta - 2- \kappa'/
\end{equation}
Then we have
$\displaystyle
					\|v_t+w_t\|_{\HolBesSp{\gamma}}+
					\|R^{\perp}(v_t+w_t)\|_{\HolBesSp{\gamma}}
					\lesssim_{\theta,\gamma,\kappa,\kappa'}
						t^{
						-\frac{1}{\theta}
							\left(
								\gamma - \frac32 \theta +2+\theta q
								 -(\theta -1)\kappa' 
							\right)
													}
						\|(v,w)\|_{\sols{T}{\kappa}{\kappa'}}.
$
In particular,
$\displaystyle
					\|v_t+w_t\|_{{\mathcal C}^{(\kappa'-\kappa)(\theta -1)}}
+
					\|R^\perp(v_t+w_t)\|_{{\mathcal C}^{(\kappa'-\kappa)(\theta -1)}}
					\lesssim_{\theta,\gamma,\kappa,\kappa'}
						t^{\frac{1}{\theta}
							\left( 
								 \frac32 \theta -2 -\theta q
								 +(\theta -1)\kappa 
							\right)						
						 }
						\|(v,w)\|_{\sols{T}{\kappa}{\kappa'}}.
$
\end{corollary}

\begin{proof}
From (\ref{eq:180206-200}),
we can combine two estimates
in 
Proposition \ref{prop:180131-12}.
In particular,
by letting
$\gamma = (\kappa' -\kappa)(\theta-1)$
we obtain the second estimate.
\end{proof}

We transform
Corollary \ref{cor:180201-2}
into the form in which we use.
\begin{corollary}\label{cor:180206-2}
Let
$(v,w) \in \sols{T}{\kappa}{\kappa'}$.
Let $\Phi=R^{\perp}(Y+v+w)$.
Then one has
\[
\|\Phi_t\|_{L^\infty(\Torus^2)}
\lesssim
\|\Phi_t\|_{{\mathcal C}^{-2\theta+4+\kappa'}}
\lesssim
\|\Phi_t\|_{{\mathcal C}^{2\theta-3-\kappa}}
\lesssim
\|Y_t\|_{{\mathcal C}^{2\theta-3-\kappa}}
+
t^{\frac32-\frac2\theta-q+\frac{\theta-1}{\theta}\kappa}
				\|(v,w)\|_{\sols{T}{\kappa}{\kappa'}}
				\quad (t>0).
\]
\end{corollary}

\begin{lemma}
Under assumptions
on
Lemmas \ref{lm.20171227}
and \ref{rem_20171114-1} 
and
Proposition \ref{prop:180131-11},
we have
 \begin{align}\nonumber
 \label{eq:180122-14}
\lefteqn{
\| \Phi_t-\Phi_s\|_{L^{\infty}(\Torus^2)}
}\nonumber\\
& \lesssim_{\theta,\gamma,\kappa,\kappa'} (t-s)^{2-\frac{3}{\theta}-\frac{\kappa'}{\theta}} 
 \| Y\|_{\cL_T^{2\theta-3-\kappa,\frac{2\theta-3-\kappa}{\theta}}}+
						s^{-q+\kappa'}
						(t-s)^{ \frac32 -\frac{2}{\theta} -\kappa'}
						\|v\|_{
						\cL_T^{q-\kappa', 
			 \frac32\theta -2 -\kappa', 1-\frac{\kappa'}{\theta} }
						 }\\
						 &\quad +
 s^{-q'+\kappa'-\kappa}
						(t-s)^{\frac{1}{\theta}-\kappa'}
						\|w\|_{
						\cL_T^{ 
						q' -\kappa'+\kappa, 
			 \frac72 \theta - 5 - \theta\kappa',1-\kappa' 	}
						}. 
 \end{align}
\end{lemma}

\begin{proof}
By the definition of the function spaces
and
Lemmas \ref{lm.20171227}
and \ref{rem_20171114-1} 
we have
\[
\| R^{\perp} (Y_t -Y_s)\|_{L^{\infty}(\Torus^2)} 
\lesssim \| Y_t -Y_s \|_{\HolBesSp{\kappa' -\kappa}}
 \lesssim (t-s)^{2-\frac{3}{\theta}-\frac{\kappa'}{\theta}} 
 \| Y\|_{\cL_T^{2\theta-3-\kappa,\frac{2\theta-3-\kappa}{\theta}}}.
\]
Meanwhile,
the definition of the function spaces,
Lemma \ref{lm.20171227}
and
Proposition \ref{prop:180131-11}
yield
\begin{align*}
\|R^{\perp} (v_t-v_s)\|_{L^\infty(\Torus^2)}
					&\lesssim
						\|v_t-v_s\|_{\HolBesSp{(\theta -1)\kappa'}}
					\lesssim
						s^{-q+\kappa'}
						(t-s)^{ \frac32 -\frac{2}{\theta} -\kappa'}
						\|v\|_{
						\cL_T^{q-\kappa', 
			 \frac32\theta -2 -\kappa', 1-\frac{\kappa'}{\theta} }
						 },
 \\
 \|R^{\perp} (w_t-w_s)\|_{L^\infty(\Torus^2)}
					&\lesssim
						\|w_t-w_s\|_{\HolBesSp{\frac72 \theta -6}}
					\lesssim
						s^{-q'+\kappa'-\kappa}
						(t-s)^{\frac{1}{\theta}-\kappa'}
						\|w\|_{
						\cL_T^{ 
						q' -\kappa'+\kappa, 
			 \frac72 \theta - 5 - \theta\kappa',1-\kappa' 	}
						}. 
 \end{align*}
 Consequently,
(\ref{eq:180122-14}) follows.
\end{proof}
%%%%%%%%%%%%%%%%%%%%%%%%%%%
%\newpage

\subsection{Integration maps for paracontrolled QGE}

Let 
${\mathbf X} =(X, V, Y, Z, W, \hat{Z}, \hat{W}) \in \drivers{T}{\kappa}$ and
 $v_0\in\HolBesSp{\alpha}$ for some $\alpha \in \RealNum$.
For $(v,w)\in\sols{T}{\kappa}{\kappa'}$
 we set
\begin{equation}\label{eq:180227-1}
\Phi=R^{\perp}(Y+v+w)
\end{equation}
as before and
\begin{equation}\label{eq:180227-2}
\com (v, w)_t = P^{\theta/2}_t v_0 + I [\Phi \rpara \nabla X] _t
-
\Phi_t \rpara V_t.
\end{equation}
Define a mapping
$F$ on $\sols{T}{\kappa}{\kappa'}$
by
\begin{equation}
	\label{eq_20161002045615}
	F(v,w)
	=
		\Phi \rpara \nabla X
\end{equation}
and a mapping $G$ on $\sols{T}{\kappa}{\kappa'}$ by
\begin{align}
	\label{eq_20161002050746}
	G(v,w)
	&=
		\Phi \lpara \nabla X 
		+
		Z + R^{\perp}w \reso \nabla X + \Phi \cdot W
		 + 
		 \{ R^{\perp} (\Phi \rpara V) 
		 - \Phi \rpara R^{\perp} V \}
		 \reso \nabla X 
		 \nonumber\\
		& \qquad + 
		 \comC (\Phi, R^{\perp} V, \nabla X)
		 +R^{\perp}	\com (v,w) \reso \nabla X		
\nonumber\\
		& \quad
		+
		\hat{Z} + R^{\perp}X \cdot \nabla w 
		 + R^{\perp} X \cdot \{\nabla \Phi \rpara V \} +
		 R^{\perp} X (\rpara + \lpara)
		 \{\Phi \rpara \nabla V \} +\Phi \cdot \hat{W}
		 \nonumber\\
		& \qquad 
		+
				\comC (\Phi, \nabla V, R^{\perp} X)
				 +R^{\perp} X \cdot \nabla \com (v,w)	
		 +	\Phi \cdot \nabla (Y+v+w)
		 + (X+Y+v+w).\nonumber
\end{align}
In the above definitions of $F$ and $G$, 
all operations on the right-hand side were done at each fixed time $t \in [0,T]$.
The role of each term in the long definition of $G$ will be explained later.

%\vspace{7mm}

Let $f,g \in C_T\HolBesSp{\alpha}$ for some $\alpha \in \RealNum$.
We write
\[
R^{\perp} f \rpara \nabla g :=R^{\perp}_1 f \rpara \partial_1 g
+R^{\perp}_2 f \rpara \partial_2 g.
\]
We also define 
$R^{\perp} f \lpara \nabla g, R^{\perp} f \reso \nabla g, R^{\perp} f \cdot \nabla g, 
R^{\perp} f \rpara I(\nabla g)$
in an analogous way.

\begin{remark}
The precise meanings of the simplified symbols
used in the definitions of $F$ and $G$ are as follows:
In this remark, $f, g, h \in C_T\HolBesSp{\alpha}$ 
and $K =(k_1,k_2) \in (C_T\HolBesSp{\alpha})^2$ for some $\alpha \in \RealNum$.

More complex symbols are precisely given 
(when they are well defined)
as follows:
\begin{align}
\nonumber
R^{\perp} \{ R^{\perp} f \rpara K \} \reso \nabla h 
&:=
\sum\limits_{i,j =1}^2
R^{\perp}_i \{ R^{\perp}_j f \rpara k_j \} \reso \partial_i h,
\\
\nonumber
\{ R^{\perp} f \rpara R^{\perp} K \} \reso \nabla h 
&:=
\sum\limits_{i,j =1}^2
\{ R^{\perp}_j f \rpara R^{\perp}_i k_j \} \reso \partial_i h,
\\
\nonumber
\comC (R^{\perp} f, R^{\perp} K, \nabla h)
&:=
\sum\limits_{i,j =1}^2 \comC (R^{\perp}_j f, R^{\perp}_i k_j, \partial_i h),
\\
\nonumber
R^{\perp} h \cdot \{\nabla R^{\perp}f \rpara k\}
&:=
\sum\limits_{i,j =1}^2
R^{\perp}_i h \cdot \{\partial _i R^{\perp}_j f \rpara k_j\}
\\
\nonumber
R^{\perp} h \star \{ R^{\perp} f \rpara \nabla K \} 
&:=
\sum\limits_{i,j =1}^2
R^{\perp}_i h \star \{ R^{\perp}_j f \rpara \partial_i k_j \} ,
\\
\nonumber
\comC (R^{\perp} f, \nabla K, R^{\perp} h)
&:=
\sum\limits_{i,j =1}^2
\comC (R^{\perp}_j f, \partial_i k_j, R^{\perp}_i h).
\end{align}
\end{remark}
In the second to last equality,
$\star$ stands for any one of $\cdot, \lpara, \reso, \rpara$.

%%%%%%%%%%%%%%%%%%\vspace{10mm}

The map
$
	\mathcal{M}
	=
		(\mathcal{M}^1,\mathcal{M}^2)
$
is defined on $\sols{T}{\kappa}{\kappa'}$ by
\begin{align}
	\label{eq_20161003034438}
	(\mathcal{M}^1(v,w)_t,
	\mathcal{M}^2(v,w)_t)
	&=
		(P^{\theta /2}_tv_0
		+
			I[F(v,w) ]_t,
		P^{\theta /2}_t w_0
		+
		I[ G(v,w) ]_t)
		\quad (0<t \le T)
		\end{align}
for every initial value
$
	(v_0,w_0)
$.
We will use the fractional Schauder estimate to show that
this is a well-defined map from $\sols{T}{\kappa}{\kappa'}$ to itself
and has good property.
Note the map $\mathcal{M}$ depends on the driver
 ${\mathbf X} \in \drivers{T}{\kappa} $ and the initial value $(v_0,w_0)$.
To make this dependency explicit, 
we will sometimes write 
$\mathcal{M}_{{\mathbf X}, (v_0,w_0)} (v,w)$ for $\mathcal{M} (v,w)$.

We interpret
the QGE equation as a fixed point problem 
for $\mathcal{M}:(v,w)\mapsto(\mathcal{M}^1(v,w),\mathcal{M}^2(v,w))$
as follows:
\begin{definition}\label{eq_20160921004656}
	For every
	$
		(v_0,w_0)
		\in
			\HolBesSp{ \frac32 \theta -2 -\theta q
								 +(\theta -1)\kappa' }
			\times
			\HolBesSp{\frac72 \theta -5 -\theta q' -\theta\kappa}
	$
	and
	$
		{\mathbf X} \in\drivers{T}{\kappa}
	$,
	we consider the system
	\begin{gather}\label{eq_20160920090410}
			\begin{cases}
				v_t
				=
					\mathcal{M}^1(v,w)_t,\\
				w_t
				=
					\mathcal{M}^2(v,w)_t
			\end{cases}
			\quad (0<t \le T).
	\end{gather}
	If there exist $T_\ast>0$ and $(v,w)\in\sols{T_\ast}{\kappa}{\kappa'}$ 
	 satisfying \eqref{eq_20160920090410},
	 then $(v,w)$ is called a local solution to the paracontrolled 
	 QGE on $[0,T_\ast]$ with initial condition $(v_0,w_0)$.
\end{definition}

About the regularity of $(v_0,w_0)$, one should note that
due to \eqref{eq_20171124_1}
 the regularity of $v_0$ is worse than 
 that of the first component $X$ of a driver ${\mathbf X}$
(if $0 < \kappa <\kappa'$ are chosen small enough).

%%%%%%%%%%%%%%%%%
%%%%%%%%%%%%%%%%%%%%%%%%%%%%%%%%%%%%%%%%%%%%%%%%%%%%%%%%%%%%%%%%%%
%% Section
%\newpage
%%%%%%%%%%%%%%%%%%%%%%%%%%%%%%%%%%%%%%%%%%%%%%%%%%%%%%%%%%%%%%%%%%%
\section{Estimates of integration map $\mathcal{M}$}

In this section 
we estimate the integration map 
$\mathcal{M}=\mathcal{M}_{{\bf X},(v_0,w_0)}= (\mathcal{M}^1, \mathcal{M}^2)$
using the various inequalities we proved in the preceding sections.

\begin{proposition}\label{pr:180213-1}
Let $T \in (0,1]$
and $0<\kappa <\kappa' \ll 1$ be as in \eqref{eq:180213-3}.
Then, for every $
		(v_0,w_0)
		\in
			\HolBesSp{ \frac32 \theta -2 -\theta q
								 +(\theta -1)\kappa' }
			\times
			\HolBesSp{\frac72 \theta -5 -\theta q' -\theta\kappa}
	$
	and
	${\mathbf X} \in \drivers{T}{\kappa}$, the map
	$
		\mathcal{M}=\mathcal{M}_{{\bf X},(v_0,w_0)}:
			\sols{T}{\kappa}{\kappa'}
			\to
			\sols{T}{\kappa}{\kappa'}
				$
	is well defined.
	Moreover, there exist positive constants $\const[1], \const[2]$ 
 and 
	 $a$ such that
	 the following estimate holds:
	For every 
	$
		(v,w)\in\sols{T}{\kappa}{\kappa'}
	$,
	\begin{multline*}
		\|\mathcal{M} (v,w)\|_{
		\sols{T}{\kappa}{\kappa'}
				}
		\leq
			\const[1]
			(\|v_0\|_{\HolBesSp{\frac32 \theta -2 - \theta q + (\theta -1)\kappa'}}
			+
				\|w_0\|_{\HolBesSp{\frac72 \theta -5 - \theta q' -\theta\kappa}}
			)\\
			\quad+
			\const[2]
			T^{a}
			\left(1+
				\|v_0\|_{\HolBesSp{\frac32 \theta -2 - \theta q + (\theta -1)\kappa'}}
				+
				\|(v,w)\|_{\sols{T}{\kappa}{\kappa'}}^3
			\right).
	\end{multline*}
	 Here,
the constants  $a$ and $\const[1]$ depend only on $\kappa, \kappa'$
	 and $\const[2]$ depends only on $\kappa, \kappa'$ and ${\bf X}$. 
	 More precisely, $\const[2]$ is given by an at most 
	 third-order polynomial in $\|{\mathbf X}\|_{\drivers{1}{\kappa}}$ for fixed $\kappa, \kappa'$.
	 	 	 	\end{proposition}

The goal of this section is 
to prove Proposition \ref{pr:180213-1}.
The proof is immediate from 
Propositions \ref{prop_20161003143833} and \ref{prop_20161003143917} below.
In the following subsections 
(in particular, in Lemma \ref{lem_20161002130857}, 
Proposition \ref{prop_20161003143833}, Lemma \ref{lem_20161003140013}, Lemma \ref{lem_20171219} and
Proposition \ref{prop_20161003143917}), 
we implicitly assume that
$
		(v_0,w_0)
		\in
			\HolBesSp{ \frac32 \theta -2 -\theta q
								 +(\theta -1)\kappa' }
			\times
			\HolBesSp{\frac72 \theta -5 -\theta q' -\theta\kappa}
	$,
	${\mathbf X} \in \drivers{T}{\kappa}$,
	$0< t \le T \le 1$
and $0<\kappa <\kappa' \ll 1$ are as in \eqref{eq:180213-3}.

%%%%%%%%

\subsection{Estimates of $\mathcal{M}^1$}
First, we estimate $F$.

Let
\begin{equation}\label{eq:180122-22}
\eta := - \frac{1}{\theta}\left(
								 \frac32 \theta -2 -\theta q
								 +(\theta -1)\kappa \right).
\end{equation}

\begin{lemma}\label{lem:180201-3}
Assume $2\theta-3-\kappa>(\theta-1)(\kappa'-\kappa)$
and $t>0$. 
Then we have
\[
\| \Phi_t\|_{L^\infty(\Torus^2)}
\lesssim_{\theta,\kappa,\kappa'}
\|{\mathbf X}\|_{\drivers{T}{\kappa}} 
 +t^{-\eta}
				\|(v,w)\|_{\sols{T}{\kappa}{\kappa'}}.
\]
\end{lemma}

\begin{proof}
Simply use Corollary \ref{cor:180206-2} and 
$\|Y_t\|_{{\mathcal C}^{2\theta-3-\kappa}} \le 
\|{\mathbf X}\|_{\drivers{T}{\kappa}}$.
\end{proof}

\begin{lemma}\label{lem_20161002130857}
	For any $(v,w)\in\sols{T}{\kappa}{\kappa'}$ and $0<t\leq T$,
	we have $F(v,w)_t\in\HolBesSp{ \frac12 \theta -2 -\kappa}$ and
	\begin{align*}
		\|F(v,w)_t\|_{\HolBesSp{\frac12 \theta -2-\kappa}}
		\leq
			\const
			(
				\|{\mathbf X}\|_{\drivers{1}{\kappa}}
				+
				t^{-\eta}
				\|(v,w)\|_{\sols{T}{\kappa}{\kappa'}}
			)
			\|{\mathbf X}\|_{\drivers{1}{\kappa}},
	\end{align*}
	where $\const >0$ is a constant depending only on $\kappa$ and $\kappa'$.
	In particular, $F(v,w) \in \cE_T^{\eta} \HolBesSp{\frac12 \theta -2 -\kappa}$
with the estimate
\[
\|F(u,v)\|_{\cE_T^{\eta} \HolBesSp{\frac12 \theta -2 -\kappa}}
\lesssim
1+
				\|(v,w)\|_{\sols{T}{\kappa}{\kappa'}}
\]
with the implicit constant depending on ${\mathbf X}$.
\end{lemma}

\bigskip 

\begin{proof}
By Proposition \ref{prop_20160926062106}(1),
we obtain
\[
\| F(v,w)_t\|_{ \HolBesSp{ \frac12 \theta -2 - \kappa}}
\lesssim
\| \Phi_t\|_{L^\infty(\Torus^2)} 
\| \nabla X_t\|_{ \HolBesSp{\frac12 \theta -2 - \kappa}}.
\]
Thus
it remains to combine
Lemmas \ref{lem:180208-7} and
\ref{lem:180201-3}.
\end{proof}

%\vspace{10mm}

Next, we estimate $\mathcal{M}^1 (v,w)$ by using the fractional Schauder estimate.
\begin{proposition}\label{prop_20161003143833}
	The map
	$
		\mathcal{M}^1:
			\sols{T}{\kappa}{\kappa'}
			\to
			\cL_T^{
q -\kappa', \frac32 \theta -2-\kappa',1-\frac{\kappa'}{\theta}
			}
	$
	is well defined. Moreover, 
	there exist positive constants $\const[1], \const[2]$ such that
	 the following estimate holds:
	For every 
	$
		(v,w)\in\sols{T}{\kappa}{\kappa'}
	$,
	\begin{align*}
\lefteqn{
\|\mathcal{M}^1(v,w)\|_{\cL_T^{
q -\kappa', \frac32 \theta -2-\kappa',1-\frac{\kappa'}{\theta}
		}
		}
}\\
		&\leq
\const[1]
\|v_0\|_{\HolBesSp{ \frac32 \theta -2 -\theta q
								 +(\theta -1)\kappa' }}
			+
			\const[2]
			T^{
			\frac{1}{\theta} \left\{ 
		\frac32 \theta -2 -(\theta -1) \kappa' - (2- \theta) \kappa
						 \right\}
						 			}
			(1+\|(v,w)\|_{\sols{T}{\kappa}{\kappa'}}).
	\end{align*}
	Here, $\const[1]$ depends only on $\kappa, \kappa'$
	 and $\const[2]$ depends only on $\kappa, \kappa'$ and ${\bf X}$. 
	 More precisely, $\const[2]$ is given by an at most 
	 second-order polynomial in $\|{\mathbf X}\|_{\drivers{1}{\kappa}}$ for fixed $\kappa, \kappa'$.
\end{proposition}

\bigskip 
\begin{proof}
In Lemma \ref{lem_20161002130857}, we showed
 $F (v,w) \in \cE_T^{\eta } \HolBesSp{\frac12 \theta -2 -\kappa}$,
where $\eta$ is given by (\ref{eq:180122-22}).
We refine this estimate.
We now invoke 
Proposition \ref{prop_20170804}(2) with 
\begin{gather*}
\alpha = \frac12 \theta -2-\kappa
\le
\gamma = \frac32 \theta -2 - \theta q+ (\theta - 1) \kappa'
<
\beta = \frac32 \theta -2 -\kappa', 
\quad
\quad
\delta =1 - \frac{\kappa'}{\theta}.
\end{gather*}
Since $\alpha - \theta \eta + \theta -\gamma 
= \frac32 \theta -2 -(\theta -1) \kappa' - (2- \theta) \kappa >0$, we
are in the position of using 
Proposition \ref{prop_20170804}(2).
Combining
Proposition \ref{prop_20170804}(2)
with Lemma \ref{lem_20161002130857},
we have
\begin{align*}
\left\| I [F(v,w)] \right\|_{\cL_T^{
q -\kappa', \frac32 \theta -2-\kappa',1-\frac{\kappa'}{\theta}
}
}
		&\lesssim
		T^{
		\frac{1}{\theta} \left\{ 
		\frac32 \theta -2 -(\theta -1) \kappa' - (2- \theta) \kappa
						 \right\}
		}
			\|F(v,w)\|_{\cE_T^{\eta}\HolBesSp{\frac12 \theta -2-\kappa}}
			\\
			&\lesssim
			T^{
\frac{1}{\theta} \left\{ 
		\frac32 \theta -2 -(\theta -1) \kappa' - (2- \theta) \kappa
						 \right\}			
			}
			(1+\|(v,w)\|_{\sols{T}{\kappa}{\kappa'}}).
						\end{align*}

It remains to handle $ P^{\theta /2}_{\cdot} v_0 $.
We use 
Proposition \ref{prop_20170804}(1) with 
\[
\alpha = \frac32 \theta -2 -\theta q
								 +(\theta -1)\kappa', 
								 \quad
\beta = \frac32 \theta -2 -\kappa', 
\quad
\delta =1 - \frac{\kappa'}{\theta}
\]
to obtain
\[
\left\| P^{\theta /2}_{\cdot} v_0 \right\|_{
\cL_T^{q -\kappa', \frac32 \theta -2-\kappa',1-\frac{\kappa'}{\theta}}
}
		\lesssim
		 \|v_0\|_{\HolBesSp{ 
		 \frac32 \theta -2 -\theta q
								 +(\theta -1)\kappa'		 }},
		 		\]
				which completes the proof.
		\end{proof}

%%
%
% To be deleted 
%%%%%%%%%%%%%%%%%%%%%%
%%% 
%\newpage

\subsection{Estimates of $\mathcal{M}^2$ 
}
In this subsection we
will control $\mathcal{M}^2$.
First, we estimate the comutator $\com (v, w)$.
From the assumption of the parameters we have
\begin{gather}\label{eq:180101-1}
-\frac12 \theta +2 + 2\kappa' -\kappa>0,\\
\label{eq.20171228_1}
\max \Bigl\{ -\frac72 +q + \frac{6}{\theta},
-2 +q' + \frac{3}{\theta},-2+\frac{4}{\theta}+q,
4-\frac{7}{\theta}
\Bigr\} 
<
\frac{5}{\theta} -3 +2q,\\
\label{eq:180101-4}
\Bigl(\frac{4}{\theta} -1 + \frac{\kappa +\kappa'}{\theta} \Bigr) 
- 2-\frac{3}{\theta}-\frac{\kappa'}{\theta}
<1.
\end{gather}

\begin{lemma}\label{lem_20161003140013}
	For any
	$
		(v,w)\in\sols{T}{\kappa}{\kappa'}
	$
	and
	$0<t\leq T$,
	we have
	\begin{align*}
		\|\com(v,w)_t \|_{\HolBesSp{-\frac12 \theta +2+\kappa'}}
		&\le \const t^{- ( \frac{5}{\theta} -3 +2q - \frac{\theta -1}{\theta}\kappa') }
			(1+
								 \|v_0\|_{\HolBesSp{\frac32 \theta -2 - \theta q + (\theta -1)\kappa'				 }}
				+
				\|(v,w)\|_{\sols{T}{\kappa}{\kappa'}}
			).
	\end{align*}
Here, $\const$ is a positive constants depending only on $\kappa$, $\kappa'$ and 
$\|{\mathbf X} \|_{\drivers{1}{\kappa}}$, which is given by an at most second-order polynomial in
	 $\|{\mathbf X} \|_{\drivers{1}{\kappa}}$.
	
\end{lemma}

\begin{proof}
We write
\begin{align*}
{\rm I}=P^{\theta /2}_t v_0, \
{\rm II}=\Phi_t \rpara P^{\theta /2}_t V_0, \
{\rm III}= \int_0^t ( \Phi_t - \Phi_s) \rpara P^{\theta /2}_{t-s} [\nabla X_s] ds, \
{\rm IV}=
\int_0^t [ P^{\theta /2}_{t-s}, \Phi_s \rpara]
\nabla X_s ds,
\end{align*}
so that
\begin{align}\label{eq.0720_01}
\com (v, w)_t
=
P^{\theta /2}_t v_0 + I [\Phi \rpara \nabla X]_t - \Phi_t \rpara V_t
={\rm I}
 - {\rm II}
-{\rm III}
+
{\rm IV}.
\end{align}

We will
estimate ${\rm I},\ldots,{\rm IV}$.

\paragraph{Estimate of ${\rm I}$}
We have 
\[
\| P^{\theta /2}_t v_0\|_{\HolBesSp{ -\frac12 \theta +2+\kappa'}}
\lesssim
t^{- ( 
-2 +\frac{4}{\theta} +q + \frac{2-\theta}{\theta}\kappa'
) } 
\| v_0\|_{\HolBesSp{ \frac32 \theta -2 - \theta q + (\theta -1)\kappa' }}
\]
from Proposition \ref{prop_20160927051055}(1).
It remains to use
(\ref{eq.20171228_1}) to control the negative power
in the right-hand side above.

\paragraph{Estimate of ${\rm II}$}

We use
Proposition \ref{prop_20160926062106}(1)
and
Corollary \ref{cor:180206-2}
to have
\begin{align*}
 \| \Phi_t \rpara P^{\theta /2}_t V_0 \|_{\HolBesSp{- \frac12 \theta +2+\kappa'}}
& \lesssim
 \| \Phi_t\|_{L^{\infty} }
 \| P^{\theta /2}_t V_0 \|_{\HolBesSp{- \frac12 \theta +2+\kappa'}}\\
& \lesssim
 \{ \|Y_t\|_{\HolBesSp{2 \theta -3-\kappa}}
 + t^{
								 \frac32 -\frac{2}{\theta} - q
								 +\frac{\theta -1}{\theta}\kappa } 
 \|(v,w)\|_{\sols{T}{\kappa}{\kappa'}}
 \} 
 \| P^{\theta /2}_t V_0 \|_{\HolBesSp{- \frac12 \theta +2+\kappa'}}.
\end{align*}

We use
Proposition \ref{prop_20160927051055}(1)
to have
\begin{align*}
 \| \Phi_t \rpara P^{\theta /2}_t V_0 \|_{\HolBesSp{- \frac12 \theta +2+\kappa'}}
 &\lesssim
 \{ \|Y_t\|_{\HolBesSp{2 \theta -3-\kappa}}
 + t^{
								 \frac32 -\frac{2}{\theta} - q
								 +\frac{\theta -1}{\theta}\kappa } 
 \|(v,w)\|_{\sols{T}{\kappa}{\kappa'}}
 \} 
 t^{ 2-\frac{4}{\theta}-\frac{\kappa+\kappa'}{\theta}}
\| V_0 \|_{\HolBesSp{\frac32 \theta -2-\kappa}}.
\end{align*}
Since
\[
\frac32-\frac2{\theta}-q<0,
\]
we have
\begin{align*}
 \| \Phi_t \rpara P^{\theta /2}_t V_0 \|_{\HolBesSp{- \frac12 \theta +2+\kappa'}}
 &\lesssim
 \{ \|{\mathbf X} \|_{\drivers{1}{\kappa}}
 + t^{
								 \frac32 -\frac{2}{\theta} - q
								 +\frac{\theta -1}{\theta}\kappa } 
 \|(v,w)\|_{\sols{T}{\kappa}{\kappa'}}
 \} 
 t^{ 2-\frac{4}{\theta}-\frac{\kappa+\kappa'}{\theta}}
\| V_0 \|_{\HolBesSp{\frac32 \theta -2-\kappa}}\\
 &\le
 \{ \|{\mathbf X} \|_{\drivers{1}{\kappa}}
 + 
 \|(v,w)\|_{\sols{T}{\kappa}{\kappa'}}
 \} 
 t^{
								 \frac32 -\frac{2}{\theta} - q
								 +\frac{\theta -1}{\theta}\kappa + 2-\frac{4}{\theta}-\frac{\kappa+\kappa'}{\theta}}
\| V_0 \|_{\HolBesSp{\frac32 \theta -2-\kappa}}\\
 &=
 \{ \|{\mathbf X} \|_{\drivers{1}{\kappa}}
 + 
 \|(v,w)\|_{\sols{T}{\kappa}{\kappa'}}
 \} 
 t^{
								 \frac72 -\frac{6}{\theta} - q
								 +\frac{\theta -1}{\theta}\kappa-\frac{\kappa+\kappa'}{\theta}}
\| V_0 \|_{\HolBesSp{\frac32 \theta -2-\kappa}}.
 \end{align*}
This term satisfies the desired estimate due to \eqref{eq.20171228_1}.

\paragraph{Estimate of ${\rm III}$}

From Lemma \ref{lem:180122-1} and
Proposition \ref{prop_20160927051055}(1) we can also see that
\begin{equation}\label{eq:180122-15}
 \| P^{\theta /2}_{t-s} [\nabla X_s] \|_{\HolBesSp{- \frac12 \theta +2+\kappa'}}
 \lesssim 
 (t-s)^{-( \frac{4}{\theta} -1 + \frac{\kappa +\kappa'}{\theta})}
 \| \nabla X_s\|_{\HolBesSp{\frac12 \theta -2-\kappa}}
 \lesssim 
 (t-s)^{-( \frac{4}{\theta} -1 + \frac{\kappa +\kappa'}{\theta})}
 \| {\mathbf X} \|_{\drivers{1}{\kappa}}.
 \end{equation}

We use 
Proposition \ref{prop_20160926062106}(1)
to have
\[
\Bigl\|{\rm III} \Bigr\|_{\HolBesSp{- \frac12 \theta +2+\kappa' }}
\lesssim
 \int_0^t \| \Phi_t - \Phi_s\|_{L^{\infty}} 
 \| P^{\theta /2}_{t-s} \nabla X_s\|_{\HolBesSp{- \frac12 \theta +2+\kappa'}} ds
\]
keeping in mind that
$- \frac12 \theta +2+\kappa'>0$.

If we combine 
(\ref{eq:180122-14}) and (\ref{eq:180122-15}),
then we obtain
\begin{align*}
\Bigl\|{\rm III} \Bigr\|_{\HolBesSp{- \frac12 \theta +2+\kappa' }}
 &\lesssim
 \| {\mathbf X} \|_{\drivers{1}{\kappa}} 
 \int_0^t 
 \Bigl\{ 
 (t-s)^{2-\frac{3}{\theta}-\frac{\kappa'}{\theta}} \| {\mathbf X} \|_{\drivers{1}{\kappa}} 
 +
 s^{-q+\kappa'}
						(t-s)^{ \frac32 -\frac{2}{\theta} -\kappa'}
						\|(v,w)\|_{\sols{T}{\kappa}{\kappa'}}
						 \\
						 &\qquad \qquad
						 + 
s^{-q'+\kappa'-\kappa}
						(t-s)^{\frac{1}{\theta}-\kappa'} \|(v,w)\|_{\sols{T}{\kappa}{\kappa'}}
						\Bigr\}
 (t-s)^{-( \frac{4}{\theta} -1 + \frac{\kappa +\kappa'}{\theta})}
 ds\\
 &=
 \| {\mathbf X} \|_{\drivers{1}{\kappa}} 
 \int_0^t 
 \Bigl\{ 
 (t-s)^{3-\frac{7}{\theta}-\frac{\kappa+2\kappa'}{\theta}} 
\| {\mathbf X} \|_{\drivers{1}{\kappa}} 
 +
 s^{-q+\kappa'}
						(t-s)^{ \frac52 -\frac{6}{\theta} -\frac{\kappa +(1+\theta)\kappa'}{\theta}}
						\|(v,w)\|_{\sols{T}{\kappa}{\kappa'}}
						 \\
						 &\qquad \qquad
						 + 
s^{-q'+\kappa'-\kappa}
						(t-s)^{1-\frac{3}{\theta}-\frac{\kappa +(1+\theta)\kappa'}{\theta}} \|(v,w)\|_{\sols{T}{\kappa}{\kappa'}}
						\Bigr\}
 ds.
\end{align*}
Since we have
(\ref{eq:180208-11}),
(\ref{eq:180101-4}) and
\[
4-\frac{7}{\theta}>0, \quad
\frac72 -\frac{6}{\theta} -q>0, \quad
2-\frac{3}{\theta}>0,
\]
we are in the position of using
Proposition \ref{prop:180131-1}
to have
\begin{align*}
\lefteqn{
\Bigl\|{\rm III} \Bigr\|_{\HolBesSp{- \frac12 \theta +2+\kappa' }}
}\\
&\lesssim
 \| {\mathbf X} \|_{\drivers{1}{\kappa}} 
 (
 t^{4-\frac{7}{\theta}-\frac{\kappa+2\kappa'}{\theta}} \| {\mathbf X} \|_{\drivers{1}{\kappa}} 
 +
 t^{\frac52 -\frac{6}{\theta} -q-\frac{\kappa +\kappa'}{\theta}}
						\|(v,w)\|_{\sols{T}{\kappa}{\kappa'}}
+
t^{2-\frac{3}{\theta}-q'-\frac{(1+\theta)\kappa +\kappa'}{\theta}} \|(v,w)\|_{\sols{T}{\kappa}{\kappa'}}).
\end{align*}
Thus,
from (\ref{eq.20171228_1})
we conclude
\begin{align*}
\Bigl\|{\rm III} \Bigr\|_{\HolBesSp{- \frac12 \theta +2+\kappa' }}
 &\lesssim
 \|{\mathbf X} \|_{\drivers{1}{\kappa}}^2 
 + 
 t^{- (\frac{5}{\theta} -3 +2q - \frac{\theta -1}{\theta}\kappa'
)
 }
 \|{\mathbf X} \|_{\drivers{1}{\kappa}}\|(v,w)\|_{\sols{T}{\kappa}{\kappa'}}.
 \end{align*}

 %
 % 
%
%

%%%%%%%%%%%\vspace{20mm}

\paragraph{Estimate of ${\rm IV}$}
We write
\begin{align*}
{\rm IV}_Y=
\int_0^t [ P^{\theta /2}_{t-s}, R^{\perp}Y_s \rpara]
\nabla X_s ds,
%\\
\quad
{\rm IV}_{v+w}=
\int_0^t [ P^{\theta /2}_{t-s}, R^{\perp}(v_s +w_s) \rpara]
\nabla X_s ds
%\\
%
\end{align*}
We use Proposition \ref{prop_20160927051159}
and $-3 + \frac{7}{\theta} <1$ to have
\begin{align*}
\Bigl\|{\rm IV}_Y \Bigr\|_{\HolBesSp{- \frac12 \theta +2+\kappa'}}
\lesssim
 \int_0^t (t-s)^{3-\frac{7}{\theta}-\frac{\kappa' +2\kappa}{\theta}} 
 \| R^{\perp} Y_s\|_{\HolBesSp{ 2\theta -3 -\kappa}} 
 \| \nabla X_s\|_{\HolBesSp{\frac12 \theta -2 -\kappa }} ds
 \lesssim
 \|{\mathbf X} \|_{\drivers{1}{\kappa}}^2.
 \end{align*}
Meanwhile,
Lemma \ref{lem:180122-1}
and
Corollary \ref{cor:180201-2}
yield
\begin{align*}
\Bigl\|{\rm IV}_{v+w}
 \Bigr\|_{\HolBesSp{- \frac12\theta+2+\kappa'}}
 & \lesssim
 \|{\mathbf X} \|_{\drivers{1}{\kappa}}
\|(v,w)\|_{\sols{T}{\kappa}{\kappa'}} \int_0^t (t-s)^{3-\frac{7}{\theta}-\frac{\kappa' +2\kappa}{\theta}} 
s^{-\frac12+\frac{1}{\theta}+\frac{\kappa}{\theta}-q+\frac{\theta-1}{\theta}\kappa'}
ds
 \\
& \sim
t^{- ( \frac{6}{\theta}-\frac72 +q 
							 + \frac{\kappa+(2-\theta)\kappa'}{\theta})} 
\|{\mathbf X} \|_{\drivers{1}{\kappa}}\|(v,w)\|_{\sols{T}{\kappa}{\kappa'}}.
 \end{align*}
Using \eqref{eq.20171228_1} we obtain the desired estimate.
\end{proof}

%%%%%%%%%%%%%%%%%%%%%%
%%% 
%\newpage

Now we estimate $G$.
Arithmetic shows
\begin{equation}\label{eq:180101-19}
\frac12\theta-2-\kappa<0,
\end{equation}
\begin{align}
\label{eq:180101-11}
-\frac52\theta +5+ \theta \kappa'
>0>
\frac52\theta-5- \kappa', \quad
-\frac52\theta +5+ \theta \kappa'
+
\frac52\theta-5- \kappa'
=(\theta-1)\kappa'>0,
\end{align}
\begin{equation}\label{eq:180122-20}
\frac12 + q - \frac{1}{\theta} + \frac{2- \theta}{\theta}\kappa' 
\le \frac{5}{\theta} -3 +2q - \frac{\theta -1}{\theta}\kappa',
\end{equation}
\begin{equation}\label{eq:180122-161}
\frac{2}{\theta}-1+q+\kappa
\le 
 \frac{5}{\theta} -3 +2q - \frac{\theta -1}{\theta}\kappa',
\end{equation}	
\begin{equation}\label{eq:180122-16}
 \frac{6}{\theta}-\frac72 + q +\frac{3-\theta}{\theta} \kappa' 
\le 
 \frac{5}{\theta} -3 +2q - \frac{\theta -1}{\theta}\kappa',
\end{equation}	
\begin{equation}\label{eq:180101-21}
-\frac{1}{\theta} \left\{
\frac32 \theta -2 - \theta q +(\theta -1) \kappa \right\}
\le \frac{5}{\theta} -3 +2q - \frac{\theta -1}{\theta}\kappa',
\end{equation}
\begin{equation}\label{eq:180101-15}
\frac72 \theta -5 -\theta q' - \theta\kappa 
					\le 0 \le 
						-2\theta +4 +2\kappa'
					\le
						\frac32 \theta -2-\kappa'
						\end{equation}
if $0<\kappa<\kappa' \ll 1$.
Since $4\theta > 7$
and $0<\kappa<\kappa' \ll 1$, 
it holds 
\begin{equation}\label{eq:180122-18}
-4 + \frac{7}{\theta} +q' 
+ \frac{ \kappa' +\theta\kappa}{\theta} < 
\frac{5}{\theta} -3 +2q - \frac{\theta -1}{\theta}\kappa',
\end{equation}
\begin{equation}\label{eq:180101-13}
(2 \theta -3 -\kappa)+(2 \theta -4 -\kappa)=4\theta-7-2\kappa>0.
\end{equation}

\begin{lemma}\label{lem_20171219}
	For any $(v,w)\in\sols{T}{\kappa}{\kappa'}$ and $0<t\leq T$, we have
	\begin{eqnarray*}
	\lefteqn{
		\|G(v,w)_t\|_{\HolBesSp{\frac52\theta-5-(\kappa' +\kappa)}}
		}
		\\
		&\le
\const
			 t^{- ( \frac{5}{\theta} -3 +2q - \frac{\theta -1}{\theta}\kappa') }
			 			\Big(1+
				 \|v_0\|_{\HolBesSp{\frac32 \theta -2 - \theta q + (\theta -1)\kappa'}}
				+
					\|(v,w)\|_{\sols{T}{\kappa}{\kappa'}}^3
					+\|(v,w)\|_{\sols{T}{\kappa}{\kappa'}}
			\Big).
	\end{eqnarray*}
	Here, $\const$ is a positive constant depending only on $\kappa$, $\kappa'$, 
	and $\|{\mathbf X}\|_{\drivers{1}{\kappa}}$
	and it is given by a third-order polynomial 
	in $\|{\mathbf X}\|_{\drivers{1}{\kappa}}$.
In particular,
\[
\|G(v,w)\|_{{\mathcal E}^{ \frac{5}{\theta} -3 +2q - \frac{\theta -1}{\theta}\kappa'}_T\HolBesSp{\frac52\theta-5-(\kappa' +\kappa)}}
\le \const\Big(1+
				 \|v_0\|_{\HolBesSp{\frac32 \theta -2 - \theta q + (\theta -1)\kappa'}}
				+
					\|(v,w)\|_{\sols{T}{\kappa}{\kappa'}}^3
					+\|(v,w)\|_{\sols{T}{\kappa}{\kappa'}}
			\Big).
\]
\end{lemma}

Noteworthy in Lemma \ref{lem_20171219} 
is the fact that we have
\[
\frac{5}{\theta} -3 +2q - \frac{\theta -1}{\theta}\kappa'<1.
\]
\bigskip
We set
\begin{equation}\label{eq:180227-3}
T_1=\Phi \cdot \nabla (Y+v+w), \quad
T_2= 
		 \comC (\Phi, R^{\perp} V, \nabla X)
+
				\comC (\Phi, \nabla V, R^{\perp} X), \quad
T_3=
		\Phi \lpara \nabla X,
\end{equation}
\[
T_4=R^{\perp}w \reso \nabla X, \quad
T_5= R^{\perp}X \cdot \nabla w, \quad
T_6=\{ R^{\perp} (\Phi \rpara V) - \Phi \rpara R^{\perp} V \}
		 \reso \nabla X , \quad
\]
\[
T_7= R^{\perp} X \cdot \{\nabla \Phi \rpara V \}, \quad
T_8=
		 R^{\perp} X (\rpara + \lpara)
		 \{\Phi \rpara \nabla V \} +\Phi \cdot \hat{W} + \Phi \cdot W,
\]
\[
T_9=Z+\hat{Z}+X+Y+v+w, \quad
T_{10}=R^{\perp}	\com (v,w) \reso \nabla X
				 +R^{\perp} X \cdot \nabla \com (v,w),
\]
so that
\[
G(v,w)=\sum\limits_{j=1}^{10}T_j.
\]
In Step $j$ below,
we estimate $T_j$,
$j=1,2,\ldots,10$.

\begin{proof}
{\bf [Step 1]}~
First, we calculate $T_1$.
Insert the definition of $\Phi$ to obtain $9$ terms.
Among the $9$ terms,
the most difficult one is $R^{\perp} v \cdot \nabla v$,
which determines the indices of our solution spaces.
We can use
Corollary \ref{cor:180131-1}(3)
in view of
(\ref{eq:180101-11})
to have
\[
\| R^{\perp} v_t \cdot \nabla v_t \|_{ \HolBesSp{ \frac52\theta-5- \kappa'}} 
\lesssim
 \| R^{\perp} v_t \|_{ \HolBesSp{ -\frac52\theta +5+ \theta \kappa' }} 
 \| \nabla v_t \|_{ \HolBesSp{ \frac52\theta-5- \kappa' }}.
\]
As before,
from
Lemmas \ref{lem:180122-1} and \ref{lm.20171227}
\begin{equation}\label{eq:180430-1}
\| R^{\perp} v_t \cdot \nabla v_t \|_{ \HolBesSp{ \frac52\theta-5- \kappa'}} 
\lesssim
 \| v_t \|_{ \HolBesSp{ -\frac52\theta +5+ \theta \kappa' }} 
 \| v_t \|_{ \HolBesSp{ \frac52\theta -4 - \kappa' }}.
\end{equation}
From
Proposition \ref{prop:180131-12}(1),
we deduce
\begin{align*}
\|v_t\|_{\HolBesSp{-\frac52\theta +5+ \theta \kappa' }}
					&\lesssim
						t^{
							-\frac{1}{\theta}
							\left\{ 
								-4\theta+7 +\theta q
								 +\kappa' 
							\right\}
					},
\|v\|_{\cL_T^{q-\kappa',\frac32\theta -2 -\kappa',1-\frac{\kappa'}{\theta}}}\\
\|v_t\|_{\HolBesSp{ \frac52\theta -4 - \kappa' }}
					&\lesssim
						t^{
							-\frac{1}{\theta}
							\left\{ 
								 \theta -2 - \theta q
								 -\theta \kappa' 
							\right\}
					}
\|v\|_{\cL_T^{q-\kappa',\frac32\theta -2 -\kappa',1-\frac{\kappa'}{\theta}}}.
\end{align*}
Inserting the estimate into (\ref{eq:180430-1}),
we have
\begin{align}\label{eq:180211-3}
\| R^{\perp} v_t \cdot \nabla v_t \|_{ \HolBesSp{ \frac52\theta-5- \kappa'}} 
&\lesssim
 t^{- ( \frac{5}{\theta} -3 +2q - \frac{\theta -1}{\theta}\kappa') }
 \|v\|_{\cL_T^{q -\kappa', \frac32 \theta -2 -\kappa', 1-\frac{\kappa'}{\theta} }}^2\\
\nonumber
&\le
 t^{- ( \frac{5}{\theta} -3 +2q - \frac{\theta -1}{\theta}\kappa') }
\|(v,w)\|_{\sols{T}{\kappa}{\kappa'}}^2. 
\end{align}
Note that 
$ \| w_t \|_{ \HolBesSp{ -\frac52\theta-5+ \theta \kappa' }}$
and 
$ \| w_t \|_{ \HolBesSp{ \frac52\theta -4 - \kappa' }} $
admit similar estimates to (actually better estimates than) those of 
$ \| v_t \|_{ \HolBesSp{ -\frac52\theta-5+ \theta \kappa' }}$
and 
$ \| v_t \|_{ \HolBesSp{ \frac52\theta -4 - \kappa' }}$, respectively.
So, 
$R^{\perp} v_t \cdot \nabla w_t + R^{\perp} w_t \cdot \nabla v_t + R^{\perp} w_t \cdot \nabla w_t$
admits the same estimate.

Since
$\frac52\theta-5\le 2 \theta -4 $,
we have
\[
\|R^{\perp} Y_t \cdot \nabla Y_t
\|_{ \HolBesSp{\frac52\theta-5- \kappa' }} 
\lesssim
\|R^{\perp} Y_t \cdot \nabla Y_t
\|_{ \HolBesSp{2 \theta -4 - \kappa}}.
\]
From
(\ref{eq:180101-13})
we are in the position of using
Corollary \ref{cor:180131-1}(3)
with $\alpha=2\theta-4-\kappa$, $\beta=2\theta-3$
to have
\begin{equation}\label{eq:20200113-1}
\|R^{\perp} Y_t \cdot \nabla Y_t
\|_{ \HolBesSp{2 \theta -4 - \kappa}}
\lesssim
\|R^{\perp} Y_t 
\|_{ \HolBesSp{2 \theta -3 - \kappa}} 
\|\nabla Y_t
\|_{ \HolBesSp{2 \theta -4 - \kappa}}.
\end{equation}
Note that the condition $\theta >7/4$ is used here.
If we use 
Lemmas \ref{lem:180122-1} and \ref{lm.20171227},
then we obtain
\begin{align}%\label{eq:180211-2}
\nonumber
\|R^{\perp} Y_t \cdot \nabla Y_t
\|_{ \HolBesSp{\frac52\theta-5- \kappa' }} 
\lesssim
\| Y_t 
\|^2_{ \HolBesSp{2 \theta -3 - \kappa}} 
\le
 \| {\mathbf X}\|_{\drivers{1}{\kappa}}^2.
\end{align}

In a similar way, 
thanks to Proposition \ref{prop:180131-12},
\begin{align}%\label{eq:180213-1}
\nonumber
\|R^{\perp} v_t \cdot \nabla Y_t
\|_{ \HolBesSp{\frac52\theta-5- \kappa' }} 
&\lesssim
 \| R^{\perp} v_t \|_{ \HolBesSp{ -\frac52\theta +5+ \theta \kappa' }} 
 \| \nabla Y_t \|_{ \HolBesSp{ \frac52\theta-5- \kappa' }}
\\
 \nonumber
 &
 \lesssim
 \| R^{\perp} v_t \|_{ \HolBesSp{ -\frac52\theta +5+ \theta \kappa' }} 
 \| \nabla Y_t \|_{ \HolBesSp{ 2\theta-4- \kappa}}
\\
\nonumber
& 
\lesssim
 \| v_t \|_{ \HolBesSp{ -\frac52\theta +5+ \theta \kappa' }} 
 \| Y_t \|_{ \HolBesSp{ 2\theta-3- \kappa}}
\\
 \nonumber
 &
 \lesssim
					 t^{
							-\frac{1}{\theta}
							\left\{ 
								-4\theta+7 +\theta q
								 +\kappa' 
							\right\}
					}
										\|v\|_{\cL_T^{q-\kappa',\frac32\theta -2 -\kappa',1-\frac{\kappa'}{\theta}}}
 \| {\mathbf X}\|_{\drivers{1}{\kappa}}.
 \end{align}
In view of (\ref{eq:180122-18}),
this estimate is better than \eqref{eq:180211-3}.
The term
$R^{\perp} w_t \cdot \nabla Y_t$ admits the same estimate.

We also have 
\begin{align}
\nonumber
\|R^{\perp} Y_t \cdot \nabla v_t
\|_{ \HolBesSp{\frac52\theta-5- \kappa' }} 
&\lesssim
 \| R^{\perp} Y_t \|_{ \HolBesSp{2\theta-3- \kappa }} 
 \| \nabla v_t \|_{ \HolBesSp{(-2\theta +3 +2 \kappa')\vee( \frac52\theta-5- \kappa')}}
\\
 \nonumber
 &
 \lesssim 
 \| {\mathbf X}\|_{\drivers{1}{\kappa}} 
 \| v_t \|_{ \HolBesSp{(-2\theta +4 +2 \kappa')\vee( \frac52\theta-4- \kappa')}}
\\
 \nonumber
 &
 \lesssim
 \Bigl(
 t^{-\left(\frac{6}{\theta}-\frac72 + q +\frac{3-\theta }{\theta} \kappa'	\right)}
					\vee
t^{-\left(1-\frac{2}{\theta}-q-\kappa'\right)	}
					\Bigr)
\|v\|_{\cL_T^{q-\kappa',\frac32\theta -2 -\kappa',1-\frac{\kappa'}{\theta}}}
 \| {\mathbf X}\|_{\drivers{1}{\kappa}}.
 \end{align}
From (\ref{eq:180122-161}) and (\ref{eq:180122-16})
we obtain the desired estimate for $R^{\perp} Y_t \cdot \nabla v_t$.
The term
$R^{\perp} Y_t \cdot \nabla w_t$ can be done in the same way.
Combining these all,
we have the desired estimate for 
$\Phi \cdot \nabla (Y+v+w)$.
\begin{align*}
\|(T_{1})_t\|_{ \HolBesSp{\frac52\theta-5- \kappa' }} 
&\le
			 t^{- ( \frac{5}{\theta} -3 +2q - \frac{\theta -1}{\theta}\kappa') }
					+\|(v,w)\|_{\sols{T}{\kappa}{\kappa'}}.
\end{align*}

\bigskip 
\bigskip

\noindent
{\bf [Step 2]}~We estimate
$\comC (\Phi, R^{\perp} V, \nabla X)$
and
$\comC (\Phi, \nabla V, R^{\perp} X)$
with the help of
Proposition \ref{prop_comm_20160919055939},
the commutator estimate.

To this end we set
\begin{equation}\label{eq:180122-30}
\alpha=-2\theta +4 +2\kappa', \quad
\beta=\frac32 \theta -2 - \kappa, \quad
\gamma=\frac12 \theta -2 - \kappa.
\end{equation}
Note that $\alpha \in (0,1)$, that $\beta+\gamma<0$
and that $\alpha+\beta+\gamma=2\kappa'-2\kappa>0$.
Thus,
we deduce
\[
\|
\comC (\Phi_t, R^{\perp} V_t, \nabla X_t)
\|_{ \HolBesSp{2( \kappa' -\kappa)}}
\lesssim
\|
\Phi_t
\|_{ \HolBesSp{-2\theta +4 +2\kappa'}}
\| 
R^{\perp} V_t
\|_{ \HolBesSp{ \frac32 \theta -2 - \kappa}}
\|
 \nabla X_t
\|_{ \HolBesSp{ \frac12 \theta -2 - \kappa}}
\]
from Proposition \ref{prop_comm_20160919055939}.
Using
Lemmas \ref{lem:180122-1} and \ref{lm.20171227},
we obtain
\[
\|
\comC (\Phi_t, R^{\perp} V_t, \nabla X_t)
\|_{ \HolBesSp{2( \kappa' -\kappa)}}
\lesssim
\|
\Phi_t
\|_{ \HolBesSp{-2\theta +4 +2\kappa'}}
\| 
 V_t
\|_{ \HolBesSp{ \frac32 \theta -2 - \kappa}}
\| X_t
\|_{ \HolBesSp{ \frac12 \theta -1 - \kappa}}.
\]
Using
Corollary \ref{cor:180206-2},
we estimate
$\comC (\Phi, R^{\perp} V, \nabla X)$ as follows:
\begin{align*}
\|
\comC (\Phi_t, R^{\perp} V_t, \nabla X_t)
\|_{ \HolBesSp{2( \kappa' -\kappa)}}
&\lesssim
\{
 \|{\mathbf X}\|_{\drivers{1}{\kappa}} 
+ t^{ - ( \frac{6}{\theta}-\frac72 + q +\frac{3-\theta }{\theta} \kappa') } 
\|(v,w)\|_{\sols{T}{\kappa}{\kappa'}} 
\}
\|{\mathbf X}\|_{\drivers{1}{\kappa}}^2\\
&\lesssim
t^{ - ( \frac{6}{\theta}-\frac72 + q +\frac{3-\theta }{\theta} \kappa') } \{
 \|{\mathbf X}\|_{\drivers{1}{\kappa}} 
+ 
\|(v,w)\|_{\sols{T}{\kappa}{\kappa'}} 
\}
\|{\mathbf X}\|_{\drivers{1}{\kappa}}^2.
\end{align*}
%
%,
As a result,
from (\ref{eq:180122-16})
we obtain
\begin{align}\label{eq:180128-1}
\nonumber
\|
\comC (\Phi_t, R^{\perp} V_t, \nabla X_t)
\|_{ \HolBesSp{2( \kappa' -\kappa)}}
&\lesssim
\{
 \|{\mathbf X}\|_{\drivers{1}{\kappa}} 
+ t^{ - (
\frac{5}{\theta} -3 +2q - \frac{\theta -1}{\theta}\kappa'
)
} 
\|(v,w)\|_{\sols{T}{\kappa}{\kappa'}} 
\}
\|{\mathbf X}\|_{\drivers{1}{\kappa}}^2\\
&\lesssim t^{ - (
\frac{5}{\theta} -3 +2q - \frac{\theta -1}{\theta}\kappa'
)
} 
\{
 \|{\mathbf X}\|_{\drivers{1}{\kappa}} 
+
\|(v,w)\|_{\sols{T}{\kappa}{\kappa'}} 
\}
\|{\mathbf X}\|_{\drivers{1}{\kappa}}^2.
\end{align}

In a similar way, 
we set
\[
\alpha^\dagger=-2\theta +4 +2\kappa', \quad
\beta^\dagger=\frac32 \theta -3 - \kappa, \quad
\gamma^\dagger=\frac12 \theta -1 - \kappa.
\]
Note that
$\alpha^\dagger+\beta^\dagger+\gamma^\dagger=2\kappa'-2\kappa>0$
and
$\beta^\dagger+\gamma^\dagger=2\theta-4-\kappa'-\kappa$.
Thus
using
Proposition \ref{prop_comm_20160919055939}
and
Lemmas \ref{lem:180122-1} and \ref{lm.20171227},
we have
\begin{align*}
\|
\comC (\Phi_t, \nabla V_t, R^{\perp} X_t)
\|_{ \HolBesSp{ 2(\kappa' -\kappa)}}
&\lesssim
\|
\Phi_t
\|_{ \HolBesSp{-2\theta +4 +2\kappa'}}
\| 
\nabla V_t
\|_{ \HolBesSp{ \frac32 \theta -3 - \kappa}}
\|
R^{\perp} X_t
\|_{ \HolBesSp{ \frac12 \theta -1 - \kappa}}
\\
&\lesssim
\{
 \|{\mathbf X}\|_{\drivers{1}{\kappa}} 
+ t^{ - (
 \frac{6}{\theta}-\frac72 + q +\frac{3-\theta }{\theta} \kappa' 
) } 
\|(v,w)\|_{\sols{T}{\kappa}{\kappa'}} 
\}
\|{\mathbf X}\|_{\drivers{1}{\kappa}}^2.
\end{align*}
As a result
once again
from (\ref{eq:180122-16})
\begin{align}\label{eq:180128-2}
\|
\comC (\Phi_t, \nabla V_t, R^{\perp} X_t)
\|_{ \HolBesSp{ 2(\kappa' -\kappa)}}
\nonumber
&\lesssim
\{
 \|{\mathbf X}\|_{\drivers{1}{\kappa}} 
+ t^{ - (
\frac{5}{\theta} -3 +2q - \frac{\theta -1}{\theta}\kappa'
)
} 
\|(v,w)\|_{\sols{T}{\kappa}{\kappa'}} 
\}
\|{\mathbf X}\|_{\drivers{1}{\kappa}}\\
&\lesssim
t^{ - (
\frac{5}{\theta} -3 +2q - \frac{\theta -1}{\theta}\kappa'
)
} 
\{
 \|{\mathbf X}\|_{\drivers{1}{\kappa}} 
+ 
\|(v,w)\|_{\sols{T}{\kappa}{\kappa'}} 
\}
\|{\mathbf X}\|_{\drivers{1}{\kappa}}.
\end{align}
Since
$2(\kappa'-\kappa)>
\frac52(\theta-2)-\kappa-\kappa'$,
we
can deduce the desired estimate
of
$\comC (\Phi, R^{\perp} V, \nabla X)$
and
$\comC (\Phi, \nabla V, R^{\perp} X)$
from
(\ref{eq:180128-1})
and
(\ref{eq:180128-2}):
\[
\|(T_{2})_t\|_{\frac52(\theta-2)-\kappa-\kappa'}
\lesssim
t^{ - (
\frac{5}{\theta} -3 +2q - \frac{\theta -1}{\theta}\kappa'
)
} 
\{
 \|{\mathbf X}\|_{\drivers{1}{\kappa}} 
+ 
\|(v,w)\|_{\sols{T}{\kappa}{\kappa'}} 
\}
\|{\mathbf X}\|_{\drivers{1}{\kappa}}.
\]

\bigskip 
\bigskip

\noindent
{\bf [Step 3]}
Since we have
(\ref{eq:180101-19}),
we can use
Proposition \ref{prop_20160926062106}(2)
to have
\[
\|
\Phi_t \lpara \nabla X_t
\|_{ \HolBesSp{\frac52\theta-5- ( \kappa +\kappa')}}
\lesssim
\|
\Phi_t 
\|_{ \HolBesSp{2\theta -3 - \kappa' }}
\|
 \nabla X_t
\|_{ \HolBesSp{\frac12 \theta- 2- \kappa}}.
\]
By Lemma \ref{lem:180122-1}
\[
\|
\Phi_t \lpara \nabla X_t
\|_{ \HolBesSp{\frac52\theta-5- ( \kappa +\kappa')}}
\lesssim
\|
\Phi_t 
\|_{ \HolBesSp{2\theta -3 - \kappa' }}
\|
X_t
\|_{ \HolBesSp{\frac12 \theta- 1- \kappa}}.
\]
Lemma \ref{lem:180201-3}
yields
\begin{align*}
\|
\Phi_t \lpara \nabla X_t
\|_{ \HolBesSp{\frac52\theta-5- ( \kappa +\kappa')}}
\lesssim
\{
 \|{\mathbf X}\|_{\drivers{1}{\kappa}} 
+ t^{ - (
\frac12 + q - \frac{1}{\theta} + \frac{2- \theta}{\theta}\kappa' 
)
} 
\|(v,w)\|_{\sols{T}{\kappa}{\kappa'}} 
\}
\|{\mathbf X}\|_{\drivers{1}{\kappa}}.
\end{align*}
From (\ref{eq:180122-20}),
we conclude that
the estimate of
$
\|
\Phi_t \lpara \nabla X_t
\|_{ \HolBesSp{\frac52\theta-5- ( \kappa +\kappa')}}$
is valid:
\begin{align*}
\|(T_3)_t
\|_{ \HolBesSp{\frac52\theta-5- ( \kappa +\kappa')}}
&\lesssim
\{
 \|{\mathbf X}\|_{\drivers{1}{\kappa}} 
+ t^{ - (
\frac{5}{\theta} -3 +2q - \frac{\theta -1}{\theta}\kappa'
)
} 
\|(v,w)\|_{\sols{T}{\kappa}{\kappa'}} 
\}
\|{\mathbf X}\|_{\drivers{1}{\kappa}}\\
&\lesssim t^{ - (
\frac{5}{\theta} -3 +2q - \frac{\theta -1}{\theta}\kappa'
)
} 
\{
 \|{\mathbf X}\|_{\drivers{1}{\kappa}} 
+
\|(v,w)\|_{\sols{T}{\kappa}{\kappa'}} 
\}
\|{\mathbf X}\|_{\drivers{1}{\kappa}}.
\end{align*}

\bigskip 
\bigskip

\noindent
{\bf [Step 4]}~From
Proposition \ref{prop_20160926062106}(3),
a basic property of the resonant, 
we see that
 \begin{align*}
\|
R^{\perp} w_t \reso \nabla X_t
\|_{ \HolBesSp{ \kappa' - \kappa}}
\lesssim
\|
R^{\perp} w_t
\|_{ \HolBesSp{ - \frac12 \theta +2+ \kappa' }}
\|
 \nabla X_t
\|_{ \HolBesSp{\frac12 \theta -2 - \kappa}}
 \lesssim
 t^{ -( 
 -4 + \frac{7}{\theta} +q' 
+ \frac{ \kappa' +\theta\kappa}{\theta})
 } 
\|(v,w)\|_{\sols{T}{\kappa}{\kappa'}} 
\|{\mathbf X}\|_{\drivers{1}{\kappa}}.
\end{align*}
Thus,
from (\ref{eq:180122-18}),
\[
\|(T_4)_t
\|_{ \HolBesSp{\frac52\theta-5- \kappa' - \kappa}}
\lesssim
\|R^{\perp} w_t \reso \nabla X_t
\|_{ \HolBesSp{ \kappa' - \kappa}}
\lesssim
t^{ - (
\frac{5}{\theta} -3 +2q - \frac{\theta -1}{\theta}\kappa'
)
} 
\|(v,w)\|_{\sols{T}{\kappa}{\kappa'}} 
\|{\mathbf X}\|_{\drivers{1}{\kappa}}.
\]
\bigskip 
\bigskip

\noindent
{\bf [Step 5]}~We recall
\[
\eta=- \frac{1}{\theta}\left(
								 \frac32 \theta -2 -\theta q
								 +(\theta -1)\kappa \right)
< \frac{6}{\theta}-\frac72 + q -\frac{3-\theta }{\theta} \kappa'.
\]
See (\ref{eq:180122-22}).

By Corollary \ref{cor:180131-1}(3),
a basic property of the usual product, 
and
Lemma \ref{lem:180201-3},
we obtain
\begin{align*}
\|
\Phi_t \cdot W_t
\|_{ \HolBesSp{2\theta -4 - \kappa}}
+
\|
\Phi_t \cdot \hat{W}_t
\|_{ \HolBesSp{2\theta -4 - \kappa}}
&\lesssim
\|
\Phi_t 
\|_{ \HolBesSp{-2\theta +4 + 2\kappa' }}
(
\|
W_t
\|_{ \HolBesSp{2\theta -4 - \kappa}}
+
\|
\hat{W}_t
\|_{ \HolBesSp{2\theta -4- \kappa}}
)
\\
&\lesssim
\{ \|{\mathbf X}\|_{\drivers{1}{\kappa}} 
+ t^{ - 
( \frac{6}{\theta}-\frac72 + q -\frac{3-\theta }{\theta} \kappa')
} 
\|(v,w)\|_{\sols{T}{\kappa}{\kappa'}} 
\}
\|{\mathbf X}\|_{\drivers{1}{\kappa}}
\\
&\le t^{ - 
( \frac{6}{\theta}-\frac72 + q +\frac{3-\theta }{\theta} \kappa')
} 
\{ \|{\mathbf X}\|_{\drivers{1}{\kappa}} 
+
\|(v,w)\|_{\sols{T}{\kappa}{\kappa'}} 
\}
\|{\mathbf X}\|_{\drivers{1}{\kappa}}.
\end{align*}
Since
$\dfrac12\theta-1-\kappa<0<2\kappa'-2\kappa$
and
$
\frac12 \theta -1 - \kappa- \frac12 \theta +1 + \kappa'>0
>\frac12 \theta -1 - \kappa,
$
we are in the position of using
Corollary \ref{cor:180131-1}(3)
to have
\[
\|
R^{\perp} X_t \cdot \nabla w_t
\|_{ \HolBesSp{\frac12 \theta -1 - \kappa}}
\lesssim
\|
R^{\perp} X_t \cdot \nabla w_t
\|_{ \HolBesSp{2\kappa'- 2\kappa}}
\lesssim
\|
R^{\perp} X_t
\|_{ \HolBesSp{\frac12 \theta -1 - \kappa }}
\|
 \nabla w_t
\|_{ \HolBesSp{ - \frac12 \theta +1 + \kappa'}}.
\]
As before we use
Lemmas \ref{lem:180122-1} and \ref{lm.20171227}
to obtain
\[
\|
R^{\perp} X_t \cdot \nabla w_t
\|_{ \HolBesSp{\frac12 \theta -1 - \kappa}}
\lesssim
\|
 X_t
\|_{ \HolBesSp{\frac12 \theta -1 - \kappa }}
\|
w_t
\|_{ \HolBesSp{ - \frac12 \theta +2+ \kappa'}}.
\]
It follows from Proposition \ref{prop:180131-12}(2) that
 \begin{align*}
\|
R^{\perp} X_t \cdot \nabla w_t
\|_{ \HolBesSp{\frac12 \theta -1 - \kappa}}
&\lesssim
 t^{ -( 
 -4 + \frac{7}{\theta} +q' 
+ \frac{ \kappa' +\theta\kappa}{\theta} 
)
 } 
\|(v,w)\|_{\sols{T}{\kappa}{\kappa'}} 
\|{\mathbf X}\|_{\drivers{1}{\kappa}}.
\end{align*}
Thus,
from (\ref{eq:180122-18}),
\[
\|
R^{\perp} X_t \cdot \nabla w_t
\|_{ \HolBesSp{\frac52\theta-5- \kappa' - \kappa}}
\lesssim
\|
R^{\perp} X_t \cdot \nabla w_t
\|_{ \HolBesSp{\frac12 \theta -1 - \kappa}}
\lesssim t^{ - (
\frac{5}{\theta} -3 +2q - \frac{\theta -1}{\theta}\kappa'
)
} 
\|(v,w)\|_{\sols{T}{\kappa}{\kappa'}} 
\|{\mathbf X}\|_{\drivers{1}{\kappa}}.
\]
In total,
we have
\[
\|(T_5)_t
\|_{ \HolBesSp{\frac52\theta-5- \kappa' - \kappa}}
\lesssim t^{ - (
\frac{5}{\theta} -3 +2q - \frac{\theta -1}{\theta}\kappa'
)
} 
\|(v,w)\|_{\sols{T}{\kappa}{\kappa'}} 
\|{\mathbf X}\|_{\drivers{1}{\kappa}}.
\]

\bigskip 
\bigskip

%
%\newpage
%
\noindent
{\bf [Step 6]}
Since
(\ref{eq:180101-1})
is satisfied,
we are in the position of using Lemma \ref{lem:161117-107}.
The result is:
\[
\|
R^{\perp} (\Phi_t \rpara V_t) 
		 - \Phi_t \rpara R^{\perp} V_t
		 \|_{ \HolBesSp{ -\frac12 \theta +2 + 2\kappa' -\kappa}}
		 %\nonumber
		 \lesssim
		 \| \Phi_t \|_{ \HolBesSp{ -2\theta +4 +2 \kappa' }}
		 \| V_t \|_{ \HolBesSp{\frac32 \theta -2 - \kappa}}.
\]
If we use
Lemma \ref{lem:180201-3},
then we obtain
\begin{align*}
\|
R^{\perp} (\Phi_t \rpara V_t) 
		 - \Phi_t \rpara R^{\perp} V_t
		 \|_{ \HolBesSp{ -\frac12 \theta +2 + 2\kappa' -\kappa}}
		 &\lesssim		 
		 		 \{ \|{\mathbf X}\|_{\drivers{1}{\kappa}} 
+ t^{ - 
( \frac{6}{\theta}-\frac72 + q +\frac{3-\theta }{\theta} \kappa')
 } 
\|(v,w)\|_{\sols{T}{\kappa}{\kappa'}} \}
\|{\mathbf X}\|_{\drivers{1}{\kappa}}.
		 		 \end{align*}	
Since
$\kappa'>\kappa$,
we deduce from
Proposition \ref{prop_20160926062106}(3) 
and
Lemma \ref{lem:180208-7}
that
\begin{align*}
\|(T_6)_t\|_{ \HolBesSp{2 (\kappa' -\kappa)}}
&\lesssim
\|
\{ R^{\perp} (\Phi_t \rpara V_t) 
		 - \Phi_t \rpara R^{\perp} V_t
		 \} \|_{ \HolBesSp{ -\frac12 \theta +2 + 2\kappa' -\kappa}}
		 \| 
		 \nabla X_t
		 \|_{ \HolBesSp{\frac12 \theta -2 -\kappa}}\\
	&\lesssim		 
		 		 \{ \|{\mathbf X}\|_{\drivers{1}{\kappa}} 
+ t^{ -
( \frac{6}{\theta}-\frac72 + q +\frac{3-\theta }{\theta} \kappa')
} 
\|(v,w)\|_{\sols{T}{\kappa}{\kappa'}} \}
\|{\mathbf X}\|_{\drivers{1}{\kappa}}^2.
		 		 \end{align*}			 

As a result
from (\ref{eq:180122-16})
\begin{eqnarray*}
\|(T_6)_t\|_{ \HolBesSp{2 (\kappa' -\kappa)}}
&\lesssim&
\{
 \|{\mathbf X}\|_{\drivers{1}{\kappa}} 
+ t^{ - (
\frac{5}{\theta} -3 +2q - \frac{\theta -1}{\theta}\kappa'
)
} 
\|(v,w)\|_{\sols{T}{\kappa}{\kappa'}} 
\}
\|{\mathbf X}\|_{\drivers{1}{\kappa}}\\
&\lesssim& t^{ - (
\frac{5}{\theta} -3 +2q - \frac{\theta -1}{\theta}\kappa'
)
} 
\{
 \|{\mathbf X}\|_{\drivers{1}{\kappa}} 
+
\|(v,w)\|_{\sols{T}{\kappa}{\kappa'}} 
\}
\|{\mathbf X}\|_{\drivers{1}{\kappa}}.
\end{eqnarray*}
Since
$2(\kappa'-\kappa)>
\frac52(\theta-2)-\kappa-\kappa'$,
we deduce the desired estimate.

\bigskip 
\bigskip

\noindent
{\bf [Step 7]}~
By 
Corollary \ref{cor:180131-1}(3),
a basic property of the usual product
\[
\|
R^{\perp} X_t \cdot ( \nabla \Phi_t \rpara V_t) 
 \|_{ \HolBesSp{\frac12 \theta -1-\kappa}}
\lesssim
\| R^{\perp} X_t \|_{ \HolBesSp{\frac12 \theta -1-\kappa}}
\| \nabla \Phi_t \rpara V_t \|_{ \HolBesSp{-\frac12 \theta +1 + 2\kappa' -\kappa}}.
\]
Since $\kappa'>\kappa$,
this implies
\[
\|
R^{\perp} X_t \cdot ( \nabla \Phi_t \rpara V_t) 
 \|_{ \HolBesSp{\frac12 \theta -1-\kappa}}
\lesssim
\| R^{\perp} X_t \|_{ \HolBesSp{\frac12 \theta -1-\kappa}}
\| \nabla \Phi_t \rpara V_t \|_{ \HolBesSp{-\frac12 \theta +1 + 2\kappa' -\kappa}}.
\]
By
Proposition \ref{prop_20160926062106}(2)
and
Lemma \ref{lm.20171227}
 we have	 
\begin{align*}
\|
R^{\perp} X_t \cdot ( \nabla \Phi_t \rpara V_t) 
 \|_{ \HolBesSp{\frac12 \theta -1-\kappa}}
 &\lesssim
\| X_t \|_{ \HolBesSp{\frac12 \theta -1-\kappa}}
\| \nabla \Phi_t \|_{ \HolBesSp{ -2\theta +3 +2 \kappa' }}
\| V_t \|_{ \HolBesSp{\frac32 \theta -2 - \kappa}}^2.
\end{align*}
If we use
Lemma \ref{lem:180201-3},
then we obtain
\begin{align*}
\|
R^{\perp} X_t \cdot ( \nabla \Phi_t \rpara V_t) 
 \|_{ \HolBesSp{\frac12 \theta -1-\kappa}}
&\lesssim		
 \{ \|{\mathbf X}\|_{\drivers{1}{\kappa}} 
+ t^{ - 
( \frac{6}{\theta}-\frac72 + q +\frac{3-\theta }{\theta} \kappa')
 } 
\|(v,w)\|_{\sols{T}{\kappa}{\kappa'}} \}
\|{\mathbf X}\|_{\drivers{1}{\kappa}}^2.
 \end{align*}			 
As a result
from (\ref{eq:180122-16})
\begin{align*}
\|
R^{\perp} X_t \cdot ( \nabla \Phi_t \rpara V_t) 
 \|_{ \HolBesSp{\frac12 \theta -1-\kappa}}
&\lesssim
\{
 \|{\mathbf X}\|_{\drivers{1}{\kappa}} 
+ t^{ - (
\frac{5}{\theta} -3 +2q - \frac{\theta -1}{\theta}\kappa'
)
} 
\|(v,w)\|_{\sols{T}{\kappa}{\kappa'}} 
\}
\|{\mathbf X}\|_{\drivers{1}{\kappa}}^2\\
&\lesssim t^{ - (
\frac{5}{\theta} -3 +2q - \frac{\theta -1}{\theta}\kappa'
)
} 
\{
 \|{\mathbf X}\|_{\drivers{1}{\kappa}} 
+
\|(v,w)\|_{\sols{T}{\kappa}{\kappa'}} 
\}
\|{\mathbf X}\|_{\drivers{1}{\kappa}}^2.
\end{align*}
Since
$\frac12 \theta -1-\kappa>
\frac52(\theta-2)-\kappa-\kappa'$,
we deduce the desired estimate:
\[
\|(T_7)_t 
 \|_{ \HolBesSp{\frac12 \theta -1-\kappa}}
\lesssim t^{ - (
\frac{5}{\theta} -3 +2q - \frac{\theta -1}{\theta}\kappa'
)
} 
\{
 \|{\mathbf X}\|_{\drivers{1}{\kappa}} 
+
\|(v,w)\|_{\sols{T}{\kappa}{\kappa'}} 
\}
\|{\mathbf X}\|_{\drivers{1}{\kappa}}^2.
\]

\bigskip 
\bigskip

\noindent
{\bf [Step 8]}~
Arithemetic shows
\[
\frac12 \theta -1 -\kappa<0, \quad
\frac32 \theta -3-\kappa<0.
\]
So we are in the position of using
Proposition \ref{prop_20160926062106}(2)
to have
\[
\|
R^{\perp} X_t (\rpara + \lpara)
		 \{\Phi_t \rpara \nabla V_t \} 
 \|_{ \HolBesSp{2\theta -4 -(\kappa + \kappa') }}
\lesssim	
 	\|R^{\perp} X_t \|_{ \HolBesSp{\frac12 \theta -1 -\kappa}}
\| \Phi_t \rpara\nabla V_t \|_{ \HolBesSp{ \frac32 \theta -3-\kappa}}
\]
Proposition \ref{prop_20160926062106}(1) 
yields
\begin{align*}
\|
R^{\perp} X_t (\rpara + \lpara)
		 \{\Phi_t \rpara \nabla V_t \} 
 \|_{ \HolBesSp{2\theta -4 -(\kappa + \kappa') }}
 & \lesssim	
 	\|R^{\perp} X_t \|_{ \HolBesSp{\frac12 \theta -1 -\kappa}}
\| \Phi_t \|_{L^{\infty}}
\| \nabla V_t \|_{ \HolBesSp{ \frac32 \theta -3-\kappa}}
\end{align*}
From
Lemmas \ref{lem:180122-1} and \ref{lm.20171227}
we obtain
\begin{align*}
\|
R^{\perp} X_t (\rpara + \lpara)
		 \{\Phi_t \rpara \nabla V_t \} 
 \|_{ \HolBesSp{2\theta -4 -(\kappa + \kappa') }}
 & \lesssim	
 	\|X_t \|_{ \HolBesSp{\frac12 \theta -1 -\kappa}}
\| \Phi_t \|_{L^{\infty}}
\| V_t \|_{ \HolBesSp{ \frac32 \theta -3-\kappa}}.
\end{align*}
Thus
from
Lemma \ref{lem:180201-3}
\begin{align*}
\|
R^{\perp} X_t (\rpara + \lpara)
		 \{\Phi_t \rpara \nabla V_t \} 
 \|_{ \HolBesSp{2\theta -4 -(\kappa + \kappa') }}
 &
 \lesssim
 	 \{
	 \|{\mathbf X}\|_{\drivers{1}{\kappa}} 
+ t^{ 
\frac{1}{\theta} \{
\frac32 \theta -2 - \theta q +(\theta -1) \kappa \}
 } 
\|(v,w)\|_{\sols{T}{\kappa}{\kappa'}} 
\}
\|{\mathbf X}\|_{\drivers{1}{\kappa}}^2.
 		 \end{align*} 
As a result
from (\ref{eq:180101-21}),
\begin{align*}
\|
R^{\perp} X_t (\rpara + \lpara)
		 \{\Phi_t \rpara \nabla V_t \} 
 \|_{ \HolBesSp{2\theta -4 -(\kappa + \kappa') }}
&\lesssim
\{
 \|{\mathbf X}\|_{\drivers{1}{\kappa}} 
+ t^{ - (
\frac{5}{\theta} -3 +2q - \frac{\theta -1}{\theta}\kappa'
)
} 
\|(v,w)\|_{\sols{T}{\kappa}{\kappa'}} 
\}
\|{\mathbf X}\|_{\drivers{1}{\kappa}}^2\\
&\lesssim t^{ - (
\frac{5}{\theta} -3 +2q - \frac{\theta -1}{\theta}\kappa'
)
} 
\{
 \|{\mathbf X}\|_{\drivers{1}{\kappa}} 
+
\|(v,w)\|_{\sols{T}{\kappa}{\kappa'}} 
\}
\|{\mathbf X}\|_{\drivers{1}{\kappa}}^2.
\end{align*}		 
Since
$2\theta -4 -(\kappa + \kappa') >
\frac52(\theta-2)-\kappa-\kappa'$,
we deduce the desired estimate:
\[
\|(T_8)_t \|_{ \HolBesSp{\frac52(\theta-2)-\kappa-\kappa'}}
\lesssim t^{ - (
\frac{5}{\theta} -3 +2q - \frac{\theta -1}{\theta}\kappa'
)
} 
\{
 \|{\mathbf X}\|_{\drivers{1}{\kappa}} 
+
\|(v,w)\|_{\sols{T}{\kappa}{\kappa'}} 
\}
\|{\mathbf X}\|_{\drivers{1}{\kappa}}^2.
\]

\bigskip 
\bigskip

\noindent
{\bf [Step 9]}~From Corollary \ref{cor:180201-2} with $\gamma=0$
we deduce
\begin{align*}
\|(T_9)_t\|_{ \HolBesSp{\frac52(\theta -2) -\kappa}}
\lesssim
\|{\mathbf X}\|_{\drivers{1}{\kappa}} +\|v_t +w_t \|_{ L^{\infty}}
\lesssim 
\|{\mathbf X}\|_{\drivers{1}{\kappa}} 
+ t^{ \frac{1}{\theta} \{
\frac32 \theta -2 - \theta q +(\theta -1) \kappa 
\}
} 
\|(v,w)\|_{\sols{T}{\kappa}{\kappa'}}.
\end{align*}	As a result
from (\ref{eq:180101-21}),
\begin{align*}
\|(T_9)_t\|_{ \HolBesSp{\frac52(\theta -2) -\kappa-\kappa'}}
&\lesssim 
\|(T_9)_t\|_{ \HolBesSp{\frac52(\theta -2) -\kappa}}\\
&\lesssim 
\|{\mathbf X}\|_{\drivers{1}{\kappa}} 
+ t^{ - (
\frac{5}{\theta} -3 +2q - \frac{\theta -1}{\theta}\kappa'
)
} 
\|(v,w)\|_{\sols{T}{\kappa}{\kappa'}}\\
&\lesssim t^{ - (
\frac{5}{\theta} -3 +2q - \frac{\theta -1}{\theta}\kappa'
)
} 
(
\|{\mathbf X}\|_{\drivers{1}{\kappa}} 
+
\|(v,w)\|_{\sols{T}{\kappa}{\kappa'}}).
\end{align*}

\bigskip 
\bigskip

\noindent
{\bf [Step 10]}~
We can estimate the terms involving the commutator by using
Lemma \ref{lem_20161003140013}.
Indeed,
\begin{align*}
\|R^{\perp}	\com (v,w)_t \reso \nabla X_t	 \|_{ \HolBesSp{\kappa'- \kappa}} 
&\lesssim
\|R^{\perp}	\com (v,w)_t 	 \|_{ \HolBesSp{-\frac12 \theta +2+ \kappa'}} 
\| \nabla X_t	 \|_{ \HolBesSp{\frac12 \theta -2- \kappa}} \\
&\lesssim
\|\com (v,w)_t 	 \|_{ \HolBesSp{-\frac12 \theta +2+ \kappa'}} \|{\mathbf X}\|_{\drivers{1}{\kappa}}, 
\\
\|R^{\perp} X_t \cdot \nabla \com (v,w)_t	 \|_{ \HolBesSp{\frac12 \theta -1- \kappa}} 
&\lesssim
\|R^{\perp} X_t 	 \|_{ \HolBesSp{\frac12 \theta -1 - \kappa}} 
\| \nabla \com (v,w)_t	 \|_{ \HolBesSp{ -\frac12 \theta +1+\kappa'}} 
\\
&\lesssim
\|\com (v,w)_t\|_{ \HolBesSp{-\frac12 \theta +2+ \kappa'}} \|{\mathbf X}\|_{\drivers{1}{\kappa}}.
\end{align*}
We can use Lemma \ref{lem_20161003140013}
to handle
$\|\com (v,w)_t\|_{ \HolBesSp{-\frac12 \theta +2+ \kappa'}}$.
The result is
\begin{align*}
\lefteqn{
\|R^{\perp}	\com (v,w)_t \reso \nabla X_t	 \|_{ \HolBesSp{\kappa'- \kappa}} 
+
\|R^{\perp} X_t \cdot \nabla \com (v,w)_t	 \|_{ \HolBesSp{\frac12 \theta -1- \kappa}} 
}\\
&\lesssim		 \const t^{- ( \frac{5}{\theta} -3 +2q - \frac{\theta -1}{\theta}\kappa') }
			(1+
				 \|v_0\|_{\HolBesSp{\frac32 \theta -2 - \theta q + (\theta -1)\kappa'				 }}
				+
				\|(v,w)\|_{\sols{T}{\kappa}{\kappa'}}
			).
\end{align*}
Since
$\min\{\kappa'-\kappa, \frac12 \theta -1 -\kappa \}>
\frac52(\theta-2)-\kappa-\kappa'$,
we deduce the desired estimate:
\[
\|(T_{10})_t\|_{ \HolBesSp{\frac52(\theta-2)-\kappa-\kappa'}} \lesssim
t^{- ( \frac{5}{\theta} -3 +2q - \frac{\theta -1}{\theta}\kappa') }
			(1+
				 \|v_0\|_{\HolBesSp{\frac32 \theta -2 - \theta q + (\theta -1)\kappa'				 }}
				+
				\|(v,w)\|_{\sols{T}{\kappa}{\kappa'}}
			).
\
\]
Thus, we have estimated all the terms
$T_1,T_2,\ldots,T_{10}$ in the definition of $G(v,w)$.
\end{proof}

%%%%%%%%%%%%%%%%%%%%%%%%%%%%%%
%\vspace{10mm}

Now we are in the position to estimate ${\mathcal M}_2$.
\begin{proposition}\label{prop_20161003143917}
	The map
	$
		\mathcal{M}^2:
			\sols{T}{\kappa}{\kappa'}
			\to
			\cL_T^{q' -\kappa'+\kappa, \frac72 \theta -5 - \theta\kappa',1-\kappa'}
	$
	is well defined.
Moreover, there exist positive constants $\const[1], \const[2]$
 and 
	 $a$ such that
	 the following estimate holds:
	For every 
	$
		(v,w)\in\sols{T}{\kappa}{\kappa'}
	$,
	\begin{multline*}
		\|\mathcal{M}^2(v,w)\|_{\cL_T^{
		q' -\kappa'+\kappa, \frac72 \theta -5 - \theta\kappa',1-\kappa'
		}}
		\leq
			\const[1]
			(
				\|w_0\|_{\HolBesSp{\frac72 \theta -5 - \theta q' -\theta\kappa}}
			)\\
			\quad+
			\const[2]
			T^{a}
			\left(1+
				\|v_0\|_{\HolBesSp{\frac32 \theta -2 - \theta q + (\theta -1)\kappa'}}
				+
				\|(v,w)\|_{\sols{T}{\kappa}{\kappa'}}^3
				+\|(v,w)\|_{\sols{T}{\kappa}{\kappa'}}
			\right).
	\end{multline*}
Here, $\const[1]$ and the constant $a$ depend only on $\kappa, \kappa'$
	 and $\const[2]$ depends only on $\kappa, \kappa'$ and ${\bf X}$. 
	 More precisely, $\const[2]$ is given by an at most 
	 third-order polynomial in $\|{\mathbf X}\|_{\drivers{1}{\kappa}}$ for fixed $\kappa, \kappa'$.	
\end{proposition}

\begin{proof}
First, using
Proposition \ref{prop_20170804}(1) with 
\[
\alpha =\frac72 \theta -5 - \theta q' -\theta\kappa,
\quad
\beta = \frac72 \theta -5 -\theta\kappa',
\quad \delta =1 -\kappa',
\] 
we obtain
\[
\left\| P^{\theta /2}_{\cdot} w_0 \right\|_{\cL_T^{
q' -\kappa'+\kappa, \frac72 \theta -5 - \theta\kappa',1-\kappa'
}}
		\lesssim
		 \|w_0\|_{\HolBesSp{\frac72 \theta -5 - \theta q' -\theta\kappa }}.
		 		\]

Next, we use 
Proposition \ref{prop_20170804}(2), 
the Schauder estimate with
\begin{align*}
\alpha &= \frac52(\theta -2)-(\kappa' +\kappa), 
\quad
\beta = \frac72 \theta -5 - \theta\kappa', \quad
\gamma = \beta - \theta (q' -\kappa' +\kappa), 
\\
\delta &=1-\kappa', 
\quad
\eta = \frac{5}{\theta} -3 +2q - \frac{\theta -1}{\theta}\kappa'.
\end{align*}
It is straightforward to check that 
$\alpha \le \gamma < \alpha -\theta\eta +\theta$, 
$\gamma \le \beta < \alpha +\theta$ and $0 < \delta \le (\beta -\alpha)/\theta$
due to the assumption $\kappa /\kappa' < 3/4$.
In particular, 
\begin{eqnarray*}
\alpha -\theta\eta +\theta -\gamma =
\begin{cases}
(\theta -2) \kappa' +(\theta -1) \kappa 
& \mbox{if $\frac{11}{6} < \theta \le 2$}, \\
\theta \rho +
(\theta -2) \kappa' +(\theta -1) \kappa 
 & \mbox{if $\frac74 < \theta \le \frac{11}{6}$}
\end{cases}
\end{eqnarray*}
and hence
\[
\alpha -\theta\eta +\theta -\gamma
> \frac{ -\kappa' + 3\kappa}{4}.
\]
The right-hand side is positive 
since $\kappa /\kappa' >1/3$ is assumed.

Setting
$a :=\theta^{-1} (\alpha -\theta\eta +\theta -\gamma) >0$, we see 
from Proposition \ref{prop_20170804}(2),
the fractional
version of the Schauder estimate
that
\[
\|I[G(v,w)]\|_{\cL_T^{
q' -\kappa'+\kappa, \frac72 \theta -5 - \theta\kappa',1-\kappa'
}}
					\lesssim
					T^{a}
					\| G(v,w) \|_{
					\cE_T^{\frac{5}{\theta} -3 +2q - \frac{\theta -1}{\theta}\kappa'	
					}
					 \HolBesSp{ \frac52(\theta -2)-(\kappa' +\kappa)}}.
\]
By Lemma \ref{lem_20171219}, the estimate of $G(v,w)$
we prove the proposition.
\end{proof}

%%%%%%%%%%%%%%%%%
%%%%%%%%%%%%%%%%%%%%%%%%%%%%%%%%%%%%%%%%%%%%%%%%%%%%%%%%%%%%%%%%%%
%% Section
%\newpage
%%%%%%%%%%%%%%%%%%%%%%%%%%%%%%%%%%%%%%%%%%%%%%%%%%%%%%%%%%%%%%%%%%%
\section{Local well-posedness of paracontrolled QGE}

At the beginning of this section we prove the local Lipschitz continuity of 
the integration map $\mathcal{M}$.
Next, we prove the local well-posedness of our paracontrolled QGE, 
that is, 
the equation \eqref{eq_20160920090410}
admits a unique solution locally in time 
for every driver and initial condition 
and the solution depends continuously on these input data.
(These parts are a kind of routine
once the estimate of $\mathcal{M}$ as in Proposition \ref{pr:180213-1}
is established.
So, some of our proofs may be sketchy in this section.)

In the following proposition the local Lipschitz continuity of 
the integration map is given. 
The positive constant $a =a(\kappa, \kappa') >0$ the same as in Proposition \ref{pr:180213-1}.
\begin{proposition}\label{pr:180214-1}
Let $T \in (0,1]$
and $0<\kappa <\kappa' \ll 1$ be as in \eqref{eq:180213-3}.
For $i=1,2$, let $
		(v_0^{(i)},w_0^{(i)})
		\in
			\HolBesSp{ \frac32 \theta -2 -\theta q
								 +(\theta -1)\kappa' }
			\times
			\HolBesSp{\frac72 \theta -5 -\theta q' -\theta\kappa}
	$
	and
	${\mathbf X}^{(i)} \in \drivers{T}{\kappa}$. 
Then, there exists a positive constant $\const[3]$ depending only on $\kappa$ and $\kappa'$ such that
	 the following estimate holds:
	For 
	$(v^{(i)}, w^{(i)})\in\sols{T}{\kappa}{\kappa'}$, $i=1,2$, it holds that
	\begin{multline*}
		\|\mathcal{M}_{{\mathbf X}^{(1)}, (v_0^{(1)},w_0^{(1)})} (v^{(1)},w^{(1)})
		-
		 \mathcal{M}_{{\mathbf X}^{(2)}, (v_0^{(2)},w_0^{(2)})} (v^{(2)},w^{(2)})
		 		\|_{
		\sols{T}{\kappa}{\kappa'}
				}
				\\
		\leq
			\const[3] (1+M^2)
			(\|v_0^{(1)} - v_0^{(2)}\|_{\HolBesSp{\frac32 \theta -2 - \theta q + (\theta -1)\kappa'}}
			+
				\|w_0^{(1)} - w_0^{(2)}\|_{\HolBesSp{\frac72 \theta -5 - \theta q' -\theta\kappa}}
			)\\
			\quad+
			\const[3](1+M^2)
						T^{a}
			\left(
				\|(v^{(1)},w^{(1)}) -(v^{(2)},w^{(2)})\|_{\sols{T}{\kappa}{\kappa'}}
				+\| {\mathbf X}^{(1)}- {\mathbf X}^{(2)} \|_{ \drivers{T}{\kappa}}
			\right).
	\end{multline*}
	 Here, we set $\displaystyle M = \max_{i=1,2} \max \{ \|{\mathbf X}^{(i)}\|_{\drivers{T}{\kappa}}, 
	 \|(v^{(i)},w^{(i)})\|_{\sols{T}{\kappa}{\kappa'}}	\}$ for simplicity.
	 	 	 	\end{proposition}

\begin{proof}
The proof is essentially the same as the one for Proposition \ref{pr:180213-1}
since all the terms in the definition of ${\mathcal M}$ are ``polynomials" in 
$(v,w)$ and the components of ${\mathbf X}$.
 \end{proof}

By Propositions \ref{pr:180213-1} and \ref{pr:180214-1} above, 
we can easily show that QGE \eqref{eq_20160920090410} 
has a unique time-local solution for any driver ${\bf X}$ and initial condition $(v_0, w_0)$.
Roughly, the argument goes as follows:
If we take $M >0$ large enough compared to the norms of 
${\bf X}$ and $(v_0, w_0)$, then 
Proposition \ref{pr:180213-1} implies that $\mathcal{M}= \mathcal{M}_{{\bf X}, (v_0, w_0) }$ 
maps $B_{T_*, M}$ into itself for sufficiently small $T_* >0$.
Here, we set
$B_{T, M} = \{ (v, w)\in\sols{T}{\kappa}{\kappa'} \mid 
 \| (v,w)\|_{\sols{T}{\kappa}{\kappa'} } \le M\}$.
Moreover, Proposition \ref{pr:180214-1} implies that
$\mathcal{M}\vert_{B_{T_*, M}} \colon B_{T_*, M} \to B_{T_*, M}$ is a contraction 
if $T_* >0$ is chosen smaller.
Therefore, this map has a unique fixed point. 
Thus, our QGE \eqref{eq_20160920090410} has a unique solution
locally in time for every ${\bf X}$ and $(v_0, w_0)$.

More precisely, we have the following local well-posedness,
which is our main result for the deterministic paracontrolled QGE.
Since we can prove it by a routine argument, we omit its proof. 
(For example, see the proof of \cite[Theorem 4.27]{hin}.)

\begin{proposition}\label{pr:180215-1}
Let $0<\kappa <\kappa' \ll 1$ be as in \eqref{eq:180213-3}.
 Then, there exists a continuous mapping
 \[
 \tilde{T}_\ast:
 \HolBesSp{ \frac32 \theta -2 -\theta q
								 +(\theta -1)\kappa' }
			\times
			\HolBesSp{\frac72 \theta -5 -\theta q' -\theta\kappa}
 \times
 \drivers{1}{\kappa}
 \to
 (0,1]
 \]
 such that the following $(1)$ and $(2)$ hold{\rm:}
 \begin{enumerate}[$(1)$]
		\item	For every
				$
					(v_0,w_0)
					\in
						\HolBesSp{ \frac32 \theta -2 -\theta q
								 +(\theta -1)\kappa' }
			\times
			\HolBesSp{\frac72 \theta -5 -\theta q' -\theta\kappa}
				$
				and
				$
					{\bf X}\in\drivers{1}{\kappa}
				$,
 set $T_\ast=\tilde{T}_\ast(v_0,w_0,{\bf X})$.
 Then, the system \eqref{eq_20160920090410} admits a unique solution $(v,w)\in\sols{T_\ast}{\kappa}{\kappa'}$
				and there is a positive constant $\const$ depending only on $\kappa$, $\kappa'$, $T_\ast$ and
				$\|{\bf X}\|_{\drivers{1}{\kappa}}$
				such that
				\begin{align*}
					\|(v,w)\|_{\sols{T_\ast}{\kappa}{\kappa'}}
					\leq
						\const
						\left(
							1
							+\|v_0\|_{\HolBesSp{-\frac{2}{3}+\kappa'}}
							+\|w_0\|_{\HolBesSp{-\frac{1}{2}-2\kappa}}
						\right).
				\end{align*}
		\item[$(2)$]
		Let $\{(v_0^{(n)},w_0^{(n)})\}_{n=1}^\infty$ and $\{{\bf X}^{(n)}\}_{n=1}^\infty$
 converge to $(v_0,w_0)$ in
				$
				\HolBesSp{ \frac32 \theta -2 -\theta q
								 +(\theta -1)\kappa' }
			\times
			\HolBesSp{\frac72 \theta -5 -\theta q' -\theta\kappa}				
				$
				and
 ${\bf X}$ in $\drivers{1}{\kappa}$,
 respectively.
 Set $$T_\ast^{(n)}=\tilde{T}_\ast(v_0^{(n)},w_0^{(n)}, {\bf X}^{(n)})$$
 and let $(v^{(n)},w^{(n)})$ be a unique solution on $[0,T^{(n)}_\ast]$
 to the system \eqref{eq_20160920090410} with the initial condition $(v_0^{(n)},w_0^{(n)})$ driven by ${\bf X}^{(n)}$.
 Then, for every $0<t<T_\ast$, we have
 \begin{align*}
 \lim_{n\to\infty}
 \|(v^{(n)},w^{(n)})-(v,w)\|_{\sols{t}{\kappa}{\kappa'}}
 =
 0.
 \end{align*}
	\end{enumerate}
\end{proposition}

As usual, the local solution admits a prolongation up to 
the explosion time.
Namely, for every $(v_0, w_0)$ as in Proposition \ref{pr:180215-1} 
and ${\bf X} \in\drivers{T}{\kappa}$, $T>0$,
there exists $T_{\rm{exp}} \in (0,T]$ such that
QGE \eqref{eq_20160920090410} 
admits a unique solution $(v,w)\in\sols{t }{\kappa}{\kappa'}$
for every $t < T_{\rm{exp}}$
and 
\[
\limsup_{t \nearrow T_{\rm{exp}}} 
(
\| v_t \|_{\HolBesSp{\frac32 \theta -2 -\theta q
								 +(\theta -1)\kappa' }}
+
\| w_t \|_{\HolBesSp{\frac72 \theta -5 -\theta q' -\theta\kappa}} 
)
=\infty
\]
if $T_{\rm{exp}} <T$.
Moreover, 
the function $(v_0,w_0, {\bf X}) \mapsto T_{\rm{exp}}$ 
can be arranged lower semicontinuous.

%%%%%%%%%%%%%%%%%%%%%%%%%%%%%%%%%%%%%%%%%%%%%%%%%%%%%%%%%%%%%%%%%%
%% References 
%\newpage
%%%%%%%%%%%%%%%%%%%%%%%%%%%%%%%%%%%%%%%%%%%%%%%%%%%%%%%%%%%%%%%%%%%
\section{Relation between paracontrolled solution and classical mild solution}

In this section 
we show that the solution of our paracontrolled QGE 
coincides in a suitable sense with 
the classical mild solution of QGE
when the driver ${\mathbf X}$ is a natural enhancement of 
nice (and deterministic) $\xi$.

Assume that 
$\xi \in C_T \HolBesSp{\alpha}$, $X_0 \in \HolBesSp{\alpha}$ and
$Y_0, V_0 \in \HolBesSp{\alpha-1}$
for some $\alpha >2$.
Set
$$X_t = P_{\cdot}^{\theta /2} X_0 + I[\xi]$$
(or equivalently,
$\partial_t X_t = - ( (-\LaplaceOp)^{\theta /2} +1) X_t +\xi$ with the initial value $X_0$).
We define its natural enhancement ${\mathbf X}= (X, V, Y, Z, W, \hat{Z}, \hat{W})$
as in Example \ref{ex.20171112}.

\begin{proposition}\label{pr.20171116}
Let the notation be as above and assume $(v_0,w_0)
		\in
			\HolBesSp{(\theta -1)(-\frac12 +\kappa') }
			\times
			\HolBesSp{ -\theta\kappa}$.
We denote by $(v,w)$
a unique solution of the paracontrolled QGE \eqref{eq_20160920090410}.			
Then, $u: =X+Y+v+w$ solves 
\begin{align*}
\partial_t u_t 
&= - (-\LaplaceOp)^{\theta /2} u_t + R^{\perp} u_t \cdot \nabla u_t +\xi_t
\\
&= - ( (-\LaplaceOp)^{\theta /2} +1) u_t + R^{\perp} u_t \cdot \nabla u_t + u_t +\xi_t,
\quad
\mbox{with $u_0 = X_0+Y_0+v_0+w_0$}
\end{align*}
in the classical mild sense $($locally in time$)$,
or eqivalently,
\[
u=P^{\theta/2}_{\cdot}u_0+I[ R^{\perp} u \cdot \nabla u+ u+\xi].
\]
\end{proposition}

\begin{proof}
We use the quantities
$T_1,T_2,\ldots,T_{10}$
defined in (\ref{eq:180227-3}).
It is sufficient to show that 
\begin{align}
v+ w &=
P_{\cdot}^{\theta /2} (v_0 +w_0) + 
I [\Phi \cdot \nabla X] 
+ I[R^{\perp} X \cdot \nabla (Y+v+w)]
\label{eq_20171116_1}
\\
\nonumber
&\quad
 + I [\Phi \cdot \nabla (Y+v+w) ] + I[X +Y +v+w]\\
\nonumber
&=
P_{\cdot}^{\theta /2} (v_0 +w_0) + 
I [\Phi \cdot \nabla X] 
+ I[R^{\perp} X \cdot \nabla (Y+v+w)]
+ I [T_1] + I[X +Y +v+w] 
\nonumber
\end{align}
Indeed, adding $X +Y$ to \eqref{eq_20171116_1}
and setting $u=X+Y+v+w$, we obtain
\begin{align*}
u &=
P_{\cdot}^{\theta /2} (X_0 +Y_0+v_0 +w_0) + I[\xi]
+I [R^{\perp} X \cdot \nabla X] 
\\
&\qquad
+ 
I [R^{\perp} (u-X) \cdot \nabla X] 
+ I[R^{\perp} X \cdot \nabla (u-X)]
 + I [R^{\perp} (u-X) \cdot \nabla (u -X) ] + I[u] 
\\
&= 
P_{\cdot}^{\theta /2} (X_0 +Y_0+v_0 +w_0) 
+ I[ R^{\perp} u \cdot \nabla u+ u+\xi]
\\
&= 
P_{\cdot}^{\theta /2} (u_0) 
+ I[ R^{\perp} u \cdot \nabla u+ u+\xi],
\end{align*}
which proves the assertion of the proposition.

Now we show \eqref{eq_20171116_1}.
Before doing so,
one should note that since $v$ satisfies the first equation 
of \eqref{eq_20160920090410}, we have 
$v = \com (v,w) + \Phi \rpara V$
thanks to
(\ref{eq:180227-1})
and
(\ref{eq:180227-2}).

First, we calculate $\Phi \cdot \nabla X$:
\begin{align*}
\Phi \cdot \nabla X
&=
\Phi ( \rpara +\lpara +\reso) \nabla X
\\
&=
\Phi ( \rpara +\lpara) \nabla X
+ 
Z +R^{\perp} w \reso \nabla X
+ R^{\perp} \{ \com (v,w) + \Phi \rpara V \} \reso \nabla X
\\
&=
\Phi ( \rpara +\lpara) \nabla X
+ 
Z +R^{\perp} w \reso \nabla X
+ R^{\perp} \com (v,w) \reso \nabla X
\\
&\quad
+ \{ R^{\perp}(\Phi \rpara V) 
- \Phi \rpara R^{\perp} V \}\reso \nabla X
+
\{\Phi \rpara R^{\perp} V \}\reso \nabla X
\\
&=
\Phi ( \rpara +\lpara) \nabla X
+ 
Z +R^{\perp} w \reso \nabla X
+ R^{\perp} \com (v,w) \reso \nabla X
\\
&\quad
+ \{ R^{\perp}(\Phi \rpara V) 
- \Phi \rpara R^{\perp} V \}\reso \nabla X
+
\Phi \cdot W + \comC (\Phi, R^{\perp} V, \nabla X)\\
&=
T_3
+ 
Z 
+
T_4
+ R^{\perp} \com (v,w) \reso \nabla X
\\
&\quad
+
T_6
+
\Phi \cdot W + \comC (\Phi, R^{\perp} V, \nabla X).
\end{align*}

Next, we calculate
$R^{\perp} X \cdot \nabla (Y+v+w)$:
\begin{align*}
R^{\perp} X \cdot \nabla (Y+v+w)
&=
\hat{Z} + R^{\perp} X \cdot \nabla \{ \com (v,w) + \Phi \rpara V \}
+ R^{\perp} X \cdot \nabla w 
\\
&=
\hat{Z} + R^{\perp} X \cdot \nabla w + R^{\perp} X \cdot \nabla \com (v,w) 
+
R^{\perp} X \cdot \{ \nabla \Phi \rpara V\}
\\
& \quad
+
R^{\perp} X ( \rpara +\lpara +\reso) \{ 
 \Phi \rpara \nabla V\}
 \\
&=
\hat{Z} + R^{\perp} X \cdot \nabla w + R^{\perp} X \cdot \nabla \com (v,w) 
+
R^{\perp} X \cdot \{ \nabla \Phi \rpara V\}
\\
& \quad
+
R^{\perp} X ( \rpara +\lpara) \{ 
 \Phi \rpara \nabla V\}
+
\Phi \cdot \hat{W} 
		+
				\comC (\Phi, \nabla V, R^{\perp} X)
 \\
&=
\hat{Z} + T_5 + R^{\perp} X \cdot \nabla \com (v,w) 
+T_7
+T_8
+
\Phi \cdot \hat{W} 
		+
				\comC (\Phi, \nabla V, R^{\perp} X).
 \end{align*}
Note that we used Lemma \ref{lem:161117-106}
for the second equality.

Combining these all, we can easily check that
the right-hand side of \eqref{eq_20171116_1}
is equal to
\[
P_{\cdot}^{\theta /2} (v_0 +w_0) + I [F(v,w)+G(v,w)],
\]
where $F$ and $G$ were defined in \eqref{eq_20161002045615}.
This completes the proof.
\end{proof}

%%%%%%%%%%%%%%%%%

%%%%%%%%%%%%%%%%%
%%%%%%%%%%%%%%%%%%%%%%%%%%%%%%%%%%%%%%%%%%%%%%%%%%%%%%%%%%%%%%%%%%\newpage
%%%%%%%%%%%%%%%%%%%%%%%%%%%%%%%%%%%%%%%%%%%%%%%%%%%%%%%%%%%%%%%%%%%

%%%%%%%%%%%%%%%%%%
%%%%%%%%%%%%%%%%%%
%%%%%%%%%%%%%%%%%%\newpage
\section{Probabilistic Part: Enhancement of white noise}
\label{s8}

The main purpose of this section is to prove 
the enhancement of white noise.
In this section we do not always assume $\theta \in (7/4, 2]$
and the condition on $\theta$ will be specified in each theorem or
each proposition.
The time interval is basically $[0,T]$ for arbitrary $T>0$.
If $u\colon (-\infty, T] \to {\mathcal C}^{\alpha}$ for some $\alpha \in {\mathbf R}$, 
then as before we write 
${\mathcal I}[u]_t = \int_{-\infty}^t P_{t-s}^{\theta /2} u_s ds$, $t \in (-\infty, T]$,
whenever this makes sense.

Denote by $\xi$ the space-time white noise
associated with $L^2 ( {\mathbf R} \times {\mathbf T}^2)$.
Let $\chi \colon {\mathbf R}^2 \to {\mathbf R}$
be a smooth and even function with compact support such that $\chi (0)=1$.
See Section \ref{s7.1} below.
We write 
\begin{equation}\label{eq:180323-1}
\chi^{\ve} (x):= \chi (\ve x)
\end{equation} 
for $\ve >0$ and $x \in {\mathbf R}^2$.
Set 
\begin{equation}\label{eq:180522-1}
\xi^{\ve} (t,x) := \sum\limits_{k \in {\mathbf Z}^2} \chi^{\ve} (k) \hat\xi_t (k){\bf e}_k(x),
\end{equation}
where
the random variable
$\hat\xi (k)$ stands for the $k$th Fourier coefficient 
of $\xi$ (see Remark \ref{re:180524} below).
Heuristically, $\chi^{\ve}$ kills high frequencies. 
This is called a smooth approximation of $\xi$.
We study the stationary natural enhancement
 $
 {\mathbf X}^{\ve} 
 = (X^{\ve}, V^{\ve}, Y^{\ve}, Z^{\ve}, W^{\ve}, \hat{Z}^{\ve}, \hat{W}^{\ve})
 $
as in Example \ref{ex.20171112}.

\begin{theorem}\label{pr:180220}
Let $\theta \in (8/5, 2]$ and $\kappa >0$.
Then, there exists a random driver ${\mathbf X}$ such that
\[
{\mathbf E}[\| {\bf X} \|_{ {\mathcal X}_T^{\kappa} }^p] < \infty, 
\qquad
\lim_{\ve \searrow 0} 
{\mathbf E}[\| {\mathbf X} - {\mathbf X}^{\ve} \|_{ {\mathcal X}_T^{\kappa} }^p] =0
\]
for every $1 <p <\infty$.
Moreover, ${\mathbf X}$ does not depend on 
the cut-off function $\chi$.
\end{theorem}

Due to a few cancellations in the enhancement 
procedure, 
we do not have to do any renormalization. 
The proof of Theorem
\ref{pr:180220}
 is decomposed into Lemmas
 \ref{lm.0406_a}, \ref{lm.0406_b}, 
 \ref{lm.0406_c}, \ref{lm.0407_b},
 \ref{lm.0602_a}, \ref{lm.0612_a}, \ref{lm.0613_a}
 and
 \ref{lm.0613_ab}, below.
The rest of this section is devoted to showing
Theorem \ref{pr:180220}.

%%%%%%%%%%%%%%%%%%
%%%%%%%%%%%%%%%%%%
%%%%%%%%%%%%%%%%%%
%\newpage
%%%%%%%%%%%%%%%%%%
%%%%%%%%%%%%%%%%%%
%%%%%%%%%%%%%%%%%%\vspace{10mm}

\subsection{Notation and preliminaries}\label{s7.1}
In this subsection
we let $\theta \in (0,2]$ and
we suppose the space dimension $d$ is general.
One should be careful of the difference between subindices and superindices
of $k$ and $x$.

\begin{itemize}
\item
Let 
${\bf T}^d = ({\bf R}/{\bf Z})^d$ denote the $d$-dimensional torus.
For $k=(k^1, \ldots, k^d) \in {\bf Z}^d$ and $x=(x^1, \ldots, x^d) \in {\bf T}^d$, we set
 $k\cdot x = k^1 x^1 +\cdots + k^d x^d$ (modulo ${\mathbf Z}$)
 and 
${\bf e}_{k}(x) := \exp (2\pi \sqrt{-1} k\cdot x)$.
We set 
\begin{equation}\label{eq:200410-1}
\hat{E}:= {\bf R} \times {\bf Z}^d.
\end{equation}
An element of $\hat{E}$ is often denoted by $\mu=(\sigma,k)$, 
where $\sigma \in {\bf R}$ and $k \in {\bf Z}^d$.
Their norm-like quantities are defined by 
$
|k|_* := 1+ |k|
$
 and 
$|\mu|_* =|(\sigma,k)|_* :=1+ |\sigma|^{1/2}+|k|^{\theta /2}$,
where 
$|k| = \sqrt{(k^1)^2 +\cdots +(k^d)^2}$ as usual.
As a measure on $\hat{E}$, the product of the Lebesgue measure 
and the counting measure is used, 
which will be denoted by $d\mu =d(\sigma,k)$.

\item
For the partition of unity $\{ \rho_j \}_{ j \ge -1 }$ used in the definition of paraproduct,
we write 
\begin{equation}\label{eq:180303-1}
\psi_{\circ} (x, x')= \sum\limits_{ |j -j'| \le 1} \rho_j (x) \rho_{j'} (x').
\end{equation}

%\item
%Let 
%\begin{equation}\label{eq:180303-1}
%\chi \colon {\bf R}^d \to {\bf R}
%\end{equation}
%be a smooth function with compact support such that $\chi (0)=1$.
%We define $\chi^{\ve}$
%by (\ref{eq:180323-1}) for $\ve >0$.
%Heuristically, $\chi^{\ve}$ kills high frequencies. 

%\item
%We say that a stochastic process 
%$W = \{W(h), h \in H\}$
%defined in a complete probability space 
%$(\Omega,{\mathcal F},P)$
%is an isonormal Gaussian
%process (or a Gaussian process on $H$) if W is a centered Gaussian family
%of random variables such that 
%$E(W(h)W(g)) = \langle h, g \rangle_H$ for all 
%$h, g \in H$.
%Here and below we suppose that
%we have such $W$.

\item
Let $\xi$ be a
space-time white noise on $E:= {\bf R} \times {\bf T}^d$
which is defined on a probability space $(\Omega, {\mathcal F}, {\mathbf P})$,
that is, $\xi:\Omega \to L^2 ( E, dtdx)$ is a centered Gaussian random variable 
whose Cameron--Martin space is $L^2 ( E, dtdx)$.
(Equivalently, $\xi$ is an an isonormal Gaussian
process associated with $L^2 ( E, dtdx)$.)
For every $f \in L^2 ( E, dtdx)$,
$\langle \xi, f\rangle$ is a real-valued centered Gaussian random variable 
whose covariance is given by 
${\mathbf E} [\langle \xi, f\rangle \langle \xi, g \rangle ] = \langle f, g \rangle_{L^2 (E)}$.

\item
The orthogonal projection onto the $n$th homogeneous 
Wiener chaos ${\mathcal Q}_n$
is denoted by $\Pi_n\colon L^2 ({\mathbf P}) \to {\mathcal Q}_n$
($n \ge 0$).

For example,
\begin{equation}\label{eq:180419-1}
\Pi_0[F]={\mathbf E}[F].
\end{equation}
%More precisely,
%we let
%\[
%H_n(x)=\frac{(-1)^n}{n!}\exp\left(\frac{x^2}{2}\right)
%\frac{d^n}{d x^n}\exp\left(-\frac{x^2}{2}\right).
%\]
%We let
%\[
%{\mathcal Q}_n=
%\overline{\rm Span}\{H_n(W(h))\,:\,h \in H, \, \|h\|_H=1\}.
%\]
%The orthogonal sum $\oplus_{j=0}^n {\mathcal Q}_j$ is called 
%$n$th inhomogeneous Wiener chaos.
%
%\item
For $f=f ((s_1, x_1), \ldots, (s_n, x_n)) \in 
L^2 ( E^{n}, (dtdx)^{\otimes n})$,
${\mathcal J}_n (f)$ stands for the $n$th 
multiple Wiener-It\^o integral with respect to $\xi$.
In particular,
if $n=1$, then ${\mathcal J}_1 (f)=\langle \xi, f\rangle$.
As is well known, 
${\mathcal J}_n (f)\in {\mathcal Q}_n$
and satisfies 
\begin{equation}\label{eq:180417-1}
{\mathbf E}[ |{\mathcal J}_n (f) |^2] 
\le \| f\|_{L^2 ( E^{n})}^2.
\end{equation}
In an obvious way, ${\mathcal J}_n$ extends to 
complex-valued functions and (\ref{eq:180417-1}) still holds.
(For white noise and multiple Wiener-It\^o integrals, see 1.1 \cite{Nua} among others.)
\end{itemize}

\begin{remark}\label{re:180524}
Here, we give a simple remark on $\hat{\xi} (k), k \in {\bf Z}^d$.
Its definition is given by 
$\langle \hat{\xi} (k), h \rangle = \langle \xi, h \otimes {\bf e}_{-k}\rangle$
for (real-valued) $h \in L^2 ({\mathbf R})$. 
Clearly, $\hat{\xi} (-k) = \overline{\hat{\xi} (k)}$, ${\mathbf P}$-a.s.
Both the real and the imaginary part of $\hat{\xi} (k)$ are 
a constant multiple of the time white noise.
Indeed, 
$\hat{\xi} (k) =(\eta_1 +\sqrt{-1}\eta_2) /\sqrt{2}$ in law for every $k \neq 0$
and $\hat{\xi} (0) =\eta_1$ in law,
where $\eta_1, \eta_2$ are two independent copies of the time white noise 
associated with $L^2 ({\mathbf R})$.
Their covariance is given by 
${\mathbf E} [\langle \hat{\xi} (k) , h\rangle \langle \hat{\xi} (l) , \tilde{h} \rangle ] 
= \delta_{k+l, 0} \langle h, \tilde{h} \rangle_{L^2 ({\bf R})}$
for every (real-valued) $h, \tilde{h} \in L^2 ({\bf R})$.

%Alternatively, 
Then, 
$\xi^{\ve} = \sum\limits_{k \in {\mathbf Z}^2} \chi^{\ve} (k) \hat\xi (k) {\bf e}_k$ 
as in \eqref{eq:180522-1} 
satisfies
$\langle \xi^{\ve} , f \rangle
= \langle \xi, ({\mathcal F}^{-1}_{{\bf R}^d} \chi^{\ve})(- \bullet) * f(t, \bullet) \rangle$
for $f \in L^2 (E)$,
where 
${\mathcal F}_{{\bf R}^d}$ stands for the Fourier transform on ${\bf R}^d$
and $*$ for the convolution with respect to the space variable.
(This could be used as an alternative definition of $\xi^{\ve}$,
but the series representation as in \eqref{eq:180522-1} will be used throughout this paper).

%
%
%\vspace{20mm}
%
%Let $\{\tilde\eta_i \}_{i \in {\bf N}}$ be an i.i.d. sequence of the time white noise 
%associated with $L^2 ({\mathbf R})$
%and $\{ g_i \}_{i \in {\bf N}}$ be an orthonormal basis of $L^2({\bf T}^d)$.
%Then, as an identity in law, we have $\xi = \sum\limits_{i \in {\bf N}} \tilde\eta_i f_i$.
%(If one wants to start with $\xi$, 
%just set 
%$\tilde\eta_i = \sum\limits_{j \in {\bf N}} \langle \xi, h_j \otimes g_i \rangle h_j$,
%where $\{ h_i \}_{i \in {\bf N}}$ is an orthonormal basis of $L^2({\bf T}^d)$.)
%From this we see easily that $\xi$ can be written as 
%$\xi =\sum\limits_{k \in {\bf Z}^d} \eta_k {\bf e}_k$ in law.
%Here,
%$\{ \eta_k \}_{k \in {\bf Z}^d }$ is an identically distributed 
%sequence such that $\eta_k= \bar{\eta}_{-k}$, ${\mathbf P}$-a.s. and
%$\eta_k =(\tilde\eta_1 +\sqrt{-1}\tilde\eta_2) /\sqrt{2}$ in law for every $k$.
%Then, we have 
%${\mathbf E} [\langle \tilde{\eta}_k , h\rangle \langle \tilde{\eta}_l, \tilde{h} \rangle ] 
%= \delta_{k+l, 0} \langle h, \tilde{h} \rangle_{L^2 ({\bf R})}$
%for every real-valued $h, \tilde{h} \in L^2 ({\bf R})$.
%When $\xi$ is written in this way, we understand $\eta_k=\hat{\xi} (k)$.
%(Again, we can start with $\xi$ and write down $\eta_k$ in terms of $\xi$.)
%

Concerning the covariance, the following formal contraction rule holds:
\begin{equation}\label{eq:180510-2}
{\mathbf E} [ \hat{\xi}_{s_1} (k) \hat{\xi}_{s_2} (l) ] 
= \delta (s_1- s_2) \delta_{k+l, 0}.
\end{equation}
Though not really mathematically rigorous, 
\eqref{eq:180510-2} is used very often in mathematical physics literature. 
We will also use it in this paper
since \eqref{eq:180510-2} is quite useful and computations using this rule 
(at least those in the this paper)
can easily be made rigorous.
\end{remark}

Now we prove a couple of auxiliary estimates for later use.
\begin{lemma}\label{lm.0405_1}
Let $d \ge 1$ and $\theta \in (0,2]$.
Assume that $\alpha, \beta \in (0,2 +\frac{2d}{\theta})$ satisfy
$\alpha + \beta > 2 +\frac{2d}{\theta}$.
Then, 
\[
\int_{\hat{E}} |\mu|_*^{- \alpha} |\nu - \mu|_*^{- \beta} d\mu 
\lesssim
 |\nu|_*^{- \alpha- \beta + 2 +\frac{2d}{\theta}},
 \]
 where the implicit constant is independent of $\nu \in \hat{E}$. 
\end{lemma}

\begin{proof}
We mimic the proof of  \cite[Lemma 9.8]{gp}.
For $l \in {\mathbf N}$ with $|\nu|_* \approx 2^l$, 
we have
\[
\int_{\hat{E}} |\mu|_*^{- \alpha} |\nu - \mu|_*^{- \beta} d\mu 
\lesssim
\sum\limits_{i,j \ge 0} 2^{-\alpha i -\beta j}
\chi_{[l-l_0,\infty)}(\max(i,j))
\int_{\hat{E}}
{\bf 1}_{|\mu|_* \approx 2^i, |\nu -\mu|_* \approx 2^j}(\mu)
d\mu,
\]
where $l_0 \in {\mathbf N}$ is a fixed constant depending only on $\theta$.
We use
\[
\int_{\hat{E}}
{\bf 1}_{|\mu|_* \approx 2^i, |\nu -\mu|_* \approx 2^j}(\mu)
d\mu
\le
\int_{\hat{E}}
{\bf 1}_{ |\mu|_* \approx 2^i}(\mu)
d\mu
\lesssim
(2^i)^{2 +\frac{2d}{\theta}}
\]
and
\[
\int_{\hat{E}}
{\bf 1}_{|\mu|_* \approx 2^i, |\nu -\mu|_* \approx 2^j}(\mu)
d\mu
\le
\int_{\hat{E}}
{\bf 1}_{ |\nu -\mu|_* \approx 2^j }(\mu)
d\mu
=
\int_{\hat{E}}
{\bf 1}_{ |\mu|_* \approx 2^j }(\mu)
d\mu
\lesssim
(2^j)^{2 +\frac{2d}{\theta}}.
\]

Hence, since we are assuming $\alpha, \beta \in (0,2 +\frac{2d}{\theta})$ and
$\alpha + \beta > 2 +\frac{2d}{\theta}$,
then
\begin{align*}
\int_{\hat{E}} |\mu|_*^{- \alpha} |\nu - \mu|_*^{- \beta} d\mu 
&\le
\sum\limits_{i,j \ge 0}
\chi_{[l-l_0,\infty)}(\max(i,j))
\min(2^i,2^j)^{2+\frac{2d}{\theta}}\cdot 2^{-\alpha i-\beta j}
\\
&\lesssim
(2^l)^{- \alpha -\beta +2 +\frac{2d}{\theta} }\\
&\approx
|\nu|_*{}^{- \alpha -\beta +2 +\frac{2d}{\theta} }.
\end{align*}
Here, the summations are taken over all $(i,j)$ with the described conditions.
\end{proof}

%\vspace{7mm}

Let $\psi_{\circ}$ be a function given by
(\ref{eq:180303-1}).
\begin{lemma}\label{lm.0405_3}
{\rm (i)}~
Let $\lambda \ge 0$. Then, we have 
\[
\psi_{\circ} (k, l) \lesssim \frac{|l|_*^{\lambda}}{|k+l|_*^{ \lambda}}
 {\bf 1}_{ |k|_* \approx |l|_* } 
\qquad
(k, l \in {\bf R}^d),
\]
where the implicit constant is independent of $k, l$.
\\
{\rm (ii)}~Let $\lambda \in [0,1)$ and $k=(k^1, \ldots ,k^d), 
l =(l^1, \ldots, l^d) \in {\bf R}^d$.
Then, 
\[
\Bigl| 
\frac{\partial}{\partial k^i}\psi_{\circ} (k, l)
\Bigr|
 \lesssim 
 |l|_*^{-1 + \lambda}
 {\bf 1}_{ |k|_* \approx |l|_* } \qquad
(k, l \in {\bf R}^d),
\]
where the implicit constant is independent of $k, l$.
\end{lemma}

\begin{proof}
Modify results in \cite[Section 5]{hin} or \cite[Lemma 5.10]{hos}.
\end{proof}

\begin{proposition}\label{prop:180417-1}
Let $n\ge 0$ be an integer and 
let $\Xi$ be a ${\mathcal D}' ({\mathbf T}^d)$-valued random variable 
defined on a Gaussian probability space 
whose Fourier coefficient
$\hat\Xi (k) = \langle \Xi, {\bf e}_{-k}\rangle$ belongs to 
the $($complex-valued$)$ $n$th inhomogeneous Wiener chaos
$\oplus_{i=0}^n {\mathcal Q}_i$ 
for every $k \in {\bf Z}^d$.
Then, for every $s \in {\mathbf R}$ and $p \in (1, \infty)$, we have 
\begin{equation}\label{eq:180219-1}
{\mathbf E}[ \| \Xi \|_{ {\mathcal C}^{s} }^{p} ] 
\lesssim
\sum\limits_{j=-1}^{\infty}
2^{ j (s p + d)}
\Bigl( \sup\limits_{x \in {\mathbf T}^d}{\mathbf E}[ |\triangle_j \Xi (x) |^2 ]
\Bigr)^{p/2}.
\end{equation}
In particular, if the right-hand side of \eqref{eq:180219-1} is finite,
then $\Xi$ belongs to ${\mathcal C}^{s}$, a.s.
%
%
%Let $n=0,1,2,\ldots$ and 
% $\{\Xi (x)\}_{x \in {\mathbf T}^d }$ be a random field on ${\mathbf T}^d$
%defined on a Gaussian probability space such that
%$\Xi (x)$ belongs to 
%the $n$th inhomogeneous Wiener chaos for every $x$.
%Then, for every $s \in {\mathbf R}$ and $p \in (1, \infty)$, we have 
%%
%%
%%
%\begin{equation}\label{eq:180219-1}
%{\mathbf E}[ \| \Xi \|_{ {\mathcal C}^{s} }^{2p} ] 
%\lesssim
%\sum\limits_{j=-1}^{\infty}
%2^{ j (2s p +1)}
%\Bigl( \sup\limits_{x \in {\mathbf T}^d}{\mathbf E}[ |\triangle_j \Xi (x) |^2 ]
%\Bigr)^p.
%\end{equation}
%In particular, if the right-hand side of \eqref{eq:180219-1} is finite,
%the $\Xi$ belongs to ${\mathcal C}^{s}$, a.s.
%
\end{proposition}

\begin{proof}
This kind of results is already well known.
See \cite[Lemma 5.3]{hos} for example,
where Hoshino proved the almost the same assertion for $d=1$.
However,
the dimension $d$
of the torus does not play a major role.
Our statement and proof are slight modifications of those in \cite{hos}.

The proof of \eqref{eq:180219-1} uses the following two facts:
{\rm (i)}~The continuous embedding 
$B^{\alpha}_{p,p}({\bf T}^d) 
\hookrightarrow {\mathcal C}^{\alpha - d/p}({\bf T}^d)$ 
for 
$\alpha \in {\mathbf R}$ and $1 \le p \le \infty$. 
See \cite{Triebel-text-83} for the case of ${\mathbf R}^d$.
A similar argument works for $\Torus^d$.
{\rm (ii)}~Restricted to the $n$th inhomogeneous Wiener chaos, 
all the $L^p (\Omega, {\mathbf P})$-norms, $1 <p <\infty$, 
are equivalent (see Proposition 7.4, \cite{hu} for instance).
With these facts and the Fubini theorem, we have
\begin{eqnarray}\label{eq:180414-1}
{\mathbf E}[ \| \Xi \|_{ {\mathcal C}^{s} }^{p} ] 
&\lesssim&
{\mathbf E}[ \| \Xi \|_{ B^{s + d/p}_{p, p} }^{p} ] 
\\
&=&
\sum\limits_{j=-1}^{\infty} 
2^{ j (s p + d)}
\int_{{\bf T}^d} {\mathbf E}[ |\triangle_j \Xi (x)|^{p}] dx
\nonumber\\
&\le&
\sum\limits_{j=-1}^{\infty} 
2^{ j (s p + d)}
\int_{{\bf T}^d} {\mathbf E}[ |\triangle_j \Xi (x)|^{2}]^{p/2} dx
\nonumber\\
&\le&
\sum\limits_{j=-1}^{\infty}
2^{ j (s p +d)}
\Bigl( \sup\limits_{x \in {\mathbf T}^d}{\mathbf E}[ |\triangle_j \Xi (x) |^2 ]
\Bigr)^{p/2}.
\nonumber
\end{eqnarray}
Thus, we have obtained \eqref{eq:180219-1}.
\end{proof}

\begin{lemma}
For $\gamma >1$ and $\lambda>0$, 
\begin{equation}\label{eq:180329-4}
\int_{{\mathbf R}} \frac{ds}{\lambda^\gamma + |s|^\gamma} 
=C_\gamma \lambda^{-(\gamma-1)},
\qquad
\mbox{where} \quad
C_\gamma := \int_{{\mathbf R} } \frac{ds}{1 + |s|^\gamma} <\infty,
\end{equation}
\end{lemma}

\begin{proof}
Simply
perform the change of variables:
$s \mapsto t=\gamma^{-1}s$.
\end{proof}

%\vspace{10mm}

We will only treat the multiple Wiener-It\^o integrals
whose kernel behaves quite nicely.
In what follows,
all the random variables are defined on the probability space
$(\Omega,{\mathcal F},{\mathbf P})$ on which $\xi$ is defined
(unless otherwise specified).

We also recall
that $\hat{E}={\bf R} \times {\bf Z}^d$
as in (\ref{eq:200410-1}).

%\vspace{5mm}

\begin{definition}\label{def.gk}
A random field $A =\{ A_{t,x}\}_{(t,x) \in E}$ 
on 
$E={\bf R} \times {\bf T}^d$ 
is said to have a {\it good kernel} if 
each $A_{t,x}$ is written as
\begin{eqnarray}\label{eq:180223-5}
A_{t,x}
= {\mathcal J}_n (f_{t,x}).
\end{eqnarray}
for some $n \ge 1$ 
and some $f_{t,x} \in L^2 ( E^{n}, (dtdx)^{\otimes n}))$
satisfying the following conditions \eqref{eq:180223-1}--\eqref{eq:180223-4}:
\begin{itemize}
\item
Let $Q_0 \in L^2 (\hat{E}^n, {\mathbf C})$ be such that
\begin{equation}\label{eq:180223-1}
Q_0 ((-\sigma_1, -k_1), \ldots, (-\sigma_n, -k_n))
=
\overline{Q_0 ((\sigma_1,k_1), \ldots, (\sigma_n,k_n))},
\qquad
 \{(\sigma_i,k_i)\}_{i=1}^n \in \hat{E}^n.
\end{equation}
We define $Q_t \in L^2 (\hat{E}^n, {\mathbf C})$ by 
\begin{equation}\label{eq:180223-2}
Q_t ((\sigma_1,k_1), \ldots, (\sigma_n,k_n))
=
e^{-2\pi \sqrt{-1} ( \sigma_1 +\cdots + \sigma_n)t} 
Q_0 ((\sigma_1,k_1), \ldots, (\sigma_n,k_n)).
\end{equation}
The function
$Q_t$ is called the $Q$-function of $A$.
\item
Define $H_t \in L^2 (\hat{E}^n, {\mathbf C})$ so that
\begin{equation}\label{eq:180223-3}
Q_t((\sigma_1,k_1), \ldots, (\sigma_n,k_n)) = {\mathcal F}_{{\rm time}} H_t((\sigma_1,k_1), \ldots, (\sigma_n,k_n)) 
\end{equation}
for every $(k_1, \ldots, k_n) \in ({\mathbf Z}^d)^n$.
Here, ${\mathcal F}_{{\rm time}} H_t$ stands for the Fourier transform 
that acts on the function
\[(s_1, \ldots, s_n)
\in {\bf R}^n \mapsto H_t ((s_1,k_1), \ldots, (s_n,k_n)) \in {\mathbf C}
\]
for a fixed $(k_1, \ldots, k_n)$.
It is called the time-Fourier transform. 
The function
$H_t$ is called the $H$-function of $A$.
\item
Define a real-valued function by
$f_{t,x} \in L^2 ( E^{n}, (dtdx)^{\otimes n})$ by 
\begin{equation}\label{eq:180223-4}
f_{t,x} ((s_1,y_1), \ldots, (s_n,y_n))
=
\sum\limits_{k_1, \ldots, k_n \in {\bf Z}^d}
\prod_{i=1}^n {\bf e}_{k_i} (x -y_i)
H_t ((s_1,k_1), \ldots, (s_n,k_n))
\end{equation}
for every $\{(s_i,y_i)\}_{i=1}^n \in E^n$.
\end{itemize}
\end{definition}

\begin{remark}\label{re:180223-1}
We make some comments on Definition \ref{def.gk} above.
\begin{enumerate}
\item
Condition \eqref{eq:180223-1} is equivalent to 
\[
H_t ((s_1,-k_1), \ldots, (s_n,-k_n)) = 
\overline{ H_t ((s_1,k_1), \ldots, (s_n,k_n))}
\]
for every $t \in {\bf R}$
and
$\{(s_i,k_i)\}_{i=1}^n \in \hat{E}^n$,
which in turn is equivalent to that $f_{(t,x)}$ is real-valued.
\item
Since $Q_0 \in L^2 (E^n)$, 
the function 
\[
(\sigma_1, \ldots, \sigma_n) 
\mapsto Q_t ((\sigma_1,k_1), \ldots, (\sigma_n,k_n))
\]
belongs to $L^2 (d\sigma_1 \cdots d\sigma_n)$ for every 
$t \in {\bf R}, (k_1, \ldots, k_n) \in ({\bf Z}^d)^n$. 
Hence, ${\mathcal F}_{{\rm time}}$ in \eqref{eq:180223-3} works 
perfectly well as an isometry. In particular, 
\begin{equation}\label{eq:180223-6}
\int_{{\mathbf R}^n} ds_1 \cdots ds_n |H_t ((s_1,k_1), \ldots, (s_n,k_n))|^2
=
\int_{{\mathbf R}^n} d\sigma_1 \cdots d\sigma_n 
|Q_t ((\sigma_1,k_1), \ldots, (\sigma_n,k_n))|^2
\end{equation}
for every $t, (k_1, \ldots, k_n)$. 
If in addition
$(s_1, \ldots, s_n) \mapsto H_t ((s_1,k_1), \ldots, (s_n,k_n))$
is $L^1({\bf R}^n)$ for every 
$t, k_1, \ldots, k_n$, then
$Q_t = {\mathcal F}_{{\rm time}} H_t $ is explicitly given by
\begin{eqnarray*}
\lefteqn{
 Q_t ((\sigma_1,k_1), \ldots, (\sigma_n,k_n))
}
\\
&=&
\int_{{\bf R}^n} ds_1\cdots ds_n 
e^{ -2\pi \sqrt{-1} ( \sigma_1 s_1 +\cdots + \sigma_n s_n)}
H_t ((s_1,k_1), \ldots, (s_n,k_n)).
\end{eqnarray*}
\item
The square-integrability of $Q_0$, $Q_t$, $H_t$ and $f_{t,x}$ are mutually equivalent. 
Indeed,
\begin{equation}\label{eq:180223-7}
\| Q_0 \|_{L^2 (\hat{E}^n)} =\| Q_t \|_{L^2 (\hat{E}^n)} 
=\| H_t \|_{L^2 (\hat{E}^n)} 
=
\| f_{t,x} \|_{L^2 (E^n)},
\qquad
(t,x) \in E.
\end{equation}
\item
As a (random) function in $x$, $A_{t,x}$ can be expressed as a (random) Fourier
series as follows: 
\begin{equation}\label{eq:180223-8}
A_{t,x}
= 
\sum\limits_{l \in {\mathbf Z}^d}{\bf e}_{l} (x)
{\mathcal J}_n (g_{t}^{(l)}),
\end{equation}
where we set
\begin{equation}\label{eq:180223-9}
g_{t}^{(l)} ((s_1,y_1), \ldots, (s_n,y_n))
=
\sum\limits_{k_1+ \cdots + k_n =l}
\prod_{i=1}^n {\bf e}_{k_i} (-y_i)
H_t ((s_1,k_1), \ldots, (s_n,k_n))
\end{equation}
for $\{(s_i,y_i)\}_{i=1}^n \in E^n$.
By the Plancherel theorem,
$\sum\limits_{l \in {\mathbf Z}^d} \| g_{t}^{(l)} \|_{L^2 (E^n)}^2 =
\| f_{t,x} \|_{L^2 (E^n)}^2$
for all
$(t,x) \in E$.
\end{enumerate}
\end{remark}

Without a further condition on $Q_0$, 
we can hardly prove that $A =\{ A_{t,x}\}_{(t,x)\in E}$ in Definition \ref{def.gk}
belongs to a suitable Banach space.
A typical condition
for this purpose is as follows:
for almost all
$\mu=(\sigma, k) \in \hat{E}$
\begin{equation}\label{eq:180223-10}
\int_{\hat{E}^{n-1}}
 |Q_0 (\mu_1, \ldots, \mu_{n-1},\mu-\mu_1-\cdots-\mu_{n-1}) |^2
 d\mu_1\cdots d\mu_{n-1}
\le 
M^2 |\mu|_*^{- 2 \gamma} |k|_*^{- 2 \delta}.
\end{equation}
Here, $\gamma, \delta \in {\mathbf R}$ and $M >0$ are constants 
(a further condition on them will be imposed later). 
The integration
is actually taken over the ``hyperplane" 
$
\{ (\mu_1, \ldots, \mu_n) \in \hat{E}^n \mid \mu_1+ \cdots + \mu_n =\mu\}.
$
Hence, \eqref{eq:180223-10} remains the same
even if any argument of $Q_0$ plays a special role instead of the $n$th argument.
When $n=1$, \eqref{eq:180223-10}
simply reads as 
\begin{equation}\label{eq:180417-2}
|Q_0 (\mu) |^2
\le 
M^2 |\mu|_*^{- 2 \gamma} |k|_*^{- 2 \delta}.
\end{equation}

%\vspace{7mm}

Now we present the key lemma for enhancing the white noise.
This can be regarded as a fractional version of the corresponding results in \cite{gp, hin}.
One should note that the condition $\gamma >1$ is quite important.
It will be understood that $C_T^{\kappa} {\mathcal C}^{\alpha }= C_T {\mathcal C}^{\alpha }$
when $\kappa =0$.

\begin{lemma}\label{lm.0405_2}
Let $d \ge 1$ and $\theta \in (0,2]$.
Assume that 
$A = \{A_{t,x} \}_{(t,x) \in E}$
has a good kernel in the sense of Definition $\ref{def.gk}$.
Assume \eqref{eq:180223-10} for some
$\gamma >1$,
$\delta \in {\mathbf R}$, 
$M >0$.

Then, we have
\[
{\mathbf E}
\Bigl[
\| A \|^{2p}_{C_T^{\kappa} {\mathcal C}^{\alpha -\theta \kappa} }
\Bigr]
\lesssim M^{2p},
\qquad
p \in [0,\infty), \,\, \alpha < - \frac{\theta+d}{2} + 
\frac{\gamma \theta}{2}+\delta, 
\,\, \kappa \in \left[0, \frac{\gamma -1}{2} \wedge 1\right).
\]
Here, the implicit constant may depend on $n, p, \alpha, \kappa, \gamma, \delta, T, d, \theta$.
In particular, 
the random field $A$ almost surely belongs to 
$C_T^{\kappa} {\mathcal C}^{\alpha -\theta \kappa}$. 
%
%\footnote{Simply put, the Besov regularity of $Z$ is $\alpha^-$,
%where $\alpha = \gamma +\delta - 
%\textcolor{red}{\mbox{(parabolic space-time dim.)}} /2$}
\end{lemma}

\begin{proof}
In \cite{hin} it is assumed $\delta >0$. But, it looks redundant.
So we give a proof for the sake of completeness.
We may assume that $p\gg 1$. Indeed,
for the case $p=0$ the conclusion is trivial
and so we can use the H\"{o}lder inequality
to have the conclusion for all $p>0$.

We see from the assumption
and
(\ref{eq:180223-10}) that, for any $\kappa \in [0, (\gamma -1)/2)$
and $j=0,1,\ldots$,
\begin{align}\label{eq:180220-2}
\lefteqn{
\Bigl\| 
|\sigma_1 + \cdots + \sigma_n |^{\kappa}
\rho_j (k_1 + \cdots + k_n)
Q_0 (\mu_1, \ldots, \mu_n)
\Bigr\|_{L^2 (\hat{E}^n)}^2
}
\\
&\lesssim
\sum\limits_{k \in {\mathbf Z}^d} \rho_j (k)^2 
 \int_{{\mathbf R} }d \sigma |\sigma|^{2\kappa}
 \int_{\hat{E}^{n-1}}
 |Q_0 (\mu_1, \ldots, \mu_{n-1},\mu-\mu_1-\cdots-\mu_{n-1}) |^2\,d\mu_1\cdots\,d\mu_{n-1}
\nonumber\\
&\lesssim 
M^2
\sum\limits_{k \in {\mathbf Z}^d} \rho_j (k)^2 |k|_*^{-2 \delta}
 \int_{{\mathbf R} }d \sigma |(\sigma, k) |_*^{-2 (\gamma - 2\kappa)}
\nonumber\\
&\sim 
M^2
\sum\limits_{k \in {\mathbf Z}^d} \rho_j (k)^2
|k|_*^{-2 \delta}
|k|_*^{\theta (1- \gamma + 2\kappa)}
\sim 
M^2 (2^j)^{d+\theta - \gamma\theta -2\delta+2\kappa\theta}.
\nonumber
 \end{align}

Now we show ${\mathbf E}[ \| A_t \|_{ {\mathcal C}^{\alpha} }^{2p} ] <\infty$ for every fixed $t$
by using \eqref{eq:180219-1}.
Note that $\triangle_j$ applies to functions in the $x$-variable,
while ${\mathcal J}_n$ applies to functions in the $(s_i, y_i)$-variables.
Then, we can easily see that
\begin{equation}\label{eq:180223-16}
\triangle_j A_{t,x} = {\mathcal J}_n
\Bigl(
\sum\limits_{k_1, \ldots, k_n \in {\bf Z}^d}
\rho_j \left(\sum\limits_{i=1}^n k_i\right)
{\bf e}_{\sum\limits_{i=1}^n k_i} (x)
\prod_{i=1}^n {\bf e}_{-k_i} (y_i)
H_t ((s_1,k_1), \ldots, (s_n,k_n))
\Bigr).
\end{equation}
Since ${\mathcal J}_n$ does not act on $x$,
we have
\[
\triangle_j A_{t,x} =
\sum\limits_{k_1, \ldots, k_n \in {\bf Z}^d}
\rho_j \left(\sum\limits_{i=1}^n k_i\right)
{\bf e}_{\sum\limits_{i=1}^n k_i} (x) {\mathcal J}_n
\Bigl(
\prod_{i=1}^n {\bf e}_{-k_i} (y_i)
H_t ((s_1,k_1), \ldots, (s_n,k_n))
\Bigr)
\]
thanks to equality (\ref{eq:180223-8}).
From this equality, (\ref{eq:180417-1}), \eqref{eq:180223-2} and \eqref{eq:180223-6} (the 
$L^2$-isometric property of ${\mathcal F}_{{\rm time}}$). we have
\begin{align*}
{\mathbf E}[ |\triangle_j A_{t,x}|^2] 
&\le
\int_{{\mathbf R}^n} ds_1 \cdots ds_n
\int_{({\mathbf T}^d)^n} dy_1 \cdots dy_n
\\
&
\times
\Bigl|
\sum\limits_{k_1, \ldots, k_n \in {\bf Z}^d}
\rho_j \left(\sum\limits_{i=1}^n k_i\right)
{\bf e}_{\sum\limits_{i=1}^n k_i} (x)
\prod_{i=1}^n {\bf e}_{-k_i} (y_i)
H_t ((s_1,k_1), \ldots, (s_n,k_n))
\Bigr|^2
\\
&=
\int_{{\mathbf R}^n} ds_1 \cdots ds_n
\sum\limits_{k_1, \ldots, k_n \in {\bf Z}^d}
\left| \rho_j \left(\sum\limits_{i=1}^n k_i\right)\right|^2
|H_t ((s_1,k_1), \ldots, (s_n,k_n))|^2
\\
&=
\int_{{\mathbf R}^n} d\sigma_1 \cdots d\sigma_n
\sum\limits_{k_1, \ldots, k_n \in {\bf Z}^d}
\left| \rho_j \left(\sum\limits_{i=1}^n k_i\right)\right|^2
|Q_0 ((\sigma_1,k_1), \ldots, (\sigma_n,k_n))|^2.
\end{align*}
As a reuslt, this quantity is independent of $x$.
Combining this with
Proposition \ref{prop:180417-1}
and 
\eqref{eq:180220-2} with $\kappa=0$, we have
\begin{equation}\nonumber%\label{eq:180220-3}
{\mathbf E}[ \| A_t \|_{ {\mathcal C}^{\alpha} }^{2p} ] 
\lesssim
M^{2p}
\sum\limits_{j=-1}^{\infty}
(2^j)^{
 d+p (2\alpha +d + \theta - \gamma\theta - 2\delta)}.
\end{equation}
Since $\alpha < - (d + \theta)/2 + \gamma\theta/2 +\delta$, the right-hand side
is finite for sufficiently large $p$.
Thus,
\[
{\mathbf E}[ \| A_t \|_{ {\mathcal C}^{\alpha} }^{2p} ] 
\lesssim
1.
\]

Next we estimate 
${\mathbf E}[\| A \|^{2p}_{C_T^{\kappa} {\mathcal C}^{\alpha -\theta \kappa} }]$.
Going through an argument which is essentially the same way as above, we see that
\begin{align*}
\lefteqn{
{\mathbf E}[ |\triangle_j A_{t,x}- \triangle_j A_{t',x} |^2] 
}\\
&\le 
\int_{{\mathbf R}^n} d\sigma_1 \cdots d\sigma_n
\sum\limits_{k_1, \ldots, k_n \in {\bf Z}^d}
\left| \rho_j \left(\sum\limits_{i=1}^n k_i\right)\right|^2
|Q_0 ((\sigma_1,k_1), \ldots, (\sigma_n,k_n))|^2
\\
&
\quad
\times
| e^{-2\pi \sqrt{-1} ( \sigma_1 +\cdots + \sigma_n)t}
 -e^{-2\pi \sqrt{-1} ( \sigma_1 +\cdots + \sigma_n)t'} |^2
 \\
 &
 \lesssim
 |t -t'|^{2\kappa}
 \int_{{\mathbf R}^n} d\sigma_1 \cdots d\sigma_n
\sum\limits_{k_1, \ldots, k_n \in {\bf Z}^d}
\left| \sum\limits_{i=1}^n \sigma_i\right|^{2\kappa}
\left| \rho_j \left(\sum\limits_{i=1}^n k_i\right)\right|^2
|Q_0 ((\sigma_1,k_1), \ldots, (\sigma_n,k_n))|^2.
\end{align*}
for $t, t' \in [0,T]$ and $\kappa \in [0, (\gamma -1)/2) \cap [0,1]$.
Let
\[
\Theta=2(\alpha - \theta \kappa) p + d
+p (d + \theta - \gamma\theta - 2\delta + 2\kappa \theta).
\]
Combining this with
\eqref{eq:180219-1}
and 
\eqref{eq:180220-2}, we have
\begin{eqnarray}\label{eq:180220-3}
{\mathbf E}[ \| A_t - A_{t'} \|_{ {\mathcal C}^{\alpha -\theta \kappa} }^{2p} ] 
\sim
M^{2p} |t -t'|^{2p\kappa }
 \sum\limits_{j=-1}^{\infty}
(2^j)^{\Theta }
\sim
 M^{2p} |t -t'|^{2p\kappa }.
 \end{eqnarray}
if $p$ is sufficiently large 
since $\alpha < - (d + \theta)/2 + \gamma\theta/2 +\delta$.
We may also assume that $2\kappa p >1$ holds.
Then, 
the so-called Besov-H\"older embedding theorem yields 
\begin{align}
{\mathbf E}
\Bigl[ \Bigl( \sup\limits_{0 \le t < t' \le T }\frac{
\| A_t - A_{t'} \|_{ {\mathcal C}^{\alpha -\theta \kappa} } 
}
{
|t -t'|^{\kappa -(1/2p)}
}
\Bigr)^{2p} 
\Bigr] 
&\lesssim
\int_{[0,T]^2} \frac{
{\mathbf E}[ \| A_t - A_{t'} \|_{ {\mathcal C}^{\alpha -\theta \kappa} }^{2p}] 
}
{
|t -t'|^{1 +2p\kappa }} dtdt'.
\label{eq:180220-4}
\end{align}

We can find $\alpha <\alpha' <\alpha''$ and $0 \le \kappa < \kappa' < \kappa''$ 
such that $(\alpha', \kappa')$ and $(\alpha'', \kappa'')$
still satisfy 
\[
\alpha', \alpha'' < - \frac{\theta+d}{2} + 
\frac{\gamma \theta}{2}+\delta, 
\, \kappa', \kappa'' \in \left[0, \frac{\gamma -1}{2} \wedge 1\right)
\]
and 
$\alpha -\theta \kappa = \alpha' -\theta \kappa' = \alpha'' -\theta \kappa''$.
If $p$ is sufficiently large, we see from \eqref{eq:180220-4}
for $(\alpha', \kappa')$ and
 \eqref{eq:180220-3} for $(\alpha'', \kappa'')$ that
\begin{align*}
{\mathbf E}
\Bigl[ \Bigl( \sup\limits_{0 \le t < t' \le T }\frac{
\| A_t - A_{t'} \|_{ {\mathcal C}^{\alpha -\theta \kappa} } 
}
{
|t -t'|^{\kappa }
}
\Bigr)^{2p} 
\Bigr] 
&\lesssim
\int_{[0,T]^2} \frac{
{\mathbf E}[ \| A_t - A_{t'} \|_{ {\mathcal C}^{\alpha' -\theta \kappa'} }^{2p}] 
}
{
|t -t'|^{1 +2p\kappa' }} dtdt'
\\
&\lesssim
\int_{[0,T]^2} \frac{M^{2p}|t -t'|^{2p\kappa'' }
}
{
|t -t'|^{1 +2p\kappa' }} dtdt'\\
&\lesssim M^{2p}.
\end{align*}
This proves the lemma.
 \end{proof}

%\vspace{7mm}

Next we generalize the definition of \lq \lq random fields with a good kernel"
in Definition \ref{def.gk}. 
We assumed 
$\| Q_0 \|_{L^2 (\hat{E}^n)} =\| f_{t,x} \|_{L^2 (E^n)} <\infty$
in Definition \ref{def.gk}
in order to define $A_{t,x}$ by \eqref{eq:180223-5}.
In many cases, however, 
this condition is too restrictive.
For example, think of
the Ornstein--Uhlenbeck process $X = {\mathcal I}[\xi]$, 
which is the simplest component of our driver ${\mathbf X}$.
As we will see in the next subsection, 
$Q_0$ for $X$ in not $L^2 (E)$.
This fact motivates our generalization.

The reason why we can generalize is as follows:
The expressions \eqref{eq:180223-8}--\eqref{eq:180223-9}
suggest that
the condition ``$\| f_{t,x} \|_{L^2 (E^n)}^2 <\infty$" may be relaxed:
\begin{quote}
There exists $N \in {\mathbf N}$ such that
$\| g_{t}^{(l)} \|_{L^2 (E^n)}^2={\rm O}(|l|^N)$ for every $l$.
\end{quote}
Indeed, if 
\begin{equation}\label{eq:180223-12}
\| g_{t}^{(l)} \|_{L^2(E^n)}^2 = {\rm O}(|l|^{N})
\quad \mbox{ for some $N$ as $|l| \to \infty$}, 
\end{equation}
then 
from (\ref{eq:180417-1})
\[
{\mathbf E}[ | {\mathcal J}_n (g_{t}^{(l)})|^2] = {\rm O}(|l|^{N})
\]
and
\begin{equation}\label{eq:180223-11}
A_{t}
:= 
\sum\limits_{l \in {\mathbf Z}^d}{\bf e}_{l} (\cdot)
{\mathcal J}_n (g_{t}^{(l)})
\qquad
\mbox{in ${\mathcal D}' = {\mathcal D}'({\bf T}^d, {\bf R})$, %\quad $t \in {\mathbf R}$,
}
\end{equation}
defines a random distribution (i.e., ${\mathcal D}'$-valued random variable)
for each $t$.
(This part is a general theory.)

\begin{remark}\label{re:180223-2}
At first glance,
\eqref{eq:180223-8}
in the classical case and \eqref{eq:180223-11} 
in the generalized case
seem identical.
However, their meanings are different.
\eqref{eq:180223-8} is for an arbitrary fixed $x \in {\mathbf T}$,
while
\eqref{eq:180223-11} is 
(supposed to be) an equality in ${\mathcal D}'$
and therefore 
one can not substituting a fixed $x$ into \eqref{eq:180223-11}.
\end{remark}

\begin{definition}\label{def.gk2}
Let a measurable function
$Q_0 \colon \hat{E}^n \to {\mathbf C}$ with \eqref{eq:180223-1} be 
such that
\[
(\sigma_1, \ldots, \sigma_n) 
\mapsto Q_0 ((\sigma_1,k_1), \ldots, (\sigma_n,k_n))
\]
belongs to $L^2 (d\sigma_1 \cdots d\sigma_n)$ for every 
$k_1, \ldots, k_n \in {\bf Z}^d$. 
Define 
$Q_t$, $H_t$, $g_t^{(l)}$ by \eqref{eq:180223-2},
\eqref{eq:180223-3} and \eqref{eq:180223-9}, respectively.

By slightly abusing the terminology, 
we say that
a ${\mathcal D}'$-valued stochastic process 
 $A =\{ A_t\}_{t \in {\bf R} }$ has a {\it good kernel} 
 if there exists $Q_0$ as above such that 
 the right-hand side of \eqref{eq:180223-11} makes sense as a random 
 distribution and equals $A_t$ a.s. for every $t$. 
 \end{definition}

In the above definition,
we do not claim here that every such $Q_0$ defines a stochastic process by 
 \eqref{eq:180223-11}. 
 We need
an extra condition 
to ensure it.
For instance if we assume that 
$Q_0 \colon \hat{E}^n \to {\mathbf C}$ with \eqref{eq:180223-1}
satisfies \eqref{eq:180223-10} for some $\gamma >1$, $\delta \in {\mathbf R}$,
$M>0$.
Then, a calculation quite similar to \eqref{eq:180220-2} shows that
$\| g_{t}^{(l)} \|_{L^2(E^n)}^2 \lesssim |l|_*^{-2\delta -\theta (\gamma-1)}$
and that
$\| g_{t}^{(k_1+ \cdots+ k_n)} \|_{L^2(E^n)}$ donimates the $L^2$-norm of 
$(\sigma_1, \ldots, \sigma_n) 
\mapsto Q_0 ((\sigma_1,k_1), \ldots, (\sigma_n,k_n))$ for any given 
$(k_1, \ldots, k_n)$.
Therefore, under this assumption the right-hand side of
 \eqref{eq:180223-11} makes sense.

Going deeper in this direction, we can naturally generalize 
our key lemma (Lemma \ref{lm.0405_2}) as follows:
\begin{lemma}\label{lm.0405_2B}
Let $d \ge 1$, $\theta \in (0,2]$, $p \ge 0$, 
$\gamma >1$,
$\delta \in {\mathbf R}$, $M >0$ and 
$
0 \le \kappa < \frac{\gamma -1}{2} \wedge 1.
$
Assume that 
$Q_0 \colon \hat{E}^n \to {\mathbf C}$ with \eqref{eq:180223-1}
satisfies condition \eqref{eq:180223-10}.
Then, 
$A = \{A_{t} \}_{t \in {\bf R}}$ defined by \eqref{eq:180223-11}
has a good kernel in the sense of Definition \ref{def.gk2}.
Moreover, all the conclusions of Lemma \ref{lm.0405_2} 
remain true
as long as
\begin{equation}\label{eq:180510-1}
\alpha < - \frac{\theta+d}{2} + 
\frac{\gamma \theta}{2}+\delta.
\end{equation}
\end{lemma}

\begin{proof}
The proof of Lemma \ref{lm.0405_2} works without any modification.
To understand this,
simply observe that 
what is actually computed in the proof of Lemma \ref{lm.0405_2}
is always $\triangle_j A_{t,x}$, not $A_{t,x}$ itself.

Since we have already seen that $A_t$ is a random distribution, 
$\triangle_j A_t$ is well-defined as a finite Fourier series
(and hence can be evaluated at every $x$):
\begin{equation}\nonumber%\label{eq:180223-13}
\triangle_j A_{t,x}
= 
\sum\limits_{l \in {\mathbf Z}^d} \rho_j (l) {\bf e}_{l} (x)
{\mathcal J}_n (g_{t}^{(l)})
=
{\mathcal J}_n \left( \sum\limits_{l \in {\mathbf Z}^d} \rho_j (l) {\bf e}_{l} (x) g_{t}^{(l)} \right).
\end{equation}
By the definition of $g_{t}^{(l)}$, this coincides with \eqref{eq:180223-16}.
The rest is exactly the same as in the corresponding parts of the
 proof of Lemma \ref{lm.0405_2},
where the isometry of ${\mathcal F}_{{\rm time}}$
 and condition \eqref{eq:180223-10} are used, 
but the assumption ``$\| Q_0 \|_{L^2 (\hat{E}^n)} <\infty$" is not.
\end{proof}

%%%%%%%%%%%%%%%%%
%%%%%%%%%%%%%%%%%
%\newpage
%%%%%%%%%%%%%%%%%
%%%%%%%%%%%%%%%%%

\subsection{Fractional version of Ornstein--Uhlenbeck process}
From now on we work on ${\mathbf T}^2$.
In this subsection we prove the convergence of the first component of 
the driver ${\mathbf X}^{\ve}$. 
%We give a detailed explanation for this component. 
%(Explanations for the other components will be simpler.)

What we call (the fractional version of) the
 Ornstein--Uhlenbeck process is the solution of the 
 following linearized equation:
\begin{equation}\nonumber
\partial_t X= -(-\Delta)^{\theta /2} X -X + \xi.
\end{equation}
Its stationary solution should be given by 
\begin{equation}\label{eq:180221-2}
X_{t} = {\mathcal I}[\xi]_t
=
 \int_{- \infty}^t e^{- (t-s) ( (-\Delta)^{\theta /2}+1)} (\xi_s) ds 
=
\sum\limits_{k \in {\bf Z^2}} 
{\bf e}_k (\bullet)
 \int_{- \infty}^{\infty}
h(t-s, k) \hat{\xi}_s (k) ds,
 \end{equation}
where $h: \hat{E}= {\bf R} \times {\bf Z}^2 \to {\bf R}$ is defined by
\begin{equation}\label{def.h}
h(t, k) := {\bf 1}_{(0, \infty)} (t) e^{-\{ (2\pi |k|)^{\theta} +1\} t }.
\end{equation}
At this stage, it is not completely obvious whether $X$ makes sense.
%
%However, since $X_{t,x}$ cannot be expressed as a Wiener-It\^o integral, 
%Lemma \ref{lm.0405_2} does not apply to $X$ directly.

So, we first compute the smooth approximation $X^{\ve}$.
Set, for $\ve \in (0,1]$,
\begin{equation} \label{eq.0914_1}
X_{t,x}^{\ve} 
:=
\sum\limits_{k \in {\bf Z^2}} \chi^{\ve} (k)
{\bf e}_k (x)
 \int_{- \infty}^{\infty}
h(t-s, k) \hat{\xi}_s (k) ds
=
{\mathcal J}_1
(f^{\ve}_{t,x}),
 \end{equation}
 where 
 \begin{equation} \label{eq:180221-1}
 f^{\ve}_{t,x} (s, y) 
 = 
 \sum\limits_{k \in {\bf Z^2}} \chi^{\ve} (k)
{\bf e}_k (x- y)
 h(t- s, k).
 \end{equation}
Thanks to the truncation by $\chi$, 
$f^{\ve}_{t,x}$ (the integrand in ${\mathcal J}_1$) 
belongs to $L^2 (E)$.
Hence, the right-hand side of \eqref{eq.0914_1} is well defined.

Let $(s,k), (\sigma,k) \in \hat{E}$.
Set 
\begin{eqnarray}\label{eq.0406_1}
H_t^{X} (s,k) &:=& h(t-s, k),
\\
Q_t^{X} (\sigma,k) &:=& 
\frac{e^{- 2\pi \sqrt{-1} \sigma t}}{ -2\pi \sqrt{-1} \sigma + (2\pi |k|)^{\theta} +1 }
=
e^{-2\pi \sqrt{-1} \sigma t} Q_0^{X} (\sigma,k),
\nonumber
\\
H_t^{X,\ve} (s,k) &:=& \chi^{\ve} (k)H_t^{X} (s,k), 
\qquad
Q_t^{X,\ve} (\sigma,k):= \chi^{\ve} (k)Q_t^{X} (\sigma,k).
\nonumber
\end{eqnarray}
Obviously, $Q_0^{X}$ satisfies 
(\ref{eq:180417-2}),
the assumption of Lemma \ref{lm.0405_2B} and 
hence defines a Besov space-valued process with a good kernel, 
which immediately turns out to be $X$.
Since $Q_0^{X}$ is independent of $\chi$, so is $X$.
Similarly, $X^{\ve}$ has a good kernel
 defined from $Q_0^{X,\ve}$.

%\vspace{7mm}

The following lemma will be a model case
to many other estimates:
\begin{lemma}\label{lm.0406_a}
Let $\theta \in (0,2]$.
For every $\alpha < \frac{\theta}{2} -1$ and $1 < p <\infty$, it holds that
\[
{\mathbf E} [ \| X \|_{C_T {\mathcal C}^{ \alpha}}^p ]<\infty,
\qquad
\lim_{\ve \searrow 0}
{\mathbf E} [ \| X - X^{\ve} \|_{C_T {\mathcal C}^{ \alpha}}^p ]
=0.
\]
%Moreover, $\lim_{\ve \searrow 0} X^{\ve}$ does not depend on $\chi$.
\end{lemma}

\begin{proof}
From (\ref{eq.0406_1})
\[
|Q_0^{X,\ve} (\sigma,k) |^2
\le
|Q_0^{X} (\sigma,k) |^2 
=
\frac{1}{ (2\pi \sigma)^2 + \{ (2\pi |k|)^{\theta} +1 \}^2 }
\lesssim 
|(\sigma, k)|_*^{-4},
\]
so that (\ref{eq:180417-2}) holds
with $\gamma =2$, $\delta =0$,
so that our assumption $\alpha < \frac{\theta}{2} -1$
matches (\ref{eq:180510-1}).
Then, we can use Lemma \ref{lm.0405_2B}
to see $X, X^{\ve} \in L^p (\Omega, {\mathbf P}; C_T {\mathcal C}^{ \alpha})$
if $\alpha < \frac{\theta}{2} -1$.

Next, we prove $X^{\ve}$ is Cauchy.
Since $X$ and $X^{\ve}$ have a good kernel whose $Q$-functions
are given by $Q^X_0$ and $Q^{X,\ve}_0$,
respectively, 
$X - X^{\ve}$ also has a good kernel
whose $Q$-function is given by $Q_0^{X} - Q_0^{X,\ve}$.
For every sufficiently small $a >0$, we have
\begin{align*}
|Q_0^{X} (\sigma,k)- Q_0^{X,\ve} (\sigma,k)|^2
&\lesssim
| 1 - \chi(\ve k) |^2 |Q_0^{X} (\sigma,k) |^2 
\\
&\lesssim
\frac{\ve^{\frac12\theta a} |k|^{\frac12 \theta a}}{ (2\pi \sigma)^2 + \{ (2\pi |k|)^{\theta} +1 \}^2 }
\lesssim 
\ve^{\theta a} |(\sigma, k)|_*^{- (4-a)}.
\end{align*}
Thus,
(\ref{eq:180417-2}) is fulfilled.
We use Lemma \ref{lm.0405_2B} with $\gamma =2 -a/2 \in (1,\infty)$, $\delta =0$
to prove the convergence of $X- X^{\ve}$ 
in $L^p (\Omega, {\mathbf P}; C_T {\mathcal C}^{ \alpha})$
as $\ve \searrow 0$
if $\alpha < \frac{\theta}{2} -1$
since the constant $a$ is arbitrarily small.
%
%
%Finally, we show that $X:= \lim_{\ve \searrow 0} X^{\ve}$ is independent of $\chi$.
%First recall that
%$\psi \in {\mathcal D}'$ if and only if $\psi$ is written as
% a Fourier series with at most polynomially growing Fourier coefficient.
%If $\lim_{l\to\infty}\psi_l = \psi$ in ${\mathcal D}'$, 
%then $\hat\psi (k) = \lim_{l\to\infty}\hat\psi_l (k)$ for every ${\mathbf Z}^2$.
%
%From what we have shown, there exists a ${\mathbf P}$-null set ${\mathcal N} \subset\Omega$ 
%and a subsequence $\{\ve_l\}_{l \ge 1}$ with $\ve_l \searrow 0$ as $l \to \infty$
%with the following property:
%For every $t$ and $\omega \in \Omega \setminus {\mathcal N}$, 
%$X_t = \lim_{ l \to \infty} X^{\ve_l}_t$ in ${\mathcal C}^{ \alpha}$.
%The Fouier series for $X^{\ve_l}_t$ is given \eqref{eq.0914_1}.
%Obviously,
%$ \hat{X_t} (k) =\lim_{ l \to \infty} \widehat{X^{\ve_l}_t} (k) = \int_{- \infty}^{\infty}
%h(t-s, k) \hat{\xi}_s (k) ds$
%does not involve $\chi$ since $\chi (0) =1$.
%(Therefore, \eqref{eq:180221-2} is not so formal.)
%
\end{proof}

%\vspace{7mm}

Next, we study $V={\mathcal I} [\nabla X]$.
We will check that $V_i^{\ve} ={\mathcal I} [\partial_i X^{\ve}]$ $(i =1,2)$ 
is convergent 
%in $B^s_{\infty\infty}({\bf T}^d)$ 
as $\ve \searrow 0$ if the Besov regularity %$s$
 is smaller than
$\frac32 \theta -2$.

From \eqref{eq.0914_1} we have
\begin{equation} \nonumber 
(V_{i}^{\ve})_{t,x}
=
\sum\limits_{k \in {\bf Z^2}} \chi^{\ve} (k)
{\bf e}_k (x)
(2\pi \sqrt{-1} k^i)
 \int_{- \infty}^{\infty}
 du h(t-u, k) 
 \int_{- \infty}^{\infty}
h(u-s, k) \hat{\xi}_s (k) ds
 \end{equation}
and the $H$-function and the $Q$-function in this case are given by
\begin{eqnarray} \nonumber 
H^{V_i}_t (s,k) 
&=& 
2\pi \sqrt{-1} k^i
 \int_{- \infty}^{\infty}
 du h(t-u, k) 
h(u-s, k),
\\
Q^{V_i}_t (\sigma,k) &=& 
\frac{e^{- 2\pi \sqrt{-1} \sigma t} 2\pi \sqrt{-1} k^i}
{ \{ -2\pi \sqrt{-1} \sigma + (2\pi |k|)^{\theta} +1 \}^2 }
=
e^{-2\pi \sqrt{-1} \sigma t} Q^{V_i}_0 (\sigma,k).
\nonumber
\end{eqnarray}
We set
\begin{eqnarray}
H_t^{V_i,\ve} (s,k) &=& \chi^{\ve} (k)H_t^{V_i} (s,k), 
\qquad
Q_t^{V_i,\ve} (\sigma,k)= \chi^{\ve} (k)Q_t^{V_i} (\sigma,k).
\nonumber
\end{eqnarray}
This shows that $V_{i}^{\ve}$ has a good kernel,
whose $H$-function and $Q$-function are 
given by $H_t^{V_i,\ve}$ and $Q_t^{V_i,\ve}$, respectively.
It will turn out that $Q_0^{V_i}$ satisfies 
the assumption of Lemma \ref{lm.0405_2B}.
The corresponding process is denoted by $V_i$.
Since $Q_0^{V_i}$ is independent of $\chi$, so is $V_i$ ($i=1,2$).

\begin{lemma}\label{lm.0406_b}
Let $\theta \in (2/3,2]$.
Then
for every $\alpha < \frac{3}{2}\theta -2$, $1 < p <\infty$ and $i=1,2$, 
\[
{\mathbf E} [ \| V_i \|_{C_T {\mathcal C}^{ \alpha}}^p ]
<\infty,
\qquad
\lim_{\ve \searrow 0}
{\mathbf E} [ \| V_i - V_i^{\ve} \|_{C_T {\mathcal C}^{ \alpha}}^p ]
=0.
\]
%Moreover, $\lim_{\ve \searrow 0} V_i^{\ve}$ does not depend on $\chi$.
\end{lemma}

\begin{proof}
We set
\[
(\gamma,\delta)=\left(4-\frac{2}{\theta},0\right)
\]
to apply Lemma \ref{lm.0405_2B},
so that our assumption $\alpha < \frac{3\theta}{2} -2$
matches (\ref{eq:180510-1}).
Note that $4 -2/\theta >1$ if and only if $\theta >2/3$.
In this case, we have 
\[
|Q_0^{V_i,\ve} (\sigma,k) |^2 
\lesssim
|Q_0^{V_i} (\sigma,k) |^2 
\lesssim
\frac{|k|^2}{[ (2\pi \sigma)^2 + \{ (2\pi |k|)^{\theta} +1 \}^2 ]^2 }
\lesssim 
|(\sigma, k)|_*^{- 2\gamma}.
\]
The rest is omitted since it is 
quite similar to the proof of Lemma \ref{lm.0406_a}.
\end{proof}

%%%%%%%%%%%%%%%%%

%\newpage
\subsection{Convergence of $Y:= {\mathcal I}[R^\perp X \cdot \nabla X]$}

In this subsection 
we prove that 
\[
Y^{\ve}:= {\mathcal I}[R_2 X^{\ve} \cdot \partial_1 X^{\ve}] 
- {\mathcal I}[R_1 X^{\ve} \cdot \partial_2 X^{\ve}]
\]
is convergent 
%in $B^s_{\infty\infty}({\bf T}^d)$ 
as $\ve \searrow 0$ if the Besov regularity %$s$ 
is smaller than $2\theta -3$.

Firstly, $R_i X_{t,x}^{\ve}$ has the following expression ($i=1,2$):
\begin{eqnarray}\label{eq.0406_2}
R_i X_{t,x}^{\ve} 
&=&
\sum\limits_{k=(k^1, k^2) \in {\bf Z^2}} \chi^{\ve} (k)
\frac{2\pi \sqrt{-1}k^i}{2\pi |k|} {\bf 1}_{k \neq 0}
{\bf e}_k (x)
 \int_{- \infty}^{\infty}
h(t-s, k) \hat{\xi}_s (k) ds
\\
&=&
{\mathcal J}_1
\Bigl(
\sum\limits_{k \in {\bf Z^2}} \chi^{\ve} (k)
\frac{2\pi \sqrt{-1}k^i}{2\pi |k|} {\bf 1}_{k \neq 0}
{\bf e}_k (x- \star)
 h(t- \bullet, k)
\Bigr),
\nonumber
 \end{eqnarray}
whose $H$-function and $Q$-function are as follows:
For $(s,k), (\sigma,k) \in \hat{E}$,
\begin{eqnarray}\label{eq.0406_3}
H^{R_i X}_t (s,k) &=& h(t-s, k) \frac{2\pi \sqrt{-1}k^i}{2\pi |k|} {\bf 1}_{k \neq 0},
\\
Q^{R_i X}_t (\sigma,k) &=& 
\frac{e^{-2\pi \sqrt{-1} \sigma t}}{- 2\pi \sqrt{-1} \sigma + (2\pi |k|)^\theta +1 }
\frac{ 2\pi \sqrt{-1}k^i }{2\pi |k|} {\bf 1}_{k \neq 0}.
\nonumber
\\
H^{R_i X,\ve}_t (s,k) &=& \chi^{\ve} (k) H^{R_i X}_t (s,k),
\qquad 
Q^{R_i X,\ve}_t (\sigma,k) = \chi^{\ve} (k) Q^{R_i X}_t (\sigma,k).
\nonumber
\end{eqnarray}
Here, $h$ is given by (\ref{def.h}).

Secondly, $\partial_j X_{t,x}^{\ve}$ has the following expression
($j =1,2$):
\begin{eqnarray}\label{eq.0406_4}
\partial_j X_{t,x}^{\ve} 
&=&
\sum\limits_{l=(l^1, l^2) \in {\bf Z^2}} \chi^{\ve} (l)
(2\pi \sqrt{-1} l^j)
{\bf e}_l (x)
 \int_{- \infty}^{\infty}
h(t-s, l) \hat{\xi}_s (l) ds
\\
&=&
{\mathcal J}_1
\Bigl(
\sum\limits_{l \in {\bf Z^2}} \chi^{\ve} (l)
(2\pi \sqrt{-1} l^j)
{\bf e}_l (x- \star)
 h(t- \bullet, l)
\Bigr).
\nonumber
 \end{eqnarray}
In this case, 
the $H$-function and the $Q$-function are as follows:
For $(s,k), (\sigma,l) \in \hat{E}$,
\begin{eqnarray}\label{eq.0406_5}
H^{\partial_j X}_t (s,l) &=& h(t-s, l) (2\pi \sqrt{-1}l^j),
\\
Q^{\partial_j X}_t (\sigma,l) 
&=& 
\frac{e^{- 2\pi \sqrt{-1} \sigma t} 2\pi \sqrt{-1}l^j}
{- 2\pi \sqrt{-1} \sigma + (2\pi |l|)^\theta +1 }
\nonumber
\\
H^{\partial_j X,\ve}_t (s,l) &=& \chi^{\ve} (l) H^{\partial_j X}_t (s,l),
\qquad
Q^{\partial_j X,\ve}_t (\sigma,l) = \chi^{\ve} (l) Q^{\partial_j X}_t (\sigma,l).
\nonumber
\end{eqnarray}

Hence, 
${\mathcal I}[R_i X^{\ve} \cdot \partial_j X^{\ve}]_{t,x}$ 
has the following expression:
\begin{eqnarray}\label{eq.0406_6}
{\mathcal I} [R_i X^{\ve} \cdot \partial_j X^{\ve} ]_{t,x}
&=&
 \sum\limits_{k, l \in {\bf Z^2}}
 \chi^{\ve} (k) \chi^{\ve} (l) {\bf e}_{k+l} (x)
 \int_{- \infty}^{\infty}
du
h (t-u, k+l) 
\\
&&
\times 
\frac{2\pi \sqrt{-1}k^i }{2\pi |k|} {\bf 1}_{k \neq 0} \cdot (2\pi \sqrt{-1} l^j)
\nonumber\\
&&
\times
 \int_{- \infty}^{\infty}\int_{- \infty}^{\infty}
h(u-s_1, k) h(u-s_2, l) \hat{\xi}_{s_1} (k) \hat{\xi}_{s_2} (l) ds_1 ds_2.
\nonumber
 \end{eqnarray}

Since the right-hand side of \eqref{eq.0406_6} above is a 
homogeneous polynomial in $\xi$ of degree two, 
the following Wiener chaos decomposition holds:
\[
{\mathcal I}[R_i X^{\ve} \cdot \partial_j X^{\ve}]_{t,x}
=
\Pi_2 ({\mathcal I}[R_i X^{\ve} \cdot \partial_j X^{\ve}]_{t,x}) 
+\Pi_0 ({\mathcal I}[R_i X^{\ve} \cdot \partial_j X^{\ve}]_{t,x}).
\]
Now we calculate the second order term
on the right-hand side of (\ref{eq.0406_6}).
It is easy to see that
\[
\Pi_2 ({\mathcal I}[R_i X^{\ve} \cdot \partial_j X^{\ve}]_{t,x}) 
=
{\mathcal J}_2 
\Bigl(
f_{(t,x)}^{\ve}
\Bigr),
\]
where
\begin{eqnarray}\label{eq.0406_7A}\nonumber
f_{(t,x)}^{\ve} ((y_1,s_1),(y_2,s_2))
&=&
\sum\limits_{k, l \in {\bf Z}^2}
{\bf e}_{k} (x- y_1) {\bf e}_{l} (x- y_2) \chi^{\ve} (k) \chi^{\ve} (l)
H_t ((s_1, k), (s_2, l))
\end{eqnarray}
with
\begin{eqnarray}\label{eq.0406_7B}\nonumber
H_t ((s_1, k), (s_2, l))
=
 \int_{- \infty}^{\infty}
du
h (t-u, k+l)
\Bigl[ \frac{2\pi \sqrt{-1}k^i}{2\pi |k|} {\bf 1}_{k \neq 0} h(u-s_1, k) \Bigr]
\Bigl[
 (2\pi \sqrt{-1}l^j)h(u-s_2, l) \Bigr]
 \nonumber
 \end{eqnarray}
 for
 $(s_1,k), (s_2,l) \in \hat{E}$.

It is straightforward to check that
\begin{eqnarray}\label{eq.0406xyz}
Q_t ((\sigma_1, k), (\sigma_2, l))
&=&
\frac{e^{- 2\pi \sqrt{-1} (\sigma_1 +\sigma_2) t}}
{ -2\pi \sqrt{-1} (\sigma_1 +\sigma_2) + (2\pi |k+l|)^\theta +1 }
\\
&&
\times
\frac{\sqrt{-1} k^i |k|^{-1} {\bf 1}_{k \neq 0}}{- 2\pi \sqrt{-1} \sigma_1 + (2\pi |k|)^\theta +1 }
\cdot
\frac{2\pi \sqrt{-1}l^j }{ -2\pi \sqrt{-1} \sigma_2 + (2\pi |l|)^\theta +1 }.
\nonumber
\end{eqnarray}
These show that 
$\Pi_2 ({\mathcal I}[R_1 X^{\ve} \cdot \partial_2 X^{\ve}]) $ has a good kernel
given by 
\begin{equation}\label{eq:180221-3}
Q^{\ve}_t ((\sigma_1, k), (\sigma_2, l)) = \chi^{\ve} (k)\chi^{\ve} (l)Q_t ((\sigma_1, k), (\sigma_2, l)).
\end{equation}
It will turn out that $Q_0$ above satisfies 
the assumption of Lemma \ref{lm.0405_2B} 
and hence it defines a Besov space-valued process 
$\Pi_2 ({\mathcal I}[R_i X \cdot \partial_j X])$.

In fact, the zeroth order term vanishes.
By the contraction rule
${\mathbf E} [ \hat{\xi}_{s_1} (k) \hat{\xi}_{s_2} (l) ] 
= \delta (s_1- s_2) \delta_{k+l, 0}$, 
we see from
(\ref{eq:180419-1}),
(\ref{eq:180510-2}) and (\ref{eq.0406_6}) that
\begin{eqnarray}
\Pi_0 ({\mathcal I}[R_i X^{\ve} \cdot \partial_j X^{\ve}]_{t,x}) 
&=&
{\mathbf E} [ {\mathcal I}[R_i X^{\ve} \cdot \partial_j X^{\ve}]_{t,x} ] 
\nonumber\\
&=&
 \sum\limits_{k \in {\bf Z^2}}
 \chi^{\ve} (k)^2
\frac{2\pi \sqrt{-1}k^i}{2\pi |k|} {\bf 1}_{k \neq 0} \cdot 2\pi \sqrt{-1}(-k^j)
\nonumber\\
&&
\times
 \int_{- \infty}^{\infty}
du
h (t-u, 0) 
\int_{- \infty}^{\infty}
h(u-s_1, k)^2 ds_1,
\label{eq.0606_1}
\end{eqnarray}
from which 
$\Pi_0 ({\mathcal I}[R_2 X^{\ve} \cdot \partial_1 X^{\ve}]_{t,x} 
- {\mathcal I}[R_1 X^{\ve} \cdot \partial_2 X^{\ve}]_{t,x}) =0$ immediately follows.
Thanks to this cancellation, 
we do not need
any renormalization for $Y$ given by (\ref{eq:180430-2}).

We give a main result of this subsection.
Set 
\begin{equation}\label{eq:180430-2}
Y 
:= 
\Pi_2 ({\mathcal I}[R_2 X \cdot \partial_1 X])
- 
\Pi_2 ({\mathcal I}[R_1 X \cdot \partial_2 X]),
\end{equation}
which is independent of $\chi$ by definition.

\begin{lemma}\label{lm.0406_c}
Let $\theta \in (3/2, 2]$.
Then, for every $0<\alpha < 2\theta -3$ and $1 < p <\infty$, it holds that
\[ {\mathbf E}
 [\|Y\|_{C_T {\mathcal C}^{ \alpha } }^p 
 +
 \|Y\|_{C_T^{ \alpha/\theta} {\mathcal C}^{ 0 } }^p
 ]
 <\infty,
\qquad
\lim_{\ve \searrow 0} {\mathbf E}
 [\|Y-Y^{\ve}\|_{C_T {\mathcal C}^{ \alpha } }^p 
 +
 \|Y -Y^{\ve}\|_{C_T^{ \alpha/\theta} {\mathcal C}^{ 0 } }^p
 ]
=0.
\]
%Moreover, $\lim_{\ve \searrow 0} Y^{\ve}$ does not depend on $\chi$.
\end{lemma}

\begin{proof}
We prove the lemma for $\theta \in (3/2, 2)$.
The proof for the case $\theta =2$ is just a slight modification.

It is sufficient to estimate $\Pi_2 ({\mathcal I}[R_i X^{\ve} \cdot \partial_j X^{\ve}]_{t,x})$ 
and $\Pi_2 ({\mathcal I}[R_i X \cdot \partial_j X]_{t,x})$ 
for $i \neq j$.
Let $Q_t^{\ve}$ be the corresponding $Q$-function given by \eqref{eq.0406xyz}
 and \eqref{eq:180221-3}.
We use Lemma \ref{lm.0405_2B} with
\[
(\gamma,\delta)=\left(5-\frac{4}{\theta},0\right).
\]
Let
$(\sigma_1,k), (\sigma_2,l) \in E$.
It is easy to see from (\ref{eq.0406xyz}) and Lemma \ref{lm.0405_1} that 
\[
|Q^{\ve}_0 ((\sigma_1, k), (\sigma_2, l)) |^2 
\lesssim
|Q_0 ((\sigma_1, k), (\sigma_2, l)) |^2 
\lesssim 
|(\sigma_1 +\sigma_2, k+l)|_*^{-4}
|(\sigma_1, k)|_*^{-4}
|(\sigma_2, l)|_*^{-4 (1 -\frac{1}{\theta})}.
\]
Then, if $\theta \in (3/2, 2)$, we can use Lemma \ref{lm.0405_1} to obtain that
\[
\int_{\hat{E}}
 |Q^{\ve}_0 ( (\sigma_1, k), (\tau-\sigma_1, m-k)) |^2
\, d(\sigma_1, k)
\lesssim
 |(\tau, m) |_*^{- 2\gamma},
\]
where the implicit constant is independent of $\ve$.
Now we use Lemma \ref{lm.0405_2B}.
The rest is similar and is omitted.
\end{proof}

%%%%%%%%%%%%%%%%%%%%%%%%%%%%%%%
%%%%%%%%%%%%%%%%%%%%%%%%%%%%%%%%%%%
%\newpage
%%%%%%%%%%%%%%%%%%%%%%%%%%%%%%%%
%%%%%%%%%%%%%%%%%%%%%%%%%%%%%%%%%%%

\subsection{Convergence of $\hat{W}:=R^{\perp} X \reso \nabla {\mathcal I} [\nabla X]$}

In this subsection we prove that 
\[\hat{W}_i^{\ve}=
R_2 X^{\ve} \reso \partial_1 {\mathcal I}[\partial_i X^{\ve} ]
- R_1 X^{\ve} \reso \partial_2 {\mathcal I}[\partial_i X^{\ve} ]
\qquad
\mbox{($i=1,2$).}
\]
is convergent %in $B^s_{\infty\infty}({\bf T}^d)$ 
as $\ve \searrow 0$ if the Besov regularity %$s$
 is smaller than $2\theta -4$.

It is straightforward to see that
$ \partial_j {\mathcal I}[\partial_i X^{\ve} ]$ has the following expression ($i,j =1,2$):
\begin{eqnarray}\label{eq.0407_1}
\partial_j {\mathcal I} [\partial_i X^{\ve} ]_{t,x}
&=&
 \sum\limits_{l \in {\bf Z^2}}
 \chi^{\ve} (l) {\bf e}_{l} (x) (2\pi \sqrt{-1} l^j)
 \\
 && \qquad
\times
 \int_{- \infty}^{\infty}
 du h(t-u, l)
 \int_{- \infty}^{\infty} H_u^{ \partial_i X} (s,l) 
\hat{\xi}_{s} (l) ds,
 \nonumber\
 \end{eqnarray}
where 
$H_u^{ \partial_i X}$ is the $H$-function for $ \partial_i X$ given by (\ref{eq.0406_5}).

Hence, we can easily see from (\ref{eq.0406_3}) that
\begin{eqnarray}\label{eq.0407_2}
\lefteqn{
[R_2 X^{\ve}]_{t,x} \reso
\partial_1 {\mathcal I} [\partial_i X^{\ve} ]_{t,x}
}
\\
&=&
 \sum\limits_{k, l \in {\bf Z^2}}
 \chi^{\ve} (k) \chi^{\ve} (l) 
 \psi_{\circ} (k,l) {\bf e}_{k+l} (x) 
 \int_{- \infty}^{\infty} H_t^{ R_2 X} (s_1,k) 
\hat{\xi}_{s_1} (k) ds_1 
 \nonumber\\
 && 
\times
 (2\pi \sqrt{-1} l^1)
 \int_{- \infty}^{\infty}
 du h(t-u, l)
 \int_{- \infty}^{\infty} H_u^{ \partial_i X} (s_2,l) 
\hat{\xi}_{s_2} (l) ds_2
\nonumber\\
&=&
\Pi_2 ( [R_2 X^{\ve}]_{t,x} \reso
\partial_1 {\mathcal I} [\partial_i X^{\ve} ]_{t,x}) 
+\Pi_0 ([R_2 X^{\ve}]_{t,x} \reso
\partial_1 {\mathcal I} [\partial_i X^{\ve} ]_{t,x}).
 \nonumber\
 \end{eqnarray}

We calculate the second order term 
on the right-hand side of (\ref{eq.0407_2}).
It holds that
\[
 \Pi_2 ( [R_2 X^{\ve}]_{t,x} \reso
\partial_1 {\mathcal I} [\partial_i X^{\ve} ]_{t,x}) 
=
{\mathcal J}_2 
\Bigl(
f_{(t,x)}^{\ve}
\Bigr),
\]
where
\begin{eqnarray}\label{eq.0406_7C}\nonumber
f_{(t,x)}^{\ve} ((s_1, y_1),(s_2, y_2))
&=&
\sum\limits_{k, l \in {\bf Z}^2}
{\bf e}_{k} (x- y_1) {\bf e}_{l} (x- y_2) \chi^{\ve} (k) \chi^{\ve} (l)
H_t ((s_1, k), (s_2, l))
\end{eqnarray}
with
\begin{eqnarray}\label{eq.0406_7}
H_t ((s_1, k), (s_2, l))
=
 \psi_{\circ} (k,l) H_t^{ R_2 X} (s_1,k) (2\pi \sqrt{-1} l^1)
 \int_{- \infty}^{\infty}
 du h(t-u, l)
 H_u^{ \partial_i X} (s_2,l).
 \end{eqnarray}
Recall that
$H^{ R_i X}$ is the $H$-function for $ R_i X$ given by (\ref{eq.0406_3}).

By taking the Fourier transform with respect to the time variables, we have
\begin{eqnarray}\label{eq.0407_3}
Q_t ((\sigma_1, k), (\sigma_2, l))
&=&
e^{- 2\pi \sqrt{-1} (\sigma_1 +\sigma_2) t} 
\psi_{\circ} (k,l)
Q_0^{ R_2 X} (\sigma_1,k) (2\pi \sqrt{-1} l^1)
\\
&&
\qquad
\times
\frac{ Q_0^{ \partial_i X} (\sigma_2,l) }
{ - 2\pi \sqrt{-1} \sigma_2 + (2\pi |l|)^\theta +1 }
\nonumber\\
&=&
e^{- 2\pi \sqrt{-1} (\sigma_1 +\sigma_2) t} \psi_{\circ} (k,l)
\frac{\sqrt{-1} k^2 |k|^{-1} {\bf 1}_{k \neq 0}}
{- 2\pi \sqrt{-1} \sigma_1 + (2\pi |k|)^\theta +1 }
\nonumber\\
&&
\qquad
\times
\frac{ 2\pi \sqrt{-1} l^1 }
{ -2\pi \sqrt{-1} \sigma_2 + (2\pi |l|)^\theta +1 }
\cdot 
\frac{ 2\pi \sqrt{-1} l^i }
{- 2\pi \sqrt{-1} \sigma_2 + (2\pi |l|)^\theta +1 }.
\nonumber
\end{eqnarray}
This shows that 
$ \Pi_2 ( R_2 X^{\ve} \reso
\partial_1 {\mathcal I} [\partial_i X^{\ve} ]) $ has the following good kernel:
\begin{equation}\label{eq:180222-1}
Q_t^{\ve} ((\sigma_1, k), (\sigma_2, l))
=
 \chi^{\ve} (k) \chi^{\ve} (l) Q_t ((\sigma_1, k), (\sigma_2, l)).
\end{equation}
It will turn out that $Q_0$ above satisfies 
the assumption of Lemma \ref{lm.0405_2B} 
and hence it defines a Besov space-valued process 
$\Pi_2 ( R_2 X \reso
\partial_1 {\mathcal I} [\partial_i X ])$.
The computation for 
$ R_1 X^{\ve} \reso
\partial_2 {\mathcal I} [\partial_i X^{\ve} ]$
is essentially the same
and $\Pi_2 ( R_1 X \reso
\partial_2 {\mathcal I} [\partial_i X ])$
is also well defined.

Next we show the zeroth order term vanishes.
Using the contraction rule (\ref{eq:180510-2}) to (\ref{eq.0407_2})
and the fact that $ \psi_{\circ} (k, -k) = \psi_{\circ} (k,k)=1$,
\begin{eqnarray*} %\label{eq.04011_1}
\Pi_0 ([R_2 X^{\ve}]_{t,x} \reso
\partial_1 {\mathcal I} [\partial_i X^{\ve} ]_{t,x})
&=&
\sum\limits_{l \in {\bf Z^2}}
\chi^{\ve} (l)^2
 (2\pi \sqrt{-1} l^1)
 \nonumber\\
 && 
\times \int_{- \infty}^{\infty} du h(t-u, l)
 \int_{- \infty}^{\infty} 
 H_t^{ R_2 X} (s, -l) 
 H_u^{ \partial_i X} (s,l) ds
 \nonumber\\
 &=& 
 \sum\limits_{l \in {\bf Z^2}}
\chi^{\ve} (l)^2 
 (2\pi \sqrt{-1} l^1) \frac{ \sqrt{-1} (-l^2) {\bf 1}_{l \neq 0}}{|l|} (2\pi \sqrt{-1} l^i)
 \nonumber\\
 && 
\times \int_{- \infty}^{\infty} du h(t-u, l)
 \int_{- \infty}^{\infty} ds h(t-s, -l)h(u-s, l)
 \nonumber\\ 
 &=& 
 \sum\limits_{l \in {\bf Z^2}}
\chi^{\ve} (l)^2 
 (2\pi \sqrt{-1} l^1) \frac{ \sqrt{-1} (- l^2) {\bf 1}_{l \neq 0}}{|l|} 
\frac{ 2\pi \sqrt{-1} l^i}
{ 4 \{ (2\pi |l|)^\theta +1\}^{2}}.
 \nonumber\\ 
 \end{eqnarray*}
Hence, 
$\Pi_0 ([R_2 X^{\ve}]_{t,x} \reso
\partial_1 {\mathcal I} [\partial_i X^{\ve} ]_{t,x} 
-[R_1 X^{\ve}]_{t,x} \reso
\partial_2 {\mathcal I} [\partial_i X^{\ve} ]_{t,x}) =0$.
Thanks to this cancellation, 
we do not need any renormalization for $\hat{W}$
 given by (\ref{eq:180430-3}) below. 

%\vspace{7mm}

Now we set 
\begin{equation}\label{eq:180430-3}
\hat{W}_i =\Pi_2 ( R_2 X \reso
\partial_1 {\mathcal I} [\partial_i X ])-
\Pi_2( R_1 X \reso
\partial_2 {\mathcal I} [\partial_i X ])
\end{equation}
for $i=1,2$.

\begin{lemma}\label{lm.0407_b}
Let $\theta \in (3/2, 2]$.
Then, for every $\alpha < 2\theta -4$, $1 < p <\infty$ and $i=1,2$, 
\[
{\mathbf E}
 [\|\hat{W}_i\|_{C_T {\mathcal C}^{ \alpha } }^p 
 ]<\infty
,\qquad
\lim_{\ve \searrow 0} {\mathbf E}
 [\|\hat{W}_i- \hat{W}^{\ve}_i\|_{C_T {\mathcal C}^{ \alpha } }^p 
 ]
=0.
\]
\end{lemma}

\begin{proof}
We prove the lemma for $\theta \in (3/2, 2)$.
The proof for the case $\theta =2$ is just a slight modification.

Let $Q_t$ and 
$Q_t^{\ve}$ be as in \eqref{eq.0407_3} and \eqref{eq:180222-1}, respectively.
As we have explained, they (should) correspond to
$ \Pi_2 ( R_2 X \reso \partial_1 {\mathcal I} [\partial_i X ])$
and its smooth approximation, respectively. 
We let
\[
\gamma=5-\frac{6}{\theta}, \quad \delta=0
\]
to use Lemma \ref{lm.0405_2B}.
(Note that $\gamma>1$ if and only if $\theta >3/2$.)
Since $ \psi_{\circ} (k,l)$ is bounded, we see from Lemma \ref{lm.0405_1} that
\[
|Q_0^{\ve} ((\sigma_1, k), (\sigma_2, l)) |^2 
\lesssim 
|Q_0 ((\sigma_1, k), (\sigma_2, l)) |^2 
\lesssim 
|(\sigma_1, k)|_*^{-4}
|(\sigma_2, l)|_*^{-8 (1 -\frac{1}{\theta})}
\]
and 
\[
\int_{\hat{E}}
 |Q_0 ( (\sigma_1, k), (\tau-\sigma_1, m-k)) |^2\,d(\sigma_1,k)
\lesssim
 |(\tau, m) |_*^{- 2(5 - \frac{6}{\theta}) }.
\]
%
%Using Lemma \ref{lm.0405_2B}, 
%we prove the case $\theta \in (3/2, 2)$.
%
%The case $\theta =2$ can be done essentially in the same way.
\end{proof}

%%%%%%%%%%%%%%%%%

%\newpage
\subsection{Convergence of $W:= R^{\perp} {\mathcal I} [\nabla X] \reso \nabla X$}

In this subsection we prove that 
\[
W_i^{\ve}=
R_2 {\mathcal I} [\partial_i X^{\ve}] \reso \partial_1 X^{\ve}
-
R_1 {\mathcal I} [\partial_i X^{\ve}] \reso \partial_2 X^{\ve}
\qquad
\mbox{($i=1,2$).}
\]
is convergent %in $B^s_{\infty\infty}({\bf T}^d)$ 
as $\ve \searrow 0$ if the Besov regularity 
%$s$ 
is smaller than $2\theta -4$.
For brevity we write $[R_2 {\mathcal I} [ \partial_i X^{\ve}] \reso
\partial_1 X^{\ve}]_{t,x}
$ for $[R_2 {\mathcal I} [ \partial_i X^{\ve}] ]_{t,x} \reso
[\partial_1 X^{\ve}]_{t,x}$.

It is straightforward to see that
$ R_j {\mathcal I}[\partial_i X^{\ve} ]$ has the following expression ($i,j =1,2$):
\begin{eqnarray}\label{eq.0601_1}
R_j {\mathcal I} [\partial_i X^{\ve} ]_{t,x}
&=&
 \sum\limits_{k \in {\bf Z^2}}
 \chi^{\ve} (k) {\bf e}_{k} (x) \frac{2\pi \sqrt{-1} k^j}{2 \pi |k|} {\bf 1}_{k \neq 0}
 \\
 && \qquad
\times
 \int_{- \infty}^{\infty}
 du h(t-u, k)
 \int_{- \infty}^{\infty} H_u^{ \partial_i X} (s,k) 
\hat{\xi}_{s} (k) ds,
 \nonumber\
 \end{eqnarray}
where 
$H^{ \partial_i X}$ is the $H$-function for $ \partial_i X$
explicitly given by (\ref{eq.0406_5}).

Hence, we can easily see from (\ref{eq.0406_5}) that
\begin{eqnarray}\label{eq.0601_2}
\lefteqn{
[ R_2 {\mathcal I} [ \partial_i X^{\ve}] \reso
\partial_1 X^{\ve} ]_{t,x}
}
\\
&=&
 \sum\limits_{k, l \in {\bf Z^2}}
 \chi^{\ve} (k) \chi^{\ve} (l) \psi_{\circ} (k,l)
 {\bf e}_{k+l} (x) 
 \nonumber\\
 && 
\times
 \frac{2\pi \sqrt{-1} k^2}{2 \pi |k|} {\bf 1}_{k \neq 0}
 \int_{- \infty}^{\infty}
 du h(t-u, k)
 \int_{- \infty}^{\infty} H_u^{ \partial_i X} (s_1,k) 
\hat{\xi}_{s_1} (k) ds_1
 \nonumber\\
 && 
\times
 \int_{- \infty}^{\infty} H_t^{ \partial_1 X} (s_2,l) 
\hat{\xi}_{s_2} (l) ds_2
\nonumber\\
&=&
\Pi_2 ( [ R_2 {\mathcal I} [ \partial_i X^{\ve}] \reso
\partial_1 X^{\ve} ]_{t,x})
+
\Pi_0 ( [ R_2 {\mathcal I} [ \partial_i X^{\ve}] \reso
\partial_1 X^{\ve} ]_{t,x}).
 \nonumber\
 \end{eqnarray}

Let us calculate the second order term 
on the right-hand side of (\ref{eq.0601_2}).
It holds that
\[
\Pi_2 ( [ R_2 I [ \partial_i X^{\ve}] \reso
\partial_1 X^{\ve} ]_{t,x})
=
{\mathcal J}_2 
\Bigl(
f_{(t,x)}^{\ve}
\Bigr),
\]
where
\begin{eqnarray}\label{eq.0406_7D}\nonumber
f_{(t,x)}^{\ve} ((s_1,y_1),(s_2,y_2))
&=&
\sum\limits_{k, l \in {\bf Z}^2}
{\bf e}_{k} (x- y_1) {\bf e}_{l} (x- y_2) \chi^{\ve} (k) \chi^{\ve} (l)
H_t ((s_1, k), (s_2, l))
\end{eqnarray}
with
\begin{eqnarray}\label{eq.0406_7E}
H_t ((s_1, k), (s_2, l))
=
 \psi_{\circ} (k,l)
 \frac{2\pi \sqrt{-1} k^2}{2 \pi |k|} {\bf 1}_{k \neq 0}
 \int_{- \infty}^{\infty}
 du h(t-u, k) H_u^{ \partial_i X} (s_1,k) 
 \cdot 
 H_t^{ \partial_1 X} (s_2,l).
 \nonumber
 \end{eqnarray}

By taking the Fourier transform with respect to the time variables, we have
\begin{eqnarray}\label{eq.0602_0}
Q_t ((\sigma_1, k), (\sigma_2, l))
&=&
e^{- 2\pi \sqrt{-1} (\sigma_1 +\sigma_2) t}
 \psi_{\circ} (k,l)
 \frac{2\pi \sqrt{-1} k^2}{2 \pi |k|} {\bf 1}_{k \neq 0}
\\
&&
\qquad
\times
\frac{ Q_0^{ \partial_i X} (\sigma_1, k) }
{ -2\pi \sqrt{-1} \sigma_1 + (2\pi |k|)^\theta +1 }
Q_0^{ \partial_1 X} (\sigma_2, l) 
\nonumber\\
&=&
e^{- 2\pi \sqrt{-1} (\sigma_1 +\sigma_2) t}
 \psi_{\circ} (k,l) \frac{2\pi \sqrt{-1} k^2}{2 \pi |k|} {\bf 1}_{k \neq 0}
\nonumber
 \\
&&
\quad
\times
\frac{ 2\pi \sqrt{-1} k^i }
{ \{ -2\pi \sqrt{-1} \sigma_1 + (2\pi |k|)^\theta +1 \}^2 }
\cdot
\frac{ 2\pi \sqrt{-1} l^1 }
{ - 2\pi \sqrt{-1} \sigma_2 + (2\pi |l|)^\theta +1 }.
\nonumber
\end{eqnarray}
Therefore, 
$\Pi_2 ( R_2 {\mathcal I} [ \partial_i X^{\ve}] \reso
\partial_1 X^{\ve})$ has the following good kernel:
\begin{equation}\label{eq:180222-2}
Q_t^{\ve} ((\sigma_1, k), (\sigma_2, l))
=
 \chi^{\ve} (k) \chi^{\ve} (l) Q_t ((\sigma_1, k), (\sigma_2, l)).
\end{equation}
It will turn out that $Q_0$ above satisfies 
the assumption of Lemma \ref{lm.0405_2B} 
and hence it defines a Besov space-valued process 
$\Pi_2(R_2 {\mathcal I} [ \partial_i X] \reso
\partial_1 X)$.
The computation for 
$
R_1 {\mathcal I} [ \partial_i X^{\ve}] \reso
\partial_2 X^{\ve}
$
is essentially the same and 
$\Pi_2(R_1 {\mathcal I} [ \partial_i X] \reso
\partial_2 X)$
is also well defined.

Next we show the zeroth order term 
on the right-hand side of (\ref{eq.0601_2}) vanishes.
Using $\psi_{\circ} (k, -k) = \psi_{\circ} (k, k)=1$
and applying the contraction rule
${\mathbf E} [ \hat{\xi}_{s_1} (k) \hat{\xi}_{s_2} (l) ] 
= \delta (s_1- s_2) \delta_{k+l, 0}$
to (\ref{eq.0601_2}),

\begin{eqnarray*}
\Pi_0 ( [ R_2 {\mathcal I} [ \partial_i X^{\ve}] \reso
\partial_1 X^{\ve}]_{t,x})
&=&
\sum\limits_{k \in {\bf Z^2}}
\chi^{\ve} (k)^2
 \frac{2\pi \sqrt{-1} k^2}{2 \pi |k|} {\bf 1}_{k \neq 0}
 \nonumber\\
 && 
\times \int_{- \infty}^{\infty} du h(t-u, k)
 \int_{- \infty}^{\infty} 
 H_t^{ \partial_i X} (s, k) 
 H_u^{ \partial_1 X} (s, -k) ds
 \nonumber\\
 &=& 
 \sum\limits_{k \in {\bf Z^2}}
\chi^{\ve} (k)^2 
 \frac{ \sqrt{-1} k^2 {\bf 1}_{k \neq 0}}{|k|} 
 \frac{
 (2\pi \sqrt{-1} k^i) (-2\pi \sqrt{-1} k^1) 
 }
 {
 [ 2 \{ (2 \pi |k|)^\theta +1 \} ]^{2}
 }.
 \end{eqnarray*}
Again, we have
$\Pi_0 ( [R_2 {\mathcal I} [ \partial_i X^{\ve}] \reso
\partial_1 X^{\ve}]_{t,x}
-
[ R_1 {\mathcal I} [ \partial_i X^{\ve}] \reso
\partial_2 X^{\ve}]_{t,x}
) =0$.
Due to this cancellation, 
we do not need any renormalization for $W$
 given by (\ref{eq:180430-4}) below.

We set 
\begin{equation}\label{eq:180430-4}
W_i= \Pi_2(R_2 {\mathcal I} [ \partial_i X] \reso
\partial_1 X)- \Pi_2(R_1 {\mathcal I} [ \partial_i X] \reso
\partial_2 X)
\end{equation} 
for $i=1,2$.

\begin{lemma}\label{lm.0602_a}
Let $\theta \in (3/2, 2]$.
Then, for every $\alpha < 2\theta -4$, $1 < p <\infty$ and $i=1,2$, 
\[
{\mathbf E}
 [\|W_i\|_{C_T {\mathcal C}^{ \alpha } }^p 
 ]<\infty
,\qquad
\lim_{\ve \searrow 0} {\mathbf E}
 [\|W_i- W^{\ve}_i\|_{C_T {\mathcal C}^{ \alpha } }^p 
 ]
=0.
\]
\end{lemma}

\begin{proof}
We let
\[
\gamma=4-\frac{4}{\theta}, \quad \delta=2.
\]
Let $Q_t$ and 
$Q_t^{\ve}$ be as in \eqref{eq.0602_0} and \eqref{eq:180222-2}, respectively.
%As we have seen, the latter
They (should) correspond to
$\Pi_2 ( [ R_2 {\mathcal I} [ \partial_i X] \reso
\partial_1 X ])$ 
and its smooth approximation, respectively.
Since we have $\psi_{\circ} (k,l)^2 |k|^2 |l|^2 \lesssim |k|_*^4$,
we see that 
\begin{align*}
|Q_0^{\ve} ((\sigma_1, k), (\sigma_2, l)) |^2 
&\lesssim 
|Q_0 ((\sigma_1, k), (\sigma_2, l)) |^2 
\\
&\lesssim 
\psi_{\circ} (k,l)^2 |k|^2 |l|^2
|(\sigma_1, k)|_*^{-8}
|(\sigma_2, l)|_*^{-4}\\
&\le 
|(\sigma_1, k)|_*^{ -8 (1 -\frac{1}{\theta})}
|(\sigma_2, l)|_*^{-4}\\
&=
|(\sigma_1, k)|_*^{ -2\gamma}
|(\sigma_2, l)|_*^{-2\delta}.
\end{align*}
The rest is essentially the same as in the proof of Lemma \ref{lm.0407_b}.
\end{proof}

%%%%%%%%%%%%%%%%%%%%%%%%%%%%%%%
%%%%%%%%%%%%%%%%%%%%%%%%%%%%%%%%%%%
%\newpage
%%%%%%%%%%%%%%%%%%%%%%%%%%%%%%%%
%%%%%%%%%%%%%%%%%%%%%%%%%%%%%%%%%

\subsection{Convergence of $Z:=R^{\perp} Y\reso \nabla X$}

In this subsection we prove that 
\begin{eqnarray}\label{def.0612}
Z^{\ve} := R^{\perp} Y^{\ve} \reso \nabla X^{\ve}
&=&
(R_2 Y^{\ve}) \reso (\partial_1 X^{\ve})- (R_1 Y^{\ve}) \reso (\partial_2 X^{\ve})
\\
&=&
R_2 {\mathcal I}[R_2 X^{\ve} \cdot \partial_1 X^{\ve} - R_1 X^{\ve} \cdot \partial_2 X^{\ve}]
\reso (\partial_1 X^{\ve})
\nonumber\\
&& \qquad
-
R_1 {\mathcal I}[R_2 X^{\ve} \cdot \partial_1 X^{\ve} - R_1 X^{\ve} \cdot \partial_2 X^{\ve}]
\reso (\partial_2 X^{\ve})
\nonumber
\end{eqnarray}
is convergent %in $B^s_{\infty\infty}({\bf T}^d)$ 
as $\ve \searrow 0$ if the Besov regularity %$s$ 
is smaller than $\frac52 \theta - 5$.
We do not need any renormalization.

Let us fix notation.
For $j \in \{1,2\}$, $j'=3-j$ stands for the other element of $\{1,2\}$. 
Hence, if $j =1$ (resp. $j =2$), then $j' =2$ (resp. $j' =1$).

Let $i, j \in \{1,2\}$.
By straightforward computation, we have
\begin{eqnarray*}%\label{eq.0605_1}
R_i {\mathcal I} [R_j X^{\ve} \cdot \partial_{j'} X^{\ve} ]_{t,x}
&=&
\sum\limits_{k, l \in {\bf Z^2}}
 \chi^{\ve} (k) \chi^{\ve} (l) {\bf e}_{k+l} (x) 
 \frac{2\pi \sqrt{-1} (k^i + l^i)}{2 \pi |k+l|} {\bf 1}_{k +l \neq 0}
 \int_{- \infty}^{\infty} du h(t-u, k+l) \\
 &&
 \times 
 \int_{- \infty}^{\infty} H_u^{ R_j X} (s_1,k) 
\hat{\xi}_{s_1} (k) ds_1
 \int_{- \infty}^{\infty} H_u^{ \partial_{j'} X} (s_2,l) 
\hat{\xi}_{s_2} (l) ds_2
 \end{eqnarray*}
and
\[
\partial_{i'} X^{\ve}_{t,x}
=
\sum\limits_{m\in {\bf Z^2}}
\chi^{\ve} (m) {\bf e}_{m} (x) 
 \int_{- \infty}^{\infty} H_t^{ \partial_{i'} X} (s_3, m) 
\hat{\xi}_{s_3} (m) ds_3.
\]
Here, $H_u^{ R_j X}$ and $H_u^{ \partial_{j'} X}$
are given by \eqref{eq.0406_3} and \eqref{eq.0406_5}, respectively.
For simplicity we set $Z^{i,j;\ve}_{t,x} :=
R_i {\mathcal I} [R_j X^{\ve} \cdot \partial_{j'} X^{\ve} ]_{t,x} 
\reso \partial_{i'} X^{\ve}_{t,x}$.
Then, we have
\begin{eqnarray}\label{eq.0605_1}
Z^{i,j;\ve}_{t,x} 
&=&
\sum\limits_{k, l, m \in {\bf Z^2}}
 \chi^{\ve} (k) \chi^{\ve} (l) \chi^{\ve} (m) {\bf e}_{k+l+m} (x) \psi_{\circ} (k+l,m)
 \frac{2\pi \sqrt{-1} (k^i + l^i)}{2 \pi |k+l|} {\bf 1}_{k +l \neq 0}
 \\
 &&
 \times
 \int_{- \infty}^{\infty} du h(t-u, k+l) 
 \int_{- \infty}^{\infty} H_u^{ R_j X} (s_1,k) 
\hat{\xi}_{s_1} (k) ds_1
 \nonumber\\
 & &\times
 \int_{- \infty}^{\infty} H_u^{ \partial_{j'} X} (s_2,l) 
\hat{\xi}_{s_2} (l) ds_2 
 \int_{- \infty}^{\infty} H_t^{ \partial_{i'} X} (s_3, m) 
\hat{\xi}_{s_3} (m) ds_3
\nonumber\\
 &=& \Pi_3 ( Z^{i,j;\ve}_{t,x}) + \Pi_1 ( Z^{i,j;\ve}_{t,x}).
 \nonumber\
 \end{eqnarray}

Let us calculate the third order term 
on the right-hand side of (\ref{eq.0605_1}).
Note that the superindex $(i,j)$ will be omitted in many places
when no confusion seems likely.
It is easy to see that
\[
\Pi_3 ( Z^{i,j;\ve}_{t,x}) 
=
{\mathcal J}_3
\Bigl(
f_{(t,x)}^{\ve}
\Bigr),
\]
where
\begin{eqnarray}\nonumber
\lefteqn{
f_{(t,x)}^{\ve} ((s_1,y_1),(s_2,y_2),(s_3,y_3))
}\\
\nonumber
&=&
\sum\limits_{k, l, m \in {\bf Z}^2}
{\bf e}_{k} (x- y_1) {\bf e}_{l} (x- y_2) {\bf e}_{m} (x- y_3)
\chi^{\ve} (k) \chi^{\ve} (l) \chi^{\ve} (m)
H_t ((s_1, k), (s_2, l), (s_3, m))
\end{eqnarray}
with
\begin{eqnarray}\nonumber
H_t ((s_1, k), (s_2, l), (s_3,m))
&=&
 \psi_{\circ} (k+l, m) \frac{2\pi \sqrt{-1} (k^i + l^i)}{2 \pi |k+l|} {\bf 1}_{k +l \neq 0} 
\\
&& \times
 \int_{- \infty}^{\infty}
 du h(t-u, k+l) 
 H_u^{ R_j X} (s_1,k) H_u^{ \partial_{j'} X} (s_2,l) H_t^{ \partial_{i'} X} (s_3, m). \nonumber
 \end{eqnarray}
%Note that the superindex $(i,j)$ is omitted.

By taking the Fourier transform with respect to the time variables, we have
%\[
%Q_t ((\sigma_1, k), (\sigma_2, l), (\sigma_3, m))
%=
%Q_t ((\sigma_1, k), (\sigma_2, l), (\sigma_3, m))
%\]
%with
\begin{eqnarray}\label{eq.0602_1}
\lefteqn{
Q_t ((\sigma_1, k), (\sigma_2, l), (\sigma_3, m))
}
\\
&=&
e^{- 2\pi \sqrt{-1} (\sigma_1 +\sigma_2+ \sigma_3) t}
 \psi_{\circ} (k+l, m)\frac{2\pi \sqrt{-1} (k^i + l^i)}{2 \pi |k+l|} {\bf 1}_{k +l \neq 0} 
 \nonumber\\
 && \times
 \frac{ Q_0^{ R_j X} (\sigma_1, k) Q_0^{ \partial_{j'} X} (\sigma_2, l) }
{- 2\pi \sqrt{-1} (\sigma_1 + \sigma_2) + (2\pi |k+l|)^\theta +1 }
Q_0^{ \partial_{i'} X} (\sigma_3, m).
\nonumber
\end{eqnarray}
Therefore, 
$ \Pi_3 ( Z^{i,j;\ve}_{t,x})$ has a good kernel whose $Q$-function is given by
\[
Q_t^{\ve} ((\sigma_1, k), (\sigma_2, l), (\sigma_3, m))
=\chi^{\ve} (k)\chi^{\ve} (l)\chi^{\ve} (m)
Q_t ((\sigma_1, k), (\sigma_2, l), (\sigma_3, m)).
\]

As we will see below, 
$Q_0$ above satisfies 
the assumption of Lemma \ref{lm.0405_2B} 
and hence it defines a Besov space-valued process. 
We denote it by $\Pi_3 ( Z^{i,j})$.

%\vspace{7mm}

\begin{lemma}\label{lm.0605_a}
Let the notation be as above and $\theta \in (8/5, 2]$. 
Then, for every $\alpha < \frac52 \theta -5$, $1 < p <\infty$ and $i,j =1, 2$, 
\begin{align*}
{\mathbf E}
 [
 \|\Pi_3 ( Z^{i,j;\ve}) \|_{C_T {\mathcal C}^{ \alpha } }^p 
 ]<\infty
,\qquad 
\lim_{\ve \searrow 0}
{\mathbf E}[
 \|
 \Pi_3 ( Z^{i,j;\ve})
 -
\Pi_3 ( Z^{i,j}) 
 \|_{C_T {\mathcal C}^{ \alpha } }^p 
 ]
=0.
\end{align*}
\end{lemma}

\begin{proof}
Using Lemma \ref{lm.0405_2} with $\gamma = 6 -\frac{8}{\theta}$ and $\delta=0$,
we prove the lemma.
Note that $\gamma >1$ if and only if $\theta > 8/5$.
For simplicity, we prove the case $\theta \in (8/5, 2)$.
(The case $\theta =2$ can be shown essentially in the same way.)
It is easy to see from \eqref{eq.0602_1} that 
\begin{eqnarray*}
|Q_0((\sigma_1, k), (\sigma_2, l), (\sigma_3, m)) |^2 
&\lesssim&
\frac{ \psi_{\circ} (k+l, m)^2}{ |(\sigma_1 + \sigma_2, k+l)|_*^4}
\frac{1}{|(\sigma_1 , k)|_*^4}
\frac{|l|^2}{|(\sigma_2 , l)|_*^4}\frac{|m|^2}{|(\sigma_3 , m)|_*^4}.
\end{eqnarray*}
First, take a sum of this quantity 
over $(\sigma_1 , k)$ and $(\sigma_2 , l)$ such that
$(\sigma_1 + \sigma_2, k+l) =(\hat\tau, \hat{n})$ for a given $(\hat\tau, \hat{n})$.
By Lemma \ref{lm.0405_1} and the same argument as in the proof of Lemma \ref{lm.0406_a}, 
it is dominated by a constant multiple of 
\[
\frac{ \psi_{\circ} (\hat{n}, m)^2}{ |(\hat\tau, \hat{n})|_*^{2 (5 -\frac{4}{\theta})} }
\frac{|m|^2}{|(\sigma_3 , m)|_*^4}
\lesssim 
\frac{ \psi_{\circ} (\hat{n}, m)^2}{ |(\hat\tau, \hat{n})|_*^{ 2 (5 -\frac{4}{\theta})} }
\frac{|\hat{n}|^2}{|(\sigma_3 , m)|_*^4}
\lesssim
\frac{1}{ |(\hat\tau, \hat{n})|_*^{ 2 (5 -\frac{6}{\theta}) } }
\frac{1}{|(\sigma_3 , m)|_*^4}.
\]
Note that, since $\theta >4/3$, we can use Lemma \ref{lm.0405_1} again.
The sum over 
$( \hat\tau, \hat{n})$ and $(\sigma_3 , m)$ such that
$(\hat\tau + \sigma_3, \hat{n} +m) =(\tau, n)$ for a given $(\tau, n)$
is dominated by a constant multiple of $|(\tau , n)|_*^{-2 ( 6 -\frac{8}{\theta})}$.

Summing up, we have shown that
\[
\int_{\hat{E}^2}
 |Q_0 ( (\sigma_1, k), (\sigma_2, l), (\tau-\sigma_1-\sigma_2, n-k-l)) |^2
\,d\sigma_1\,d\sigma_2\,dk\,dl
\lesssim 
 |(\tau , n)|_*^{-2 ( 6 -\frac{8}{\theta}) }. 
\]
\end{proof}

%\vspace{7mm}

Next, we calculate the first order term 
on the right-hand side of (\ref{eq.0605_1}).

\begin{lemma}\label{lm.0605_b}
Let the notation be as above and $\theta \in (8/5, 2]$. 
Then, for every $\alpha < \frac52 \theta -5$ and $1 < p <\infty$, 
there exists 
$\Pi_1 (
 Z^{2,2} - Z^{2,1} - Z^{1,2} + Z^{1,1}) \in L^p (\Omega, {\mathbf P} ;C_T {\mathcal C}^{ \alpha}) 
 $
 which is independent of $\alpha, p, \chi$
 such that
\begin{align*}
\lim_{\ve \searrow 0}
{\mathbf E}
 [
 \|
 \Pi_1 ( Z^{2,2;\ve} - Z^{2,1;\ve} - Z^{1,2;\ve} + Z^{1,1;\ve})
 -
\Pi_1 (
 Z^{2,2} - Z^{2,1} - Z^{1,2} + Z^{1,1}
) 
 \|_{C_T {\mathcal C}^{ \alpha} }^p 
 ]
=0.
\end{align*}
\end{lemma}

\begin{proof}
First we calculate $\Pi_1 ( Z^{i,j;\ve}_{t,x})$.
Applying the contraction rule
\begin{eqnarray}\label{eq.0613_z}
\Pi_1 [ \hat{\xi}_{s_1} (k) \hat{\xi}_{s_2} (l) \hat{\xi}_{s_3} (m) ] 
&=& \delta (s_1- s_2) \delta_{k+l, 0} \hat{\xi}_{s_3} (m)
\\
&&
+ \delta (s_1- s_3) \delta_{k+m, 0} \hat{\xi}_{s_2} (l)
+ \delta (s_2 - s_3) \delta_{l+m, 0} \hat{\xi}_{s_1} (k)
\nonumber
\end{eqnarray}
to (\ref{eq.0605_1}), we have 
\begin{eqnarray}
\Pi_1 ( Z^{i,j;\ve}_{t,x})
=:
A^{i,j;\ve}_{t,x} (1)+ A^{i,j;\ve}_{t,x} (2)+A^{i,j;\ve}_{t,x} (3),
\nonumber
\end{eqnarray}
where $A^{i,j;\ve}_{t,x} (\nu)$, $\nu=1,2,3$, are given as follows:
\begin{eqnarray}
A^{ i,j;\ve}_{t,x} (1)
&=&
\sum\limits_{m \in {\bf Z^2}}
 \chi^{\ve} (m) {\bf e}_{m} (x) \psi_{\circ} (0,m)
 \int_{- \infty}^{\infty} H_t^{ \partial_{i'} X} (s_3, m) 
\hat{\xi}_{s_3} (m) ds_3
\nonumber \\
 &&
 \times
 \sum\limits_{k \in {\bf Z^2}}
 \chi^{\ve} (k)^2 \int_{- \infty}^{\infty} du h(t-u, 0) 
 \int_{- \infty}^{\infty} H_u^{ R_j X} (s_1,k) 
 H_u^{ \partial_{j'} X} (s_1, -k) ds_1. 
\nonumber
\end{eqnarray}
By \eqref{eq.0606_1},
the second factor (i.e. the sum over $k$) above equals
$\Pi_0 ({\mathcal I}[R_j X^{\ve} \cdot \partial_{j'} X^{\ve}]_{t,x})$. 
Therefore, $A^{ i,2;\ve}_{t,x} (1)- A^{ i,1;\ve}_{t,x} (1) =0$.

Similarly, we have
\begin{eqnarray}
A^{ i,j;\ve}_{t,x} (2)
&=&
\sum\limits_{k, l \in {\bf Z^2}}
 \chi^{\ve} (k)^2 \chi^{\ve} (l) {\bf e}_{l} (x) \psi_{\circ} (k+l, -k)
 \frac{2\pi \sqrt{-1} (k^i + l^i)}{2 \pi |k+l|} {\bf 1}_{k +l \neq 0}
 \label{eq.0606_02}\\
 &&
 \times
 \int_{- \infty}^{\infty} du h(t-u, k+l) 
 \int_{- \infty}^{\infty} H_u^{ R_j X} (s_1,k) H_t^{ \partial_{i'} X} (s_1, -k) ds_1
 \nonumber\\
 & &\times
 \int_{- \infty}^{\infty} H_u^{ \partial_{j'} X} (s_2,l) 
\hat{\xi}_{s_2} (l) ds_2
 \nonumber\\
 &=&
 \sum\limits_{l \in {\bf Z^2}}\chi^{\ve} (l) {\bf e}_{l} (x)
 \Bigl[
 \sum\limits_{k \in {\bf Z^2}}
 \chi^{\ve} (k)^2 \psi_{\circ} (k+l, -k)
 \frac{2\pi \sqrt{-1} (k^i + l^i)}{2 \pi |k+l|} {\bf 1}_{k +l \neq 0}
 \nonumber\\
 && 
 \quad \times 
 \frac{ - (2\pi \sqrt{-1} k^j)(2\pi \sqrt{-1} k^{i'})}{2 \pi |k|} {\bf 1}_{k \neq 0}
 \frac{1}{2 \{ ( 2 \pi |k|)^\theta +1\}}
 \nonumber\\
 && 
 \quad \times 
 \int_{- \infty}^{t} du e^{- (t-u) \{ ( 2 \pi |k+l|)^\theta+ ( 2 \pi |k|)^\theta+2\}}
 \int_{- \infty}^{\infty} H_u^{ \partial_{j'} X} (s_2,l) 
\hat{\xi}_{s_2} (l) ds_2 \Bigr]
 \nonumber
\end{eqnarray}
and 
\begin{eqnarray}
A^{i,j;\ve}_{t,x} (3)
&=&
\sum\limits_{k, l \in {\bf Z^2}}
 \chi^{\ve} (k) \chi^{\ve} (l)^2 {\bf e}_{k} (x) \psi_{\circ} (k+l, -l)
 \frac{2\pi \sqrt{-1} (k^i + l^i)}{2 \pi |k+l|} {\bf 1}_{k +l \neq 0}
 \label{eq.0606_03}\\
 &&
 \times
 \int_{- \infty}^{\infty} du h(t-u, k+l) 
 \int_{- \infty}^{\infty} H_u^{ R_j X} (s_1,k) 
\hat{\xi}_{s_1} (k) ds_1
 \nonumber\\
 & &\times
 \int_{- \infty}^{\infty} H_u^{ \partial_{j'} X} (s_2,l) 
 H_t^{ \partial_{i'} X} (s_2, -l) ds_2
 \nonumber\\
 &=& 
 \sum\limits_{k \in {\bf Z^2}}\chi^{\ve} (k) {\bf e}_{k} (x)
 \Bigl[
 \sum\limits_{l \in {\bf Z^2}}
 \chi^{\ve} (l)^2 \psi_{\circ} (k+l, -l)
 \frac{2\pi \sqrt{-1} (k^i + l^i)}{2 \pi |k+l|} {\bf 1}_{k +l \neq 0}
 \nonumber\\
 && 
 \quad \times 
 \frac{ - (2\pi \sqrt{-1} l^{j'})(2\pi \sqrt{-1} l^{i'})}{ 2 \{ (2\pi |l|)^\theta +1\} } 
 \nonumber\\
 && 
 \quad \times 
 \int_{- \infty}^{t} du e^{- (t-u) \{ ( 2 \pi |k+l|)^\theta+ ( 2 \pi |l|)^\theta +2 \} }
 \int_{- \infty}^{\infty} H_u^{ R_{j} X} (s_1,k) 
\hat{\xi}_{s_1} (k) ds_1 \Bigr].
\nonumber
\end{eqnarray}
In what follows, the superscript ``$i,j;$" in $A^{i,j;\ve}_{t,x} (\nu)$
(and in other corresponding quantities)
will be supressed 
when no confusion seems likely.

Now we calculate the $H$-function and the $Q$-function
associated with $A^{\ve}_{t,x} (2)$.
Let $(s_2,l) \in \hat{E}$.
From \eqref{eq.0606_02}, we see that
\begin{eqnarray*}
H_t^{\ve} (s_2, l) 
&=&
 \chi^{\ve} (l) \sum\limits_{k \in {\bf Z^2}}
 \chi^{\ve} (k)^2 \psi_{\circ} (k+l, -k)
 \frac{2\pi \sqrt{-1} (k^i + l^i)}{2 \pi |k+l|} {\bf 1}_{k +l \neq 0}
 \frac{ - (2\pi \sqrt{-1} k^j)(2\pi \sqrt{-1} k^{i'})}
 {4 \pi |k| \{ (2 \pi |k|)^\theta +1\} } {\bf 1}_{k \neq 0}
 \nonumber\\
 && 
 \quad \times 
 \int_{- \infty}^{t} du e^{- (t-u) \{ ( 2 \pi |k+l|)^\theta+ ( 2 \pi |k|)^\theta+2 \} }
 H_u^{ \partial_{j'} X} (s_2,l)
\end{eqnarray*}
and 
\begin{eqnarray*}
Q_t^{\ve} (\sigma_2, l) 
&=& 
e^{- 2\pi \sqrt{-1}\sigma_2 t}
\chi^{\ve} (l) 
 \sum\limits_{k \in {\bf Z^2}}
 \chi^{\ve} (k)^2 \psi_{\circ} (k+l, -k)
 \frac{2\pi \sqrt{-1} (k^i + l^i)}{2 \pi |k+l|} {\bf 1}_{k +l \neq 0}
 \nonumber\\
 && 
 \quad \times \frac{ - (2\pi \sqrt{-1} k^j)(2\pi \sqrt{-1} k^{i'})}
 {4 \pi |k| \{ ( 2 \pi |k|)^\theta +1 \} } {\bf 1}_{k \neq 0}
 \nonumber\\
&& 
 \quad \times 
 \frac{1}{-2\pi \sqrt{-1} \sigma_2 +( 2 \pi |k+l|)^\theta +( 2 \pi |k|)^\theta +2
 }
 \cdot
 \frac{2\pi \sqrt{-1} l^{j'}}{-2\pi \sqrt{-1} \sigma_2 +( 2 \pi |l|)^\theta +1 }, 
 \end{eqnarray*}
which shows $A^{\ve}_{t,x} (2)$ has a good kernel 
whose $Q$-function is given by $Q^{\ve}_t$.
We define $Q_t (\sigma_2, l) $ by formally setting $\ve =0$
(or replacing both $\chi^{\ve} (k)$ and $\chi^{\ve} (l)$ by $1$)
 in the above 
expression of $Q_t^{\ve} (\sigma_2, l)$.

Assume that $\theta \in (8/5,2)$. 
(The case $\theta =2$ can be done essentially in the same way.)
Take sufficiently small $a>0$ and set $\gamma := 6 - \frac{8 +2a}{\theta}$.
Then, under the condition on $\theta$, we may assume $0<\gamma <2$
and $0 <2\theta -3-a <1$.
Recall that 
 $\psi_{\circ} (k+l, k) \lesssim |k+l|_*^{ 2\theta -3-a} |l|_*^{-( 2\theta -3-a)}$
by Lemma \ref{lm.0405_3}. 
We can dominate $|Q_0 (\sigma_2, l)|^2$ as follows:
\begin{eqnarray}\label{eq.0606_04}
|Q_0 (\sigma_2, l) |^2
&\lesssim& 
\Bigl\{
\sum\limits_{k \in {\bf Z^2}}
\psi_{\circ} (k+l, -k)
\frac{1}{|k|_*^{\theta -1} }
\frac{|l|}{ |k+l|^\theta + |k|^\theta +1}
\frac{1}{ |(\sigma_2, l)|_*^{2- \gamma} }
\Bigr\}^2 
|(\sigma_2, l)|_*^{-2 \gamma}
\\
 &\lesssim& 
\Bigl\{
\sum\limits_{k \in {\bf Z^2}}
\frac{ |k+l|_*^{ 2\theta -3-a} }{ |l|_*^{ 2\theta -3-a}}
\frac{1}{|k|_*^{\theta -1} }
\frac{ |l| }{ |k+l|^\theta + |k|^\theta +1}
\frac{1}{ |l|_*^{ -2\theta +4 +a}}
\Bigr\}^2 |(\sigma_2, l)|_*^{-2 \gamma}
 \nonumber\\
 &\lesssim& 
\Bigl\{
\sum\limits_{k \in {\bf Z^2}}
|k|_*^{-2- a}
\Bigr\}^2 
|(\sigma_2, l)|_*^{-2\gamma}
\lesssim |(\sigma_2, l)|_*^{-2\gamma}.
\nonumber
 \end{eqnarray}
Here, the implicit constant may depend on the constant $a$. 
Since $6 - \frac{8}{\theta} >1$
and $a >0$ is arbitrarily small, 
we can apply Lemma \ref{lm.0405_2} with $\gamma = 6 - \frac{8 +2a}{\theta}$ and $\delta =0$ 
to
$A^{\ve}_{t,x} (2)$ to obtain the desired estimate.
We can also estimate $|Q_0 (\sigma_2, l) - Q^{\ve}_0 (\sigma_2, l) |^2$
in the same way using that 
$|1 - \chi^{\ve} (k)| \lesssim\ve^b |k|^b$ for every sufficiently small $b >0$.

The Besov space-valued process associated with this $Q_0$
will be denoted by $A^{i,j} (2)$.

Next we calculate the $H$-function and the $Q$-function
associated with $A^{\ve}_{t,x} (3)$.
Let $(s_1,k) \in E$.
In a similar way as above, we see that
\begin{eqnarray*}
H_t^{\ve} (s_1, k) &=& 
\chi^{\ve} (k)
\sum\limits_{l \in {\bf Z^2}}
 \chi^{\ve} (l)^2 \psi_{\circ} (k+l, -l)
 \frac{2\pi \sqrt{-1} (k^i + l^i)}{2 \pi |k+l|} {\bf 1}_{k +l \neq 0}
 \frac{ - (2\pi \sqrt{-1} l^{j'})(2\pi \sqrt{-1} l^{i'})}
 {2 \{ ( 2 \pi |l|)^\theta +1 \} } 
 \nonumber\\
 && 
 \quad \times 
 \int_{- \infty}^{t} du e^{- (t-u) \{( 2 \pi |k+l|)^\theta+ ( 2 \pi |l|)^\theta+2 \}}
 H_u^{ R_{j} X} (s_1,k)
\end{eqnarray*}
and 
\begin{eqnarray*}
Q_t^{\ve}(\sigma_1, k) 
&=& 
e^{- 2\pi \sqrt{-1}\sigma_1 t} \chi^{\ve} (k)
\sum\limits_{l \in {\bf Z^2}}
 \chi^{\ve} (l)^2 \psi_{\circ} (k+l, -l)
 \frac{2\pi \sqrt{-1} (k^i + l^i)}{2 \pi |k+l|} {\bf 1}_{k +l \neq 0}
 \nonumber\\
 && 
 \quad \times 
 \frac{C(l)}{-2\pi \sqrt{-1} \sigma_1 + B_{k,l} }
 \cdot
 \frac{D (k)}{-2\pi \sqrt{-1} \sigma_1 +(2 \pi |k|)^\theta +1}.
 \end{eqnarray*}
Here and in what folllows, we set 
\begin{align*}
C (l) &= - \frac{ (2\pi \sqrt{-1} l^{j'})(2\pi \sqrt{-1} l^{i'})}
 {2 \{ ( 2 \pi |l|)^\theta +1\} },
 \qquad 
D(k) = \frac{2\pi \sqrt{-1} k^j}{2 \pi |k|} {\bf 1}_{k \neq 0},
\\
B_{k,l} &= ( 2 \pi |k+l|)^\theta + ( 2 \pi |l|)^\theta +2,
\qquad
B^{\prime}_l =2\{ (2 \pi |l|)^\theta +1 \}.
\qquad 
\end{align*}
Thus, $A^{\ve}_{t,x} (3)= A^{i,j;\ve}_{t,x} (3)$ also has a good kernel and its $Q$-function is 
given by $Q^{\ve} (= Q^{i,j;\ve})$ above.

One should note here that 
\begin{eqnarray}
\sum\limits_{l \in {\bf Z^2}}
 \chi^{\ve} (l)^2 \psi_{\circ} (l, -l)
 \frac{2\pi \sqrt{-1} l^i}{2 \pi |l|} {\bf 1}_{l \neq 0}
 \frac{ C(l) }{-2\pi \sqrt{-1} \sigma_1 + B^{\prime}_l } 
 \label{eq.0612_01}
 \end{eqnarray}
is symmetric in $i$ and $i'$ for fixed $j$.
Therefore, 
if we modify the definition a little bit as 
\begin{eqnarray*}
\tilde{Q}_t^{\ve}(\sigma_1, k) 
&=& 
e^{- 2\pi \sqrt{-1}\sigma_1 t} \chi^{\ve} (k)
\sum\limits_{l \in {\bf Z^2}}
 \chi^{\ve} (l)^2 C(l)
 \\
 && \times
 \Bigl\{
 \frac{2\pi \sqrt{-1} (k^i + l^i)}{2 \pi |k+l|} {\bf 1}_{k +l \neq 0}
 \frac{ \psi_{\circ} (k+l, -l) }{-2\pi \sqrt{-1} \sigma_1 + 
 B_{k,l} } - 
 \frac{2\pi \sqrt{-1} l^i}{2 \pi |l|} {\bf 1}_{l \neq 0}
 \frac{ \psi_{\circ} (l, -l) }{-2\pi \sqrt{-1} \sigma_1 + B^{\prime}_l } 
 \Bigr\}
 \nonumber\\
 && 
 \qquad\qquad
 \cdot
 \frac{D (k)}{-2\pi \sqrt{-1} \sigma_1 +(2 \pi |k|)^\theta +1},
 \end{eqnarray*}
then we can easily see that 
$Q^{2,2;\ve} - Q^{2,1;\ve} - Q^{1,2;\ve} +Q^{1,1;\ve} 
=
\tilde{Q}^{2,2;\ve} - \tilde{Q}^{2,1;\ve} - \tilde{Q}^{1,2;\ve} +\tilde{Q}^{1,1;\ve}
$.
Hence, it is enough to estimate these modified $Q$-functions.
We define $\tilde{Q}_t (\sigma_1, k) = \tilde{Q}^{i,j}_t (\sigma_1, k)$ by formally setting $\ve =0$
(or replacing both $\chi^{\ve} (k)$ and $\chi^{\ve} (l)$ by $1$)
 in the above 
expression of $\tilde{Q}_t^{\ve} (\sigma_1, k)= \tilde{Q}^{i,j;\ve}_t (\sigma_1, k)$.

In the following we will prove that 
this $\tilde{Q}_0$ satisfies the assumption of Lemma \ref{lm.0405_2B} 
for the desired $\gamma >1$.
The Besov space-valued process associated with this $\tilde{Q}_0$
will be denoted by $\tilde{A}^{i,j} (3)$.

By the triangle inequality, we have
\begin{eqnarray}
\lefteqn{
\Bigl|
 \frac{k^i +l^i}{ |k+l|} {\bf 1}_{k +l \neq 0}
 \psi_{\circ} (k+l, -l)
 \frac{1}{ -2\pi \sqrt{-1} \sigma_1 + B_{k,l}}
 - 
 \frac{l^i}{ |l|} {\bf 1}_{l \neq 0}
 \psi_{\circ} (l, -l) 
 \frac{1}{ -2\pi \sqrt{-1} \sigma_1 + B_l^{\prime}}
 \Bigr|
}
\nonumber\\
&\le&
 \frac{|k|}{ |k+l|} {\bf 1}_{k +l \neq 0}
 \psi_{\circ} (k+l, -l)
\Bigl| \frac{1}{ -2\pi \sqrt{-1} \sigma_1 + B_{k,l}}
\Bigr|
\nonumber\\
&&
+
|l| \cdot
\Bigl|
 \frac{1}{ |k+l|} {\bf 1}_{k +l \neq 0}
 - \frac{1}{ |l|} {\bf 1}_{l \neq 0}
 \Bigr|
 \cdot
 \psi_{\circ} (k+l, -l) 
 \cdot
 \Bigl|
 \frac{1}{ -2\pi \sqrt{-1} \sigma_1 + B_{k,l}}
 \Bigr|
 \nonumber\\
&&
+
 |\psi_{\circ} (k+l, -l) -\psi_{\circ} (l, -l) | 
 \cdot
 \Bigl|
 \frac{1}{ -2\pi \sqrt{-1} \sigma_1 + B_{k,l}}
 \Bigr|
 \nonumber\\
&&
+
 \psi_{\circ} (l, -l) 
 \cdot
 \Bigl|
 \frac{1}{ -2\pi \sqrt{-1} \sigma_1 + B_{k,l}} - \frac{1}{ -2\pi \sqrt{-1} \sigma_1 + B_l^{\prime}}
 \Bigr|
 \nonumber\\
&=& 
 : J_1 +\cdots +J_4.
 \nonumber
 \end{eqnarray}

Contribution from $J_1$ is dominated as follows: 
Assume that $\theta \in (8/5,2)$. 
(The case $\theta =2$ can be done essentially in the same way.)
Take sufficiently small $a>0$ and set $\gamma := 6 - \frac{8 +2a}{\theta}$.
We use 
 $\psi_{\circ} (k+l, -l) \lesssim |k+l|_*^{ 2\theta -3-a} |k|_*^{-( 2\theta -3-a)}$
again. 
\begin{eqnarray}
\lefteqn{
\Bigl|
\sum\limits_{l \in {\bf Z^2}}
 \chi^{\ve} (l)^2
 \frac{\sqrt{-1} J_1 C(l)D(k)}{-2\pi \sqrt{-1} \sigma_1 +(2 \pi |k|)^\theta +1} 
 \Bigr|^2
 }
 \label{ineq.0612_02}
 \\
&\lesssim& 
\Bigl\{
\frac{ |k|}{ |k+l|} {\bf 1}_{k+l \neq 0}
\frac{ |k+l|_*^{ 2\theta -3-a} }{ |k|_*^{ 2\theta -3-a}}
\frac{|l|_*^{2-\theta }}{ |k+l|_*^\theta + |l|_*^\theta}
\frac{1}{ |(\sigma_1, k)|_*^{2- \gamma} }
\Bigr\}^2 
|(\sigma_1, k)|_*^{- 2\gamma}
\nonumber\\
 &\lesssim& 
\Bigl\{
\sum\limits_{l \in {\bf Z^2}}
|k+l|_*^{-2-a}
\Bigr\}^2 
|(\sigma_1, k)|_*^{- 2\gamma}
\lesssim |(\sigma_1, k)|_*^{-2\gamma}.
\nonumber
 \end{eqnarray}
Here we used that $1/ |(\sigma_1, k)|_*^{2- \gamma} \le 1/|k|_*^{-2\theta +4+a}$.

We estimate the contribution from $J_2$. 
By straightforward calculation, we have that
\begin{eqnarray*}
\Bigl|
 \frac{1}{ |k+l|} {\bf 1}_{k +l \neq 0}
 - \frac{1}{ |l|} {\bf 1}_{l \neq 0}
 \Bigr|
\le
\frac{\bigl| |l| -|k+l| \bigr|}
{|k+l|_* |l|_* }
\le
\frac{|k|}{|k+l|_* |l|_* }
\end{eqnarray*}
and therefore for any $\kappa \in (0,1)$
\begin{eqnarray*}
\Bigl|
 \frac{1}{ |k+l|} {\bf 1}_{k +l \neq 0}
 - \frac{1}{ |l|} {\bf 1}_{l \neq 0}
 \Bigr|
 \psi_{\circ} (k+l, -l)
 \lesssim
 \frac{|k|^{\kappa}}{|k+ l|_*^{\kappa} |l|_*}.
\end{eqnarray*}
Taking $1 -\kappa = 2\theta -3 -a$, we see that 
\begin{eqnarray}%
\lefteqn{
\Bigl|
\sum\limits_{l \in {\bf Z^2}}
 \frac{\sqrt{-1} J_2 C(l)D(k)}{-2\pi \sqrt{-1} \sigma_1 
 +(2 \pi |k|)^\theta +1} 
 \Bigr|^2
 }
 \label{ineq.0612_03}
 \\
&\lesssim& 
\Bigl\{
\sum\limits_{l \in {\bf Z^2}} 
 |l| \frac{|k|^{-2\theta +4+a}}{|k+ l|_*^{-2\theta +4+a} |l|_*}.
\frac{ |l|_*^{2- \theta} }{ |k+l|_*^\theta + |l|_*^\theta }
\frac{1}{ |(\sigma_1, k)|_*^{2- \gamma} }
\Bigr\}^2
 |(\sigma_1, k)|_*^{-2\gamma}
\nonumber\\
 &\lesssim& 
\Bigl\{
\sum\limits_{l \in {\bf Z^2}}
|k+l|_*^{-2-\kappa}
\Bigr\}^2 
|(\sigma_1, k)|_*^{-2\gamma}
\lesssim |(\sigma_1, k)|_*^{-2\gamma}.
\nonumber
 \end{eqnarray}

We estimate the contribution from $J_3$. 
Due to Lemma \ref{lm.0405_3} (ii) and the boundedness of $\psi_{\circ}$,
we have for any $\lambda, \kappa \in (0,1)$ that
\begin{eqnarray*}
 |\psi_{\circ} (k+l, -l) -\psi_{\circ} (l, -l) | 
 &\lesssim&
 |\psi_{\circ} (k+l, -l) -\psi_{\circ} (l, -l) |^{\kappa}
 \nonumber\\
 &\lesssim&
 |(k+l) -l |^{\kappa} 
 |l|_*^{ (-1+\lambda) \kappa}\\
& \le& |k |^{\kappa} 
 |l|_*^{ (-1+\lambda) \kappa}.
 \end{eqnarray*}

We take $1 -\kappa = 2\theta -3 -a$ for sufficiently small $a>0$
and then we choose $\lambda$ so small that 
$(1-\lambda)( -2\theta +4+a)+2\theta -2 >2$ holds. 
Then, we have
\begin{eqnarray}%\label{eq.0606_04}
\lefteqn{
\Bigl|
\sum\limits_{l \in {\bf Z^2}}
% \chi^{\ve} (l)^2
 \frac{\sqrt{-1} J_3 C(l)D(k)}{-2\pi \sqrt{-1} \sigma_1 + (2 \pi |k|)^\theta +1} 
 \Bigr|^2
 }
 \label{ineq.0612_04}
 \\
&\lesssim& 
\Bigl\{
\sum\limits_{l \in {\bf Z^2}} 
 \frac{|k|^{\kappa}}{| l|_*^{(1 -\lambda)\kappa} }.
\frac{|l|_*^{2- \theta} }{ |k+l|_*^\theta + |l|_*^\theta}
\frac{1}{ |(\sigma_1, k)|_*^{2- \gamma} }
\Bigr\}^2 
|(\sigma_1, k)|_*^{- 2\gamma}
\nonumber\\
 &\lesssim& 
\Bigl\{
\sum\limits_{l \in {\bf Z^2}}
|l|_*^{- \{ (1-\lambda)( -2\theta +4+a)+2\theta -2\}}
\Bigr\}^2 
|(\sigma_1, k)|_*^{-2\gamma}
\lesssim |(\sigma_1, k)|_*^{-2\gamma}.
\nonumber
 \end{eqnarray}

We estimate the contribution from $J_4$. 
For any $\kappa \in (0,1)$,
\begin{eqnarray}\label{ineq.0613_B}
 \Bigl|
 \frac{1}{ -2\pi \sqrt{-1} \sigma_1 + B_{k,l}} - \frac{1}{ -2\pi \sqrt{-1} \sigma_1 + B_l^{\prime}}
 \Bigr|
&\le&
\frac{ | B_{k,l} - B_l^{\prime} |^{\kappa} ( B_{k,l} + B_l^{\prime})^{1- \kappa } }
{B_{k,l} B_l^{\prime}}
\\
&\lesssim&
\frac{ \bigl| |k+l|^\theta - |l|^\theta \bigr|^{\kappa} 
 (|k+l|_*^\theta+ |l|_*^\theta)^{ 1- \kappa}}
{|l|_*^\theta (|k+l|_*^\theta + |l|_*^\theta)}
 \nonumber
 \\
&\lesssim&
\frac{ \{ |k|^{\theta/2} |l|^{\theta/2}+ |k|^\theta \}^{\kappa} }
{|l|_*^\theta (|k+l|_*^\theta+ |l|_*^\theta)^{ \kappa}},
 \nonumber
 \end{eqnarray}
where we used that 
$(r + \tilde{r})^{\theta/2} - r^{\theta/2} \le \tilde{r}^{\theta/2}$
for all $r, \tilde{r} \ge 0$
since $0<\theta \le 2$.

We set $\kappa \theta /2 := -2\theta +4+a$
for sufficiently small $a>0$.
Since $\theta \in (8/5,2)$,
such a $\kappa \in (0,1)$ exists if $a$ is small enough.
In the same way as above,
\begin{eqnarray}\label{ineq.0612_05}
\lefteqn{
\Bigl|
\sum\limits_{l \in {\bf Z^2}}
 \frac{\sqrt{-1} J_4 C(l)D(k)}{-2\pi \sqrt{-1} \sigma_1 + (2\pi |k|)^\theta +1} 
 \Bigr|^2
 }
 \\
&\lesssim& 
\Bigl\{
\sum\limits_{l \in {\bf Z^2}} 
|l|_*^{2-\theta}
\frac{ |k|^{\kappa\theta/2} \{ |l|^{\kappa\theta/2}+ |k|^{\kappa\theta/2} \}}
{|l|_*^\theta (|k+l|_*^\theta+ |l|_*^\theta)^{ \kappa}}
\frac{1}{|(\sigma_1, k)|_*^{2- \gamma}}
\Bigr\}^2 
|(\sigma_1, k)|_*^{-2 \gamma}
 \nonumber \\
 &\lesssim& 
\Bigl\{
 \sum\limits_{l \in {\bf Z^2}} 
\frac{1}{
|l|_*^{2\theta-2}(|k+l|_*^\theta+ |l|_*^\theta)^{ \kappa/2}
} 
\Bigr\}^2 
|(\sigma_1, k)|_*^{-2 \gamma} \lesssim
|(\sigma_1, k)|_*^{-2\gamma}.
\nonumber
 \end{eqnarray}
Here, we used the fact that 
$(2\theta -2) + \kappa\theta/2 = 2+a >2$.

We plan to use Lemma \ref{lm.0405_2} with $\gamma =6 - (8+2a)/\theta$ 
 and $\delta=0$.
Combining \eqref{eq.0612_01}--\eqref{ineq.0612_05},
we show that $| \tilde{Q}_0 (\sigma_1, k) |^2\vee | \tilde{Q}_0^{\ve} (\sigma_1, k) |^2
 \lesssim
|(\sigma_1, k)|_*^{-2\gamma}$
for arbitrarily small $a>0$. 

In a similar way, 
we can also estimate $|\tilde{Q}_0 (\sigma_1, k) - \tilde{Q}^{\ve}_0 (\sigma_1, k) |^2$
using that 
$|1 - \chi^{\ve} (k)| \lesssim\ve^b |k|^b$ for every sufficiently small $b >0$. 

Finally, by setting 
\begin{eqnarray*}
\Pi_1 (
 Z^{2,2} - Z^{2,1} - Z^{1,2} + Z^{1,1}) 
 &=&
\{ A^{2,2}(2) - A^{2,1}(2) - A^{1,2} (2) + A^{1,1} (2) \}
\\
&& \quad +\{ \tilde{A}^{2,2}(3) - \tilde{A}^{2,1}(3) - \tilde{A}^{1,2} (3) + \tilde{A}^{1,1} (3) \},
 \end{eqnarray*}
 we prove the lemma.
\end{proof}

Summing up, we have shown the following lemma:
Set 
$$
Z := \sum\limits_{k=1,3} \Pi_k (
 Z^{2,2} - Z^{2,1} - Z^{1,2} + Z^{1,1}).
$$
For the definition of 
$Z^{\ve} :=R^{\perp} Y^{\ve}\reso \nabla X^{\ve}$,
see \eqref{def.0612}.

\begin{lemma}\label{lm.0612_a}
Let $\theta \in (8/5, 2]$. 
Then, for every $\alpha < \frac52 \theta -5$ and $1 < p <\infty$, 
\[
{\mathbf E}
 [
 \|Z \|_{C_T {\mathcal C}^{\alpha}}^p 
 ]
 <\infty,
\qquad
\lim_{\ve \searrow 0}
{\mathbf E}
 [
 \|
Z- Z^{\ve}
 \|_{C_T {\mathcal C}^{\alpha}}^p 
 ]
=0.
\]
\end{lemma}

%%%%%%%%%%%%%%%%%

%\newpage
\subsection{Convergence of $\hat{Z} = \nabla Y \cdot R^{\perp} X$}

In this subsection we prove that 
\begin{eqnarray}\label{def.0612_2}
\hat{Z}^{\ve} := \nabla Y^{\ve} \cdot R^{\perp} X^{\ve}
&=&
(\partial_1 Y^{\ve}) \cdot (R_2 X^{\ve})- (\partial_2 Y^{\ve}) \cdot (R_1 X^{\ve})
\\
&=&
\partial_1 {\mathcal I}[R_2 X^{\ve} \cdot \partial_1 X^{\ve} 
 - R_1 X^{\ve} \cdot \partial_2 X^{\ve}]
\cdot (R_2 X^{\ve})
\nonumber\\
&& \qquad
-
\partial_2 {\mathcal I}[R_2 X^{\ve} \cdot \partial_1 X^{\ve} 
 - R_1 X^{\ve} \cdot \partial_2 X^{\ve}]
\cdot (R_1 X^{\ve})
\nonumber
\end{eqnarray}
is convergent %in $B^s_{\infty\infty}({\bf T}^d)$ 
as $\ve \searrow 0$ if the Besov regularity %$s$ 
is smaller than $\frac52 \theta - 5$.
We do not need any renormalization.
The proof is quite similar to, but somewhat simpler than the one for $Z$ in the previous subsection.

Let $i, j \in \{1,2\}$.
By straightforward computation, we have
\begin{eqnarray*}
\partial_i {\mathcal I} [R_j X^{\ve} \cdot \partial_{j'} X^{\ve} ]_{t,x}
&=&
\sum\limits_{k, l \in {\bf Z^2}}
 \chi^{\ve} (k) \chi^{\ve} (l) {\bf e}_{k+l} (x) 2\pi \sqrt{-1} (k^i + l^i)
 \int_{- \infty}^{\infty} du h(t-u, k+l) \\
 &&
 \times 
 \int_{- \infty}^{\infty} H_u^{ R_j X} (s_1,k) 
\hat{\xi}_{s_1} (k) ds_1
 \int_{- \infty}^{\infty} H_u^{ \partial_{j'} X} (s_2,l) 
\hat{\xi}_{s_2} (l) ds_2
 \end{eqnarray*}
and
\[
R_{i'} X^{\ve}_{t,x}
=
\sum\limits_{m\in {\bf Z^2}}
\chi^{\ve} (m) {\bf e}_{m} (x) 
 \int_{- \infty}^{\infty} H_t^{ R_{i'} X} (s_3, m) 
\hat{\xi}_{s_3} (m) ds_3.
\]
We set $\hat{Z}^{i,j;\ve}_{t,x} :=
\partial_i {\mathcal I} [R_j X^{\ve} \cdot \partial_{j'} X^{\ve} ]_{t,x} 
\cdot R_{i'} X^{\ve}_{t,x}$.
Then, we have
\begin{eqnarray}\label{eq.0613_1}
\hat{Z}^{i,j;\ve}_{t,x} 
&=&
\sum\limits_{k, l, m \in {\bf Z^2}}
 \chi^{\ve} (k) \chi^{\ve} (l) \chi^{\ve} (m) {\bf e}_{k+l+m} (x) 
 2\pi \sqrt{-1} (k^i + l^i)
 \\
 &&
 \times
 \int_{- \infty}^{\infty} du h(t-u, k+l) 
 \int_{- \infty}^{\infty} H_u^{ R_j X} (s_1,k) 
\hat{\xi}_{s_1} (k) ds_1
 \nonumber\\
 & &\times
 \int_{- \infty}^{\infty} H_u^{ \partial_{j'} X} (s_2,l) 
\hat{\xi}_{s_2} (l) ds_2 
 \int_{- \infty}^{\infty} H_t^{ R_{i'} X} (s_3, m) 
\hat{\xi}_{s_3} (m) ds_3
\nonumber\\
 &=& \Pi_3 ( \hat{Z}^{i,j;\ve}_{t,x}) + \Pi_1 ( \hat{Z}^{i,j;\ve}_{t,x}).
 \nonumber\
 \end{eqnarray}

Let us calculate the third order term 
on the right-hand side of (\ref{eq.0613_1}).
It is easy to see that
\[
\Pi_3 ( \hat{Z}^{i,j;\ve}_{t,x}) 
=
{\mathcal J}_3
\Bigl(
f_{(t,x)}^{\ve}
\Bigr),
\]
where
\begin{eqnarray}\nonumber
\lefteqn{
f_{(t,x)}^{\ve} ((s_1,y_1),(s_2, y_2),(s_3, y_3))
}\\
\nonumber
&=&
\sum\limits_{k, l, m \in {\bf Z}^2}
{\bf e}_{k} (x- y_1) {\bf e}_{l} (x- y_2) {\bf e}_{m} (x- y_3)
\chi^{\ve} (k) \chi^{\ve} (l) \chi^{\ve} (m)
H_t ((s_1, k), (s_2, l), (s_3, m))
\end{eqnarray}
with
\begin{eqnarray}\nonumber
\lefteqn{
H_t ((s_1, k), (s_2, l), (s_3,m))
}
\nonumber\\
&=&
2\pi \sqrt{-1} (k^i + l^i)
 \int_{- \infty}^{\infty}
 du h(t-u, k+l) 
 H_u^{ R_j X} (s_1,k) H_u^{ \partial_{j'} X} (s_2,l) H_t^{ R_{i'} X} (s_3, m). \nonumber
 \end{eqnarray}

By Fourier transform with respect to the time variables, we have
\begin{eqnarray}\label{eq.0613_2}
Q_t ((\sigma_1, k), (\sigma_2, l), (\sigma_3, m))
&=&
e^{- 2\pi \sqrt{-1} (\sigma_1 +\sigma_2+ \sigma_3) t}
2\pi \sqrt{-1} (k^i + l^i)
 \\
 && \times
 \frac{ Q_0^{ R_j X} (\sigma_1, k) Q_0^{ \partial_{j'} X} (\sigma_2, l) }
{- 2\pi \sqrt{-1} (\sigma_1 + \sigma_2) + (2\pi |k+l|)^\theta +1) }
Q_0^{ R_{i'} X} (\sigma_3, m).
\nonumber
\end{eqnarray}
Therefore, 
$ \Pi_3 ( \hat{Z}^{i,j;\ve}_{t,x})$ has a good kernel
whose $Q$-function is given by
\[
Q_t^{\ve} ((\sigma_1, k), (\sigma_2, l), (\sigma_3, m))
=\chi^{\ve} (k)\chi^{\ve} (l)\chi^{\ve} (m)
Q_t ((\sigma_1, k), (\sigma_2, l), (\sigma_3, m)).
\]

As we will see below, 
$Q_0$ above satisfies 
the assumption of Lemma \ref{lm.0405_2B} 
and hence it defines a Besov space-valued process. 
We denote it by $\Pi_3 ( \hat{Z}^{i,j})$.

%\vspace{7mm}

\begin{lemma}\label{lm.0613_a}
Let the notation be as above and $\theta \in (8/5, 2]$. 
Then, for every $\alpha < \frac52 \theta -5$, $1 < p <\infty$ and $i,j =1, 2$, 
\begin{align*}
{\mathbf E}
 [
 \|\Pi_3 ( \hat{Z}^{i,j}) \|_{C_T {\mathcal C}^{ \alpha } }^p 
 ]<\infty
,\qquad 
\lim_{\ve \searrow 0}
{\mathbf E}
 [
 \|
 \Pi_3 ( \hat{Z}^{i,j;\ve})
 -
\Pi_3 ( \hat{Z}^{i,j}) 
 \|_{C_T {\mathcal C}^{ \alpha } }^p 
 ]
=0.
\end{align*}
\end{lemma}

%\vspace{7mm}

%Next, we calculate the first order term 
%on the right-hand side of (\ref{eq.0613_1}).

\begin{proof}
We let
\[
\gamma=6-\frac{8}{\theta}, \delta=0
\]
to use Lemma \ref{lm.0405_2}.
Note that $\gamma>1$ if and only if $\theta > 8/5$.
We prove the case $\theta \in (8/5, 2)$ only.
It is easy to see from \eqref{eq.0613_2} that 
\begin{eqnarray*}
|Q_0((\sigma_1, k), (\sigma_2, l), (\sigma_3, m)) |^2 
&\lesssim&
\frac{ |k+l|^2}{ |(\sigma_1 + \sigma_2, k+l)|_*^4}
\frac{1}{|(\sigma_1 , k)|_*^4}
\frac{|l|^2}{|(\sigma_2 , l)|_*^4}\frac{1}{|(\sigma_3 , m)|_*^4}
\nonumber\\
&\lesssim&
\frac{ 1}{ |(\sigma_1 + \sigma_2, k+l)|_*^{4(1 -1/\theta)} }
\frac{1}{|(\sigma_1 , k)|_*^4}
\frac{1}{|(\sigma_2 , l)|_*^{4(1 -1/\theta)}}
\frac{1}{|(\sigma_3 , m)|_*^4}.
\end{eqnarray*}
By using Lemma \ref{lm.0405_1} twice in the same way as in
the proof of Lemma \ref{lm.0605_a}, we have
\[
\int_{\hat{E}^2}
 |Q_0 ( (\sigma_1, k), (\sigma_2, l), (\tau-\sigma_1-\sigma_2, n-k-l)) |^2
 \,
 d(\sigma_1, k) d(\sigma_2, l)\lesssim 
 |(\tau , n)|_*^{- 2\gamma}.
\]
Using Lemma \ref{lm.0405_2},
we prove the lemma.
\end{proof}

%\vspace{7mm}

Next, we calculate the first order term 
on the right-hand side of (\ref{eq.0613_1}).

\begin{lemma}\label{lm.0613_b}
Let the notation be as above and $\theta \in (8/5, 2]$. 
Then, for every $\alpha < \frac52 \theta -5$ and $1 < p <\infty$, 
there exists 
$\Pi_1 (
 \hat{Z}^{2,2} - \hat{Z}^{2,1} - \hat{Z}^{1,2} + \hat{Z}^{1,1}) \in L^p (\Omega, {\mathbf P} ;C_T {\mathcal C}^{ \alpha}) 
 $
 which is independent of $\alpha, p, \chi$
 such that
\begin{align*}
\lim_{\ve \searrow 0}
{\mathbf E}
 [
 \|
 \Pi_1 ( \hat{Z}^{2,2;\ve} - \hat{Z}^{2,1;\ve} - \hat{Z}^{1,2;\ve} + \hat{Z}^{1,1;\ve})
 -
\Pi_1 (
 \hat{Z}^{2,2} - \hat{Z}^{2,1} - \hat{Z}^{1,2} + \hat{Z}^{1,1}
) 
 \|_{C_T {\mathcal C}^{ \alpha} }^p 
 ]
=0.
\end{align*}
\end{lemma}

\begin{proof}
First we calculate $\Pi_1 ( Z^{i,j;\ve}_{t,x})$.
Applying the contraction rule \eqref{eq.0613_z} to \eqref{eq.0613_1}, 
we have 
\begin{eqnarray}
\Pi_1 ( Z^{i,j;\ve}_{t,x})
=:
\hat{A}^{ i,j;\ve}_{t,x} (1)+ \hat{A}^{i,j;\ve}_{t,x} (2)+ \hat{A}^{i,j;\ve}_{t,x} (3),
\nonumber
\end{eqnarray}
where $\hat{A}^{\ve}_{t,x} (\nu)$, $\nu=2,3$, are given as follows:
(the contribution from
$\hat{A}^{i,j;\ve}_{t,x} (1)$ cancels out 
for the same reason as in the proof of Lemma \ref{lm.0605_b}.
So we do not write it down.)
\begin{eqnarray}
\hat{A}^{i,j;\ve}_{t,x} (2)
&=&
\sum\limits_{k, l \in {\bf Z^2}}
 \chi^{\ve} (k)^2 \chi^{\ve} (l) {\bf e}_{l} (x) 
 2\pi \sqrt{-1} (k^i + l^i)
 \label{eq.0613_3}\\
 &&
 \times
 \int_{- \infty}^{\infty} du h(t-u, k+l) 
 \int_{- \infty}^{\infty} H_u^{ R_j X} (s_1,k) H_t^{ R_{i'} X} (s_1, -k) ds_1
 \nonumber\\
 & &\times
 \int_{- \infty}^{\infty} H_u^{ \partial_{j'} X} (s_2,l) 
\hat{\xi}_{s_2} (l) ds_2
 \nonumber\\
 &=&
 \sum\limits_{l \in {\bf Z^2}}\chi^{\ve} (l) {\bf e}_{l} (x)
 \Bigl[
 \sum\limits_{k \in {\bf Z^2}}
 \chi^{\ve} (k)^2 2\pi \sqrt{-1} (k^i + l^i)
 \nonumber\\
 && 
 \quad \times 
 \frac{2\pi \sqrt{-1} k^j}{2 \pi |k|} {\bf 1}_{k \neq 0} \cdot 
 \frac{-2\pi \sqrt{-1} k^{i'}}{2 \pi |k|} {\bf 1}_{k \neq 0} \cdot 
 \frac{1}{2 \{ ( 2 \pi |k|)^\theta +1\}}
 \nonumber\\
 && 
 \quad \times 
 \int_{- \infty}^{t} du e^{- (t-u) \{ ( 2 \pi |k+l|)^\theta
 + ( 2 \pi |k|)^\theta +2\}
 }
 \int_{- \infty}^{\infty} H_u^{ \partial_{j'} X} (s_2,l) 
\hat{\xi}_{s_2} (l) ds_2 \Bigr]
 \nonumber
\end{eqnarray}
and 
\begin{eqnarray}
\hat{A}^{i,j;\ve}_{t,x} (3)
&=&
\sum\limits_{k, l \in {\bf Z^2}}
 \chi^{\ve} (k) \chi^{\ve} (l)^2 {\bf e}_{k} (x)2\pi \sqrt{-1} (k^i + l^i)
 \label{eq.0613_4}\\
 &&
 \times
 \int_{- \infty}^{\infty} du h(t-u, k+l) 
 \int_{- \infty}^{\infty} H_u^{ R_j X} (s_1,k) 
\hat{\xi}_{s_1} (k) ds_1
 \nonumber\\
 & &\times
 \int_{- \infty}^{\infty} H_u^{ \partial_{j'} X} (s_2,l) 
 H_t^{ R_{i'} X} (s_2, -l) ds_2
 \nonumber\\
 &=& 
 \sum\limits_{k \in {\bf Z^2}}\chi^{\ve} (k) {\bf e}_{k} (x)
 \Bigl[
 \sum\limits_{l \in {\bf Z^2}}
 \chi^{\ve} (l)^2 2\pi \sqrt{-1} (k^i + l^i)
 \frac{ (2\pi \sqrt{-1} l^{j'})(2\pi \sqrt{-1} l^{i'})}{ 2\pi |l| \cdot 2 \{ ( 2 \pi |l|)^\theta +1\} } 
 {\bf 1}_{l \neq 0} 
 \nonumber\\
 && 
 \quad \times 
 \int_{- \infty}^{t} du e^{- (t-u) \{ ( 2 \pi |k+l|)^\theta +( 2 \pi |l|)^\theta +2\}
  }
 \int_{- \infty}^{\infty} H_u^{ R_{j} X} (s_1,k) 
\hat{\xi}_{s_1} (k) ds_1 \Bigr].
\nonumber
\end{eqnarray}

Now we calculate the $H$-function and the $Q$-function
associated with $\hat{A}^{ i,j;\ve}_{t,x} (2)$.
Let $(s_2,l) \in E$,
From \eqref{eq.0613_3}, we see that
\begin{eqnarray*}
H_t^{\ve} (s_2, l) 
&=&
 \chi^{\ve} (l) \sum\limits_{k \in {\bf Z^2}}
 \chi^{\ve} (k)^2 2\pi \sqrt{-1} (k^i + l^i)
 \frac{2\pi \sqrt{-1} k^j}{2 \pi |k|} {\bf 1}_{k \neq 0} \cdot 
 \frac{-2\pi \sqrt{-1} k^{i'}}{2 \pi |k|} {\bf 1}_{k \neq 0} \cdot 
 \frac{1}{2 \{ ( 2 \pi |k|)^\theta +1 \}}
 \nonumber\\
 && 
 \quad \times 
 \int_{- \infty}^{t} du e^{- (t-u) \{ ( 2 \pi |k+l|)^\theta +( 2 \pi |k|)^\theta +2\}
 }
 H_u^{ \partial_{j'} X} (s_2,l) 
\end{eqnarray*}
and 
\begin{eqnarray*}
\lefteqn{
Q_t^{\ve}(\sigma_2, l) 
}
\\
&=& 
e^{- 2\pi \sqrt{-1}\sigma_2 t}
 \chi^{\ve} (l) \sum\limits_{k \in {\bf Z^2}}
 \chi^{\ve} (k)^2 2\pi \sqrt{-1} (k^i + l^i)
 \frac{2\pi \sqrt{-1} k^j}{2 \pi |k|} {\bf 1}_{k \neq 0} \cdot 
 \frac{-2\pi \sqrt{-1} k^{i'}}{2 \pi |k|} {\bf 1}_{k \neq 0} \cdot 
 \nonumber\\
 && 
 \times 
 \frac{1}{2 \{ ( 2 \pi |k|)^\theta +1)}
 \cdot
 \frac{1}{-2\pi \sqrt{-1} \sigma_2 + ( 2 \pi |k+l|)^\theta +( 2 \pi |k|)^\theta +2 }
 \cdot
 \frac{2\pi \sqrt{-1} l^{j'}}{-2\pi \sqrt{-1} \sigma_2 + (2 \pi |l|)^\theta +1}.
 \end{eqnarray*}
This shows $\hat{A}^{\ve}_{t,x} (2)$ has a good kernel 
whose $Q$-function is given by $Q^{\ve}_t$.
We define $Q_t (\sigma_2, l) $ by formally setting $\ve =0$
(or replacing both $\chi^{\ve} (k)$ and $\chi^{\ve} (l)$ by $1$)
 in the above 
expression of $Q_t^{\ve} (\sigma_2, l)$.

We proceed in a similar way to the proof of Lemma \ref{lm.0605_b}.
Take $a >0$ sufficiently small and set $\gamma = 6 - (8 +2a)/\theta$.
Since 
$
|l| \le |l|^{ -2\theta +4 +a} (|k+l| +|k|)^{2\theta -3 -a}
$,
we can dominate $|Q_0 (\sigma_2, l)|^2$ as follows:
\begin{eqnarray}\label{eq.0613_5}
|Q_0 (\sigma_2, l) |^2
&\lesssim& 
\Bigl\{
\sum\limits_{k \in {\bf Z^2}}
|k|_*^{-\theta}
\frac{|k+l|}{ |k+l|_*^{\theta} + |k|_*^{\theta}}
\frac{|l|}{|(\sigma_2, l)|_*^{2-\gamma}}
\Bigr\}^2
|(\sigma_2, l)|_*^{-2\gamma}
\\
 &\lesssim& 
\Bigl\{
\sum\limits_{k \in {\bf Z^2}}
|k|_*^{-\theta}
\frac{|k+l| (|k+l| + |k|)^{ 2\theta -3 -a} }{ |k+l|_*^{\theta} + |k|_*^{\theta}}
\Bigr\}^2 |(\sigma_2, l)|_*^{-2\gamma}
 \nonumber\\
 &\lesssim& 
 \Bigl\{
\sum\limits_{k \in {\bf Z^2}}
|k|_*^{-(2 + a)}
\Bigr\}^2 
|(\sigma_2, l)|_*^{- 2\gamma} 
\lesssim |(\sigma_2, l)|_*^{- 2\gamma}.
 \nonumber
 \nonumber
 \end{eqnarray}
 This estimate \eqref{eq.0613_5}
 is essentially the same as \eqref{eq.0606_04}.
Hence, we can apply Lemma \ref{lm.0405_2} to 
$\hat{A}^{\ve} (2)= \hat{A}^{i,j;\ve} (2)$
with $\gamma = 6 - (8 +2a)/\theta$ and $\delta =0$. 

The Besov space-valued process associated with this $Q_0$
will be denoted by $\hat{A}^{i,j} (2)$.

Next, we will calculate the $H$-function and the $Q$-function
associated with
the function $\hat{A}^{\ve}_{t,x} (3)$.
For
any $(s_1,k)$, 
$(\sigma_1,k) \in E$,
\begin{eqnarray*}
H_t^{\ve} (s_1, k) 
&=&
\chi^{\ve} (k) \sum\limits_{l \in {\bf Z^2}}
 \chi^{\ve} (l)^2 2\pi \sqrt{-1} (k^i + l^i)
 \frac{ (2\pi \sqrt{-1} l^{j'})(2\pi \sqrt{-1} l^{i'})}
 { 2\pi |l| \cdot 2 \{ ( 2 \pi |l|)^\theta +1\} } 
 {\bf 1}_{l \neq 0} 
 \nonumber\\
 && 
 \quad \times 
 \int_{- \infty}^{t} du e^{- (t-u) \{ (2 \pi |k+l|)^\theta + (2 \pi |l|)^\theta +2 \}}
 H_u^{ R_{j} X} (s_1,k) 
\end{eqnarray*}
and 
\begin{eqnarray*}
%\lefteqn{
Q_t^{\ve} (\sigma_1, k) 
%}\\
&=& 
e^{- 2\pi \sqrt{-1}\sigma_1 t}\chi^{\ve} (k)
 \sum\limits_{l \in {\bf Z^2}}
 \chi^{\ve} (l)^2 2\pi \sqrt{-1} (k^i + l^i)
 \frac{ (2\pi \sqrt{-1} l^{j'})(2\pi \sqrt{-1} l^{i'})}{ 2\pi |l| 
 %\cdot 2 \{ (2 \pi |l|)^\theta +1\}
 B_l^{\prime}
 } 
 {\bf 1}_{l \neq 0} 
 \nonumber\\
 && 
 \quad \times 
 \frac{1}{-2\pi \sqrt{-1} \sigma_1
 + 
 %(2 \pi |k+l|)^\theta + (2 \pi |l|)^\theta +2 
 B_{k,l}
 }
 \cdot
 \frac{2\pi \sqrt{-1} k^{j'} }{2\pi |k|} {\bf 1}_{k \neq 0}
 \frac{1}{-2\pi \sqrt{-1} \sigma_1 +(2 \pi |k|)^\theta +1}.
 \end{eqnarray*}
As in the proof of Lemma \ref{lm.0605_b}, 
we modify the $Q^{\ve}
=Q^{i,j;\ve}$ 
a little bit as follows:
\begin{eqnarray*}
\tilde{Q}_t^{\ve} (\sigma_1, k) 
=
\tilde{Q}_t^{i,j;\ve} (\sigma_1, k) 
&=& 
e^{- 2\pi \sqrt{-1}\sigma_1 t}\chi^{\ve} (k)
 \sum\limits_{l \in {\bf Z^2}}
 \chi^{\ve} (l)^2 
 \frac{ (2\pi \sqrt{-1} l^{j'})(2\pi \sqrt{-1} l^{i'})}{ 2\pi |l| 
 B_l^{\prime}
 } 
 {\bf 1}_{l \neq 0} 
 \nonumber\\
 && 
 \quad \times 
 \Bigl\{
 \frac{ 2\pi \sqrt{-1} (k^i + l^i) }{-2\pi \sqrt{-1} \sigma_1
 + 
 B_{k,l}
 }
 -
 \frac{ 2\pi \sqrt{-1} l^i }{-2\pi \sqrt{-1} \sigma_1
 + 
 B_{l}^{\prime}
 }
 \Bigr\}
 \nonumber\\
 &&
 \quad
 \times
 \frac{2\pi \sqrt{-1} k^{j'} }{2\pi |k|} {\bf 1}_{k \neq 0}
 \frac{1}{-2\pi \sqrt{-1} \sigma_1 +(2 \pi |k|)^\theta +1}.
 \end{eqnarray*}
We have again
$Q^{2,2;\ve} - Q^{2,1;\ve} - Q^{1,2;\ve} +Q^{1,1;\ve} 
=
\tilde{Q}^{2,2;\ve} - \tilde{Q}^{2,1;\ve} - \tilde{Q}^{1,2;\ve} +\tilde{Q}^{1,1;\ve}
$.
We define $\tilde{Q}_t (\sigma_1, k) = \tilde{Q}^{i,j}_t (\sigma_1, k)$ by formally setting $\ve =0$
(or replacing both $\chi^{\ve} (k)$ and $\chi^{\ve} (l)$ by $1$)
 in the above 
expression of $\tilde{Q}_t^{\ve} (\sigma_1, k)= \tilde{Q}^{i,j;\ve}_t (\sigma_1, k)$.

By
 the same argument as 
in the proof of Lemma \ref{lm.0605_b}, 
we can show 
that 
$$
| \tilde{Q}_0 (\sigma_1, k) |^2\vee | \tilde{Q}_0^{\ve} (\sigma_1, k) |^2\lesssim
|(\sigma_1, k)|_*^{- 2 \{6 - (8 +2a)/\theta\}}
$$
for every $a >0$ small enough.
%where the implicit constant may depend on $a$. 
(This time it is actually less cumbersome
since $\psi_{\circ}$ does not appear.)
%We use the inequalities \eqref{ineq.0613_B} and
%$|k| \le |k|^{-2\theta +4+a} (|k+l| +|l|)^{2\theta -3 -a}$.)
By Lemma \ref{lm.0405_2}, 
there exists a Besov space-valued process $\hat{A}^{\sim, i,j}(3)$
which corresponds to $\tilde{Q}_0^{i,j}$. 
We can also show in the same way as before that, 
 for every $a >0$ small enough, there exists $b>0$ such that
$| \tilde{Q}_0 (\sigma_1, k) - \tilde{Q}_0^{\ve} (\sigma_1, k) |^2
\lesssim
\ve^b |(\sigma_1, k)|_*^{- 2 \{6 - (8 +2a)/\theta\}}$.
 By Lemma \ref{lm.0405_2}, this implies that
 the process corresponding to $\tilde{Q}_0^{i,j;\ve}$
 converges to $\hat{A}^{\sim, i,j}(3)$.
 
By setting 
\begin{eqnarray*}
\Pi_1 (
 \hat{Z}^{2,2} - \hat{Z}^{2,1} - \hat{Z}^{1,2} + \hat{Z}^{1,1}) 
 &=&
\{ \hat{A}^{2,2}(2) - \hat{A}^{2,1}(2) - 
 \hat{A}^{1,2} (2) + \hat{A}^{1,1} (2) \}
\\
&& \quad +\{ \hat{A}^{\sim, 2,2}(3) - \hat{A}^{\sim,2,1}(3) 
- \hat{A}^{\sim, 1,2} (3) + \hat{A}^{\sim, 1,1} (3) \},
 \end{eqnarray*}
 we finish our proof. 
\end{proof}

Summing up, we have shown the following Lemma in this subsection:
Set $$\hat{Z} = \sum\limits_{k=1,3} \Pi_k (
 \hat{Z}^{2,2} - \hat{Z}^{2,1} - \hat{Z}^{1,2} + \hat{Z}^{1,1}).
$$
For the definition of 
$\hat{Z}^{\ve} := \nabla Y^{\ve} \cdot R^{\perp} X^{\ve}$,
see \eqref{def.0612_2}.

\begin{lemma}\label{lm.0613_ab}
Let $\theta \in (8/5, 2]$. 
Then, for every $\alpha < \frac52 \theta -5$ and $1 < p <\infty$, 
\[
{\mathbf E}
 [
 \|\hat{Z} \|_{C_T {\mathcal C}^{\alpha}}^p 
 ]
 <\infty,
\qquad
\lim_{\ve \searrow 0}
{\mathbf E}
 [
 \|
\hat{Z}- \hat{Z}^{\ve}
 \|_{C_T {\mathcal C}^{\alpha}}^p 
 ]
=0.
\]
%
%
%Let the notation be as above and $\theta \in (8/5, 2]$. 
%Then, for every sufficiently small $\kappa >0$,
%\begin{align*}
%\lim_{\ve \searrow 0}
%{\mathbf E}
% [
% \|
%\hat{Z}^{\ve}
% -\hat{Z}
% \|_{C_T {\mathcal C}^{ \textcolor{red}{\frac52 \theta -5} -\kappa } }^p 
% ]
%=0
%\end{align*}
%for all $1 < p <\infty$.
%%
%For sufficiently small $\kappa >0$, 
%\[
%(t,x) \mapsto
%\{ 
%\nabla Y \cdot R^{\perp} X
%\}_{t,x}
%\]
%belongs to 
%$C_T^{\kappa} {\mathcal C}^{\alpha}$
%a.s.
%with \textcolor{red}{$\alpha =0^-$} (i=1,2).
%Here, $\nabla Y \cdot R^{\perp} X$ is defined in \eqref{def.0612_2}.
\end{lemma}

%%%%%%%%%%%%%%%%%
%%%%%%%%%%%%%%%%%%%%%%%%%%%%%%%%%%%%%%%%%%%%%%%%%%%%%%%%%%%%%%%%%%
%% References 
%\newpage
%%%%%%%%%%%%%%%%%%%%%%%%%%%%%%%%%%%%%%%%%%%%%%%%%%%%%%%%%%%%%%%%%%%

%%%%%%%%%%%%%%%%%
%%%%%%%%%%%%%%%%%%%%%%%%%%%%%%%%%%%%%%%%%%%%%%%%%%%%%%%%%%%%%%%%%%

%%%%%%%%%%%%%%%%%%%%%%%%%%%%%%%%%%%%%%%%%%%%
%\vspace{30mm}
%%%%%%%%%%%%%%%%%%%%%%%

\section{Acknowledgement}

The authors are grateful to 
Professors Ryo Takada and Masato Hoshino
at Kyushu University
for fruitful discussions on many issues 
including the histories of QGEs and paracontrolled calculus.

The first author is supported by Grant-in-Aid for Scientific Research (C) (15K04922), the
Japan Society for the Promotion of Science.
The second author is supported by Grant-in-Aid for Scientific Research (C) (16K05209), the
Japan Society for the Promotion of Science.
The second author is also supported by
Department of Mathematics Analysis and the Theory of functions,
Peoples' Friendship University of Russia, Moscow, Russia.

%%%%%%%%%%%% affiliations %%%%%%%%%%%%%

\vspace{10mm}

\begin{flushleft}
\begin{tabular}{ll}
Yuzuru \textsc{Inahama}
\\
Faculty of Mathematics,  
\\
Kyushu University,
\\
Motooka 744, Nishi-ku, Fukuoka 819-0395, JAPAN.
\\
Email: {\tt inahama@math.kyushu-u.ac.jp}
\end{tabular}
\end{flushleft}

\bigskip

\begin{flushleft}
\begin{tabular}{ll}
Yoshihiro \textsc{Sawano}
\\
Graduate School of Science and Engineering,
\\
Chuo University, 
\\
1-13-27 Kasuga, Bunkyo-Ku, Tokyo,
Japan\\
\\
+
Department of Mathematics Analysis and the Theory of functions,
\\
Peoples' Friendship University of Russia, Moscow, Russia.
\\
Email: {\tt ysawano@tmu.ac.jp}
\end{tabular}
\end{flushleft}
\end{document}